%% file: main.tex
\documentclass[10pt,leqno]{article}
\usepackage{graphicx}
\usepackage{indentfirst,csquotes}

\usepackage[english]{babel}
\usepackage{tikz}
\usepackage{multicol}
\usepackage{wrapfig}
\usepackage{quiver}
\usepackage{amscd,epsfig}
\usepackage{hyperref}

\makeatletter
\DeclareFontEncoding{LS2}{}{\@noaccents}
\makeatother
\DeclareFontSubstitution{LS2}{stix}{m}{n}

\DeclareSymbolFont{largesymbolsstix}{LS2}{stixex}{m}{n}

\DeclareMathDelimiter{\llbrace}{\mathopen}{largesymbolsstix}{"E8}{largesymbolsstix}{"0E}
\DeclareMathDelimiter{\rrbrace}{\mathclose}{largesymbolsstix}{"E9}{largesymbolsstix}{"0F}

\usepackage{stmaryrd}
\usepackage{epigraph}
\usepackage{geometry}
\usepackage{bbold}
\usepackage{tikz-cd}
\usepackage[style=alphabetic]{biblatex}
\usepackage{authblk}
\usepackage{amssymb,amsthm,amsmath}
\usepackage{xcolor,paralist,fancyhdr,etoolbox}

\title{Pre-Lie algebras up to homotopy with divided powers and homotopy of operadic mapping spaces} 
\date{\today}
\author{Marvin VERSTRAETE}

\AtEndDocument{%
  \par
  \medskip
  \begin{tabular}{@{}l@{}}%
    \textsc{Univ. Lille, CNRS, UMR 8524 - Laboratoire Paul Painlevé, F-59000 Lille, France}\\
    \textit{E-mail address}: \texttt{marvin.verstraete@univ-lille.fr}
  \end{tabular}}

\begin{document}

\newtheorem{defi}{Definition}[section]
\newtheorem{thm}[defi]{Theorem}
\newtheorem{thmA}{Theorem}
\newtheorem*{thm*}{Theorem}
\newtheorem{cor}[defi]{Corollary}
\newtheorem{remarque}[defi]{Remark}
\newtheorem{example}[defi]{Example}
\newtheorem{lm}[defi]{Lemma}
\newtheorem{prop}[defi]{Proposition}
\newtheorem{algo}[defi]{Algorithme}
\newtheorem{cons}[defi]{Construction}
\newtheorem{propdef}[defi]{Proposition-Definition}

\setcounter{tocdepth}{2}

\renewcommand*{\thethmA}{\Alph{thmA}}

\maketitle

\newcommand{\antishriek}{\mbox{\footnotesize{\rotatebox[origin=c]{180}{$!$}}}}
\newcommand{\dperm}{\Delta_{\text{\normalfont Perm}}}
\newcommand{\dgperm}{\Delta_{\Gamma\text{\normalfont Perm}}}

\newcommand{\dgMod}{$\text{dg}\widehat{\text{Mod}}_\mathbb{K}$}

\newcommand{\biduledroite}{
    $\displaystyle\left( \underline{01}+ \underline{12}+\sum_{k\geq 1}\begin{tikzpicture}[baseline={([yshift=-.5ex]current bounding box.center)},scale=1]
    \node (ix) at (0,0) {$ \underline{01}$};
    \node (1x) at (-1,1) {$ \underline{12}$};
    \node (2x) at (0,1) {$\overset{k}{\cdots}$};
    \node (3x) at (1,1) {$\underline{12}$};
    \draw (ix) -- (1x);
    \draw (ix) -- (3x);
    \end{tikzpicture}\right)$}

\newcommand{\biduledroitebis}{
    $\displaystyle\left(\underline{01}+ \underline{12}+\sum\begin{tikzpicture}[baseline={([yshift=-.5ex]current bounding box.center)},scale=1]
    \node (ix) at (0,0) {$ \underline{01}$};
    \node (1x) at (-1,1) {$ \underline{12}$};
    \node (2x) at (0,1) {${\cdots}$};
    \node (3x) at (1,1) {$\underline{12}$};
    \draw (ix) -- (1x);
    \draw (ix) -- (3x);
    \end{tikzpicture}\right)$}

\DeclareRobustCommand\longtwoheadrightarrow
     {\relbar\joinrel\twoheadrightarrow}

\let\thefootnote\relax

\begin{abstract}
    \noindent The purpose of this memoir is to study pre-Lie algebras up to homotopy with divided powers, and to use this algebraic structure for the study of mapping spaces in the category of operads. We define a new notion of algebra called $\Gamma\Lambda\mathcal{PL}_\infty$-algebra which characterizes the notion of $\Gamma(\mathcal{P}re\mathcal{L}ie_\infty,-)$-algebra. We also define a notion of a Maurer-Cartan element in complete $\Gamma\Lambda\mathcal{PL}_\infty$-algebras which generalizes the classical definition in Lie algebras. We prove that for every complete brace algebra $A$, and for every $n\geq 0$, the tensor product $ A\otimes\Sigma N^*(\Delta^n)$ is endowed with the structure of a complete $\Gamma\Lambda\mathcal{PL}_\infty$-algebra, and define the simplicial Maurer-Cartan set $\mathcal{MC}_\bullet(A)$ associated to $A$ as the Maurer-Cartan set of $ A\otimes\Sigma N^*(\Delta^\bullet)$. We compute the homotopy groups of this simplicial set, and prove that the functor $\mathcal{MC}_\bullet(-)$ satisfies a homotopy invariance result, which extends the Goldman-Millson theorem in dimension $0$. As an application, we give a description of mapping spaces in the category of non-symmetric operads in terms of this simplicial Maurer-Cartan set. We establish a generalization of the latter result for symmetric operads.
\end{abstract}

\tableofcontents

\input{intro}

\input{chap2}

\input{chap3}

\printbibliography[heading=bibintoc,title={References}]

\end{document}

%% file: intro.tex
\section*{Introduction}
\phantomsection
\addcontentsline{toc}{section}{Introduction}

The usual category of topological spaces comes equipped with a functor $\text{\normalfont Map}_{\mathcal{T}op}(-,-):\mathcal{T}op^{op}\times\mathcal{T}op\longrightarrow\text{\normalfont sSet}$ which endows $\mathcal{T}op$ with the structure of a simplicial category (see for instance \cite[$\mathsection$2.1.1]{fresselivre}). This functor can be used in order to handle higher homotopies in the category $\mathcal{T}op$. For every topological spaces $X,Y$, the connected components of $\text{\normalfont Map}_{\mathcal{T}op}(X,Y)$ are in bijection with homotopy classes of morphisms $X\longrightarrow Y$, while the homotopy groups encode higher homotopy relations. This approach allows us to use tools from algebraic topology in order to study homotopy morphisms from $X$ to $Y$. The functor $\text{\normalfont Map}_{\mathcal{T}op}(-,-)$ is defined as follows. For every $X\in\mathcal{T}op$, we define two functors $X\otimes -:\text{\normalfont sSet}\longrightarrow\mathcal{T}op$ and $X^-:\text{\normalfont sSet}^{op}\longrightarrow\mathcal{T}op$ by
$$X\otimes K:=X\times|K|\ ;\ X^K:=\text{\normalfont Mor}_{\mathcal{T}op}(|K|,X),$$

\noindent for every $X\in\mathcal{T}op$ and $K\in\text{\normalfont sSet}$, where $|K|\in\mathcal{T}op$ denotes the geometric realization of the simplicial set $K$. For every $X,Y\in\mathcal{T}op$ and $K\in\text{\normalfont sSet}$, we have the isomorphism
$$\text{\normalfont Mor}_{\mathcal{T}op}(X\otimes K,Y)\simeq\text{\normalfont Mor}_{\mathcal{T}op}(X,Y^K).$$

\noindent We then define $\text{\normalfont Map}_{\mathcal{T}op}(X,Y)$ as the simplicial set $\text{\normalfont Mor}_{\mathcal{T}op}(X\otimes\Delta^\bullet,Y)$, where, for every $n\geq 0$, we denote by $\Delta^n$ the fundamental $n$-simplex.\\

In a general model category $C$, we have an analogue of the functors $X\otimes -$ and $X^-$. Such functors are defined by giving the image of $\Delta^n$ for every $n\geq 0$. The cosimplicial set $X\otimes\Delta^\bullet$ is called a \textit{cosimplicial frame} associated to $X$, while $X^{\Delta^\bullet}$ is called a \textit{simplicial frame} associated to $X$ (see for instance \cite[$\mathsection$3.2.2,$\mathsection$3.2.7]{fresselivre}). However, we only have a zig-zag of weak-equivalences of simplicial sets between $\text{\normalfont Mor}_C(X\otimes\Delta^\bullet,Y)$ and $\text{\normalfont Mor}_C(X,Y^{\Delta^\bullet})$ instead of an isomorphism, provided that $X$ is cofibrant and $Y$ is fibrant. This still allows us to construct a simplicial set $\text{\normalfont Map}_C(X,Y)$, which is unique up to a zig-zag of weak-equivalences. As in $\mathcal{T}op$, the connected components of $\text{\normalfont Map}_C(X,Y)$ are in bijection with homotopy classes of morphisms $X\longrightarrow Y$.\\

In this memoir, we provide an approach in order to study the homotopy of such mapping spaces in the category of non-symmetric operads and in the category of symmetric operads, where in both cases we consider operads defined in the category of differential graded $\mathbb{K}$-modules (dg $\mathbb{K}$-modules for short).\\

The category of operads in dg $\mathbb{K}$-modules inherits a model structures (see \cite{hinich1997homological} and \cite{hinicherratum}). Therefore, we can use the above theory to construct mapping spaces in the category of operads. We aim to give an explicit description of these mapping spaces, extending known results in characteristic zero to the case where the ground ring $\mathbb{K}$ is a field of positive characteristic. More precisely, we describe these mapping spaces as simplicial Maurer-Cartan spaces associated to some pre-Lie algebra up to homotopy with divided powers, a notion that we also define and study in this memoir.\\

We review the known results in characteristic $0$ before explaining our results in details in positive characteristic.

\subsection*{State-of-the-art in characteristic $0$}
\phantomsection
\addcontentsline{toc}{subsection}{State-of-the-art in characteristic $0$}

A comprehensive study of the homotopy type of a mapping spaces in the category of symmetric operads has already been done in the case $char(\mathbb{K})=0$. Let $\mathcal{C}$ be a coaugmented connected cooperad and $\mathcal{P}$ be an augmented connected operad. The computation of the homotopy groups can be deduced from a description of a mapping space $\text{\normalfont Map}_{\Sigma \mathcal{O}p^0}(B^c(\mathcal{C}),\mathcal{P})$ given in \cite{yalin}, in the context of properads. For this purpose, we use an explicit simplicial frame $\mathcal{P}^{\Delta^\bullet}$ associated to the operad $\mathcal{P}$, given by:
$$\mathcal{P}^{\Delta^\bullet}:=\mathcal{P}\otimes\Omega^\ast(\Delta^\bullet),$$

\noindent where, for every $n\geq 0$, we denote by $\Omega^\ast(\Delta^n)$ the Sullivan algebra of de Rham polynomial forms on $\Delta^n$ (see for instance \cite[$\mathsection$2.1]{bousfield}). The $n$-simplices of $\text{\normalfont Map}_{\Sigma \mathcal{O}p^0}(B^c(\mathcal{C}),\mathcal{P})$ then correspond to elements in $\text{\normalfont Hom}_{\Sigma\text{\normalfont Seq}_\mathbb{K}}(\overline{\mathcal{C}},\overline{\mathcal{P}})\widehat{\otimes}\Omega^\ast(\Delta^n)$ which satisfy some equations, where $\widehat{\otimes}$ is the complete tensor product associated to the complete filtered dg modules $\text{\normalfont Hom}_{\Sigma\text{\normalfont Seq}_\mathbb{K}}(\overline{\mathcal{C}},\overline{\mathcal{P}})$ and $\Omega^\ast(\Delta^n)$. These equations can be written by using the Lie algebra structure on $\text{\normalfont Hom}_{\text{\normalfont Seq}_\mathbb{K}}(\overline{\mathcal{C}},\overline{\mathcal{P}})$ induced by its pre-Lie algebra structure (see for instance \cite[$\mathsection$6.4.4]{loday} for a definition of the pre-Lie product). We recall the definitions in the paragraphs to follow.\\

Recall that if $L$ is a complete Lie algebra, then a Maurer-Cartan element is an element $\tau\in L_{-1}$ such that
$$d(\tau)+\frac{1}{2}[\tau,\tau]=0.$$

\noindent We denote by $\mathcal{MC}(L)$ the set of Maurer-Cartan elements in $L$. Note that every $\tau\in\mathcal{MC}(L)$ induces a differential $d_\tau$ defined by
$$d_\tau(x)=d(x)+[x,\tau].$$

\noindent We let $L^\tau$ be the dg $\mathbb{K}$-module $L$ endowed with the differential $d_\tau$. Using that $\Omega^\ast(\Delta^n)$ is endowed with the structure of a commutative algebra for every $n\geq 0$, the dg $\mathbb{K}$-module $L\widehat{\otimes}\Omega^\ast(\Delta^n)$ is endowed with the structure of a Lie algebra. We define the \textit{simplicial Maurer-Cartan set} associated to $L$ as
$$\mathcal{MC}_\bullet(L)=\mathcal{MC}(L\widehat{\otimes}\Omega^\ast(\Delta^\bullet)).$$

\noindent From \cite[Theorem 3.12]{yalin}, for every coaugmented cooperad $\mathcal{C}$, and for every augmented operad $\mathcal{P}$, we obtain the following description:
$$\text{\normalfont Map}_{\Sigma \mathcal{O}p^0}(B^c(\mathcal{C}),\mathcal{P})=\mathcal{MC}_\bullet(\text{\normalfont Hom}_{\Sigma\text{\normalfont Seq}_\mathbb{K}}(\overline{\mathcal{C}},\overline{\mathcal{P}})).$$

The computation of the homotopy groups of the simplicial set $\text{\normalfont Map}_{\Sigma \mathcal{O}p}(B^c(\mathcal{C}),\mathcal{P})$ can be deduced from the general computations of the homotopy groups of $\mathcal{MC}_\bullet(L)$ associated to a given complete Lie algebra $L$. These computations have been made in \cite[Theorem 1.1]{berglund}. Explicitly, if $L$ is a complete Lie algebra, then, for every $\tau\in\mathcal{MC}(L)$ and $k\geq 0$, we have the isomorphism
$$\pi_{k+1}(\mathcal{MC}_\bullet(L),\tau)\simeq H_k(L^\tau),$$

\noindent where $H_0(L^\tau)$ is endowed with the group structure $BCH$ given by the Baker-Campbell-Hausdorff formula.\\

The computation of the connected components of $\text{\normalfont Map}_{\Sigma \mathcal{O}p}(B^c(\mathcal{C}),\mathcal{P})$ can be achieved by using the pre-Lie deformation theory developed in \cite{dotsenko}. Recall that a pre-Lie algebra is a dg $\mathbb{K}$-module $L$ endowed with a linear morphism $\star:L\otimes L\longrightarrow L$ such that
$$(x\star y)\star z-x\star (y\star z)=(-1)^{|y||z|}((x\star z)\star y-x\star (z\star y)).$$

\noindent In particular, any pre-Lie algebra $L$ is endowed with the structure of a Lie algebra with the bracket $[x,y]=x\star y-(-1)^{|x||y|}y\star x$. In \cite{dotsenko}, the author generalized the Lie deformation theory to the pre-Lie context. Explicitly, a Maurer-Cartan element $\tau$ in a pre-Lie algebra $L$ is an element $\tau\in L_{-1}$ such that
$$d(\tau)+\tau\star\tau=0.$$

The gauge group $(L_0,BCH,0)$ can also be written in terms of pre-Lie operations. We consider the subset $1+L_0\subset\mathbb{K}\oplus L$. Under some convergence hypothesis, we define the circular product $\circledcirc:L\times (1+L_0)\longrightarrow L$ by
$$x\circledcirc(1+y)=\sum_{n\geq 0}\frac{1}{n!}x\{\underbrace{y,\ldots,y}_n\},$$

\noindent for every $x\in L$ and $y\in L_0$, where we denote by $-\{-,\ldots,-\}$ the symmetric brace operations associated to $L$ (see \cite{oudom} or \cite{lada}). We can restrict this product to an operation on $1+L_0$ defined by
$$(1+x)\circledcirc (1+y)=1+y+\sum_{n\geq 0}\frac{1}{n!}x\{\underbrace{y,\ldots,y}_n\}$$

\noindent for every $x,y\in L_0$. Then the triple $(1+L_0,\circledcirc,1)$ is a group isomorphic to the gauge group (see \cite[Theorem 2]{dotsenko}). The group $(1+L_0,\circledcirc,1)$ also acts on $\mathcal{MC}(L)$ via
$$(1+\mu)\cdot\tau=(\tau+\mu\star\tau-d(\tau))\circledcirc(1+\mu)^{\circledcirc-1}.$$

\noindent We define the Deligne groupoid $\text{\normalfont Deligne}(L)$ as the category with $\mathcal{MC}(L)$ as set of objects, and $(1+L_0,\circledcirc,1)$ as hom-sets.\\

\noindent Using a cylinder object associated to $B^c(\mathcal{C})$ (see \cite[Theorem 3.2.14]{fresselivre}) and \cite[Corollary 2]{dotsenko}, we obtain a bijection
$$\pi_0\text{\normalfont Map}_{\Sigma\mathcal{O}p^0}(B^c(\mathcal{C}),\mathcal{P})\simeq \pi_0\text{\normalfont Deligne}(\text{\normalfont Hom}_{\Sigma\text{\normalfont Seq}_\mathbb{K}}(\overline{\mathcal{C}},\overline{\mathcal{P}})),$$

\noindent where the right hand-side denotes the set of isomorphism classes of $\text{\normalfont Deligne}(\text{\normalfont Hom}_{\Sigma\text{\normalfont Seq}_\mathbb{K}}(\overline{\mathcal{C}},\overline{\mathcal{P}}))$.

\subsection*{Objectives and Results}
\phantomsection
\addcontentsline{toc}{subsection}{Objective and Results}

If $char(\mathbb{K})>0$, then the simplicial set $\mathcal{P}\otimes\Omega^\ast(\Delta^\bullet)$ given in \cite{yalin} is no longer a simplicial frame associated to $\mathcal{P}$, as the cohomology of $\Omega^\ast(\Delta^n)$ is not $0$ for every $n\geq 0$.\\

The first description of $\pi_0\text{\normalfont Map}_{\Sigma\mathcal{O}p^0}(B^c(\mathcal{C}),\mathcal{P})$ has been generalized to the positive characteristic context in \cite{moi} by using a $\Gamma(\mathcal{P}re\mathcal{L}ie,-)$-algebra structure on $\text{\normalfont Hom}_{\Sigma\text{\normalfont Seq}_\mathbb{K}}(\overline{\mathcal{C}},\overline{\mathcal{P}})$. Recall briefly that a $\Gamma(\mathcal{P}re\mathcal{L}ie,-)$-algebra is a dg module endowed with operations $-\{-,\ldots,-\}_{r_1,\ldots,r_n}$, defined for every integers $r_1,\ldots,r_n\geq 0$, and which mimic the operations
$$x\{y_1,\ldots,y_n\}_{r_1,\ldots,r_n}=\frac{1}{\prod_i r_i!}x\{\underbrace{y_1,\ldots,y_1}_{r_1},\ldots,\underbrace{y_n,\ldots,y_n}_{r_n}\}.$$

\noindent This notion has been studied in the non-graded context in \cite{cesaro}, and generalized to the graded context in \cite{moi}. Following the formulas of \cite{dotsenko}, the pre-Lie deformation theory can be generalized to a deformation theory controlled by $\Gamma(\mathcal{P}re\mathcal{L}ie,-)$-algebras, which is valid over a ring with positive characteristic. For every $\Gamma(\mathcal{P}re\mathcal{L}ie,-)$-algebra $L$, we thus have a notion of Deligne groupoid $\text{\normalfont Deligne}(L)$ (see \cite[Proposition-Definition 2.30]{moi}). Using a $\Gamma(\mathcal{P}re\mathcal{L}ie,-)$-algebra structure on $\text{\normalfont Hom}_{\Sigma\text{\normalfont Seq}_\mathbb{K}}(\overline{\mathcal{C}},\overline{\mathcal{P}})$ (see \cite[Corollary 2.18]{moi}), we retrieve, by \cite[Theorem 3.6]{moi}, a bijection
$$\pi_0\text{\normalfont Map}_{\Sigma\mathcal{O}p^0}(B^c(\mathcal{C}),\mathcal{P})\simeq \pi_0\text{\normalfont Deligne}(\text{\normalfont Hom}_{\Sigma\text{\normalfont Seq}_\mathbb{K}}(\overline{\mathcal{C}},\overline{\mathcal{P}})).$$

In this memoir, we construct an explicit cosimplicial frame associated to $B^c(\mathcal{C})$, in the case where $\mathcal{C}$ is a non-symmetric cooperad. Explicitly, for ever $n\geq 0$, we construct a twisting derivation $\partial^n$ on the operad $\mathcal{F}(\overline{\mathcal{C}}\otimes\Sigma^{-1}N_*(\Delta^n))$ such that
$$B^c(\mathcal{C})\otimes\Delta^\bullet:=(\mathcal{F}(\overline{\mathcal{C}}\otimes\Sigma ^{-1} N_*(\Delta^\bullet)),\partial^\bullet)$$

\noindent is a cosimplicial frame associated to $B^c(\mathcal{C})$ where $N_*(\Delta^n)$ is the normalized chain complex of the simplicial set $\Delta^n$. The $n$-simplices of a mapping space from $B^c(\mathcal{C})$ to $\mathcal{P}$ can then be identified with elements of $\text{\normalfont Hom}_{\text{\normalfont Seq}_\mathbb{K}}(\overline{\mathcal{C}},\overline{\mathcal{P}})\otimes\Sigma N^*(\Delta^n)$ which satisfy some equations. Our purpose is to interpret these equations as Maurer-Cartan equations. Our main ideas are the following. We deal with  $\Gamma(\mathcal{P}re\mathcal{L}ie_\infty,-)$-algebra structures, where $\mathcal{P}re\mathcal{L}ie_\infty$ denotes the operad that governs pre-Lie algebras up to homotopy. The key point is that if $A$ is a brace algebra and if $N$ is an algebra over the Barratt-Eccles operad $\mathcal{E}$, then $A\otimes N$ is a $\Gamma(\mathcal{P}re\mathcal{L}ie_\infty,-)$-algebra. Using this result with $A=\text{\normalfont Hom}_{\text{\normalfont Seq}_\mathbb{K}}(\overline{\mathcal{C}},\overline{\mathcal{P}})$ (which is a brace algebra by \cite[Proposition 6.4.2]{loday} and \cite[Proposition 1]{gerstenhaber2}) and $N=N^*(\Delta^n)$ (see \cite{fresseberger}) precisely give the desired equations.\\


The $\mathcal{P}re\mathcal{L}ie_\infty$-algebras, also called \textit{pre-Lie algebras up to homotopy}, have been studied in \cite{chapoton}. The author characterized the data of a $\mathcal{P}re\mathcal{L}ie_\infty$-algebra structure on $L$ as the data of brace operations which satisfy some identities. We denote these brace operations by $-\llbrace-,\ldots,-\rrbrace$ in this memoir, and we assume that these operations defined on the suspension $\Sigma L$. As for the study of the monad $\Gamma(\mathcal{P}re\mathcal{L}ie,-)$ in \cite{cesaro}, we prove that giving a $\Gamma(\mathcal{P}re\mathcal{L}ie_\infty,-)$-algebra structure on $L$ is equivalent to giving weighted brace operations $-\llbrace-,\ldots,-\rrbrace_{r_1,\ldots,r_n}$ on the suspension $\Sigma  L$ which are similar to the operations
$$x\llbrace y_1,\ldots,y_n\rrbrace_{r_1,\ldots,r_n}=\frac{1}{\prod_i r_i!}x\llbrace\underbrace{y_1,\ldots,y_1}_{r_1},\ldots,\underbrace{y_n,\ldots,y_n}_{r_n}\rrbrace.$$

\noindent We give an other characterization of such objects that will emphasize a notion of $\infty$-morphism. For any graded $\mathbb{K}$-module $V$, we set
$$\Gamma\text{\normalfont Perm}^c(V)=\bigoplus_{n\geq 0} V\otimes (V^{\otimes n})^{\Sigma _n}.$$

\noindent We prove that $\Gamma\text{\normalfont Perm}^c(V)$ is endowed with a coproduct $\dgperm$ which, in some sense, is compatible with the coproduct defined in \cite[$\mathsection$2.3]{chapoton} on $\text{\normalfont Perm}^c(V)$. We then define the category $\Gamma\Lambda\mathcal{PL}_\infty$ formed by pairs $(V,Q)$ where $V$ is a graded $\mathbb{K}$-module and $Q$ a coderivation on $\Gamma\text{\normalfont Perm}^c(V)$ of degree $-1$ such that $Q^2=0$. A morphism in $\Gamma\Lambda\mathcal{PL}_\infty$, also called an $\infty$-morphism, is a morphism of coalgebras which preserve the coderivations. We prove that $L$ is a $\Gamma(\mathcal{P}re\mathcal{L}ie_\infty,-)$-algebra if and only if $\Sigma  L\in\Gamma\Lambda\mathcal{PL}_\infty$. We also define a notion of (complete) filtered $\Gamma\Lambda\mathcal{PL}_\infty$-algebra. The category of complete $\Gamma\Lambda\mathcal{PL}_\infty$-algebras is denoted by $\widehat{\Gamma\Lambda\mathcal{PL}_\infty}$.\\

Given an object $V\in\widehat{\Gamma\Lambda\mathcal{PL}_\infty}$, a Maurer-Cartan element is a degree $0$ element $x\in V$ such that
$$d(x)+\sum_{n\geq 1}x\llbrace x\rrbrace_n=0.$$

\noindent We denote by $\mathcal{MC}(V)$ the set formed by these objects. We prove that any $\infty$-morphism $\phi:V\rightsquigarrow W$ induces a map
$$\mathcal{MC}(\phi):\mathcal{MC}(V)\longrightarrow\mathcal{MC}(W)$$

\noindent so that $\mathcal{MC}:\widehat{\Gamma\Lambda\mathcal{PL}_\infty}\longrightarrow\text{\normalfont Set}$ is a functor.\\

The motivation for using $\Gamma(\mathcal{P}re\mathcal{L}ie_\infty,-)$-algebras is given by the following theorem.

\begin{thmA}\label{theoremD}
    Let $\mathcal{B}race$ be the operad which governs brace algebras (see \cite[Proposition 2]{chapotonbis}). There exists an operad morphism $\mathcal{P}re\mathcal{L}ie_\infty\longrightarrow \mathcal{B}race\underset{\text{\normalfont H}}{\otimes} \mathcal{E}$ which fits in a commutative square
\[\begin{tikzcd}
	{\mathcal{P}re\mathcal{L}ie_\infty} & {\mathcal{B}race\underset{\text{\normalfont H}}{\otimes}\mathcal{E}} \\
	{\mathcal{P}re\mathcal{L}ie} & {\mathcal{B}race}
	\arrow[from=1-1, to=1-2]
	\arrow[from=1-1, to=2-1]
	\arrow[from=1-2, to=2-2]
	\arrow[from=2-1, to=2-2]
\end{tikzcd}.\]
\end{thmA}

\noindent As brace algebras are endowed with the structure of a $\Gamma(\mathcal{P}re\mathcal{L}ie,-)$-algebra (see \cite[Theorem 2.15]{moi}), Theorem \ref{theoremD} implies that every $\mathcal{B}race\underset{\text{\normalfont H}}{\otimes}\mathcal{E}$-algebra $L$ is a $\Gamma(\mathcal{P}re\mathcal{L}ie_\infty,-)$-algebra, via the composite
\begin{center}
\begin{tikzcd}
\Gamma(\mathcal{P}re\mathcal{L}ie_\infty,L)\arrow[r] & \Gamma(\mathcal{B}race\underset{\text{\normalfont H}}{\otimes} \mathcal{E},L) & \mathcal{S}(\mathcal{B}race\underset{\text{\normalfont H}}{\otimes} \mathcal{E},L)\arrow[l, "\simeq"']\arrow[r] & L.
\end{tikzcd}
\end{center}

Using that the normalized cochain complex $N^*(X)$ of a simplicial set $X$ admits the structure of an algebra over the Barratt-Eccles operad (see \cite{fresseberger}) and Theorem \ref{theoremD}, we define the simplicial Maurer-Cartan set associated to a complete brace algebra $A$ as
$$\mathcal{MC}_\bullet(A)=\mathcal{MC}(A\otimes \Sigma N^*(\Delta^\bullet)).$$

\noindent In particular, the $0$-vertices are identified with Maurer-Cartan elements in $A$, when using its underlying $\Gamma(\mathcal{P}re\mathcal{L}ie,-)$-algebra structure (see \cite[Theorem 2.15]{moi}). We explicitly compute the connected components and the homotopy groups of $\mathcal{MC}_\bullet(A)$.

\begin{thmA}\label{theoremE}
    For every complete brace algebra $A$, the simplicial set $\mathcal{MC}_\bullet(A)$ is a Kan complex. Moreover, we have the following computations for every $\tau\in\mathcal{MC}(A)$.
    \begin{itemize}
    \item $\pi_0(\mathcal{MC}_\bullet( A))\simeq\pi_0\text{\normalfont Deligne}(A),$ where $\text{\normalfont Deligne}(A)$ denotes the Deligne groupoid associated to the $\Gamma(\mathcal{P}re\mathcal{L}ie,-)$-algebra $A$ (see \cite[Proposition-Definition 2.30]{moi});
        \item
$\pi_1(\mathcal{MC}_\bullet(A),\tau)\simeq\{\mu\in A_0\ |\ d(\mu)=\tau+\mu\langle\tau\rangle-\tau\circledcirc(1+\mu)\}/\sim_\tau,$ where $\sim_\tau$ is the equivalence relation such that $\mu\sim_\tau \mu'$ if and only if there exists $\psi\in A_1$ such that
$$\mu-\mu'=d(\psi)+\psi\langle\tau\rangle+\sum_{p,q\geq 0}\tau\langle\underbrace{\mu,\ldots,\mu}_{p},\psi,\underbrace{\mu',\ldots,\mu'}_q\rangle.$$

\item $\pi_2(\mathcal{MC}_\bullet(A),\tau)\simeq (H_1(A^\tau),\ast_\tau,0)$, where $\ast_\tau$ is the group structure on $H_1(A^\tau)$ such that
$$[\mu]\ast_\tau[\mu']=[\mu+\mu'+\tau\langle\mu,\mu'\rangle].$$

\item 
$\pi_{n+1}(\mathcal{MC}_\bullet(A),\tau)\simeq H_{n}(A^\tau)$ for every $n\geq 3$.
\end{itemize}
\end{thmA}

We have the following homotopy invariance result, which extends the Goldman-Millson theorem in dimension 0.

\begin{thmA}\label{theoremF}
    Let $\Theta:A\longrightarrow B$ be a morphism of complete brace algebras such that $\Theta$ is a weak equivalence in $\text{\normalfont dgMod}_\mathbb{K}$. Then $\mathcal{MC}_\bullet(\Theta):\mathcal{MC}_\bullet(A)\longrightarrow\mathcal{MC}_\bullet(B)$ is a weak equivalence.
\end{thmA}

We use this new deformation theory for the study of the homotopy of mapping spaces in the category of non symmetric operads. For every non-symmetric coaugmented cooperad $\mathcal{C}$ such that $\mathcal{C}(0)=0$, and for every $n\geq 0$, we construct a twisting derivation $\partial^n$ on the operad $\mathcal{F}(\overline{\mathcal{C}}\otimes\Sigma ^{-1} N_*(\Delta^n))$ such that
$$B^c(\mathcal{C})\otimes\Delta^\bullet:=(\mathcal{F}(\overline{\mathcal{C}}\otimes\Sigma ^{-1} N_*(\Delta^\bullet)),\partial^\bullet)$$

\noindent is a cosimplicial frame associated to $B^c(\mathcal{C})$. This leads to the following theorem.

\begin{thmA}\label{theoremG}
    Let $\mathcal{C}$ be a coaugmented cooperad and $\mathcal{P}$ be an augmented operad such that $\mathcal{C}(0)=\mathcal{P}(0)=0$ and $\mathcal{C}(1)=\mathcal{P}(1)=\mathbb{K}$. Then we have an isomorphism of simplicial sets
    $$\text{\normalfont Map}_{\mathcal{O}p}(B^c(\mathcal{C}),\mathcal{P})\simeq\mathcal{MC}_\bullet(\text{\normalfont Hom}_{\text{\normalfont Seq}_\mathbb{K}}(\overline{\mathcal{C}},\overline{\mathcal{P}})).$$
\end{thmA}

The computation of the connected components and the homotopy groups of $\text{\normalfont Map}_{\mathcal{O}p}(B^c(\mathcal{C}),\mathcal{P})$ can then be achieved by using Theorem \ref{theoremE}.\\

In the symmetric context, the derivation $\partial^n$ constructed above on $\mathcal{F}(\overline{\mathcal{C}}\otimes\Sigma^{-1}N_*(\Delta^n))$ does not preserve the action of the symmetric group for every $n\geq 2$. We instead consider a $\Sigma_\ast$-cofibrant replacement of $B^c(\mathcal{C})$ given by the map $B^c(\mathcal{C}\underset{\text{\normalfont H}}{\otimes}\textbf{Sur}_\mathbb{K})\overset{\sim}{\longrightarrow} B^c(\mathcal{C})$, where $\textbf{Sur}_\mathbb{K}$ is the surjection cooperad defined in \cite[Theorem A.1]{pdoperads}. Using that the action of $\Sigma_n$ on $\overline{\mathcal{C}}(n){\otimes}\overline{\normalfont\textbf{Sur}}_\mathbb{K}(n)$ is free for every $n\geq 1$, we construct a twisting derivation $\partial^n$ on the operad $\mathcal{F}(\overline{\mathcal{C}}\underset{\text{\normalfont H}}{\otimes}\overline{\normalfont\textbf{Sur}}_\mathbb{K}\otimes\Sigma^{-1}N_*(\Delta^n))$ such that
$$B^c(\mathcal{C}\underset{\text{\normalfont H}}{\otimes}{\normalfont\textbf{Sur}}_\mathbb{K})\otimes\Delta^\bullet:=(\mathcal{F}(\overline{\mathcal{C}}\underset{\text{\normalfont H}}{\otimes}\overline{\normalfont\textbf{Sur}}_\mathbb{K}\otimes\Sigma^{-1} N_*(\Delta^\bullet)),\partial^\bullet)$$

\noindent is a cosimplicial frame associated to $B^c(\mathcal{C}\underset{\text{\normalfont H}}{\otimes}\textbf{Sur}_\mathbb{K})$. We deduce the following theorem.

\begin{thmA}\label{theoremH}
    Let $\mathcal{C}$ be a symmetric coaugmented cooperad and $\mathcal{P}$ be a symmetric augmented operad such that $\mathcal{C}(0)=\mathcal{P}(0)=0$ and $\mathcal{C}(1)=\mathcal{P}(1)=\mathbb{K}$. Then $\Sigma \text{\normalfont Hom}_{\Sigma\text{\normalfont Seq}_\mathbb{K}}(\overline{\mathcal{C}}\underset{\text{\normalfont H}}{\otimes}\overline{\normalfont\textbf{Sur}}_\mathbb{K}\otimes N_*(\Delta^\bullet),\overline{\mathcal{P}})$ is endowed with the structure of a $\widehat{\Gamma\Lambda\mathcal{PL}_\infty}$-algebra such that we have an isomorphism of simplicial sets
    $$\text{\normalfont Map}^h_{\Sigma\mathcal{O}p^0}(B^c(\mathcal{C}),\mathcal{P})\simeq\mathcal{MC}(\Sigma \text{\normalfont Hom}_{\Sigma\text{\normalfont Seq}_\mathbb{K}}(\overline{\mathcal{C}}\underset{\text{\normalfont H}}{\otimes}\overline{\normalfont\textbf{Sur}}_\mathbb{K}\otimes N_*(\Delta^\bullet),\mathcal{P})),$$

    \noindent where $\text{\normalfont Map}^h_{\Sigma\mathcal{O}p^0}(B^c(\mathcal{C}),\mathcal{P}):=\text{\normalfont Map}_{\Sigma\mathcal{O}p^0}(B^c(\mathcal{C}\underset{\text{\normalfont H}}{\otimes}{\normalfont\textbf{Sur}}_\mathbb{K}),\mathcal{P})$
\end{thmA}

\subsection*{Organization of the memoir}
\phantomsection
\addcontentsline{toc}{subsection}{Organization of the memoir}

In the first part of this memoir, we recall notions that will be useful for explaining our results. In $\mathsection$\ref{sec:211}, we explain our conventions on the context of differential graded $\mathbb{K}$-modules (dg $\mathbb{K}$-modules) in which we carry out our constructions. We examine in particular the definition of dg $\mathbb{K}$-modules which are complete with respect to a filtration and which we use in the definition of Maurer-Cartan elements. In $\mathsection$\ref{sec:212}, we review our conventions on operads, and recall the precise definition of algebras with divided powers over an operad. In $\mathsection$\ref{sec:213}, we recall the definition of the operad that governs brace algebras and its expression in terms of $\mathbb{K}$-modules of planar rooted trees. In $\mathsection$\ref{sec:214}, we recall the definition of the Barratt-Eccles and the definition of the action of this operad on the cochain algebra of simplicial sets through an intermediate operad given by an operad of surjections.\\

In the second part, we study the structure of a $\Gamma(\mathcal{P}re\mathcal{L}ie_\infty,-)$-algebra. In $\mathsection$\ref{sec:221}, we recall the construction of the operad $\mathcal{P}re\mathcal{L}ie_\infty$ and a characterization of $\mathcal{P}re\mathcal{L}ie_\infty$-algebras in terms of twisting coderivation on cofree $\text{\normalfont Perm}$-coalgebras. In $\mathsection$\ref{sec:222}, we explain the definition of the category $\Gamma\Lambda\mathcal{PL}_\infty$. In $\mathsection$\ref{sec:223}, we explain the definition of the weighted brace operations $-\{-,\ldots,-\}_{r_1,\ldots,r_n}$ and of the notion of a Maurer-Cartan element in a complete $\widehat{\Gamma\Lambda\mathcal{PL}_\infty}$-algebra. In $\mathsection$\ref{sec:224}, we explain the equivalence between $\Gamma\Lambda\mathcal{PL}_\infty$-algebras and $\Gamma(\mathcal{P}re\mathcal{L}ie_\infty,-)$-algebras (up to a shift).\\

The goal of part $3$ is to define the morphism $\mathcal{P}re\mathcal{L}ie_\infty\longrightarrow\mathcal{B}race\underset{\text{\normalfont H}}{\otimes}\mathcal{E}$ and to prove Theorem \ref{theoremD}. We actually obtain this morphism as a composite $\mathcal{P}re\mathcal{L}ie_\infty\overset{(1)}{\longrightarrow}\mathcal{B}race\underset{\text{\normalfont H}}{\otimes} B^c(\Lambda^{-1}\mathcal{B}race^\vee)\overset{(2)}{\longrightarrow}\mathcal{B}race\underset{\text{\normalfont H}}{\otimes}\mathcal{E}$ where $(2)$ is induced by a morphism $B^c(\Lambda^{-1}\mathcal{B}race^\vee)\longrightarrow\mathcal{E}$. In $\mathsection$\ref{sec:231}, we explain the construction of the latter morphism $B^c(\Lambda^{-1}\mathcal{B}race^\vee)\longrightarrow\mathcal{E}$. Then, from the general bar duality theory of algebras over operads, every $\mathcal{E}$-algebra $A$ comes with a twisting morphism on $\mathcal{B}race^\vee(\Sigma A)$. In $\mathsection$\ref{sec:232}, we make this twisting morphism explicit in the case $A=N^*(\Delta^n)$ for every $n\geq 0$. We will use this description to control the $\mathcal{P}re\mathcal{L}ie_\infty$-algebra structure on $A\otimes N^*(\Delta^n)$ for every $n\geq 0$, when we study the simplicial Maurer-Cartan set associated to brace algebras. In $\mathsection$\ref{sec:233}, we explain the definition of $(1)$ to complete our construction of the morphism $\mathcal{P}re\mathcal{L}ie_\infty\longrightarrow\mathcal{B}race\underset{\text{\normalfont H}}{\otimes}\mathcal{E}$ and the proof of Theorem \ref{theoremD}.\\

In the fourth part, we define and study our notion of a simplicial Maurer-Cartan set associated to a complete brace algebra. In $\mathsection$\ref{sec:241}, we define this simplicial set and prove that it is a Kan complex. In $\mathsection$\ref{sec:242}, we prove Theorem \ref{theoremE}, which gives a computation of the connected component and the homotopy groups of the simplicial Maurer-Cartan set associated to a complete brace algebra. In $\mathsection$\ref{sec:242b}, we give an interpretation of the first differentials computed in $\mathsection$\ref{sec:232} by computing the first simplices of the simplicial Maurer-Cartan set associated to a chosed complete brace algebra defined in the context of associated algebras up to homotopy. In $\mathsection$\ref{sec:243}, we prove Theorem \ref{theoremF}, which is an extension of the classical Goldman-Millson theorem for Lie algebras. In $\mathsection$\ref{sec:244}, we prove that, in characteristic $0$, our simplicial Maurer-Cartan set is related to the simplicial Maurer-Cartan set defined for Lie algebras via a zig-zag of weak-equivalences.\\

In the fifth part, we show that we can describe a mapping space from the cobar construction of a coaugmented non-symmetric cooperad to an augmented non-symmetric operad as a simplicial Maurer-Cartan set associated to a complete brace algebra. In $\mathsection$\ref{sec:251}, we recall the definition of the free operad generated by a sequence in terms of planar rooted trees with inputs, and recall the model structure on the category of operads. In $\mathsection$\ref{sec:253}, we construct a cosimplicial frame associated to the cobar construction of a coaugmented cooperad. In $\mathsection$\ref{sec:254}, we prove Theorem \ref{theoremG}, which shows that we can describe a mapping space from the cobar construction of a coaugmented non-symmetric cooperad $\mathcal{C}$ to an augmented non-symmetric operad $\mathcal{P}$ as as the simplicial Maurer-Cartan set associated to the complete brace algebra $\text{\normalfont Hom}_{\text{\normalfont Seq}_\mathbb{K}}(\overline{\mathcal{C}},\overline{\mathcal{P}})$.\\

In the last part of this memoir, we show that we can describe a mapping space from the cobar construction of a coaugmented symmetric cooperad to an augmented symmetric operad as a degree-wise Maurer-Cartan set of some $\widehat{\Gamma\Lambda\mathcal{PL}_\infty}$-algebras. In $\mathsection$\ref{sec:261}, we recall the definition of the free operad generated by a symmetric sequence in terms of planar rooted trees with inputs, and recall the model structure on the category of symmetric connected operads. In $\mathsection$\ref{sec:262}, we construct a cosimplicial frame associated to the cobar construction of a symmetric coaugmented cooperad. In $\mathsection$\ref{sec:263}, we prove Theorem \ref{theoremH}, we show that we can describe a mapping space from the cobar construction of a coaugmented cooperad $\mathcal{C}$ to an augmented operad $\mathcal{P}$ as a degree-wise Maurer-Cartan set of some $\widehat{\Gamma\Lambda\mathcal{PL}_\infty}$-algebras.

\subsection*{Acknowledgements}
\phantomsection
\addcontentsline{toc}{subsection}{Acknowledgements}

I acknowledge support from the Labex CEMPI (ANR-11-LABX-0007-01).\\ I also acknowledge the strong support of Benoit Fresse for the reading and discussions around this memoir.

%% file: chap2.tex
\section{Conventions and notations}

The goal of this section is to give recollections that will be needed in this memoir, and to set on our notations and conventions.\\

In $\mathsection$\ref{sec:211}, we recall basic definitions on dg $\mathbb{K}$-modules, and give the notation used in this memoir. We also give our definitions and notation on the notion of a (complete) filtered dg $\mathbb{K}$-module, with underlying category \dgMod.\\

In $\mathsection$\ref{sec:212}, we recall the notion of an operad and a cooperad. We also recall the operation of the Hadamard tensor product, which is widely used in this memoir. From the definition, we recall the notion of (co)operadic suspension, and study the (co)algebras over suspensions. We finally recall the notion of a $\Gamma(\mathcal{P},-)$-algebra associated to an operad $\mathcal{P}$ such that $\mathcal{P}(0)=0$.\\

In $\mathsection$\ref{sec:213}, we recall the definition of the operad $\mathcal{B}race$ which governs brace algebras in terms of planar rooted trees. We also set on our notations and conventions on planar rooted trees in this subsection.\\

In $\mathsection$\ref{sec:214}, we recall the definition of the Barratt-Eccles and surjection operads, following the conventions of \cite{fresseberger}. We also recall an important example of an algebra over such operads, given by the normalized cochain complex of simplicial sets.\\

In $\mathsection$\ref{sec:211bis}, we recall notions and notations on permutations and symmetric groups. More precisely, we recall the notion of shuffle permutations which gives a set of representatives of $\Sigma_m/\Sigma_{r_1}\times\cdots\times\Sigma_{r_n}$ where $r_1+\cdots+r_n=m$.

\subsection{The category $\text{\normalfont dgMod}_\mathbb{K}$}\label{sec:211}

Let $\mathbb{K}$ be a field. In this memoir, we work in the category \dgMod\ that we aim to define in this subsection.\\

A \textit{graded $\mathbb{K}$-module} is a $\mathbb{K}$-module $V$ equipped with a decomposition
$$V\simeq\bigoplus_{n\in\mathbb{Z}} V_n.$$

\noindent Given such a decomposition, an element $x\in V$ is \textit{homogeneous} if $x\in V_n$ for some $n\in\mathbb{Z}$. The integer $n$ is called the \textit{degree} of $x$, and denoted by $|x|$. A {morphism} of graded $\mathbb{K}$-modules of degree $d$ is a morphism $f:V\longrightarrow W$ of $\mathbb{K}$-modules such that $f(V_n)\subset W_{n+d}$. We denote by $\text{\normalfont Hom}(V,W)_d$ the $\mathbb{K}$-module formed by such morphisms. We set
$$\text{\normalfont Hom}(V,W):=\bigoplus_{d\in\mathbb{Z}}\text{\normalfont Hom}(V,W)_d.$$

\noindent We denote by $\text{\normalfont gMod}_\mathbb{K}$ the category formed by graded $\mathbb{K}$-modules with as set of morphisms from $V$ to $W$ the $\mathbb{K}$-module $\text{\normalfont Hom}(V,W)_0$. The \textit{dual} graded $\mathbb{K}$-module of $V$, denoted by $V^\vee$, is defined by $V^\vee=\text{\normalfont Hom}(V,\mathbb{K})$ where $\mathbb{K}$ is the graded $\mathbb{K}$-module with only one degree $0$ component given by $\mathbb{K}$. Explicitly, we have
$$V^\vee\simeq\bigoplus_{n\in\mathbb{Z}}V_{-n}^\vee.$$

\noindent If $V$ is finite dimensional, then, given a basis $x_1,\ldots,x_n$ of $V$, we endow $V^\vee$ with the basis $x_1^\vee,\ldots,x_n^\vee$ where, for every $1\leq i\leq n$, the linear form $x_i^\vee\in\text{\normalfont Hom}(V,\mathbb{K})$ is defined by
$$x_i^\vee(x_j)=\left\{\begin{array}{ll}
    1 & \text{if }i=j \\
    0 & \text{else}
\end{array}\right..$$

A \textit{differential} on $V$ is a degree $-1$ morphism $d_V:V\longrightarrow V$ such that $d_V\circ d_V=0$. We usually omit the index $V$ if there is no ambiguity on the ambient $\mathbb{K}$-module. The pair $(V,d_V)$ is called a \textit{differential graded $\mathbb{K}$-module} (or dg $\mathbb{K}$-module). A morphism of dg $\mathbb{K}$-modules is a morphism of graded $\mathbb{K}$-modules which commutes with the differentials. If $V$ and $W$ are dg $\mathbb{K}$-modules, then the graded $\mathbb{K}$-module $\text{\normalfont Hom}(V,W)$ comes equipped with a differential $d=d_{\text{\normalfont Hom}(V,W)}$ defined by
$$d(f)=d_W\circ f-(-1)^{|f|}f\circ d_V.$$

\noindent We denote by $\text{dgMod}_\mathbb{K}$ the category formed by dg $\mathbb{K}$-modules with as hom-sets the previous dg $\mathbb{K}$-module. This category is endowed with the structure of a symmetric monoidal category: the tensor product $V\otimes W$ of two elements $V,W\in\text{dgMod}_\mathbb{K}$ is the usual tensor product of $\mathbb{K}$-modules, with as degree $n$ component
$$(V\otimes W)_n=\bigoplus_{p+q=n}V_p\otimes W_q.$$

\noindent The differential on $V\otimes W$ is defined by
$$d_{V\otimes W}(v\otimes w)=d_V(v)\otimes w+(-1)^{|v|}v\otimes d_W(w).$$

\noindent The symmetry operator $\tau:V\otimes W\longrightarrow W\otimes V$ is defined by
$$\tau(v\otimes w)=(-1)^{|v||w|}w\otimes v.$$

\noindent The tensor product $f\otimes g$ of two morphisms of dg $\mathbb{K}$-modules $f:V\longrightarrow V'$ and $g:W\longrightarrow W'$ is defined by
$$(f\otimes g)(v\otimes w)=(-1)^{|g||v|}f(v)\otimes g(w).$$

Note that, as in the non-graded setting, we have an isomorphism of dg $\mathbb{K}$-modules
$$\text{\normalfont Hom}(U\otimes V,W)\overset{\simeq}{\longrightarrow}\text{\normalfont Hom}(U,\text{\normalfont Hom}(V,W))$$

\noindent for every dg $\mathbb{K}$-modules $U,V$ and $W$, defined by sending a morphism $f:U\otimes V\longrightarrow W$ to the morphism which sends $u\in U$ to the morphism $v\in V\longmapsto f(u\otimes v)\in W$.\\

\noindent If $V$ is finite dimensional, we also have an isomorphism of dg $\mathbb{K}$-modules
$$\text{\normalfont Hom}(V,W)\overset{\simeq}{\longrightarrow}W\otimes V^\vee.$$

\noindent This morphism is defined by sending $f\in\text{\normalfont Hom}(V,W)$ to $\sum_{i=1}^nf(e_i)\otimes e_i^\vee$, where $e_1,\ldots,e_n$ is a chosen basis of $V$. Using the two above isomorphisms gives
$$(V\otimes W)^\vee\simeq W^\vee\otimes V^\vee,$$

\noindent provided that $V$ is finite dimensional.\\

If we set $V=A^{\otimes n}, V'=C^{\otimes k}$ and $W=B^{\otimes n}, W'=D^{\otimes k}$ for some dg $\mathbb{K}$-modules $A,B,C,D$ and $n,k\geq 0$, then $f\otimes g$ is a morphism from $A^{\otimes n}\otimes B^{\otimes n}$ to $C^{\otimes k}\otimes D^{\otimes k}$. For our needs, it will sometimes be more convenient to see $f\otimes g$ as a morphism from $(A\otimes B)^{\otimes n}$ to $(C\otimes D)^k$.

\begin{defi}\label{otimesmodif}
    Let $f:A^{\otimes n}\longrightarrow C^{\otimes k}$ and $g:B^{\otimes n}\longrightarrow D^{\otimes k}$. We denote by $f\widetilde{\otimes} g:(A\otimes B)^{\otimes n}\longrightarrow (C\otimes D)^{\otimes k}$ the morphism defined by the following commutative diagram:
    \begin{center}
    $\begin{tikzcd}
        A^{\otimes n}\otimes B^{\otimes n}\arrow[r, "f\otimes g"]& C^{\otimes k}\otimes D^{\otimes k}\\
        (A\otimes B)^{\otimes n}\arrow[r, "f\widetilde{\otimes} g"']\arrow[u, "\simeq"] & (C\otimes D)^{\otimes k}\arrow[u, "\simeq"']
    \end{tikzcd}$
\end{center}

\noindent where we consider the isomorphisms given by the symmetry operator.
\end{defi}

We now recall the definition of the suspension of dg $\mathbb{K}$-modules.

\begin{defi}
    Let $k\in\mathbb{Z}$ and $V\in \text{\normalfont dgMod}_\mathbb{K}$. We denote by $\Sigma ^k$ the dg $\mathbb{K}$-module generated by one degree $k$ element $\Sigma^k\in\Sigma ^k$ with $0$ as differential. We define the {\normalfont $k$-suspension} of $V$ as
    $$\Sigma ^k V= \Sigma ^k\otimes V.$$

    For every $v\in V$, we set $\Sigma^kv:=\Sigma^k\otimes v$. We also set $\Sigma^1 =\Sigma$.
\end{defi}

For every $n,k\geq 0$, we have an isomorphism of $\mathbb{K}$-modules $(\Sigma ^kV)_n\simeq V_{n-k}$. Besides, giving a degree $k$ morphism $V\longrightarrow W$ is equivalent to giving a degree $0$ morphism $V\longrightarrow\Sigma ^k W$, and also equivalent to giving a degree $0$ morphism $\Sigma ^{-k}V\longrightarrow W$.\\

\noindent Note that, for every $k\in\mathbb{Z}$, the $k$-suspension defines an endofunctor in the category $\text{\normalfont dgMod}_\mathbb{K}$: for every $f\in\text{\normalfont Hom}(V,W)$, we define $\Sigma^kf\in\text{\normalfont Hom}(\Sigma^kV,\Sigma^kW)$ by
$$(\Sigma^kf)(\Sigma^kv)=(-1)^{k|f|}\Sigma^kf(v)$$

\noindent for every $v\in V$.\\

We now make explicit the notion of filtration that we consider in this memoir.

\begin{defi}
    Let $V\in\text{\normalfont dgMod}_\mathbb{K}$. A {\normalfont filtration} on $V$ is a sequence $(F_nV)_{n\geq 1}$ of sub dg $\mathbb{K}$-modules of $V$ such that
    $$\cdots\subset F_n V\subset F_{n-1}V\subset\cdots\subset F_1V =V.$$

    \noindent A dg $\mathbb{K}$-module endowed with a filtration is called a {\normalfont filtered} dg $\mathbb{K}$-module. A filtered dg $\mathbb{K}$-module $V$ is said to be {\normalfont complete} if we have an isomorphism
    $$V\simeq \lim_{\longleftarrow}V/F_nV.$$
\end{defi}

For every filtered dg $\mathbb{K}$-module $V$, the \textit{completion} of $V$ with respect to its filtration is the filtered dg $\mathbb{K}$-module defined by
$$\widehat{V}=\lim_{\longleftarrow}V/F_nV,$$

\noindent with as filtration
$$F_m\widehat{V}=\lim_{\longleftarrow}F_mV/(F_nV\cap F_mV).$$

\noindent We immediately see that $\widehat{V}$ is complete.

\begin{remarque}
    If $V$ is complete with respect to the filtration $(F_nV)_{n\geq 1}$, then $\bigcap_{n\geq 1}F_nV=0$. This implies that if $x\in V$ is such that $x\in F_kV\implies x\in F_{k+1}V$ for every $k\geq 1$, then $x=0$.
\end{remarque}

Let $V,W\in\text{dgMod}_\mathbb{K}$ be two complete filtered dg $\mathbb{K}$-modules. We say that a morphism $f:V\longrightarrow W$ \textit{preserves the filtrations} if it satisfies, for all $n\geq 1$,
$$f(F_nV)\subset F_nW.$$

\noindent The complete filtered dg $\mathbb{K}$-modules together with the filtration preserving morphisms define a category denoted by \dgMod. If $V$ and $W$ are filtered, then their tensor product $V\otimes W$ is also filtered with
$$F_n(V\otimes W)=\sum_{p+q=n}F_pV\otimes F_qW.$$

\noindent However, this filtered dg $\mathbb{K}$-module is not complete in general, even if $V$ and $W$ are so. We therefore define the \textit{complete tensor product} by
$$V\widehat{\otimes} W=\lim_{\longleftarrow}(V\otimes W)/{F_n(V\otimes W)}.$$

\noindent We can check that the category \dgMod\ endowed with $\widehat{\otimes}$ is a symmetric monoidal category.

\subsection{The notion of an operad and a cooperad}\label{sec:212}

We briefly recall the notion of an operad and its dual notion, the notion of a cooperad. We will mostly follow \cite{fresselivre} and \cite{loday}.\\

Let $\text{\normalfont Seq}_\mathbb{K}$ be the category whose objects are sequences in $\text{\normalfont dgMod}_\mathbb{K}$. For every $M,N\in\text{\normalfont Seq}_\mathbb{K}$, we denote by $\text{\normalfont Hom}(M,N)$ the sequence defined for every $n\geq 0$ by
$$\text{\normalfont Hom}(M,N)(n)=\text{\normalfont Hom}(M(n),N(n)).$$

\begin{defi}
    A {\normalfont symmetric sequence} is a sequence $M\in\text{\normalfont Seq}_\mathbb{K}$ such that, for every $n\geq 0$, the dg $\mathbb{K}$-module $M(n)$ comes equipped with an action of $\Sigma_n$ on it. A morphism of symmetric sequences is a morphism of sequences which preserves the actions of the symmetric groups.
\end{defi}

We denote by $\Sigma\text{\normalfont Seq}_\mathbb{K}$ the subcategory of symmetric sequences. Note that if $M,N\in\Sigma\text{\normalfont Seq}_\mathbb{K}$, then $\text{\normalfont Hom}(M,N)\in\Sigma\text{\normalfont Seq}_\mathbb{K}$ with the action defined by

$$(\sigma\cdot f)(x)=\sigma\cdot f(\sigma^{-1}\cdot x)$$

\noindent for every $n\geq 0, f\in\text{\normalfont Hom}(M(n),N(n)), x\in M(n)$ and $\sigma\in\Sigma_n$.

\begin{defi}\quad
\begin{itemize}
    \item An {\normalfont operad} is a symmetric sequence $\mathcal{P}\in\Sigma \text{\normalfont Seq}_\mathbb{K}$ endowed with composition products
    $$\gamma:\mathcal{P}(n)\otimes\mathcal{P}(r_1)\otimes\cdots\otimes\mathcal{P}(r_n)\longrightarrow\mathcal{P}\left(\sum_i r_i\right),$$

    \noindent which satisfy associativity, unit and symmetry axioms. The underlying category is denoted by $\Sigma \mathcal{O}p$.

    \item Dually, a {\normalfont cooperad} is a symmetric sequence $\mathcal{C}\in\Sigma \text{\normalfont Seq}_\mathbb{K}$ endowed with composition coproducts
    $$\Delta:\mathcal{C}\left(\sum_i r_i\right)\longrightarrow\mathcal{C}(n)\otimes\mathcal{C}(r_1)\otimes \cdots\otimes\mathcal{C}(r_n),$$

    \noindent which satisfy coassociativity, counit and symmetry axioms. The underlying category is denoted by $\Sigma\mathcal{O}p^c$.
\end{itemize}
\end{defi}

By forgetting the action of the symmetric groups, and the symmetry axioms, we have the notion of a non-symmetric (co)operad. We denote by $\mathcal{O}p$ and $\mathcal{O}p^c$ the underlying categories.

\begin{remarque}\label{coopdual}
    If $\mathcal{P}$ is an operad such that $\mathcal{P}(n)$ is finite dimensional for every $n\geq 0$ and $\mathcal{P}(0)=0$, then the symmetric sequence
    $$\mathcal{P}^\vee(n):=\mathcal{P}(n)^\vee$$

    \noindent is endowed with the structure of a cooperad given by the dualization of the operadic structure of $\mathcal{P}$.
\end{remarque}

If $\mathcal{P}$ is an operad, then we define, for every $n,m\geq 0$ and $1\leq i\leq n$, the \textit{$i$-th partial composition morphism} by
\[\begin{tikzcd}
	{\circ_i:\mathcal{P}(p)\otimes\mathcal{P}(q)} & {\mathcal{P}(p)\otimes\mathbb{K}\otimes\cdots\otimes\underset{i}{\mathcal{P}(q)}\otimes\cdots\otimes\mathbb{K}} & {\mathcal{P}(p+q-1)}
	\arrow["\simeq", from=1-1, to=1-2]
	\arrow["\gamma", from=1-2, to=1-3]
\end{tikzcd}\]

\noindent where we plug operadic units in places $j\neq i$. Dually, if $\mathcal{C}$ is a cooperad, we define the \textit{$i$-th partial decomposition morphism} by
\[\begin{tikzcd}
	{\Delta_i:\mathcal{C}(p+q-1)} & {\mathcal{C}(p)\otimes\mathbb{K}\otimes\cdots\otimes\underset{i}{\mathcal{C}(q)}\otimes\cdots\otimes\mathbb{K}} & {\mathcal{C}(p)\otimes\mathcal{C}(q)}
	\arrow["\Delta", from=1-1, to=1-2]
	\arrow["\simeq", from=1-2, to=1-3]
\end{tikzcd}\]

\noindent where we plug cooperadic counits in places $j\neq i$.


\begin{example}
    For every $n\geq 0$, we set $\mathcal{C}om(n)=\mathbb{K}$ endowed with the trivial $\Sigma_n$-action. The isomorphism $\mathbb{K}\otimes\mathbb{K}\simeq\mathbb{K}$ endows $\mathcal{C}om$ with the structure of an operad called the {\normalfont commutative operad}.
\end{example}

\begin{example}
    Let $V\in\text{\normalfont dgMod}_\mathbb{K}$. We define the symmetric sequences $\text{\normalfont End}_V$ and $\text{\normalfont CoEnd}_V$ by
    $$\begin{array}{rll}
        \text{\normalfont End}_V(n) & = & \text{\normalfont Hom}(V^{\otimes n},V); \\
        \text{\normalfont CoEnd}_V(n) & = & \text{\normalfont Hom}(V,V^{\otimes n}).
    \end{array}$$

    \noindent These symmetric sequences are endowed with the structure of a symmetric operad defined as follows. Let $f\in\text{\normalfont End}_V(p),g\in\text{\normalfont End}_V(q)$ and $1\leq i\leq p$. We set
    $$f\circ_i g=f\circ ( id_V\otimes\cdots\otimes\underset{i}{g}\otimes\cdots\otimes id_V).$$

    \noindent Let $\phi\in\text{\normalfont CoEnd}_V(p),\psi\in\text{\normalfont CoEnd}_V(q)$ and $1\leq i\leq p$. We set,
    $$\phi\circ_i\psi=(-1)^{|\phi||\psi|}(id_V\otimes\cdots\otimes\underset{i}{\psi}\otimes\cdots\otimes id_V)\circ\phi.$$
    
    These operads are called respectively called the {\normalfont endomorphism} and {\normalfont coendomorphism} operads generated by $V$.
\end{example}

\begin{remarque}\label{remlabel}
    Let $n \geq 0$ and $\mathcal{P}$ be an operad. The elements of $\mathcal{P}(n)$ are seen as operations with abstract variables labeled by $1,\ldots,n$. In this memoir, we often label these variables by elements of a finite set $X$ with $n$ elements. This can be formalized as follows. Let $X$ be a set with $n$ elements, and $\Sigma(n,X)$ be the set of bijections from $\llbracket 1,n\rrbracket$ to $X$. We set
    $$\mathcal{P}(X)=(\mathcal{P}(n)\otimes \mathbb{K}[\Sigma(n,X)])_{\Sigma_n},$$

    \noindent where we make coincide the action of $\Sigma_n$ on $\mathcal{P}(n)$ with its action by right translation on $\Sigma(n,X)$. The group of permutations on $X$ acts on $\mathcal{P}(X)$ by left translation on $\Sigma(n,X)$. Note that, for every finite sets $X,Y$ with $n$ elements, every bijection $u:X\longrightarrow Y$ induces a morphism $u_\ast:\mathcal{P}(X)\longrightarrow\mathcal{P}(Y)$, so that $\mathcal{P}$ defines a functor from the category of finite sets to the category of dg $\mathbb{K}$-modules.\\
    
    \noindent For our needs, we apply the above construction to totally ordered finite sets. In this setting, we can shape operadic compositions on finite sets in the following way. Let $X=x_1<\cdots <x_n$ and $Y=y_1<\cdots <y_n$ be two disjoint totally ordered sets. We denote by $x:\llbracket 1,n\rrbracket\longrightarrow X$ and $y:\llbracket 1,m\rrbracket\longrightarrow Y$ the unique order preserving maps. Let $1\leq i\leq n$. We set
    $$X\sqcup_{i}Y=x_1<\cdots <x_{i-1}<y_1<\cdots <y_n<x_{i+1}<\cdots < x_n.$$
    
    \noindent Let $z:\llbracket 1,n+m-1\rrbracket\longrightarrow X\sqcup_i Y$ be the order preserving map. We define
    $$\circ_{i}:\mathcal{P}(X)\otimes\mathcal{P}(Y)\longrightarrow\mathcal{P}(X\sqcup_iY)$$

    \noindent by the following commutative diagram:
\[\begin{tikzcd}
	{\mathcal{P}(n)\otimes\mathcal{P}(m)} & {\mathcal{P}(n+m-1)} \\
	{\mathcal{P}(X)\otimes\mathcal{P}(Y)} & {\mathcal{P}(X\sqcup_{i}Y)}
	\arrow["{\circ_i}", from=1-1, to=1-2]
	\arrow["{x_\ast\otimes y_\ast}"', from=1-1, to=2-1]
	\arrow["\simeq", from=1-1, to=2-1]
	\arrow["{z_\ast}", from=1-2, to=2-2]
	\arrow["\simeq"', from=1-2, to=2-2]
	\arrow["{\circ_{i}}"', from=2-1, to=2-2]
\end{tikzcd}.\]
\end{remarque}

\begin{defi}
    Let $\mathcal{P}$ be an operad. A {\normalfont $\mathcal{P}$-algebra} $A$ (respectively {\normalfont $\mathcal{P}$-coalgebra} $C$) is the data of a dg $\mathbb{K}$-module $A$ (resp. $C$) together with an operad morphism $\mathcal{P}\longrightarrow\text{\normalfont End}_A$ (respectively $\mathcal{P}\longrightarrow\text{\normalfont CoEnd}_C$).
\end{defi}

If $A$ is a $\mathcal{P}$-algebra with associated morphism $\phi:\mathcal{P}\longrightarrow\text{\normalfont End}_A$, for every $p\in\mathcal{P}(n)$, we set $p^A:=\phi(p)\in\text{\normalfont Hom}(A^{\otimes n},A)$.\\ Analogously, if $C$ is a $\mathcal{P}$-coalgebra with associated morphism $\phi:\mathcal{P}\longrightarrow\text{\normalfont CoEnd}_A$, we set $p^C:=\phi(p)\in\text{\normalfont Hom}(C,C^{\otimes n})$.

\begin{remarque}
    The above definition is such that the dual dg $\mathbb{K}$-module $C^\vee$ of every $\mathcal{P}$-coalgebra $C$ is endowed with the structure of a $\mathcal{P}$-algebra. We define an operad morphism $\phi^\vee:\mathcal{P}\longrightarrow\text{\normalfont End}_{C^\vee}$ as follows. Let $\phi:\mathcal{P}\longrightarrow\text{\normalfont CoEnd}_C$ be the coalgebra structure given by $\mathcal{P}$ on $C$. For every $n\geq 0$, $p\in\mathcal{P}(n)$ and $u_1,\ldots,u_n\in C^\vee$, we set
    $$\phi^\vee(p)(u_1\otimes\cdots\otimes u_n)=\phi(p)^\vee(u_1\otimes\cdots\otimes u_n).$$
\end{remarque}

In the following, we also use the notion of a complete $\mathcal{P}$-algebra. A \textit{filtered} $\mathcal{P}$-algebra is a filtered dg $\mathbb{K}$-module $A$ endowed with the structure of a $\mathcal{P}$-algebra such that, for every $n\geq 0$ and $p\in\mathcal{P}(n)$, the morphism $p^A:A^{\otimes n}\longrightarrow A$ preserves the filtrations. A \textit{complete} $\mathcal{P}$-algebra is a filtered $\mathcal{P}$-algebra which is complete with respect to its filtration.\\

We define a monoidal structure on $\Sigma \mathcal{O}p$ and $\Sigma \mathcal{O}p^c$ given by the \textit{Hadamard tensor product}.

\begin{defi}
    Let $\mathcal{P},\mathcal{Q}$ be two operads and $\mathcal{C},\mathcal{D}$ be two cooperads.
    \begin{itemize}
        \item The {\normalfont Hadamard tensor product} of $\mathcal{P}$ and $\mathcal{Q}$ is the operad $\mathcal{P}\underset{\text{\normalfont H}}{\otimes}\mathcal{Q}$ defined by
$$(\mathcal{P}\underset{\text{\normalfont H}}{\otimes}\mathcal{Q})(n)=\mathcal{P}(n)\otimes\mathcal{Q}(n)$$

\noindent and equipped with the tensor-wise operadic composition product and the diagonal action of $\Sigma _n$.

\item The {\normalfont Hadamard tensor product} of $\mathcal{C}$ and $\mathcal{D}$ is the cooperad $\mathcal{C}\underset{H}{\otimes}\mathcal{D}$ defined by
$$(\mathcal{C}\underset{\text{\normalfont H}}{\otimes}\mathcal{D})(n)=\mathcal{C}(n)\otimes\mathcal{D}(n)$$

\noindent and equipped with the tensor-wise cooperadic composition coproduct and the diagonal action of $\Sigma _n$.
    \end{itemize}
\end{defi}

We now define the notion of suspension in the category of operads and cooperads.

\begin{defi}
    Let $\mathcal{P}$ be an operad and $\mathcal{C}$ be a cooperad. We set $\Lambda^k=\text{\normalfont End}_{\Sigma ^k}$ and $\Lambda:=\Lambda^1$.

    \begin{itemize}
        \item The {\normalfont $k$ operadic suspension} of $\mathcal{P}$ is the operad $\Lambda^k\mathcal{P}$ defined by
        $$\Lambda^k\mathcal{P}=\Lambda^k\underset{\text{\normalfont H}}{\otimes}\mathcal{P}.$$

        \item The {\normalfont $k$ cooperadic suspension} of $\mathcal{C}$ is the cooperad $\Lambda^{k}\mathcal{C}$ defined by
        $$\Lambda^{k}\mathcal{C}= {(\Lambda^k)^\vee}\underset{\text{\normalfont H}}{\otimes}\mathcal{C}.$$
    \end{itemize}
\end{defi}

Accordingly, $\Lambda^k\mathcal{P}(n)\simeq\Sigma ^{k(1-n)}\mathcal{P}(n)$ and $\Lambda^k\mathcal{C}(n)\simeq\Sigma ^{k(1-n)}\mathcal{C}(n)$.\\

We have an isomorphism of operads $\Lambda^{-k}\Lambda^k\mathcal{P}\longrightarrow\mathcal{P}$ given by $(\Sigma^{-k(1-n)}(\Sigma^{k(1-n)}p))\longmapsto -(-1)^{\frac{n(n+1)}{2}k}p$ for every $p\in\mathcal{P}(n)$.\\

Note that, for every $k\in\mathbb{Z}$, the dg $\mathbb{K}$-module $\Sigma^k$ is a $\Lambda^k$-algebra. We thus have the following.

\begin{prop}
    Let $\mathcal{P}$ be an operad and $\mathcal{C}$ be a cooperad. Let $V$ be a dg $\mathbb{K}$-module.

    \begin{itemize}
        \item Giving a structure of $\mathcal{P}$-algebra on $V$ is equivalent to giving a structure of $\Lambda^k\mathcal{P}$-algebra on $\Sigma ^kV$.
        \item Giving a structure of $\mathcal{C}$-coalgebra on $V$ is equivalent to giving a structure of $\Lambda^k\mathcal{C}$-coalgebra on $\Sigma ^k V$.
    \end{itemize}
\end{prop}

Any operad $\mathcal{P}$ gives a monad $\mathcal{S}(\mathcal{P},-):\text{\normalfont dgMod}_{\mathbb{K}}\longrightarrow\text{\normalfont dgMod}_\mathbb{K}$ called the \textit{Schur functor} and defined by
$$\mathcal{S}(\mathcal{P},V)=\bigoplus_{n\geq 0}\mathcal{P}(n)\otimes_{\Sigma_n} V^{\otimes n},$$

\noindent where we consider the action of $\Sigma_n$ on $\mathcal{P}(n)$, and the action of $\Sigma_n$ on $V^{\otimes n}$ by permutation. The monadic structure is given by the composite
\[\begin{tikzcd}
	{\mathcal{S}(\mathcal{P},\mathcal{S}(\mathcal{P},V))} & {\mathcal{S}(\mathcal{P}\circ\mathcal{P},V)} & {\mathcal{S}(\mathcal{P},V),}
	\arrow["\simeq", from=1-1, to=1-2]
	\arrow["{\mathcal{S}(\gamma,V)}", from=1-2, to=1-3]
\end{tikzcd}\]

\noindent where we denote by $\circ$ the composition of symmetric sequences. Note that the algebras over the monad $\mathcal{S}(\mathcal{P},-)$ are precisely the $\mathcal{P}$-algebras.\\

\noindent If $\mathcal{P}(0)=0$, we also have a monad $\Gamma(\mathcal{P},-):\text{\normalfont dgMod}_{\mathbb{K}}\longrightarrow\text{\normalfont dgMod}_\mathbb{K}$ defined by
$$\Gamma(\mathcal{P},V)=\bigoplus_{n\geq 1}\mathcal{P}(n)\otimes^{\Sigma_n} V^{\otimes n}.$$

\noindent We refer to \cite[$\mathsection$1.1.18]{fressebis} for the description of this monadic structure. We only note that we have a morphism of monads
$$Tr:\mathcal{S}(\mathcal{P},V)\longrightarrow\Gamma(\mathcal{P},V)$$

\noindent given by the trace map. This is an isomorphism as soon as $char(\mathbb{K})=0$. It is however no longer an isomorphism in general when $char(\mathbb{K})\neq 0$.

\begin{defi}
    Let $\mathcal{P}$ be an operad such that $\mathcal{P}(0)=0$. A {\normalfont $\mathcal{P}$-algebra with divided powers} is a $\Gamma(\mathcal{P},-)$-algebra.
\end{defi}

Note that every $\mathcal{P}$-algebra with divided powers is in particular a $\mathcal{P}$-algebra through the trace map.

\begin{prop}\label{gammashift}
    Let $\mathcal{P}$ be an operad such that $\mathcal{P}(0)=0$ and $V$ be a dg $\mathbb{K}$-module. Let $k\in\mathbb{Z}$. Then $V$ is a $\Gamma(\mathcal{P},-)$-algebra if and only if $\Sigma ^k V$ is a $\Gamma(\Lambda^k\mathcal{P},-)$-algebra.
\end{prop}

\begin{proof}
    Let $V$ be a $\Gamma(\mathcal{P},-)$-algebra. We endow $\Sigma ^kV$ with the structure of a $\Gamma(\Lambda^k\mathcal{P},-)$-algebra via the composite
$$\Gamma(\Lambda^k\mathcal{P},\Sigma ^k V)\longrightarrow\Sigma ^k\Gamma(\mathcal{P},V)\longrightarrow\Sigma ^k V,$$

\noindent where the first morphism comes from the fact that $\text{\normalfont End}_{\Sigma ^k}(n)\otimes (\Sigma ^k)^{\otimes n}$ is isomorphic to $\Sigma ^k$ endowed with the trivial $\Sigma _n$ action. The fact that this endows $\Sigma ^k V$ with a $\Gamma(\Lambda^k\mathcal{P},-)$-algebra structure is an immediate verification.
\end{proof}

\subsection{On trees and the operad $\mathcal{B}race$}\label{sec:213}

In this section, we recall the notion of a tree and define the operad $\mathcal{B}race$. The notion of a brace algebra was introduced in \cite[Definition 1]{gerstenhaber2}, while an explicit construction of their governing operad $\mathcal{B}race$ is given in \cite[$\mathsection$2.1-2.2]{chapotonbis}.



\begin{defi}
    We call (planar) {\normalfont $n$-tree} any simply connected graph endowed with a special vertex called the {\normalfont root} and a labeling of its set of vertices from $1$ to $n$. We put the root at the bottom by convention:
$$\begin{tikzpicture}[baseline={([yshift=-.5ex]current bounding box.center)},scale=0.6]
    \node[draw,circle,scale=0.6] (i) at (0,0) {$5$};
    \node[draw,circle,scale=0.6] (1) at (-1,1) {$6$};
    \node[draw,circle,scale=0.6] (1b) at (-1.5,2) {$3$};
    \node[draw,circle,scale=0.6] (1bb) at (-0.5,2) {$7$};
    \node[draw,circle,scale=0.6] (2) at (0,1) {$1$};
    \node[draw,circle,scale=0.6] (3) at (1,1) {$4$};
    \node[draw,circle,scale=0.6] (3b) at (1,2) {$2$};
    \draw (2) -- (i);
    \draw (1) -- (1b);
    \draw (1) -- (1bb);
    \draw (i) -- (1);
    \draw (i) -- (3);
    \draw (3) -- (3b);
    \end{tikzpicture}.$$
    
    We denote by $\mathcal{PRT}(n)$ the set of planar rooted trees with $n$ vertices. For every $T\in\mathcal{PRT}(n)$, we set $|T|=n$ and $r(T)$ denotes the root of the tree $T$.
\end{defi}



In some situation, it is more convenient to label an $n$-tree by a finite set with $n$ elements endowed with a total ordered relation. If $X$ is such a set, we denote by $\mathcal{PRT}(X)$ the set of $n$-trees labeled with elements of $X$. Note that since there is a unique order preserving bijection $\llbracket 1,n\rrbracket\longrightarrow X$, there is a canonical bijection $\mathcal{PRT}(n)\longrightarrow\mathcal{PRT}(X)$. For instance, the tree $T$ shown in the above definition can be seen in $\mathcal{PRT}(a_1<\cdots < a_7)$ as
$$\begin{tikzpicture}[baseline={([yshift=-.5ex]current bounding box.center)},scale=0.6]
    \node[draw,circle,scale=0.6] (i) at (0,0) {$a_5$};
    \node[draw,circle,scale=0.6] (1) at (-1,1) {$a_6$};
    \node[draw,circle,scale=0.6] (1b) at (-1.5,2) {$a_3$};
    \node[draw,circle,scale=0.6] (1bb) at (-0.5,2) {$a_7$};
    \node[draw,circle,scale=0.6] (2) at (0,1) {$a_1$};
    \node[draw,circle,scale=0.6] (3) at (1,1) {$a_4$};
    \node[draw,circle,scale=0.6] (3b) at (1,2) {$a_2$};
    \draw (2) -- (i);
    \draw (1) -- (1b);
    \draw (1) -- (1bb);
    \draw (i) -- (1);
    \draw (i) -- (3);
    \draw (3) -- (3b);
    \end{tikzpicture}.$$

\begin{prop}[{\cite[Proposition 2]{chapotonbis}}]
    Let $\mathcal{B}race$ be the symmetric sequence defined by $\mathcal{B}race=\mathbb{K}[\mathcal{PRT}(n)]$. Then $\mathcal{B}race$ is endowed with the structure of an operad. Its algebras are given by dg $\mathbb{K}$-modules $A$ endowed with morphisms $-\langle -,\ldots,-\rangle:A^{\otimes n+1}\longrightarrow A$ for any $n\geq 0$ such that $x\langle\rangle=x$ and
    $$x\langle y_1,\ldots,y_n\rangle\langle z_1,\ldots,z_p\rangle=\sum\pm x\langle Z_1,y_1\langle Z_2\rangle,\ldots,Z_{2n-1},y_n\langle Z_{2n}\rangle, Z_{2n+1}\rangle$$

    \noindent for every $x,y_1,\ldots,y_n,z_1,\ldots,y_p\in A$, where the sum runs over all consecutive subsets such that $Z_1\sqcup\cdots\sqcup Z_{2n+1}=(z_1,\ldots,z_p)$.
\end{prop}

Note that every tree $T$ with $|T|\geq 2$ can be uniquely written as $T=\gamma(F_n(\begin{tikzpicture}[baseline={([yshift=-.5ex]current bounding box.center)},scale=0.6]
    \node[draw,circle,scale=0.6] (2) at (0,0) {1};\end{tikzpicture},T_1,\ldots,T_n))$ where we denote by
$$F_n=\begin{tikzpicture}[baseline={([yshift=-.5ex]current bounding box.center)},scale=0.6]
    \node[draw,circle,scale=0.6] (2) at (0,0) {1};
    \node[scale=0.6] (3) at (0,1) {$\cdots$};
    \node[draw,circle,scale=0.6] (1) at (-1,1) {$2$};
    \node[draw,circle,scale=0.4] (4) at (1,1) {$n+1$};
    \draw (2) -- (1);
    \draw (4) -- (2);
    \end{tikzpicture}$$
\noindent the \textit{corolla} with $n$ leaves.\\

In the next sections, in order to have formulas which preserve the symmetric groups actions on $\mathcal{B}race$, we pick an explicit set of representatives of $\mathcal{B}race(n)$ as a free $\Sigma _n$-set. We achieve this by setting a total order relation on the set of vertices $V_T$ which we call the \textit{canonical order}. For every $a\in\mathbb{N}^\ast$, we set $V_{\begin{tikzpicture}[baseline={([yshift=-.5ex]current bounding box.center)},scale=0.6]
    \node[draw,circle,scale=0.6] (2) at (0,0) {$a$};\end{tikzpicture}}=a$, and define by induction,
    $$V_{\gamma(F_n(\begin{tikzpicture}[baseline={([yshift=-.5ex]current bounding box.center)},scale=0.6]
    \node[draw,circle,scale=0.6] (2) at (0,0) {$a$};\end{tikzpicture},T_1,\ldots,T_n))}=a<V_{T_1}<\cdots<V_{T_n}$$

\noindent for every tree $T_1,\ldots,T_n$. For instance, if we set
\begin{equation*}
        T=\begin{tikzpicture}[baseline={([yshift=-.5ex]current bounding box.center)},scale=0.6]
    \node[draw,circle,scale=0.6] (i) at (0,0) {$5$};
    \node[draw,circle,scale=0.6] (1) at (-1,1) {$6$};
    \node[draw,circle,scale=0.6] (1b) at (-1.5,2) {$3$};
    \node[draw,circle,scale=0.6] (1bb) at (-0.5,2) {$7$};
    \node[draw,circle,scale=0.6] (2) at (0,1) {$1$};
    \node[draw,circle,scale=0.6] (3) at (1,1) {$4$};
    \node[draw,circle,scale=0.6] (3b) at (1,2) {$2$};
    \draw (2) -- (i);
    \draw (1) -- (1b);
    \draw (1) -- (1bb);
    \draw (i) -- (1);
    \draw (i) -- (3);
    \draw (3) -- (3b);
    \end{tikzpicture}
\end{equation*}

\noindent then $V_T=5<6<3<7<1<4<2$.

\begin{defi}\label{canonicaltree}
    A tree $T\in\mathcal{PRT}(a_1<\cdots <a_n)$ is {\normalfont canonical} (or in the canonical order) if
    $$V_T=a_1<\cdots <a_n.$$

    \noindent We let $\sigma _T\in\Sigma _{|T|}$ to be the unique permutation such that $\sigma _T^{-1}\cdot T$ is in the canonical order.

\end{defi}



For instance, if we consider the above tree, then $\sigma _T=(5637142)$ and
\begin{equation*}
        \sigma _T^{-1}\cdot T=\begin{tikzpicture}[baseline={([yshift=-.5ex]current bounding box.center)},scale=0.6]
    \node[draw,circle,scale=0.6] (i) at (0,0) {$1$};
    \node[draw,circle,scale=0.6] (1) at (-1,1) {$2$};
    \node[draw,circle,scale=0.6] (1b) at (-1.5,2) {$3$};
    \node[draw,circle,scale=0.6] (1bb) at (-0.5,2) {$4$};
    \node[draw,circle,scale=0.6] (2) at (0,1) {$5$};
    \node[draw,circle,scale=0.6] (3) at (1,1) {$6$};
    \node[draw,circle,scale=0.6] (3b) at (1,2) {$7$};
    \draw (2) -- (i);
    \draw (1) -- (1b);
    \draw (1) -- (1bb);
    \draw (i) -- (1);
    \draw (i) -- (3);
    \draw (3) -- (3b);
    \end{tikzpicture}
\end{equation*}

\noindent is in the canonical order in $\mathcal{PRT}(1<\cdots <7)$.


    

\begin{defi}
    Let $X$ be a totally finite ordered set and $T\in\mathcal{PRT}(X)$.
    \begin{itemize}
        \item A {\normalfont subtree $S\subset T$} of $T$ is an induced simply connected subgraph of $T$ whose set of vertices is seen as a subset $Y$ of $X$ endowed with the induced order relation. Note that $V_S\subset V_T$ as ordered sets.

        \item If $S\subset T$, we define the tree $T/S\in\mathcal{PRT}(X\setminus Y\cup\{S\})$ obtained from $T$ by contracting the tree $S$ on the root of $S$, denoted by $S$ in the labeling of $T/S$. The totally ordered set $X\setminus Y\cup\{S\}$ is obtained by changing $r(S)$ into $S$, and removing all the non-root vertices of $S$ in $X$.
    \end{itemize}

    A subtree $S\subset T$ is {\normalfont non-trivial} if neither $|S|\neq 1$ nor $|T/S|\neq 1$.
\end{defi}

\begin{remarque}
    Let $X$ be a totally finite ordered set. Let $T\in\mathcal{PRT}(X)$ and $S\subset T$. If $T$ is canonical, then so are $S$ and $T/S$.
\end{remarque}

\textit{Example:} If

$$\begin{array}{lll}
        T=\begin{tikzpicture}[baseline={([yshift=-.5ex]current bounding box.center)},scale=0.6]
    \node[draw,circle,scale=0.6] (i) at (0,0) {$5$};
    \node[draw,circle,scale=0.6] (1) at (-1,1) {$6$};
    \node[draw,circle,scale=0.6] (1b) at (-1.5,2) {$3$};
    \node[draw,circle,scale=0.6] (1bb) at (-0.5,2) {$7$};
    \node[draw,circle,scale=0.6] (2) at (0,1) {$1$};
    \node[draw,circle,scale=0.6] (3) at (1,1) {$4$};
    \node[draw,circle,scale=0.6] (3b) at (1,2) {$2$};
    \draw (2) -- (i);
    \draw (1) -- (1b);
    \draw (1) -- (1bb);
    \draw (i) -- (1);
    \draw (i) -- (3);
    \draw (3) -- (3b);
    \end{tikzpicture},
\end{array}$$

\noindent then
$$S=\begin{tikzpicture}[baseline={([yshift=-.5ex]current bounding box.center)},scale=0.6]
    \node[draw,circle,scale=0.6] (i) at (-1,0) {$5$};
    \node[draw,circle,scale=0.6] (1) at (-1,1) {$6$};
    \node[draw,circle,scale=0.6] (1b) at (-1.5,2) {$3$};
    \node[draw,circle,scale=0.6] (1bb) at (-0.5,2) {$7$};
    \draw (1) -- (1b);
    \draw (1) -- (1bb);
    \draw (i) -- (1);
    \end{tikzpicture}\in\mathcal{PRT}(3<5<6<7)$$

\noindent is a subtree of $T$ such that
\begin{center}
    \begin{equation*}
        T/S=\begin{tikzpicture}[baseline={([yshift=-.5ex]current bounding box.center)},scale=0.6]
    \node[draw,circle,scale=0.6] (i) at (0,0) {$S$};
    \node[draw,circle,scale=0.6] (2) at (0,1) {$1$};
    \node[draw,circle,scale=0.6] (3) at (1,1) {$4$};
    \node[draw,circle,scale=0.6] (3b) at (1,2) {$2$};
    \draw (2) -- (i);
    \draw (i) -- (3);
    \draw (3) -- (3b);
    \end{tikzpicture}\in\mathcal{PRT}(1<2<S<4). 
    \end{equation*}
\end{center}



\subsection{On the Barratt-Eccles and the surjection operads}\label{sec:214}

We devote this subsection to recollections on the Barratt-Eccles operad and the surjection operad. We will mostly follow conventions of \cite{fresseberger}.

\begin{defi}
    We let $\mathcal{E}(r)_d$ to be the $\mathbb{K}$-module spanned by $(d+1)$-tuples 
    $$(w_0,\ldots,w_d)\in (\Sigma _r)^{d+1}$$

    \noindent with the identification $(w_0,\ldots,w_d)\equiv0$ if $w_i=w_{i+1}$ for some $i$. We denote by $\mathcal{E}(r)$ the dg $\mathbb{K}$-module with $\mathcal{E}(r)_d$ as degree $d$ component. The differential on $\mathcal{E}(r)$ is defined by
    $$d(w_0,\ldots,w_d)=\sum_{i=0}^d(-1)^{i}(w_0,\ldots,\hat{w_i},\ldots,w_d).$$
\end{defi}

We also have an action of $\Sigma _r$ on $\mathcal{E}(r)$ given by the diagonal action and the left translation of $\Sigma _r$ on itself.

\begin{prop}
    The symmetric sequence $\mathcal{E}$ is an operad called the \normalfont{Barratt-Eccles operad}.
\end{prop}

We refer to \cite[$\mathsection$1.1]{fresseberger} for an explicit description of the composition product. We have an operad morphism $\mathcal{E}\longrightarrow\mathcal{C}om$ obtained by sending each degree $0$ element to $1$, and sending each non-degree $0$ element to $0$. This morphism is a weak equivalence arity-wise.

\begin{remarque}
    The operad $\mathcal{E}$ has the structure of a {\normalfont Hopf operad}. Namely, we have an operad morphism $\Delta_\mathcal{E}:\mathcal{E}\longrightarrow\mathcal{E}\underset{\text{\normalfont H}}{\otimes}\mathcal{E}$ defined by
$$\Delta_\mathcal{E}(w_0,\ldots,w_d)=\sum_{k=0}^d(w_0,\ldots,w_k)\otimes (w_k,\ldots,w_d).$$
\end{remarque}

We now aim to define the surjection operad $\chi$.

\begin{defi}
    Let $r,d\geq 0$. A surjective map $u:\llbracket 1,r+d\rrbracket\longrightarrow\llbracket 1,r\rrbracket$ is {\normalfont degenerate} if $u(i)=u(i+1)$ for some $i\in\llbracket 1,r+d-1\rrbracket$. We let $\chi(r)_d$ to be the $\mathbb{K}$-module spanned by non-degenerate surjective maps from $\llbracket 1,r+d\rrbracket$ to $\llbracket 1,r\rrbracket$.
\end{defi}

In practice, we represent a surjection $u:\llbracket 1,r+d\rrbracket\longrightarrow\llbracket 1,r\rrbracket$ by a sequence of values:
$$(u(1)\cdots u(r+d)).$$


\begin{defi}
    Let $u\in\chi(r)_d$. An integer $k\in\llbracket 1,r+d\rrbracket$ is called a {\normalfont caesura} if $u(k)$ does not represent the last occurrence of its value in $u$.
\end{defi}

We sometimes represent a surjection by its \textit{table arrangement}, which is defined as follows. Let $u\in\chi(r)_d$. We cut $u$ at the caesuras, in the sense that we set
$$u=(u_0(1)\cdots u_0(r_0))\cdots(u_d(1)\cdots u_d(r_d)),$$

\noindent where $\sum_i r_i=r+d$, and where $u_0(r_0),\ldots,u_{d-1}(r_{d-1})$ represent caesuras of $u$. We then write $u$ as
$$u=\left|\begin{array}{ccc}
         u_0(1) & \cdots & u_0(r_0) \\
         \vdots &  &\vdots \\
         u_d(1) & \cdots & u_d(r_d)\\
    \end{array}\right..$$



We have an obvious action of $\Sigma _r$ on $\chi(r)_d$ given by the pre-composition.

\begin{prop}[{see \cite[$\mathsection$1.2]{fresseberger}}]
The symmetric sequence $\chi$ is endowed with the structure of a symmetric operad and is called the {\normalfont surjection operad}.
\end{prop}

In fact, the surjection operad $\chi$ is a quotient of the Barratt-Eccles operad $\mathcal{E}$. The quotient map is called the \textit{table reduction morphism}.

\begin{prop}\label{tablered}
    There exists an operad morphism $TR:\mathcal{E}\longrightarrow\chi$ called the table reduction morphism which is surjective arity-wise.
\end{prop}

We refer to \cite{fresseberger} for more details on the morphism $TR$. We only recall its definition. Let $w=(w_0,\ldots,w_d)\in\mathcal{E}(r)_d$. We set
$$TR(w)=\sum_{r_0+\cdots+r_d=r+d}\left|\begin{array}{cccc}
    w_0'(1) & \cdots & w_0'(r_0-1) & w_0'(r_0)  \\
    \vdots & & \vdots & \vdots\\
    w_d'(1) & \cdots & w_d'(r_d-1) & w_d'(r_d)
\end{array}\right.$$

\noindent where each row $w_i'(1)\cdots w_i'(r_i)$ represents the first $r_i$ integers occurring in the permutation $w_i$ such that the values $w_i'(1)\cdots w_i'(r_i-1)$ do not occur in
$$\left|\begin{array}{ccc}
    w_0'(1) & \cdots & w_0'(r_0-1) \\
    \vdots & & \vdots\\
    w_{i-1}'(1) & \cdots & w_{i-1}'(r_i-1)
\end{array}\right..$$

An important example of $\chi$-algebra is given by the normalized cochain complex associated to a simplicial set.

\begin{defi}
    Let $X$ be a simplicial set and, for every $k\geq 0$, let $C_k(X)$ be the $\mathbb{K}$-module spanned by $X_k$. We define a differential on $C_*(X)$ by setting, for every $x\in X_k$,
    $$d(x)=\sum_{i=0}^k (-1)^id_i(x),$$
    
    \noindent where we denote by $d_0,\ldots,d_k:X_k\longrightarrow X_{k-1}$ the face maps. We then set
    $$N_k(X)=C_k(X)/\left(\sum_{i=0}^{k}s_i C_{k-1}(X)\right),$$

    \noindent where we denote by $s_0,\ldots,s_{k-2}:X_{k-1}\longrightarrow X_k$ the degeneracy maps. The dg $\mathbb{K}$-module $N_*(X)$ is called the {\normalfont normalized chain complex of $X$}. Its dual dg $\mathbb{K}$-module, denoted by $N^*(X)$, is called the {\normalfont normalized cochain complex of $X$.}
\end{defi}

Note that $N_*$ and $N^*$ are functors from $\text{\normalfont sSet}$ to $\text{dgMod}_\mathbb{K}$.

\begin{thm}[{\cite[$\mathsection$2]{fresseberger}}]
    Let $X\in\text{\normalfont sSet}$. Then $N_*(X)$ is a $\chi$-coalgebra, given by the interval cut operations, which is natural in $X$. As a consequence, the dg $\mathbb{K}$-module $N^*(X)$ is endowed with the structure of a $\chi$-algebra.
\end{thm}

We refer to \cite[$\mathsection$2.2.1,$\mathsection$2.2.4]{fresseberger} for an explicit description of the interval cut operations. In particular, for every simplicial set $X$, the dg $\mathbb{K}$-module $N^*(X)$ is endowed with the structure of a $\mathcal{E}$-algebra through the table reduction morphism $TR:\mathcal{E}\longrightarrow\chi$. In this memoir, we mostly consider the case $X=\Delta^n$ for some $n\geq 0$. The elements of $N_d(\Delta^n)$ are linear combination of non-decreasing sequences $a_0<\cdots<a_d$ of integers in $\llbracket 1,n\rrbracket$, which we denote by $\underline{a_0\cdots a_d}$. The normalized chain complex of $\Delta^n$ has the following fundamental property.

\begin{prop}\label{homonstar}
Let $n\geq 0$ and $0\leq k\leq n$. Then there exists a deformation retract
\[\begin{tikzcd}
	{N_*(\Delta^n)} & {N_*(\Delta^0),}
	\arrow["{h_n^k}", from=1-1, to=1-1, loop, in=145, out=215, distance=10mm]
	\arrow["{p_n}", shift left, from=1-1, to=1-2]
	\arrow["{i_n^k}", shift left, from=1-2, to=1-1]
\end{tikzcd}\]

\noindent where $i_n^k:N_*(\Delta^0)\longrightarrow N_*(\Delta^n)$ is the morphism which sends $\underline{0}$ to $\underline{k}$, and $p_n:N_*(\Delta^n)\longrightarrow N_*(\Delta^0)$ is the morphism which sends every vertex to $\underline{0}$.
\end{prop}

The claim is that we have the identities
\begin{center}
    $\begin{array}{rll}
         p_ni_n^k & = & id_{N_*(\Delta^0)};\\
         id_{N_*(\Delta^n)}-i_n^kp_n & = & dh_n^k+h_n^kd. 
    \end{array}$
\end{center}

\noindent We set $\varphi_n^k=i_n^kp_n$. The homotopy $h_n^k$ can be explicitly defined as follows. Let $\underline{a_0\cdots a_r}\in N_r(\Delta^n)$ be a non-zero element. If this sequence contains $k$, then we set $h_n^k(\underline{a_0\cdots a_r})=0$. Otherwise we set
$$h_n^k(\underline{a_0\cdots a_r})=(-1)^i \underline{a_0 \cdots \overset{i}{k}\cdots a_r},$$

\noindent where $i$ is the unique possible position to insert $k$ in $\underline{a_0\cdots a_r}$ so that we have a non decreasing sequence of integers.\\

\noindent By taking linear duals, we have a similar deformation retract on $N^*(\Delta^n)$. We will keep the same notation $h_n^k:N^{*}(\Delta^n)\longrightarrow N^{*-1}(\Delta^n)$ and $\varphi_n^k:N^*(\Delta^n)\longrightarrow N^*(\Delta^n)$ for the linear duals of $h_n^k:N_*(\Delta^n)\longrightarrow N_{*+1}(\Delta^n)$ and $\varphi_n^k:N_*(\Delta^n)\longrightarrow N_*(\Delta^n)$.\\

The dg $\mathbb{K}$-module $I=N^*(\Delta^1)$ can be used to model intervals. We indeed have a decomposition of the diagonal map $\Delta:\mathbb{K}\longrightarrow\mathbb{K}^2$ as

\[\begin{tikzcd}
	{\mathbb{K}=N^*(\Delta^0)} & {N^*(\Delta^1)} & {N^*(\Delta^0)\times N^*(\Delta^0)=\mathbb{K}^2}
	\arrow["\sim", tail, from=1-1, to=1-2]
	\arrow["{s_0}"', tail, from=1-1, to=1-2]
	\arrow["\Delta", curve={height=-18pt}, from=1-1, to=1-3]
	\arrow["{(d_0,d_1)}"', two heads, from=1-2, to=1-3]
\end{tikzcd}\]

\noindent where $s_0=(p_1)^\vee$ and $d_0=(i_1^0)^\vee$, $d_1=(i_1^1)^\vee$. We can lift such a diagram in the category of $\mathcal{P}\underset{\text{\normalfont H}}{\otimes}\mathcal{E}$-algebras for any operad $\mathcal{P}$ to get a construction of a path-object. Recall that a path objet for a $\mathcal{P}\underset{\text{\normalfont H}}{\otimes}\mathcal{E}$-algebra $R$ is a $\mathcal{P}\underset{\text{\normalfont H}}{\otimes}\mathcal{E}$-algebra $R^I$ such that the diagonal map $\Delta:R\longrightarrow R\times R$ can be described as a composite
\[\begin{tikzcd}
	R & {R^I} & {R\times R}
	\arrow["\sim", tail, from=1-1, to=1-2]
	\arrow["{s_0}"', tail, from=1-1, to=1-2]
	\arrow["\Delta", curve={height=-18pt}, from=1-1, to=1-3]
	\arrow["{(d_0,d_1)}"', two heads, from=1-2, to=1-3]
\end{tikzcd}.\]



\begin{prop}[{see \cite[$\mathsection$3.1.4,$\mathsection$3.1.9]{fresseberger}}]\label{pathobject} 
    Let $\mathcal{P}$ be an operad, and $R$ be a $\mathcal{P}\underset{\text{\normalfont H}}{\otimes}\mathcal{E}$-algebra. Then
    $$R^I=R\otimes N^*(\Delta^1)$$

    \noindent is a path objet in the category of $\mathcal{P}\underset{\text{\normalfont H}}{\otimes}\mathcal{E}$-algebras. The $\mathcal{P}\underset{\text{\normalfont H}}{\otimes}\mathcal{E}$-algebra structure on $R^I$ is given by the composite
\[\begin{tikzcd}
	{\mathcal{P}\underset{\text{\normalfont H}}{\otimes}\mathcal{E}} & {\mathcal{P}\underset{\text{\normalfont H}}{\otimes}\mathcal{E}\underset{\text{\normalfont H}}{\otimes}\mathcal{E}} & {\text{\normalfont End}_{R\otimes N^*(\Delta^1)}}
	\arrow["{id\otimes\Delta_\mathcal{E}}", from=1-1, to=1-2]
	\arrow[from=1-2, to=1-3]
\end{tikzcd},\]

\noindent where we use the $\mathcal{P}\underset{\text{\normalfont H}}{\otimes}\mathcal{E}$-algebra structure on $R$, and the $\mathcal{E}$-algebra structure on $N^*(\Delta^1)$.
\end{prop}

\subsection{Appendix: basic results on permutations}\label{sec:211bis}

In this appendix, we recall basic definitions and notations on permutations and the symmetric groups. Our conventions will follow those given in \cite[$\mathsection$1.1.7]{fressepart1}. Let $n\geq 0$. We denote by $\Sigma_n$ the symmetric group on the elements $1,\ldots, n$. For every $m,n\geq 0$, we denote by $\llbracket m,n\rrbracket$ the set of integers $k$ such that $m\leq k\leq n$. We denote by $id$ the relevant identity permutation, and we write any permutation $\sigma\in\Sigma_n$ as its sequence of values $(\sigma(1)\cdots\sigma(n))$.\\

For every $p,q\in\mathbb{N}$ and $\sigma\in\Sigma_{p},\tau\in\Sigma_{q}$, we let $\sigma\oplus\tau\in\Sigma_{p+q}$ to be the permutation defined, for every $1\leq i\leq p+q$, by 
$$(\sigma\oplus\tau)(i)=\left\{\begin{array}{cl}
    \sigma(i) & \text{if }1\leq i\leq p  \\
     \tau(p+i) & \text{if }p+1\leq i\leq p+q
\end{array}\right..$$

\noindent The operation $\oplus$ is associative in $\bigsqcup_{n\geq 0}\Sigma_n$, so that we can generalize the definition of $\oplus$ to a direct sum of $k\geq 1$ permutations $\sigma_{1}\oplus \cdots\oplus\sigma_{k}$.\\

Let $r_{1},\ldots,r_{n}\geq 0$ and $\sigma\in\Sigma_n$. We set $\textbf{r}_i=r_1+\cdots+r_{i-1}+1<\cdots <r_1+\cdots+r_{i-1}+r_i$. We define the \text{\normalfont \textit{block permutation} induced by $\sigma$ of type $(r_1,\ldots,r_r)$} by
$$\sigma_\ast(r_1,\ldots,r_n)=\textbf{r}_{\sigma(1)}\cdots\textbf{r}_{\sigma(n)}.$$

\begin{lm}[{\cite[Proposition 1.1.8]{fressepart1}}]\label{compsymmetric}
Let $\sigma\in\Sigma_n$ and  $\tau_{1}\in\Sigma_{r_{1}},\ldots,\tau_{n}\in\Sigma_{r_{n}}$. Then
$$(\tau_1\oplus\cdots\oplus\tau_n)\cdot\sigma_{*}(r_{1},\ldots,r_{n})=\sigma_{*}(r_{1},\ldots,r_{n})\cdot(\tau_{\sigma(1)}\oplus\cdots\oplus\tau_{\sigma(n)}).$$
\end{lm}

Let $\sigma\in\Sigma_{n}$ and $\tau_{1}\in\Sigma_{r_{1}},\ldots,\tau_{n}\in\Sigma_{r_{n}}$. We define the permutation $\sigma(\tau_{1},\ldots,\tau_{n})\in\Sigma_{r_{1}+\cdots+r_{n}}$ by
$$\sigma(\tau_{1},\ldots,\tau_{n})=(\tau_{1}\oplus\cdots\oplus\tau_{n})\cdot\sigma_{*}(r_{1},\ldots,r_{n}).$$

In operads theory, one needs a set of representatives of the quotient $\Sigma_m/\Sigma_{r_1}\times\cdots\times\Sigma_{r_n}$ for every $r_1,\ldots,r_n\geq 0$ such that $r_1+\cdots+r_n=m$. This leads us to the notion of \textit{shuffle permutation}. A $(r_{1},\ldots,r_{n})$-\text{\normalfont shuffle permutation} is a permutation in $\Sigma_{r_1+\cdots+r_n}$ which preserves the order on each block $\textbf{r}_1,\ldots,\textbf{r}_n$. We denote by $Sh(r_1,\ldots,r_n)$ the set of such permutations.\\ A shuffle permutation $\omega\in Sh(r_1,\ldots,r_n)$ is \textit{pointed} if it satisfies $\omega(1)<\omega(r_1+1)<\cdots<\omega(r_1+\cdots+r_{n-1}+1)$. We denote by $Sh_\ast(r_1,\ldots,r_n)$ the set of such permutations.\\

The following results allow us to write any permutations in terms of a product of a shuffle permutation with a composite of a direct sum and a block permutation.

\begin{prop}\label{shuffle}
    Let $n\geq 0$ and $r_1,\ldots,r_n\geq 0$.
    
    \begin{itemize}
        \item Every $\sigma\in\Sigma_{r_1+\cdots+r_n}$ admits a unique decomposition of the form
    $$\sigma=\omega\cdot(\tau_1\oplus\cdots\oplus\tau_n)$$
    \noindent where $\tau_i\in\Sigma_{r_i}$ and $\omega\in Sh(r_1,\ldots,r_n)$.
    \item Every $\sigma\in\Sigma_{r_1+\cdots+r_n}$ admits a unique decomposition of the form
    $$\sigma=\omega\cdot\sigma(\tau_1,\ldots,\tau_n)$$
    \noindent where $\tau_i\in\Sigma_{r_i},\sigma\in\Sigma_n$ and $\omega\in Sh_\ast(r_1,\ldots,r_n)$.
    \end{itemize}
\end{prop}

\section{On $\mathcal{P}re\mathcal{L}ie_\infty$-algebras with divided powers}

In this section, we study the structure of $\Gamma(\mathcal{P}re\mathcal{L}ie_\infty,-)$-algebras. The operad $\mathcal{P}re\mathcal{L}ie_\infty$ and its algebras have been explicitly described in \cite{chapoton}, using the computation of the Koszul dual operad of $\mathcal{P}re\mathcal{L}ie$ given in \cite{chapotonperm}.\\

In $\mathsection$\ref{sec:221}, we recall this explicit construction of $\mathcal{P}re\mathcal{L}ie_\infty$. We also focus on the characterization of the structure of a $\mathcal{P}re\mathcal{L}ie_\infty$-algebra as an algebraic structure on the suspension which we call a $\Lambda\mathcal{PL}_\infty$-algebra.\\

In $\mathsection$\ref{sec:222}, we define the notion of a $\Gamma\Lambda\mathcal{PL}_\infty$-algebra which will be the analogue, in the divided power framework, of a $\Lambda\mathcal{PL}_\infty$-algebra, and we define a notion of $\infty$-morphism of $\Gamma\Lambda\mathcal{PL}_\infty$-algebras.\\

In $\mathsection$\ref{sec:223}, we define the symmetric weighted braces associated to a $\Gamma\Lambda\mathcal{PL}_\infty$-algebra and the notion of a Maurer-Cartan element in the complete framework. We also prove that giving the structure of a $\Gamma\Lambda\mathcal{PL}_\infty$-algebra is equivalent to giving symmetric brace operations.\\

In $\mathsection$\ref{sec:224}, we prove that giving a structure of a $\Gamma(\mathcal{P}re\mathcal{L}ie_\infty,-)$-algebra is equivalent to giving a structure of a $\Gamma\Lambda\mathcal{PL}_\infty$-algebra on the suspension.

\subsection{Recollections on pre-Lie algebras up to homotopy}\label{sec:221}

We begin this section by some recollections on the operad $\text{Perm}$, which was introduced by Chapoton in \cite{chapotonperm}. Let $\text{\normalfont Perm}(n)=\mathbb{K}^n$. We denote by $(e_i^n)_{1\leq i\leq n}$ the canonical basis of $\text{\normalfont Perm}(n)$. The group $\Sigma _n$ acts on $\text{\normalfont Perm}(n)$ by 
$$\sigma \cdot e_i^n=e_{\sigma ^{-1}(i)}^n.$$

\begin{prop}\label{compperm}
    The symmetric sequence $\text{\normalfont Perm}$ is an operad with as compositions
    $$e_i^n(e_{j_1}^{n_1},\ldots,e_{j_k}^{n_k})=e_{n_1+\cdots+n_{i-1}+j_i}^{n_1+\cdots+n_k}.$$
\end{prop}

\begin{thm}[{see \cite{chapotonperm}}] The operad $\mathcal{P}re\mathcal{L}ie$ is Koszul and its Koszul dual operad is $\mathcal{P}re\mathcal{L}ie^{!}=\text{\normalfont Perm}.$
\end{thm}

This theorem implies that the operad $\mathcal{P}re\mathcal{L}ie_{\infty}=B^c(\Lambda^{-1} \text{\normalfont Perm}^\vee)$ gives a model for $\mathcal{P}re\mathcal{L}ie$-algebras up to homotopy. Such algebras have been described by Chapoton and Livernet in \cite{chapoton}. We recall this description in the following paragraphs. Let $V\in\text{\normalfont gMod}_\mathbb{K}$. We set
$$\mathcal{S}(V):=\bigoplus_{n\geq 0}(V^{\otimes n})_{\Sigma _n},$$

\noindent where we consider the usual action of $\Sigma _n$ on $V^{\otimes n}$ by permutation. Note that $\mathcal{S}(V)\simeq\mathbb{K}\oplus\overline{\mathcal{S}}(V)$ with
$$\overline{\mathcal{S}}(V):=\bigoplus_{n\geq 1}(V^{\otimes n})_{\Sigma _n}.$$

\begin{defi}[{see \cite[$\mathsection$2.3]{chapoton}}]
    The {\normalfont free Perm-coalgebra generated by $V$} is the graded $\mathbb{K}$-module $\text{\normalfont Perm}^c(V)=V\otimes\mathcal{S}(V)$ endowed with the following comultiplication:\\
    $\dperm(v_0\otimes 1)=0;$\\
    $\displaystyle\dperm(v_0\otimes v_1\cdots v_n)=\sum_{\substack{0\leq k\leq n-1\\\sigma \in Sh(k,1,n-k-1)}}\pm(v_0\otimes v_{\sigma (1)}\cdots v_{\sigma (k)})\otimes (v_{\sigma (k+1)}\otimes v_{\sigma (k+2)}\cdots v_{\sigma (n)})$\\
    \noindent for every $v_0,\ldots,v_n\in V$, where the sign in the sum is produced by permutations of the factors.
\end{defi}

The coproduct $\dperm$ satisfies the following identities (see \cite[$\mathsection$2.3]{chapoton}):
$$(id\otimes\dperm)\dperm=(\dperm\otimes id)\dperm;$$
$$(id\otimes\dperm)\dperm=(id\otimes\tau)(id\otimes\dperm)\dperm.$$

\begin{remarque}
    Let $\Delta_{\mathcal{S}(V)}:\mathcal{S}(V)\longrightarrow\mathcal{S}(V)\otimes\mathcal{S}(V)$ be the coproduct defined by $\Delta_{\mathcal{S}(V)}(1)=1\otimes 1$ and
    $$\Delta_{\mathcal{S}(V)}(v_1\cdots v_n)=\sum_{k=0}^n\sum_{\sigma \in Sh(k,n-k)}\pm(v_{\sigma (1)}\cdots v_{\sigma (k)})\otimes (v_{\sigma (k+1)}\cdots v_{\sigma (n)})$$

    \noindent for every $v_1,\ldots,v_n\in V$. Then $\dperm$ is given by the composite
\[\begin{tikzcd}
	{V\otimes\mathcal{S}(V)} & {V\otimes\mathcal{S}(V)\otimes\mathcal{S}(V)} & {V\otimes\mathcal{S}(V)\otimes\overline{\mathcal{S}}(V)} & {(V\otimes\mathcal{S}(V))\otimes (V\otimes\mathcal{S}(V))}
	\arrow["{id\otimes\Delta_{\mathcal{S}(V)}}", from=1-1, to=1-2]
	\arrow[two heads, from=1-2, to=1-3]
	\arrow["{id\otimes id\otimes i_V}", from=1-3, to=1-4]
\end{tikzcd}\]

\noindent where $i_V:\overline{\mathcal{S}}(V)\longrightarrow V\otimes\mathcal{S}(V)$ is defined by
$$v_1\cdots v_n\longmapsto \sum_{k=1}^n \pm v_k\otimes v_1\cdots\widehat{v_k}\cdots v_n$$

\noindent for every $v_1,\ldots,v_n\in V$.
\end{remarque}

Let $\pi_V:\text{\normalfont Perm}^c(V)\longrightarrow V$ be the projection on the first factor.

\begin{prop}[{\cite[$\mathsection$2.4]{chapoton}}]\label{isocoder}
    The map
    \begin{center}
        $\begin{array}{cll}
            Coder(\text{\normalfont Perm}^c(V)) & \longrightarrow & \text{\normalfont Hom}( \text{\normalfont Perm}^c(V),V)  \\
            d &\longmapsto & \pi_V\circ d
        \end{array}$
    \end{center}

    \noindent is a bijection.
\end{prop}

\begin{proof}
    We only recall the construction of the inverse bijection $\Psi$. Let $l\in\text{\normalfont Hom}(\text{\normalfont Perm}^c(V),V)$. We define $\Psi(l)\in Coder(\text{\normalfont Perm}^c(V))$ as the sum of the composite
\[\begin{tikzcd}
	{\Psi_1(l):V\otimes\mathcal{S}(V)} & {V\otimes\mathcal{S}(V)\otimes\mathcal{S}(V)} & {V\otimes\mathcal{S}(V)}
	\arrow["{id\otimes\Delta_{\mathcal{S}(V)}}", from=1-1, to=1-2]
	\arrow["{l\otimes id}", from=1-2, to=1-3]
\end{tikzcd}\]

\noindent and of the composite
\[\begin{tikzcd}
	{{\Psi_2}({l}):V\otimes\mathcal{S}(V)} & {(V\otimes\mathcal{S}(V))\otimes (V\otimes\mathcal{S}(V))} & {V\otimes\mathcal{S}(V)\otimes V} & {V\otimes\mathcal{S}(V)}
	\arrow["\dperm", from=1-1, to=1-2]
	\arrow["{id\otimes {l}}", from=1-2, to=1-3]
	\arrow[two heads, from=1-3, to=1-4]
\end{tikzcd}\]

\noindent where the last morphism is given by the projection from $\mathcal{S}(V)\otimes V$ to $\mathcal{S}(V)$. One can check that we retrieve the definition given in the proof of \cite[$\mathsection$2.4]{chapoton}.
\end{proof}


\begin{prop}[{\cite[$\mathsection$2.5]{chapoton}}]\label{caracprelieinf}
   Let $L\in\text{\normalfont gMod}_\mathbb{K}$. Giving a structure of pre-Lie algebra up to homotopy on $L$ is equivalent to giving a degree $-1$ morphism $l\in\text{\normalfont Hom}(\text{\normalfont Perm}^c(\Sigma L),\Sigma  L)$ such that, for every $x,y_1,\ldots,y_n\in\Sigma L$, we have
   \begin{multline*}
       \sum_{i=0}^n\sum_{\sigma \in Sh(i,n-i)}\pm l(l(x\otimes y_{\sigma(1)}\cdots y_{\sigma(i)})\otimes y_{\sigma(i+1)}\cdots y_{\sigma(n)})\\+\sum_{i=0}^n\sum_{\sigma \in Sh(1,i,n-i-1)}\pm l(x\otimes l(y_{\sigma(1)}\otimes y_{\sigma(2)} \cdots y_{\sigma(i+1)})\cdot y_{\sigma(i+2)}\cdots y_{\sigma(n)})=0
   \end{multline*}
   
    \noindent where the signs $\pm$ are produced by the permutations of the elements $y_1,\ldots,y_n$.
\end{prop}

In particular, if $L$ is a $\mathcal{P}re\mathcal{L}ie$-algebra up to homotopy, then $\Sigma  L$ is endowed with a differential $d$ given by the restriction $l:\Sigma  L\longrightarrow\Sigma  L$.\\

In the following, we adopt the following notation:
$$x\llbrace y_1,\ldots,y_n\rrbrace:= l(x\otimes y_1\cdots y_n).$$

\noindent We call such operations the \textit{symmetric braces} associated to the $\mathcal{P}re\mathcal{L}ie_\infty$-algebra $L$.

\begin{remarque}\label{prelieinftoprelie}
We have an operad morphism $\mathcal{P}re\mathcal{L}ie_\infty\longrightarrow\mathcal{P}re\mathcal{L}ie$ which sends $e_1^2$ to the pre-Lie product, and the other $e_1^n$'s to $0$. Thus, every $\mathcal{P}re\mathcal{L}ie$-algebra has a canonical $\mathcal{P}re\mathcal{L}ie_\infty$-algebra structure. Beware that the symmetric braces in pre-Lie algebras have nothing to do with the symmetric braces in $\mathcal{P}re\mathcal{L}ie_\infty$-algebras.
\end{remarque}

Equivalently, Proposition \ref{caracprelieinf} asserts that giving a structure of a $\mathcal{P}re\mathcal{L}ie_\infty$-algebra on $L$ is equivalent to giving a coderivation $Q\in Coder(\text{\normalfont Perm}^c(\Sigma  L))$ such that $Q^2=0$. 

\begin{defi}
    We define the category $\Lambda\mathcal{PL}_\infty$ with as set of objects the pairs $(V,Q)$, where $V\in\text{\normalfont gMod}_\mathbb{K}$ and $Q\in Coder(\text{\normalfont Perm}^c(V))$ is a degree $-1$ element such that $Q^2=0$. A morphism $\phi:(V,Q)\longrightarrow (V',Q')$ in $\Lambda\mathcal{PL}_\infty$ is a morphism of coalgebras $\phi:\text{\normalfont Perm}^c(V)\longrightarrow \text{\normalfont Perm}^c(V')$ which preserves the coderivations $Q$ and $Q'$.
\end{defi}

Usually, a morphism in $\Lambda\mathcal{PL}_\infty$ from $(V,Q)$ to $(V',Q')$ is denoted by $\phi:V\rightsquigarrow V'$ and is called an $\infty$\textit{-morphism}.

\begin{thm}\label{caracgammaprelieinf}
    A dg $\mathbb{K}$-module $L$ is a $\mathcal{P}re\mathcal{L}ie_\infty$-algebra if and only if $\Sigma  L\in\Lambda\mathcal{PL}_\infty$. Moreover, any morphism of $\mathcal{P}re\mathcal{L}ie_\infty$-algebras $\phi:L\longrightarrow L'$ gives rise to a morphism $\Sigma \phi:\Sigma L\longrightarrow\Sigma  L'$ in $\Lambda\mathcal{PL}_\infty$ which preserves the symmetric braces.
\end{thm}

\subsection{The category $\Gamma\Lambda {\mathcal{PL}}_\infty$}\label{sec:222}

In this subsection, we aim to define an analogue of the category $\Lambda\mathcal{PL}_\infty$, denoted by $\Gamma\Lambda\mathcal{PL}_\infty$, which will characterize the $\Gamma(\mathcal{P}re\mathcal{L}ie_\infty,-)$-algebras as in Theorem \ref{caracgammaprelieinf}.\\

Let $V$ be a graded $\mathbb{K}$-module. We define $\Gamma(V)$ by
$$\Gamma(V):=\bigoplus_{n\geq 0}(V^{\otimes n})^{\Sigma _n}.$$

\noindent We have $\Gamma(V)\simeq\mathbb{K}\oplus\overline{\Gamma}(V)$ with
$$\overline{\Gamma}(V):=\bigoplus_{n\geq 1}(V^{\otimes n})^{\Sigma _n}.$$

\noindent Note that we have a morphism $Tr:\mathcal{S}(V)\longrightarrow\Gamma(V)$ called the \textit{trace map} and defined by $Tr(1)=1$ and
$$Tr(v_1\cdots v_n)=\sum_{\sigma\in\Sigma_n}\pm v_{\sigma(1)}\otimes\cdots\otimes v_{\sigma(n)}$$

\noindent for every $v_1,\ldots,v_n\in V$ and $n\geq 1$. 




\begin{defi}
    For every $V\in \text{\normalfont gMod}_\mathbb{K}$, we set
    $$\Gamma \text{\normalfont Perm}^c(V):=V\otimes\Gamma(V).$$
\end{defi}

Our goal is to construct a coproduct $\dgperm:\Gamma\text{\normalfont Perm}^c(V)\longrightarrow\Gamma\text{\normalfont Perm}^c(V)\otimes\Gamma\text{\normalfont Perm}^c(V)$ which is compatible, in some sense, with the coproduct $\dperm:\text{\normalfont Perm}^c(V)\longrightarrow\text{\normalfont Perm}^c(V)\otimes\text{\normalfont Perm}^c(V)$. For this purpose, we first consider the \textit{tensor algebra} generated by $V$:
$${T}(V):=\bigoplus_{n\geq 0}V^{\otimes n}.$$

\noindent We have a coproduct $\Delta_{T(V)}:{T}(V)\longrightarrow {T}(V)\otimes {T}(V)$ defined by $\Delta_{T(V)}(1):=1\otimes 1$ and
$$\Delta_{T(V)}(v_1\otimes\cdots\otimes v_n):=\sum_{k=0}^{n} (v_1\otimes\cdots\otimes v_k)\otimes (v_{k+1}\otimes\cdots\otimes v_n).$$

\noindent for every $v_1,\ldots,v_n\in V$. Consider 
$$\overline{T}(V):=\bigoplus_{n\geq 1}V^{\otimes n}.$$

\noindent We also have a coproduct on $\overline{T}(V)$, defined, for every $v_1,\ldots,v_n\in V$, by $\Delta_{\overline{T}(V)}(v_1):=0$ and,
$$\Delta_{\overline{T}(V)}(v_1\otimes\cdots\otimes v_n):=\sum_{k=1}^{n-1}(v_1\otimes\cdots\otimes v_k)\otimes (v_{k+1}\otimes\dots\otimes v_n).$$

\noindent We embed $V\otimes\Gamma(V)\subset \overline{T}(V)$. Note that
$$\Delta_{\overline{T}(V)}(V\otimes\Gamma(V))\subset V\otimes\Gamma(V)\otimes\overline{\Gamma}(V).$$

\noindent By applying the embedding $(V^{\otimes n})^{\Sigma _n}\subset V\otimes (V^{\otimes n-1})^{\Sigma _{n-1}}$ for each $n\geq 2$, we have the inclusion $\overline{\Gamma}(V)\subset V\otimes\Gamma(V)$. We thus have obtained a coproduct
$$\dgperm:\Gamma\text{\normalfont Perm}^c(V)\longrightarrow\Gamma\text{\normalfont Perm}^c(V)\otimes\Gamma\text{\normalfont Perm}^c(V).$$

\noindent We can also identify $\dgperm$ with the composite
\[\begin{tikzcd}
	{\dgperm:V\otimes\Gamma(V)} & {V\otimes\Gamma(V)\otimes\Gamma(V)} \\
	& {V\otimes\Gamma(V)\otimes\overline{\Gamma}(V)} & {(V\otimes\Gamma(V))\otimes (V\otimes\Gamma(V))}
	\arrow["{id\otimes\Delta_{T(V)}}", from=1-1, to=1-2]
	\arrow[two heads, from=1-2, to=2-2]
	\arrow[hook, from=2-2, to=2-3]
\end{tikzcd}.\]

\begin{lm}\label{formulecoprod}
    The morphism $\dgperm:\Gamma\text{\normalfont Perm}^c(V)\longrightarrow\Gamma\text{\normalfont Perm}^c(V)\otimes\Gamma\text{\normalfont Perm}^c(V)$ satisfies the identities:
    $$(id\otimes\dgperm)\dgperm=(\Delta\otimes id)\dgperm;$$
$$(id\otimes\dgperm)\dgperm=(id\otimes \tau)(id\otimes\dgperm)\dgperm.$$

Moreover, we have the following commutative diagram:
\begin{center}
    $\begin{tikzcd}
        \text{\normalfont Perm}^c(V)\arrow[r, "\dperm"]\arrow[d, "id\otimes Tr"'] & \text{\normalfont Perm}^c(V)\otimes \text{\normalfont Perm}^c(V)\arrow[d, "(id\otimes Tr)\otimes (id\otimes Tr)"]\\
        \Gamma \text{\normalfont Perm}^c(V)\arrow[r, "\dgperm"'] & \Gamma \text{\normalfont Perm}^c(V)\otimes\Gamma \text{\normalfont Perm}^c(V)
    \end{tikzcd}$
\end{center}
\end{lm}

\begin{proof}
    The proof of this lemma comes from straightforward computations.
\end{proof}




\begin{remarque}\label{idcoprod}
    The relation $(id\otimes\dgperm)\dgperm=(id\otimes\tau)(id\otimes\dgperm)\dgperm$ implies that, for every $k\geq 1$,
    $$(\dgperm)^k(\Gamma\text{\normalfont Perm}^c(V))\subset\Gamma\text{\normalfont Perm}^c(V)\otimes(\Gamma\text{\normalfont Perm}^c(V)^{\otimes k})^{\Sigma_k}.$$

    \noindent As a consequence, since $(\Delta_{\overline{T}(V)})^k$ reduces to the identity on $V^{\otimes k+1}$, and by definition of $\dgperm$, we have the following commutative diagram:
\[\begin{tikzcd}
	{\overline{T}(V)} & {\overline{T}(V)^{\otimes k+1}} & {V\otimes V^{\otimes k}} \\
	{\Gamma\text{\normalfont Perm}^c(V)} & {\Gamma\text{\normalfont Perm}^c(V)\otimes(\Gamma\text{\normalfont Perm}^c(V)^{\otimes k})^{\Sigma_k}} & {V\otimes (V^{\otimes k})^{\Sigma_k}}
	\arrow["{(\Delta_{\overline{T}(V)})^k}", from=1-1, to=1-2]
	\arrow["{\pi_{V^{\otimes k+1}}}", curve={height=-24pt}, from=1-1, to=1-3]
	\arrow["{\pi_V^{\otimes k+1}}", from=1-2, to=1-3]
	\arrow[hook, from=2-1, to=1-1]
	\arrow["{(\dgperm)^k}", from=2-1, to=2-2]
	\arrow["{\pi_{V^{\otimes k+1}}}"', curve={height=24pt}, from=2-1, to=2-3]
	\arrow[hook, from=2-2, to=1-2]
	\arrow["{\pi_V^{\otimes k+1}}", from=2-2, to=2-3]
	\arrow[hook, from=2-3, to=1-3]
\end{tikzcd},\]

\noindent where, for every $k\geq 0$, we denote by $\pi_{V^{\otimes k}}:\overline{T}(V)\longrightarrow V^{\otimes k}$ the projection onto $V^{\otimes k}$.
\end{remarque}

\begin{defi}
    An endomorphism $d$ of $\Gamma \text{\normalfont Perm}^c(V)$ is called a {\normalfont coderivation} if it satisfies
    $$\dgperm d=(d\otimes id+ id\otimes d)\dgperm.$$

    We let $Coder(\Gamma\text{\normalfont Perm}^c(V))$ to be the $\mathbb{K}$-module spanned by coderivations.
\end{defi}

Our goal is to prove that any coderivation is characterized by its composite with $\pi_V$. We rely on the following definition.

\begin{defi}\label{defsh}
    Let $w,v_1,\ldots,v_{r}\in V$. We define $\text{\normalfont Sh}:T(V)\otimes V\longrightarrow T(V)$ by
    $$\text{\normalfont Sh}(v_1\otimes\cdots\otimes v_n;w)=\sum_{i=0}^{n}\pm v_1\otimes\cdots\otimes v_i\otimes{w}\otimes v_{i+1}\otimes\cdots\otimes v_n$$

    \noindent where the sign is given by the permutation $v_1\otimes\cdots\otimes v_n\otimes w\mapsto\pm v_1\otimes\cdots\otimes v_i\otimes{w}\otimes v_{i+1}\otimes\cdots\otimes v_n$ for every $0\leq i\leq n$. We also define analogously $\text{\normalfont Sh}:V\otimes T(V)\longrightarrow T(V)$.
\end{defi}

We immediately see that $\text{\normalfont Sh}(\Gamma(V)\otimes V)\subset\Gamma(V)$.

\begin{prop}\label{bijendo}
    The map
    \begin{center}
        $\begin{array}{cll}
            Coder(\Gamma \text{\normalfont Perm}^c(V)) & \longrightarrow & \text{\normalfont Hom}(\Gamma \text{\normalfont Perm}^c(V),V)  \\
            d &\longmapsto & \pi_V\circ d
        \end{array}$
    \end{center}

    \noindent is a bijection. We denote by $\widetilde{\Psi}$ its inverse. If we consider the inverse $\Psi$ given in the proof of Proposition \ref{isocoder}, then $\widetilde{\Psi}$ is compatible with $\Psi$ in the following sense. Let $\widetilde{l}\in\text{\normalfont Hom}(\Gamma \text{\normalfont Perm}^c(V),V)$. We define $l\in\text{\normalfont Hom}(\text{\normalfont Perm}^c(V),V)$ by the composite
\[\begin{tikzcd}
	{l:\text{\normalfont Perm}^c(V)} & {\Gamma\text{\normalfont Perm}^c(V)} & V
	\arrow["{id\otimes Tr}", from=1-1, to=1-2]
	\arrow["{\widetilde{l}}", from=1-2, to=1-3]
\end{tikzcd}.\]

\noindent Then the following diagram is commutative:
\[\begin{tikzcd}
	{\Gamma\text{\normalfont Perm}^c(V)} & {\Gamma\text{\normalfont Perm}^c(V)} \\
	{\text{\normalfont Perm}^c(V)} & {\text{\normalfont Perm}^c(V)}
	\arrow["{\widetilde{\Psi}(\widetilde{l})}", from=1-1, to=1-2]
	\arrow["{id\otimes Tr}", from=2-1, to=1-1]
	\arrow["{\Psi(l)}"', from=2-1, to=2-2]
	\arrow["{id\otimes Tr}"', from=2-2, to=1-2]
\end{tikzcd}.\]
\end{prop}

\begin{proof}
    Let $\widetilde{l}:\Gamma \text{\normalfont Perm}^c(V)\longrightarrow V$ be a morphism. We define an endomorphism $\widetilde{\Psi}(\widetilde{l})$ of $\Gamma \text{\normalfont Perm}^c(V)$ by the sum of the composite
\[\begin{tikzcd}
	{\widetilde{\Psi}_1(\widetilde{l}):V\otimes\Gamma(V)} & {V\otimes\Gamma(V)\otimes\Gamma(V)} & {V\otimes\Gamma(V)}
	\arrow["{id\otimes\Delta_{T(V)}}", from=1-1, to=1-2]
	\arrow["{\widetilde{l}\otimes id}", from=1-2, to=1-3]
\end{tikzcd}\]

\noindent and of the composite
\[\begin{tikzcd}
	{\widetilde{\Psi}_2(\widetilde{l}):V\otimes\Gamma(V)} & {(V\otimes\Gamma(V))\otimes (V\otimes\Gamma(V))} & {V\otimes\Gamma(V)\otimes V} && {V\otimes\Gamma(V)}
	\arrow["\dgperm", from=1-1, to=1-2]
	\arrow["{id\otimes id\otimes\widetilde{l}}", from=1-2, to=1-3]
	\arrow["{id\otimes\text{\normalfont Sh}(-;-)}", from=1-3, to=1-5]
\end{tikzcd}.\]

    \noindent Let $x\in V$ and $Y\in\Gamma(V)$. We write the coproduct $\dgperm(x\otimes Y)$ by using the Sweedler notation without the sum symbol, as
    $$\dgperm(x\otimes Y)=(x\otimes Y_{(1)})\otimes (x_{(2)}\otimes Y_{(2)}),$$

    \noindent where $x_{(2)}\in V$ and $Y_{(1)},Y_{(2)}\in\Gamma(V)$. Note that we have
    $$(id\otimes\Delta_{T(V)})(x\otimes Y)=(x\otimes Y_{(1)})\otimes (x_{(2)}\otimes Y_{(2)})+(x\otimes Y)\otimes 1.$$
    
    \noindent Then, by definition of $\widetilde{\Psi}_1(\widetilde{l})$ and $\widetilde{\Psi}_2(\widetilde{l})$, we have
    $$\widetilde{\Psi}_1(\widetilde{l})(x\otimes Y)=\widetilde{l}(x\otimes Y_{(1)})\otimes (x_{(2)}\otimes Y_{(2)})+\widetilde{l}(x\otimes Y)\otimes 1$$

    \noindent and
    $$\widetilde{\Psi}_2(\widetilde{l})(x\otimes Y)=\pm x\otimes\text{\normalfont Sh}(Y_{(1)};\widetilde{l}(x_{(2)}\otimes Y_{(2)})),$$

    \noindent so that 
    $$\widetilde{\Psi}(\widetilde{l})(x\otimes Y)=\widetilde{l}(x\otimes Y_{(1)})\otimes (x_{(2)}\otimes Y_{(2)})+\widetilde{l}(x\otimes Y)\otimes 1\pm x\otimes\text{\normalfont Sh}(Y_{(1)};\widetilde{l}(x_{(2)}\otimes Y_{(2)})).$$
    
    \noindent We set
    $$\dgperm(x\otimes Y_{(1)})=(x\otimes Y_{(11)})\otimes (x_{(1)}\otimes Y_{(12)});$$
    $$\dgperm(x_{(2)}\otimes Y_{(2)})=(x_{(2)}\otimes Y_{(21)})\otimes (x_{(22)}\otimes Y_{(22)}).$$

    \noindent We then have
    \begin{center}
    $\begin{array}{lll}(\widetilde{\Psi}(\widetilde{l})\otimes id)\dgperm(x\otimes Y) & = & \widetilde{l}(x\otimes Y_{(11)})\otimes (x_{(1)}\otimes Y_{(12)})\otimes (x_{(2)}\otimes Y_{(2)})\\
    & & + \widetilde{l}(x\otimes Y_{(1)})\otimes 1\otimes (x_{(2)}\otimes Y_{(2)})\\
    & & \pm (x\otimes\text{\normalfont Sh}(Y_{(11)};\widetilde{l}(x_{(1)}\otimes Y_{(12)})))\otimes (x_{(2)}\otimes Y_{(2)});\\
    (id\otimes\widetilde{\Psi}(\widetilde{l}))\dgperm(x\otimes Y) & = & \pm(x\otimes Y_{(1)})\otimes (\widetilde{l}(x_{(2)}\otimes Y_{(21)})\otimes (x_{(22)}\otimes Y_{(22)}))\\
    & & \pm (x\otimes Y_{(1)})\otimes (x_{(2)}\otimes\text{\normalfont Sh}(Y_{(21)};\widetilde{l}(x_{(22)}\otimes Y_{(22)})))\\
    & & \pm(x\otimes Y_{(1)})\otimes \widetilde{l}(x_{(2)}\otimes Y_{(2)})\otimes 1.\end{array}$\end{center}

    \noindent We now compute $\dgperm\widetilde{\Psi}(\widetilde{l})(x\otimes Y)$. The term $\dgperm\widetilde{\Psi}_1(\widetilde{l})(x\otimes Y)$ gives
    $$\widetilde{l}(x\otimes Y_{(1)})\otimes (x_{(2)}\otimes Y_{(21)})\otimes (x_{(22)}\otimes Y_{(22)})+\widetilde{l}(x\otimes Y_{(1)})\otimes 1\otimes (x_{(2)}\otimes Y_{(2)}).$$

    \noindent By using the first identity of Lemma \ref{formulecoprod} which gives
    $$(x\otimes Y_{(11)})\otimes (x_{(1)}\otimes Y_{(12)})\otimes (x_{(2)}\otimes Y_{(2)})=(x\otimes Y_{(1)})\otimes (x_{(2)}\otimes Y_{(21)})\otimes (x_{(22)}\otimes Y_{(22)}),$$

    \noindent we have that $\dgperm\widetilde{\Psi}_1(\widetilde{l})(x\otimes Y)$ is 
    $$\widetilde{l}(x\otimes Y_{(11)})\otimes (x_{(1)}\otimes Y_{(12)})\otimes (x_{(2)}\otimes Y_{(2)})+\widetilde{l}(x\otimes Y_{(1)})\otimes 1\otimes (x_{(2)}\otimes Y_{(2)}),$$

    \noindent which is exactly the first two lines occurring in $(\widetilde{\Psi}(\widetilde{l})\otimes id)\dgperm(x\otimes Y)$. The term $\dgperm\widetilde{\Psi}_2(\widetilde{l})(x\otimes Y)$ gives
    $$\pm (x\otimes\text{\normalfont Sh}(Y_{(11)};\widetilde{l}(x_{(2)}\otimes Y_{(2)})))\otimes (x_{(1)}\otimes Y_{(12)})$$
    $$\pm (x\otimes Y_{(11)})\otimes(\widetilde{l}(x_{(2)}\otimes Y_{(2)})\otimes Y_{(12)})$$
    $$\pm (x\otimes Y_{(11)})\otimes (x_{(1)}\otimes\text{\normalfont Sh}(Y_{(12)};\widetilde{l}(x_{(2)}\otimes Y_{(2)})))$$
    $$\pm (x\otimes Y_{(1)})\otimes(\widetilde{l}(x_{(2)}\otimes Y_{(2)})\otimes 1).$$

    \noindent From the second formula given in Lemma \ref{formulecoprod}, which gives
    $$(x\otimes Y_{(11)})\otimes (x_{(1)}\otimes Y_{(12)})\otimes (x_{(2)}\otimes Y_{(2)})=\pm(x\otimes Y_{(11)})\otimes (x_{(2)}\otimes Y_{(2)})\otimes (x_{(1)}\otimes Y_{(12)}),$$

    \noindent we obtain that $\dgperm\widetilde{\Psi}_2(\widetilde{l})(x\otimes Y)$ is given by
    $$\pm (x\otimes\text{\normalfont Sh}(Y_{(11)};\widetilde{l}(x_{(1)}\otimes Y_{(12)})))\otimes (x_{(2)}\otimes Y_{(2)})$$
    $$\pm (x\otimes Y_{(11)})\otimes(\widetilde{l}(x_{(1)}\otimes Y_{(12)})\otimes Y_{(2)})$$
    $$\pm (x\otimes Y_{(11)})\otimes (x_{(2)}\otimes\text{\normalfont Sh}(Y_{(2)};\widetilde{l}(x_{(1)}\otimes Y_{(12)})))$$
    $$\pm (x\otimes Y_{(1)})\otimes(\widetilde{l}(x_{(2)}\otimes Y_{(2)})\otimes 1).$$

    \noindent The first line is the remaining term in $(\widetilde{\Psi}(\widetilde{l})\otimes id)\dgperm(x\otimes Y)$, while the remaining lines give $(id\otimes\widetilde{\Psi}(\widetilde{l}))\dgperm(x\otimes Y)$ when using again the first formula of Lemma \ref{formulecoprod}. We thus have proved that $\widetilde{\Psi}(\widetilde{l})\in Coder(\Gamma\text{\normalfont Perm}^c(V))$, and $\pi_V\circ\widetilde{\Psi}(\widetilde{l})=\widetilde{l}$. We now prove that $\widetilde{\Psi}(\widetilde{l})$ is the only coderivation $Q$ such that $\pi_V\circ Q=l$. We use Remark \ref{idcoprod}, which gives
    $$\pi_{V^{\otimes k+1}}Q=\pi_V^{\otimes k+1}\dgperm^kQ=\sum_{i=1}^{k+1}(\pi_V^{\otimes i-1}\otimes\widetilde{l}\otimes\pi_V^{\otimes k-i+1})\dgperm^k,$$
    
    \noindent which proves that $Q$ is fully determined by $\widetilde{l}$. We thus have $Q=\widetilde{\Psi}(\widetilde{l})$, and then that $\Psi$ is the desired bijection. We now prove the  commutativity of the diagram. We note that, by definition of $\Psi(l),\widetilde{\Psi}(\widetilde{l})$ and $l$, the following diagram is commutative:
\[\begin{tikzcd}
	{\Gamma\text{\normalfont Perm}^c(V)} & {\Gamma\text{\normalfont Perm}^c(V)\otimes {\Gamma}(V)} & {\Gamma\text{\normalfont Perm}^c(V)} \\
	{\text{\normalfont Perm}^c(V)} & {\text{\normalfont Perm}^c(V)\otimes \mathcal{S}(V)} & {\text{\normalfont Perm}^c(V)}
	\arrow["{id\otimes\Delta_{T(V)}}"', from=1-1, to=1-2]
	\arrow["{\widetilde{\Psi}_1(\widetilde{l})}", curve={height=-18pt}, from=1-1, to=1-3]
	\arrow["{\widetilde{l}\otimes id}"', from=1-2, to=1-3]
	\arrow["{id\otimes Tr}", from=2-1, to=1-1]
	\arrow["{id\otimes\Delta_{\mathcal{S}(V)}}", from=2-1, to=2-2]
	\arrow["{\Psi_1(l)}"', curve={height=18pt}, from=2-1, to=2-3]
	\arrow["{id\otimes Tr\otimes Tr}", from=2-2, to=1-2]
	\arrow["{l\otimes id}", from=2-2, to=2-3]
	\arrow["{id\otimes Tr}"', from=2-3, to=1-3]
\end{tikzcd}.\]

\noindent By Lemma \ref{formulecoprod}, and by the formula
$$Tr(v_1\cdots v_n\cdot w)=\text{\normalfont Sh}(Tr(v_1\cdots v_n);w),$$

\noindent we also obtain the following commutative diagram:
\[\begin{tikzcd}
	{\Gamma\text{\normalfont Perm}^c(V)} & {\Gamma\text{\normalfont Perm}^c(V)\otimes\Gamma\text{\normalfont Perm}^c(V)} & {\Gamma\text{\normalfont Perm}^c(V)\otimes V} & {\Gamma\text{\normalfont Perm}^c(V)} \\
	{\text{\normalfont Perm}^c(V)} & {\text{\normalfont Perm}^c(V)\otimes\text{\normalfont Perm}^c(V)} & {\text{\normalfont Perm}^c(V)\otimes V} & {\text{\normalfont Perm}^c(V)}
	\arrow["\dgperm"', from=1-1, to=1-2]
	\arrow["{\widetilde{\Psi}_2(\widetilde{l})}", curve={height=-24pt}, from=1-1, to=1-4]
	\arrow["{id\otimes\widetilde{l}}"', from=1-2, to=1-3]
	\arrow["{id\otimes Sh(-;-)}"', from=1-3, to=1-4]
	\arrow["{id\otimes Tr}", from=2-1, to=1-1]
	\arrow["\dperm", from=2-1, to=2-2]
	\arrow["{\Psi_2(l)}"', curve={height=24pt}, from=2-1, to=2-4]
	\arrow["{id\otimes Tr\otimes id\otimes Tr}", from=2-2, to=1-2]
	\arrow["{id\otimes l}", from=2-2, to=2-3]
	\arrow["{id\otimes Tr\otimes id}"', from=2-3, to=1-3]
	\arrow[two heads, from=2-3, to=2-4]
	\arrow["{id\otimes Tr}"', from=2-4, to=1-4]
\end{tikzcd}\]

\noindent which proves the theorem.
\end{proof}

\begin{defi}
    We define the category $\Gamma\Lambda\mathcal{PL}_\infty$ with as objects the pairs $(V,Q)$ where $V$ is a graded $\mathbb{K}$-module and $Q$ a coderivation of degree $-1$ on $\Gamma\text{\normalfont Perm}^c(V)$ such that $Q^2=0$; a morphism $\phi:(V,Q)\longrightarrow (V',Q')$ is a morphism of coalgebras $\phi:\Gamma\text{\normalfont Perm}^c(V)\longrightarrow\Gamma\text{\normalfont Perm}^c(V')$ which commutes with the coderivations $Q$ and $Q'$.
\end{defi}

We usually denote a morphism $\phi:(V,Q)\longrightarrow (V',Q')$ by $\phi: V\rightsquigarrow V'$ when there is no ambiguity on $Q$ and $Q'$, and call it an $\infty$\textit{-morphism}.\\

If $\phi:\Gamma \text{\normalfont Perm}^c(V)\longrightarrow\Gamma \text{\normalfont Perm}^c(V')$ is a morphism of graded $\mathbb{K}$-modules, then we set, for all $k,n\geq 0$, 
\[\begin{tikzcd}
	{\phi_k:V\otimes (V^{\otimes k})^{\Sigma_k}} & {\Gamma\text{\normalfont Perm}^c(V)} & {\Gamma\text{\normalfont Perm}^c(V’);} \\
	{\phi^n:\Gamma\text{\normalfont Perm}^c(V)} & {\Gamma\text{\normalfont Perm}^c(V')} & {V'\otimes (V'^{\otimes n})^{\Sigma_n}.}
	\arrow[hook, from=1-1, to=1-2]
	\arrow["\phi", from=1-2, to=1-3]
	\arrow["\phi", from=2-1, to=2-2]
	\arrow["{\pi_{V^{\otimes n+1}}}", two heads, from=2-2, to=2-3]
\end{tikzcd}\]

\noindent Using these notations, a degree $-1$ coderivation $Q$ on $\Gamma\text{Perm}^c(V)$ is such that $Q^2=0$ if and only if for all $n\geq 0$,
$$\sum_{k=0}^nQ_k^0Q_n^k=0.$$

\noindent In particular, $Q_0^0$ is a differential on $V$. From now on, we endow $V\in\Gamma\Lambda\mathcal{PL}_\infty$ with the structure of a dg $\mathbb{K}$-module with differential $d=Q_0^0$. We also have that a morphism of graded $\mathbb{K}$-modules $\phi:\Gamma\text{\normalfont Perm}^c(V)\longrightarrow\Gamma\text{\normalfont Perm}^c(V')$ is a morphism of coalgebras if and only if
$$\dgperm\phi^n = \sum_{p+q=n-1}(\phi^{p}\otimes\phi^{q})\dgperm.$$


\begin{prop}\label{morphcoprod}
    Every $\infty$-morphism $\phi:V\longrightarrow W$ in $\Gamma\Lambda\mathcal{PL}_\infty$ is fully determined by the composite $\phi^{0}=\pi_W\circ\phi$.
\end{prop}

\begin{proof}
    Let $\phi$ be an $\infty$-morphism. We have that
    $$\phi^k=\pi_W^{\otimes k+1}\dgperm^k\phi^k=(\phi^0)^{\otimes k}\dgperm^k,$$

    \noindent which gives, for every $v\in\Gamma \text{\normalfont Perm}^c(V)$,
    $$\phi^k(v)=\phi^0(v_{(1)})\otimes \cdots\otimes\phi^0(v_{(k)})$$

    \noindent where we use the Sweedler notation in the coalgebra $\Gamma \text{\normalfont Perm}^c(V)$. We then see that $\phi$ is fully determined by $\phi^0$.
\end{proof}

\begin{remarque}
    This proposition implies that giving an $\infty$-morphism $\phi:V\rightsquigarrow W$ is equivalent to giving a morphism $\phi^0:\Gamma\text{\normalfont Perm}^c(V)\longrightarrow W$ such that the morphism $\phi:\Gamma\text{\normalfont Perm} ^c(V)\longrightarrow\Gamma\text{\normalfont Perm}^c(W)$ constructed in Proposition \ref{morphcoprod} satisfies
    $$\sum_{k=0}^n(Q')_k^0\phi^k_n = \sum_{k=0}^n\phi^0_kQ_n^k$$

    \noindent for every $n\geq 0$. In particular, $\phi_0^0:V\longrightarrow W$ is a morphism of dg $\mathbb{K}$-modules.
\end{remarque}

\begin{defi}
    An $\infty$-morphism $\phi:V\rightsquigarrow W$ is {\normalfont strict} if $\phi_k^0=0$ for all $k\geq 1$.
\end{defi}

Equivalently, a strict morphism $\phi:V\longrightarrow W$ is the data of a morphism of dg $\mathbb{K}$-modules $\phi:V\longrightarrow W$ such that
$$(Q')_n^0\phi^{\otimes n+1}=\phi Q^0_n$$

\noindent for every $n\geq 0$.

\subsection{Symmetric weighted braces and Maurer-Cartan elements in $\widehat{\Gamma\Lambda\mathcal{PL}_\infty}$}\label{sec:223}




In this subsection, we define weighted brace operations for  $\Gamma\Lambda\mathcal{PL}_\infty$-algebras, and prove that giving a structure of a $\Gamma\Lambda\mathcal{PL}_\infty$-algebra is equivalent to giving such operations. These operations will be analogue to the operations given in \cite[Theorem 2.6]{moi} for $\Gamma(\mathcal{P}re\mathcal{L}ie,-)$-algebras. We also define the notion of Maurer-Cartan element in complete $\Gamma\Lambda\mathcal{PL}_\infty$-algebras.\\

We first need an explicit basis of $\Gamma\text{\normalfont Perm}^c(V)$. We use the same arguments as in \cite[$\mathsection$2.1.1]{moi}. Let $\mathcal{B}$ be a basis of $V$ composed of homogeneous elements. For every $n\geq 0$, this gives a basis on $V^{\otimes n}$ which we denote by $\mathcal{B}^{\otimes n}$. We consider the action of $\Sigma_n$ on $\mathcal{B}^{\otimes n}$ by permutation of the factors without the Koszul sign rule. For every $\mathfrak{t}\in\mathcal{B}^{\otimes n}$, we denote by $X_\mathfrak{t}$ the orbit of $\mathfrak{t}$ under this action. We then have the unequivariant identity
$$V^{\otimes n}=\bigoplus_{\mathfrak{t}\in\mathcal{B}^{\otimes n}/\Sigma_n}\mathbb{K}[X_\mathfrak{t}].$$

\noindent For every $\mathfrak{t}\in\mathcal{B}^{\otimes n}$, we set $\mathbb{K}[X_\mathfrak{t}]^\pm=\mathbb{K}[X_\mathfrak{t}]$ with underlying action 
$$\sigma\cdot x=\varepsilon(\sigma,x)x$$

\noindent for every $\sigma\in\Sigma_n$ and $x\in X_\mathfrak{t}$, where we denote by $\varepsilon(\sigma,x)\in\mathbb{K}$ the Koszul sign which appears after the action of $\sigma$ on $x$. We then have the identification of $\Sigma_n$-representations:
$$V^{\otimes n}=\bigoplus_{\mathfrak{t}\in\mathcal{B}^{\otimes n}/\Sigma_n}\mathbb{K}[X_\mathfrak{t}]^\pm.$$

\noindent Let $(\mathcal{B}^{\otimes n})^s$ be the subset of $\mathcal{B}^{\otimes n}$ given by elements $\mathfrak{t}\in\mathcal{B}^{\otimes n}$ such that there exists $\sigma\in\text{\normalfont Stab}_{\Sigma_n}(\mathfrak{t})$ with $\varepsilon(\sigma,\mathfrak{t})\neq 1$. We set $(\mathcal{B}^{\otimes n})^r=\mathcal{B}^{\otimes n}\setminus(\mathcal{B}^{\otimes n})^s$. Note that, if $char(\mathbb{K})=2$, then $(\mathcal{B}^{\otimes n})^r=\mathcal{B}^{\otimes n}$, else, the subset $(\mathcal{B}^{\otimes n})^r$ is given by tensors of the form $x_1^{\otimes r_1}\otimes\cdots\otimes x_n^{\otimes r_n}$ with $x_1,\ldots,x_n\in\mathcal{B}$ pairwise distinct and $r_1,\ldots,r_n\geq 0$ such that if $x_i$ has an odd degree for some $i$, then $r_i=1$. We let $\mathcal{S}^r(V)$ to be given by the projections of $(\mathcal{B}^{\otimes n})^r$ on $\mathcal{S}(V)$.

\begin{prop}\label{orbit}
    The map $\mathcal{O}:\mathcal{S}^r(V)\longrightarrow\Gamma(V)$ defined by
    $$\mathcal{O}(x_1\cdots x_n)=\sum_{\sigma \in\Sigma _n/\text{\normalfont Stab}_{\Sigma _n}(x_1\otimes\cdots\otimes x_n)}\pm x_{\sigma ^{-1}(1)}\otimes\cdots\otimes x_{\sigma ^{-1}(n)}$$

    \noindent is an isomorphism.
\end{prop}

\begin{proof}
    It is the same arguments as in \cite[Proposition 2.5]{moi}.
\end{proof}

In the following, in order to handle both the cases $char(\mathbb{K})=2$ and $char(\mathbb{K})\neq 2$, when taking elements with associated weights, we will tacitly suppose that if $char(\mathbb{K})\neq 2$, then all odd degree elements will have an associated weight equal to $1$.



\begin{lm}\label{shcoprod}
    Let $x\in V,y_1,\ldots,y_n\in\mathcal{B}$ and $r_1,\ldots,r_n\geq 0$. Then
    $$\dgperm(x\otimes \mathcal{O}(y_1^{\otimes r_1}\cdots y_n^{\otimes r_n}))=\sum_{k=1}^n\sum_{\substack{p_i+q_i=r_i,i\neq k\\ p_k+q_k=r_k-1}}\pm(x\otimes\mathcal{O}(y_1^{\otimes p_1}\cdots y_n^{\otimes p_n}))\otimes (y_k\otimes\mathcal{O}(y_1^{\otimes q_1}\cdots y_n^{\otimes q_n})),$$

    \noindent where the sign is yielded by the shuffle
    $$x\otimes y_1^{\otimes r_1}\otimes\dots\otimes y_n^{\otimes r_n}\longmapsto \pm x\otimes y_1^{\otimes p_1}\otimes\dots\otimes y_n^{\otimes p_n}\otimes y_k\otimes y_1^{\otimes q_1}\otimes\cdots\otimes y_n^{\otimes q_n}.$$
\end{lm}

\begin{proof}
    Straightforward computations.
\end{proof}

\begin{thm}\label{relationsprelieinfinite}
    Let $V\in\Gamma\Lambda\mathcal{PL}_\infty$. Then $V$ comes equipped with operations, called {\normalfont weighted braces}, which have the following form.
    \begin{itemize}
        \item[-] If $char(\mathbb{K})=2$, then weighted braces are maps
        $$-\llbrace-,\ldots,-\rrbrace_{r_1,\ldots,r_n}:V^{\times n+1}\longrightarrow V,$$

        \noindent defined for any collections of integers $r_1,\ldots,r_n\geq 0$, which preserve the grading in the sense that
        $$V_k\llbrace V_{k_1},\ldots,V_{k_n}\rrbrace_{r_1,\ldots,r_n}\subset V_{k+k_1r_1+\cdots+k_nr_n}.$$

        \item[-] If $char(\mathbb{K})\neq 2$, by setting $V^{ev}=\bigoplus_{n\in\mathbb{Z}}V_{2n}$ and $V^{odd}=\bigoplus_{n\in\mathbb{Z}}V_{2n+1}$, weighted brace are maps
        $$-\llbrace\underbrace{-,\ldots,-}_{p},\underbrace{-,\ldots,-}_{q}\rrbrace_{r_{1},\ldots,r_{p},1,\ldots,1}:V\times (V^{ev})^{\times p}\times (V^{odd})^{\times q}\longrightarrow V,$$

        \noindent defined for any collection of integers $p,q,r_1,\ldots,r_n\geq 0$ which preserve the grading.
    \end{itemize}

    In addition, in both cases, the weighted brace operations satisfy the following formulas:
    \begin{enumerate}[(i)]
    \item $x\llbrace y_{\sigma (1)},\ldots,y_{\sigma (n)}\rrbrace_{r_{\sigma (1)},\ldots,r_{\sigma (n)}}=\pm x\llbrace y_{1},\ldots,y_{n}\rrbrace_{r_{1},\ldots,r_{n}},$
    \item $x\llbrace y_{1},\ldots,y_{i-1},y_{i},y_{i+1},\ldots,y_{n}\rrbrace_{r_{1},\ldots,r_{i-1},0,r_{i+1},\ldots,r_{n}}$\\\begin{flushright}$=x\llbrace y_{1},\ldots,y_{i-1},y_{i+1},\ldots,y_{n}\rrbrace_{r_{1},\ldots,r_{i-1},r_{i+1},\ldots,r_{n}},$\end{flushright}
    \item $x\llbrace y_{1},\ldots,\lambda y_{i},\ldots,y_{n}\rrbrace_{r_{1},\ldots,r_{i},\ldots,r_{n}}=\lambda^{r_{i}}x\llbrace y_{1},\ldots,y_{i},\ldots,y_{n}\rrbrace_{r_{1},\ldots,r_{i},\ldots,r_{n}},$
    \item $x\llbrace y_{1},\ldots,y_{i},y_{i},\ldots,y_{n}\rrbrace_{r_{1},\ldots,r_{i},r_{i+1},\ldots,r_{n}}$\\ \begin{flushright}
        $\displaystyle=\binom{r_{i}+r_{i+1}}{r_{i}}x\llbrace y_{1},\ldots,y_{i},\ldots,y_{n}\rrbrace_{r_{1},\ldots,r_{i-1},r_{i}+r_{i+1},r_{i+2},\ldots,r_{n}},$
    \end{flushright}
    \item $\displaystyle x\llbrace y_{1},\ldots,y_{i}+\widetilde{y_{i}},\ldots,y_{n}\rrbrace_{r_{1},\ldots,r_{i},\ldots,r_{n}}=\sum_{s=0}^{r_{i}}x\llbrace y_{1},\ldots,y_{i},\widetilde{y_{i}},\ldots,y_{n}\rrbrace_{r_{1},\ldots,s,r_{i}-s,\ldots,r_{n}},$
    \item $\displaystyle\sum_{p_i+q_i=r_i}\pm x\llbrace y_1,\ldots,y_n\rrbrace_{p_1,\ldots,p_n}\llbrace y_1,\ldots,y_n\rrbrace_{q_1,\ldots,q_n}$\\
    \begin{flushright}
        $\displaystyle+\sum_{k=1}^n\sum_{\substack{p_i+q_i=r_i,i\neq k\\ p_k+q_k=r_k-1}}\pm x\llbrace y_k\llbrace y_1,\ldots,y_n\rrbrace_{p_1,\ldots,p_n},y_1,\ldots,y_n\rrbrace_{1,q_1,\ldots,q_n}=0.$
    \end{flushright}
\end{enumerate}

In the converse direction, if a graded $\mathbb{K}$-module $V$ admits such operations, then $V\in\Gamma\Lambda\mathcal{PL}_\infty$.
\end{thm}

In particular, the operation $d(x):=x\llbrace\rrbrace$ is a differential. We usually endow $V\in\Gamma\Lambda\mathcal{PL}_\infty$ with the structure of a dg $\mathbb{K}$-module with differential $d$.\\

\begin{proof}
    Let $V\in\Gamma\Lambda\mathcal{PL}_\infty$. The strategy is the same as in \cite[Theorem 2.6]{moi} or \cite[Proposition 5.10]{cesaro}. Let $x,y_1,\ldots,y_n\in V$ be homogeneous elements, and $e_1,\ldots,e_n$ be formal elements with the same degrees as $y_1,\ldots,y_n$. We let $E$ to be graded $\mathbb{K}$-module spanned by $Y_1,\ldots,Y_n$. Let $\psi:\Gamma(E)\longrightarrow\Gamma(V)$ be the morphism which sends the $Y_i$'s to the $y_i$'s. We immediately see that $\psi$ is a morphism of coalgebras. We set
    $$x\llbrace y_1,\ldots,y_n\rrbrace_{r_1,\ldots,r_n}:=Q_{\sum_i r_i}^0(x\otimes\psi\mathcal{O}(Y_1^{\otimes r_1}\cdots Y_n^{\otimes r_n})).$$
    
    \noindent Formulas $(i)-(v)$ are consequences of straightforward computations. We prove formula $(vi)$. Since $\psi$ is a morphism of coalgebras, Lemma \ref{shcoprod} gives
    \begin{multline*}
        Q(x\otimes\psi\mathcal{O}(Y_1^{r_1}\cdots Y_n^{r_n}))=\sum_{p_i+q_i=r_i}\pm Q^0(x\otimes\psi\mathcal{O}(Y_1^{ p_1}\cdots Y_n^{ p_n}))\otimes\psi(\mathcal{O}(Y_1^{ q_1}\cdots Y_n^{ q_n})))\\+\sum_{k=1}^n\sum_{\substack{p_i+q_i=r_i,i\neq k\\ p_k+q_k=r_k-1}}\pm x\otimes\text{\normalfont Sh}(Q^0(y_k\otimes\psi\mathcal{O}(Y_1^{ p_1}\cdots Y_n^{ p_n}));\psi\mathcal{O}(Y_1^{ q_1} \cdots Y_n^{ q_n})).
    \end{multline*}

    \noindent For fixed $p_i$'s and $q_i$'s, we have
    $$Q^0(Q^0(x\otimes\psi\mathcal{O}(Y_1^{ p_1}\cdots Y_n^{ p_n}))\otimes\psi\mathcal{O}(Y_1^{ q_1}\cdots Y_n^{ q_n}))=x\llbrace y_1,\ldots,y_n\rrbrace_{p_1,\ldots,p_n}\llbrace y_1,\ldots,y_n\rrbrace_{q_1,\ldots,q_n},$$

    \noindent by definition of weighted brace operations. Concerning the second line, for fixed $p_i$'s, $q_i$'s and $k$, let $Z$ be a formal element with the same degree as $Q^0(y_k\otimes\psi\mathcal{O}(Y_1^{ p_1},\ldots, Y_n^{p_n}))$. We extend $\psi$ to $\psi:\Gamma(E\oplus\mathbb{K}Z)\longrightarrow\Gamma(V)$ by sending $Z$ to $Q^0(y_k\otimes\psi\mathcal{O}(Y_1^{ p_1},\ldots, Y_n^{ p_n}))$. We then have
    $$\text{\normalfont Sh}(Q^0(y_k\otimes\psi\mathcal{O}(Y_1^{ p_1}\cdots Y_n^{ p_n}));\psi\mathcal{O}(Y_1^{ q_1}\cdots Y_n^{ q_n}))=\psi\mathcal{O}(Z\cdot Y_1^{q_1}\cdots Y_n^{q_n}).$$

    \noindent Taking the image under $Q^0$ thus gives
    \begin{multline*}
        Q^0(x\otimes\text{\normalfont Sh}(Q^0(y_k\otimes\psi\mathcal{O}(Y_1^{ p_1}\cdots Y_n^{ p_n}));\psi\mathcal{O}(Y_1^{ q_1}\cdots Y_n^{ q_n})))\\=x\llbrace y_k\llbrace y_1,\ldots,y_n\rrbrace_{p_1,\ldots,p_n},y_1,\ldots,y_n\rrbrace_{1,q_1,\ldots,q_n}.
    \end{multline*}

    \noindent Since $Q^0Q=0$, formula $(vi)$ follows. We now prove the converse direction. Suppose that $V$ is a dg $\mathbb{K}$-module equipped with operations $-\llbrace-,\ldots,-\rrbrace_{r_1,\ldots,r_n}$ for all $r_1,\ldots,r_n\geq 0$ which satisfy the formulas given in the theorem. We pick a basis $\mathcal{B}$ of $V$ composed of homogeneous elements. Let $x,y_1,\ldots,y_n\in\mathcal{B}$. For all $r_1,\ldots,r_n\geq 0$, we set
    $$Q^0(x\otimes\mathcal{O}(y_1^{r_1}\cdots y_n^{r_n}))=x\llbrace y_1,\ldots,y_n\rrbrace_{r_1,\ldots,r_n}$$

    \noindent where we consider the orbit map $\mathcal{O}$ associated to the basis $\mathcal{B}$. By formulas $(iii)-(v)$ and the same computations as in \cite[Lemma 5.15]{cesaro}, this definition does not depend on the choice of $\mathcal{B}$. Let $Q=\widetilde{\Psi}(Q^0)$ be the coderivation associated to $Q^0\in\text{\normalfont Hom}(\Gamma\text{\normalfont Perm}^c(V),V)$ given by Proposition \ref{bijendo}. We need to prove that $Q^2=0$, which is equivalent to prove that $Q^0Q=0$. By Lemma \ref{shcoprod}, we have
    \begin{multline*}
        Q(x\otimes\mathcal{O}(y_1^{r_1}\cdots y_n^{r_n}))=\sum_{p_i+q_i=r_i}\pm Q^0(x\otimes\mathcal{O}(y_1^{ p_1}\cdots y_n^{ p_n}))\otimes\mathcal{O}(y_1^{ q_1}\cdots y_n^{ q_n})\\+\sum_{k=1}^n\sum_{\substack{p_i+q_i=r_i,i\neq k\\ p_k+q_k=r_k-1}}\pm x\otimes\text{\normalfont Sh}(Q^0(y_k\otimes\mathcal{O}(y_1^{ p_1}\cdots y_n^{ p_n}));\mathcal{O}(y_1^{ q_1}\cdots y_n^{ q_n})).
    \end{multline*}

    \noindent Applying $Q^0$ to this identity gives $Q^0Q=0$.
\end{proof}

\begin{remarque}
    A strict morphism $\phi:V\longrightarrow W$ preserves the braces in the sense that
    $$\phi(x\llbrace y_1,\ldots,y_n\rrbrace_{r_1,\ldots,r_n})=\pm\phi(x)\llbrace\phi(y_1),\ldots,\phi(y_n)\rrbrace_{r_1,\ldots,r_n},$$

    \noindent where $\pm$ is produced by the commutation of $\phi$ with $x$ and the $y_i^{\otimes r_i}$'s.
\end{remarque}

We aim to define the notion of a Maurer-Cartan element. To achieve this, we define the notion of a complete $\Gamma\Lambda\mathcal{PL}_\infty$-algebra.

\begin{defi}
    A {\normalfont filtered $\Gamma\Lambda\mathcal{PL}_\infty$-algebra} is a $\Gamma\Lambda\mathcal{PL}_\infty$-algebra $V$ endowed with a filtration $(F_nV)_{n\geq 1}$ such that
    $$F_m V\llbrace F_{p_1}V,\ldots,F_{p_n}V\rrbrace_{r_1,\ldots,r_n}\subset F_{m+p_1 r_1+\cdots+p_n r_n}V,$$

    \noindent for all $m,p_1,\ldots,p_n\geq 1$ and $r_1,\ldots,r_n\geq 0$. An $\infty$-morphism $\phi:V\rightsquigarrow V'$ between two filtered $\Gamma\Lambda\mathcal{PL}_\infty$-algebras is an $\infty$-morphism such that
    $$\phi_n^0(F_{k}(V^{\otimes n+1})\cap\Gamma\text{\normalfont Perm}^c(V))\subset F_k(V').$$
    
    \noindent for every $k\geq 1$, where we consider the filtration associated to a tensor product (see the end of $\mathsection$\ref{sec:211}). A filtered $\Gamma\Lambda\mathcal{PL}_\infty$-algebra is {\normalfont complete} if the map $V\longrightarrow\lim_{n\geq 1}V/F_nV$ is an isomorphism.
\end{defi}

We denote by $\widehat{\Gamma\Lambda\mathcal{PL}_\infty}$ the category formed by complete filtered $\Gamma\Lambda\mathcal{PL}_\infty$-algebras with as morphisms the $\infty$-morphisms which preserve the filtrations.

\begin{remarque}
    If $V$ is a filtered $\Gamma\Lambda\mathcal{PL}_\infty$-algebra, then its completion $\widehat{V}$ admits the structure of a complete filtered $\Gamma\Lambda\mathcal{PL}_\infty$-algebra.
\end{remarque}

\begin{defi}
    Let $V\in\widehat{\Gamma\Lambda\mathcal{PL}_\infty}$. A {\normalfont Maurer-Cartan element} is an element $x\in V_0$ such that
    $$d(x)+\sum_{n\geq 1}x\llbrace x\rrbrace_n=0.$$

    \noindent We denote by $\mathcal{MC}(V)$ the set composed of Maurer-Cartan elements.
\end{defi}






\begin{prop}
    Let $V,V'\in\widehat{\Gamma\Lambda\mathcal{PL}_\infty}$ and $\phi:V\rightsquigarrow V'$. Then $\phi$ induces a map
    \begin{center}
        $\begin{array}{rcll}
            \mathcal{MC}(\phi): & \mathcal{MC}(V) & \longrightarrow & \mathcal{MC}(V')  \\
             & x & \longmapsto & \displaystyle\sum_{n\geq 0}\phi_n^0(x\otimes x^{\otimes n})
        \end{array}$
    \end{center}

    \noindent such that $\mathcal{MC}(-):\widehat{\Gamma\Lambda\mathcal{PL}_\infty}\longrightarrow\text{\normalfont Set}$ is a functor. Moreover, if $\phi_0^0$ is an isomorphism, then $\mathcal{MC}(\phi)$ is a bijection.
\end{prop}

\begin{proof}
    Let $x\in\mathcal{MC}(V)$. We first prove that $y=\sum_{n\geq 0}\phi_n^0(x^{\otimes n+1})\in\mathcal{MC}(V')$. We have
    $$\sum_{m\geq 0}(Q')^0_m(y^{\otimes m+1})=\sum_{m\geq 0}\sum_{k\geq m}(Q')^0_m\left(\sum_{p_0+\cdots +p_{m}=k-m}\phi_{p_0}^0(x^{\otimes p_0+1})\otimes\cdots\otimes\phi_{p_{m}}^0(x^{\otimes p_{m}+1})\right).$$

    \noindent By using the proof of Proposition \ref{morphcoprod}, we have
    $$\begin{array}{lll}
            \displaystyle\sum_{m\geq 0}(Q')^0_m(y^{\otimes m+1}) & = & \displaystyle\sum_{k\geq 0}\sum_{m=0}^k(Q')^0_m\phi^m_{k}(x^{\otimes k+1}) \\
            & = & \displaystyle\sum_{k\geq 0}\sum_{m=0}^k\phi^0_mQ^m_k(x^{\otimes k+1})\\
            & = & \displaystyle\phi^0\left(\sum_{k\geq 0}Q_k(x^{\otimes k+1})\right).
        \end{array}$$

    \noindent By using that $Q=\widetilde{\Psi}(Q^0)$ (see the proof of Proposition \ref{bijendo}), we obtain
$$\sum_{k\geq 0}Q_k(x^{\otimes k+1})= \sum_{q\geq 0}\left(\sum_{p\geq 0}Q^0_{p}(x^{\otimes p+1})\right)\otimes x^{\otimes q}\pm \sum_{p\geq 0}x\otimes \text{\normalfont Sh}\left(x^{\otimes p};\sum_{q\geq 0}Q^0_q(x^{\otimes q+1})\right)=0$$

\noindent since $x\in\mathcal{MC}(V)$. The map $\mathcal{MC}(\phi)$ is thus well defined. Suppose now that $\phi_0^0$ is an isomorphism, and let $y\in\mathcal{MC}(V')$. We search $x\in\mathcal{MC}(V)$ such that
    $$\sum_{n\geq 0}\phi_n^0(x^{\otimes n+1})=y,$$

    \noindent which is equivalent to
    $$x=(\phi_0^0)^{-1}\left(y-\sum_{n\geq 1}\phi_n^0(x^{\otimes n+1})\right).$$

    \noindent We set $x_0=(\phi_0^0)^{-1}(y)$. We define a Cauchy sequence $(x_k)_k$ by induction by
    $$x_{k+1}=(\phi_0^0)^{-1}\left(y-\sum_{n\geq 1}\phi_n^0(x_k^{\otimes n+1})\right).$$

    \noindent We denote by $x$ its limit. We show that $x\in\mathcal{MC}(V)$. For every $W\in\Gamma\Lambda\mathcal{PL}_\infty$, we set
    $$\mathcal{R}(w)=d(w)+\sum_{n\geq 1}w\llbrace w\rrbrace_n$$

    \noindent for every $w\in W_0$. We apply $\mathcal{R}$ on the identity $\sum_{n\geq 0}\phi_n^0(x^{\otimes n+1})=y$, and use that $y\in\mathcal{MC}(V')$:
    $$\sum_{p\geq 0}(Q')_p^0\left(\left(\sum_{n\geq 0}\phi_n^0(x^{\otimes n+1})\right)^{\otimes p+1}\right)=0.$$

    \noindent This can be written as
    $$\sum_{n\geq 0}\sum_{p=0}^n (Q')_p^0\phi_n^p(x^{\otimes n+1})=0.$$

    \noindent Using that $\phi$ is a morphism in $\widehat{\Gamma\Lambda\mathcal{PL}_\infty}$, we obtain that
    $$\sum_{n\geq 0}\sum_{p=0}^n\phi_p^0Q_n^p(x^{\otimes n+1})=0$$

    \noindent which gives
    $$\mathcal{R}(x)=-(\phi_0^0)^{-1}\left(\sum_{n\geq 1}\sum_{p=1}^n\phi_p^0Q_n^p(x^{\otimes n+1})\right).$$

    \noindent We use the computation of $Q$ as $Q=\widetilde{\Psi}(Q^0)$:
    $$\mathcal{R}(x)=-(\phi_0^0)^{-1}\left(\sum_{n\geq 1}\sum_{p=1}^n\phi_p^0(Q^0_{n-p}(x^{\otimes n-p+1})\otimes x^{\otimes p}+x\otimes\text{\normalfont Sh}(x^{\otimes p-1};Q^0_{n-p}(x^{\otimes n-p+1})))\right).$$
\noindent We finally obtain
$$\mathcal{R}(x)=-(\phi_0^0)^{-1}\left(\sum_{p\geq 1}\phi_p^0(\mathcal{R}(x)\otimes x^{\otimes p}+x\otimes\text{\normalfont Sh}(x^{\otimes p-1};\mathcal{R}(x)))\right).$$

    \noindent From this identity, and because $\phi$ preserves the filtrations on $V$ and $V'$, we have that if $\mathcal{R}(x)\in F_kV$ for some $k\geq 1$, then $\mathcal{R}(x)\in F_{k+1}V$. Since $\mathcal{R}(x)\in F_1V$, it follows that $\mathcal{R}(x)\in\bigcap_{k\geq 1}F_kV=0$ so that $x\in\mathcal{MC}(V)$, and $\mathcal{MC}(\phi)(x)=y$ by construction. The map $\mathcal{MC}(\phi)$ is then surjective. We now prove that it is injective. Suppose that there exists $x_1,x_2\in \mathcal{MC}(V)$ such that $\mathcal{MC}(\phi)(x_1)=\mathcal{MC}(\phi)(x_2)$. Then
    $$x_1-x_2=(\phi_0^0)^{-1}\left(\sum_{n\geq 1}(x_2^{\otimes n+1}-x_1^{\otimes n+1})\right).$$

    \noindent Suppose that $x_1-x_2\in F_kV$ for some $k\geq 1$. Then there exists $\alpha_k\in F_kV$ such that $x_1=x_2+\alpha_k$. By definition of the filtration on tensor products, and because $x_2\in F_1V$, for every $n\geq 0$, we have $x_1^{\otimes n+1}=x_2^{\otimes n+1}+\alpha_k'$ where $\alpha'_k\in F_{k+1}V$ so that $x_2^{\otimes n+1}-x_1^{\otimes n+1}\in F_{k+1}V$. Since that $\phi$ preserves the filtrations, this implies $x_1-x_2\in F_{k+1}V$. We thus have $x_1=x_2$, so that $\mathcal{MC}(\phi)$ is injective.
\end{proof}

\subsection{Pre-Lie algebras up to homotopy with divided powers and $\Gamma\Lambda\mathcal{PL}_\infty$}\label{sec:224}

In this subsection, we show that giving a structure of a $\Gamma(\mathcal{P}re\mathcal{L}ie_\infty,-)$-algebra is equivalent to giving the structure of a $\Gamma\Lambda\mathcal{PL}_\infty$-algebra up to a shift.\\

Let $L$ be a dg $\mathbb{K}$-module. We make explicit a choice of a basis for $\Gamma(\mathcal{P}re\mathcal{L}ie_\infty,L)$ so that we can apply \cite[Lemma 2.3]{moi}. Let $\mathcal{B}$ be a basis of $L$. As a basis for $B^c(\Lambda^{-1}\text{\normalfont Perm}^\vee)(n)$, we consider tree monomials in $\mathcal{F}(\Sigma ^{-1}\Lambda^{-1}\overline{\normalfont \text{Perm}^\vee})(n)$ with as vertices elements of the form $\Sigma^{j-2} e_i^j$ where $j\geq 2$ and $1\leq i\leq j$ (see \cite[$\mathsection$3.1]{dotgrobner} for a definition of these trees, or also Definition \ref{treemon}). We denote by $\mathcal{TM}(n)$ the set of tree monomials with $n$ inputs. This gives a basis of $\mathcal{P}re\mathcal{L}ie_\infty(n)\otimes L^{\otimes n}$ which we denote by $\mathcal{TM}(n)\otimes\mathcal{B}^{\otimes n}$. We consider the action of $\Sigma_n$ on $\mathcal{TM}(n)$ given by the action of $\Sigma_n$ on $B^c(\Lambda^{-1}\text{\normalfont Perm}^\vee)$ where we omit the Koszul sign rule obtained after using the equivariance axioms for trees in $B^c(\Lambda^{-1}\text{\normalfont Perm}^\vee)$ in order to obtain a tree monomial. We also consider the action of $\Sigma_n$ on $\mathcal{B}^{\otimes n}$ by permutations. We deduce an action of $\Sigma_n$ on $\mathcal{TM}(n)\otimes\mathcal{B}^{\otimes n}$ defined as the diagonal action which uses the two previous actions of $\Sigma_n$ on $\mathcal{TM}(n)$ and $\mathcal{B}^{\otimes n}$. Given such an action, we can write
$$\mathcal{P}re\mathcal{L}ie_\infty(n)\otimes L^{\otimes n}=\bigoplus_{\mathfrak{t}\in (\mathcal{TM}(n)\otimes\mathcal{B}^{\otimes n})/\Sigma _n}\mathbb{K}[X_\mathfrak{t}]$$

\noindent where we denote by $X_\mathfrak{t}$ the orbit of the element $\mathfrak{t}\in\mathcal{TM}(n)\otimes\mathcal{B}^{\otimes n}$ under the above action. Now, for every $\mathfrak{t}\in\mathcal{TM}(n)\otimes\mathcal{B}^{\otimes n},\sigma \in\Sigma _n$ and $x\in X_\mathfrak{t}$, we denote by $\varepsilon(\sigma ,x)\in\mathbb{K}$ the Koszul sign which appears after the action of $\sigma $ on $x$, using the usual actions of $\Sigma_n$ on $\mathcal{TM}(n)$ and $\mathcal{B}^{\otimes n}$. We define the $\Sigma _n$-representation $\mathbb{K}[X_\mathfrak{t}]^\pm$ as $\mathbb{K}[X_\mathfrak{t}]$ endowed with the $\Sigma _n$-action given by
$$\sigma \cdot x^\pm=\varepsilon(\sigma ,x)(\sigma \cdot x)^\pm.$$

\noindent We obtain the following identification of $\Sigma _n$-representations:
$$\mathcal{P}re\mathcal{L}ie_\infty(n)\otimes L^{\otimes n}=\bigoplus_{\mathfrak{t}\in(\mathcal{TM}(n)\otimes\mathcal{B}^{\otimes n})/\Sigma _n}\mathbb{K}[X_\mathfrak{t}]^\pm.$$

\begin{lm}
    For every $n\geq 0$, let $(\mathcal{TM}(n)\otimes\mathcal{B}^{\otimes n})^r$ be the subset of $\mathcal{TM}(n)\otimes\mathcal{B}^{\otimes n}$ formed by elements $x$ such that, if $\sigma \cdot x=x$ for some $\sigma \in\Sigma _n$, then $\varepsilon(\sigma ,x)=1$. Let $\mathcal{S}^r(\mathcal{P}re\mathcal{L}ie_\infty,L)$ be the subspace of $\mathcal{S}(\mathcal{P}re\mathcal{L}ie_\infty,L)$ given by these elements. Then we have an isomorphism
    $$\mathcal{O}:\mathcal{S}^r(\mathcal{P}re\mathcal{L}ie_\infty,L)\longrightarrow\Gamma(\mathcal{P}re\mathcal{L}ie_\infty,L).$$
\end{lm}

\begin{proof}
    This comes from the previous analysis and \cite[Lemma 2.3]{moi}. See also the proof of \cite[Proposition 2.5]{moi}.
\end{proof}

\begin{lm}\label{compoprelie}
    Let $L$ be a dg $\mathbb{K}$-module. Denote by $\mu:\mathcal{S}(B^c(\text{\normalfont Perm}^\vee),\mathcal{S}(B^c(\text{\normalfont Perm}^\vee),L))\longrightarrow\mathcal{S}(B^c(\text{\normalfont Perm}^c),L)$ and $\widetilde{\mu}:\Gamma(B^c(\text{\normalfont Perm}^\vee),\Gamma(B^c(\text{\normalfont Perm}^\vee),L))\longrightarrow\Gamma(B^c(\text{\normalfont Perm}^\vee),L)$ the monadic compositions. Let $x\in L$ and $B_1,\ldots,B_n\in\mathcal{S}^r(B^c(\text{\normalfont Perm}^
    \vee),L)$ be basis elements. Then
    $$\widetilde{\mu}(\mathcal{O}\Sigma^{-1} e_1^{n+1}(x,\mathcal{O}B_1,\ldots,\mathcal{O}B_n))=\mathcal{O}(\mu(\Sigma^{-1} e_1^{n+1}(x,B_1,\ldots,B_n))).$$
\end{lm}

\begin{proof}
    The proof is identical to the proofs given in \cite[Theorem 1.5.1, Lemma 1.5.2]{cesaro}.
\end{proof}

\begin{thm}\label{thmfonda}
    A dg $\mathbb{K}$-module $(L,d)$ is a $\Gamma(\mathcal{P}re\mathcal{L}ie_\infty,-)$-algebra if and only if $\Sigma L\in\Gamma\Lambda\mathcal{PL}_\infty$ with $Q_0^0=\Sigma d$. Moreover, every morphism of $\Gamma(\mathcal{P}re\mathcal{L}ie_\infty,-)$-algebras $\phi:L\longrightarrow L'$ gives rise to a strict morphism $\Sigma \phi:\Sigma  L\longrightarrow\Sigma  L'$ in $\Gamma\Lambda\mathcal{PL}_\infty$.
\end{thm}

\begin{proof}
    Let $L$ be a $\Gamma(\mathcal{P}re\mathcal{L}ie_\infty,-)$-algebra. Then $\Sigma  L$ is a $\Gamma(\Lambda\mathcal{P}re\mathcal{L}ie_\infty,-)$-algebra by Proposition \ref{gammashift}. Since $\Lambda\mathcal{P}re\mathcal{L}ie_\infty\simeq B^c(\text{\normalfont Perm}^\vee)$, we have a morphism
    $$l:\Gamma(\Sigma ^{-1}\overline{\text{\normalfont Perm}}^c,\Sigma L)\longrightarrow \Sigma L.$$

    \noindent We then set, for homogeneous elements $x,y_1,\ldots,y_n\in\Sigma L$, and $r_1,\ldots,r_n\geq 0$,
    $$x\llbrace y_1,\ldots,y_n\rrbrace_{r_1,\ldots,r_n}:=l(\mathcal{O}(\Sigma^{-1}  e_1^{r+1}\otimes x\otimes y_1^{\otimes r_1}\otimes\cdots\otimes y_n^{\otimes r_n})),$$

    \noindent where $r=r_1+\cdots+r_n$ and where the considered orbit map is using a basis which includes $x,y_1,\ldots,y_n$. We check all formulas given in Theorem \ref{relationsprelieinfinite}. Formulas $(i)-(v)$ come from straightforward computations. We prove formula $(vi)$. We compute
    $$d( x\llbrace  y_1,\ldots,y_n\rrbrace_{r_1,\ldots,r_n})=ld(\mathcal{O}(\Sigma^{-1}  e_1^{r+1}\otimes x\otimes y_1^{\otimes r_1}\otimes\cdots\otimes y_n^{\otimes r_n})).$$

    \noindent We have
    $$\mathcal{O}(\Sigma^{-1} e_1^{r+1}\otimes x\otimes y_1^{\otimes r_1}\otimes \cdots\otimes y_n^{\otimes r_n})=\sum_{\sigma \in Sh(1,r_1,\ldots,r_n)}\sigma \cdot(\Sigma^{-1} e_1^{r+1}\otimes x\otimes y_1^{\otimes r_1}\otimes\cdots\otimes y_n^{\otimes r_n}).$$

    \noindent Let $\partial$ be the differential of $B^c(\text{\normalfont Perm}^\vee)$. Then
    \begin{multline*}
        d(x\llbrace  y_1,\ldots, y_n\rrbrace_{r_1,\ldots,r_n})= \displaystyle l\left(\sum_{\sigma \in Sh(1,r_1,\ldots,r_n)}\sigma \cdot(\partial(\Sigma^{-1} e_1^{r+1})\otimes  x\otimes y_1^{\otimes r_1}\otimes\cdots\otimes y_n^{\otimes r_n})\right)  \\
        + d(x)\llbrace  y_1,\ldots, y_n\rrbrace_{r_1,\ldots,r_n}+\sum_{k=1}^n\pm  x\llbrace  y_1,\ldots, y_k, d( y_k),\ldots, y_n\rrbrace_{r_1,\ldots, r_k-1,1,\ldots,r_n}.
    \end{multline*}

    \noindent We compute the first sum. Recall from the operadic composition in $\text{\normalfont Perm}$ (see Proposition \ref{compperm}) and from the definition of the differential in the cobar construction of a coaugmented cooperad that we have
    $$\partial(\Sigma^{-1} e_1^{r+1})=\sum_{\substack{p+q=r+2\\ p,q\geq 2}}\left(\sum_{\omega\in Sh_\ast(q,1,\ldots,1)}\omega \cdot T_{p,q}^{1,1}+\sum_{k=2}^{p}\sum_{i=1}^q\sum_{\omega\in Sh_\ast(1,\ldots,\underset{k}{q},\ldots,1)}\omega\cdot T_{p,q}^{k,i}\right),$$

    \noindent where we have set
    $$T_{p,q}^{k,i}=\begin{tikzpicture}[baseline={([yshift=-.5ex]current bounding box.center)},scale=1]
    \node[scale=1] (i) at (0,0) {$\Sigma^{-1} e_1^p$};
    \node[scale=0.8] (1) at (-3,1) {$1$};
    \node[scale=1] (1b) at (-2,1) {$\cdots$};
    \node[scale=0.8] (1bb) at (-1,1) {$k-1$};
    \node[scale=1] (2) at (0,1) {$\Sigma^{-1} e_i^q$};
    \node[scale=0.8] (2b) at (-1,2) {$k$};
    \node[scale=1] (2bb) at (0,2) {$\cdots$};
    \node[scale=0.8] (2bbb) at (1,2) {$k+q-1$};
    \node[scale=0.8] (3b) at (1,1) {$k+q$};
    \node[scale=1] (3bb) at (2,1) {$\cdots$};
    \node[scale=0.8] (3bbb) at (3,1) {$p+q-1$};
    \draw (i) -- (1);
    \draw (i) -- (1bb);
    \draw (i) -- (3);
    \draw (i) -- (2);
    \draw (2) -- (2b);
    \draw (2) -- (2bbb);
    \draw (i) -- (3b);
    \draw (i) -- (3bbb);
    \end{tikzpicture}.$$

    \noindent Let $z_1,\ldots,z_{r+1}=x,\underbrace{y_1,\ldots,y_1}_{r_1},\ldots,\underbrace{y_n,\ldots,y_n}_{r_n}$. For given $p,q\geq 2$ such that $p+q=r+2$, we need to compute the sums,
    $$S_{p,q}^1:=\sum_{\sigma \in Sh(1,r_1,\ldots,r_n)}\sum_{\omega\in Sh_\ast(q,1,\ldots,1)}\pm\sigma \omega\cdot(T_{p,q}^{1,1}\otimes z_{\omega(1)}\otimes\cdots\otimes z_{\omega(r+1)});$$
    $$S_{p,q}^2:=\sum_{k=2}^{p}\sum_{i=1}^q\sum_{\sigma \in Sh(1,r_1,\ldots,r_n)}\sum_{\omega\in Sh_\ast(1,\ldots,\underset{k}{q},\ldots,1)}\pm\sigma \omega\cdot(T_{p,q}^{k,i}\otimes z_{\omega(1)}\otimes\cdots\otimes z_{\omega(r+1)}).$$

    We first compute $S_{p,q}^1$. We claim that
    $$S_{p,q}^1=\sum_{\substack{p_i+q_i=r_i\\p_1+\cdots+p_n=p-1\\q_1+\cdots+q_n=q-1}}\pm\mathcal{O}(T_{p,q}^{1,1}\otimes x\otimes y_1^{\otimes q_1}\otimes\cdots \otimes y_n^{\otimes q_n}\otimes y_1^{\otimes p_1}\otimes \cdots\otimes y_n^{\otimes p_n}).$$

    \noindent Let $\sigma \in Sh(r_1,\ldots,r_n)$ and $\omega\in Sh_\ast(q,1,\ldots,1)$. In particular we have $\omega\in Sh(1,q-1,p-1)$ with $\omega(1)=1$. Then there exist $p_1,\ldots,p_n,q_1,\ldots,q_n$ with $p_i+q_i=r_i, p_1+\cdots+p_n=p-1$ and $q_1+\cdots+q_n=q-1$ such that
    $$z_{\omega(1)}\otimes\cdots\otimes z_{\omega(r+1)}=\pm x\otimes y_1^{\otimes q_1}\otimes\cdots\otimes y_n^{\otimes q_n}\otimes y_1^{\otimes p_1}\otimes\cdots\otimes y_n^{\otimes p_n}.$$

    \noindent Thus, every term in the left hand-side sum is part of the sum in the right hand-side. We now consider an element which occurs in the expansion of $\mathcal{O}(T_{p,q}^{1,1}\otimes x\otimes y_1^{\otimes q_1}\otimes\cdots\otimes y_n^{\otimes q_n}\otimes y_1^{\otimes p_1}\otimes\cdots\otimes y_n^{\otimes p_n})$ for some $p_i$'s,$q_i$'s as above. Let $\beta\in\Sigma _{r+1}$ be a blocks permutation which sends $x\otimes y_1^{\otimes q_1}\otimes\cdots\otimes y_n^{\otimes q_n}\otimes y_1^{\otimes p_1}\otimes\cdots\otimes y_n^{\otimes p_n}$ to $\pm x\otimes y_1^{\otimes r_1}\otimes\cdots\otimes y_n^{\otimes r_n}$. Then
    $$\mathcal{O}(T_{p,q}^{1,1}\otimes x\otimes y_1^{\otimes q_1}\otimes\cdots\otimes y_n^{\otimes q_n}\otimes y_1^{\otimes p_1}\otimes\cdots\otimes y_n^{\otimes p_n})=\pm\mathcal{O}(\beta\cdot T_{p,q}^{1,1}\otimes x\otimes y_1^{\otimes r_1}\otimes\cdots\otimes y_n^{\otimes r_n}).$$

    \noindent Let $\tau\in\Sigma _{r+1}$. We write $\tau=\sigma \eta$ where $\sigma \in Sh(1,r_1,\ldots,r_n)$ and $\eta\in\Sigma _1\times\Sigma _{r_1}\times\cdots\times\Sigma _{r_n}$. Then
    $$\tau\cdot(\mu\cdot T_{p,q}^{1,1}\otimes x\otimes y_1^{\otimes r_1}\otimes\cdots\otimes y_n^{\otimes r_n})=\sigma \cdot(\eta\beta\cdot T_{p,q}^{1,1}\otimes x\otimes y_1^{\otimes r_1}\otimes\cdots\otimes y_n^{\otimes r_n}).$$

    \noindent We write $\eta\beta=\omega\nu$ where $\omega\in Sh(1,q-1,p-1)$ and $\nu\in\Sigma _1\times\Sigma _{q-1}\times\Sigma _{p-1}$. Since $\eta\beta(1)=1$, we have $\omega(1)=1$ so that $\omega\in Sh_\ast(q,1,\ldots,1)$. We finally have
    $$\tau\cdot(\mu\cdot T_{p,q}^{1,1}\otimes x\otimes y_1^{\otimes r_1}\otimes \cdots\otimes y_n^{\otimes r_n})=\sigma \cdot(\omega\cdot T_{p,q}^{1,1}\otimes x\otimes y_1^{\otimes r_1}\otimes\cdots\otimes y_n^{\otimes r_n})$$

    \noindent so that every term in the right hand-side is part of $S_{p,q}^1$. We thus have proved the first identity.\\

    We now compute $S_{p,q}^2$. We claim that
    $$S_{p,q}^2=\sum_{j=1}^n\sum_{\substack{p_i+q_i=r_i,i\neq j\\ p_j+q_j=r_j-1\\ p_1+\cdots+p_n=p-1\\ q_1+\cdots+q_n=q-1}}\pm\mathcal{O}(T_{p,q}^{2,1}\otimes x\otimes y_j\otimes y_1^{\otimes q_1}\otimes\cdots\otimes y_n^{\otimes q_n}\otimes y_1^{\otimes p_1}\otimes\cdots\otimes y_n^{\otimes p_n}).$$
    
    \noindent Since there exists $\nu\in\Sigma _{r+1}$ such that $T_{p,q}^{k,i}=\nu\cdot T_{p,q}^{2,1}$, we can apply the same arguments as before to show that every term which occurs in $S_{p,q}^2$ is part of the right-hand side sum. Now consider some term $T_{p,q}^{2,1}\otimes x\otimes y_j\otimes y_1^{\otimes q_1}\otimes\cdots\otimes y_n^{\otimes q_n}\otimes y_1^{\otimes p_1}\otimes \cdots\otimes y_n^{\otimes p_n}$. Let $\beta\in\Sigma _{r+1}$ be a blocks permutation which sends $x\otimes y_j\otimes y_1^{\otimes q_1}\otimes\cdots\otimes y_n^{\otimes q_n}\otimes y_1^{\otimes p_1}\otimes\cdots\otimes y_n^{\otimes p_n}$ to $\pm x\otimes y_1^{\otimes r_1}\otimes\cdots\otimes y_n^{\otimes r_n}$. Then
    $$\mathcal{O}(T_{p,q}^{2,1}\otimes x\otimes y_j\otimes y_1^{\otimes q_1}\otimes\cdots\otimes y_n^{\otimes q_n}\otimes y_1^{\otimes p_1}\otimes\cdots\otimes y_n^{\otimes p_n})=\pm\mathcal{O}(\beta\cdot T_{p,q}^{2,1}\otimes x\otimes y_1^{\otimes r_1}\otimes\cdots\otimes y_n^{\otimes r_n}).$$

    \noindent Now let $\tau\in\Sigma _{r+1}$. We write $\tau=\sigma\eta$ where $\sigma \in Sh(1,r_1,\ldots, r_n)$ and $\eta\in\Sigma _1\times\Sigma _{r_1}\times\cdots\times\Sigma _{r_n}$. Then
    $$\tau\cdot(\beta\cdot T_{p,q}^{2,1}\otimes x\otimes y_1^{\otimes r_1}\otimes\cdots\otimes y_n^{\otimes r_n})=\sigma \cdot(\eta\beta\cdot T_{p,q}^{2,1}\otimes x\otimes y_1^{\otimes r_1}\otimes\cdots\otimes y_n^{\otimes r_n}).$$

    \noindent We finally write $\eta\beta=\omega\cdot \nu_\ast(1,\gamma,1,\ldots,1)$ where $\nu\in\Sigma _{p},\gamma\in\Sigma _{q},\omega\in Sh_\ast(1,\ldots,\underset{k}{q},\ldots,1)$ and $k=\nu(2)$. Since $\eta\beta(1)=1$, we have $\nu(1)=1$. We thus obtain
    $$\tau\cdot(\beta\cdot T_{p,q}^{2,1}\otimes x\otimes y_1^{\otimes r_1}\otimes\cdots\otimes y_n^{\otimes r_n})=\sigma \cdot(\omega\cdot T_{p,q}^{k,\gamma(1)}\otimes x\otimes y_1^{\otimes r_1}\otimes\cdots\otimes y_n^{\otimes r_n})$$

    \noindent so that every term in the right-hand side is part of $S_{p,q}^2$.\\

    \noindent From Lemma \ref{compoprelie}, we deduce that formula $(vi)$ of Theorem \ref{relationsprelieinfinite} is satisfied so that $\Sigma  L\in\Gamma\Lambda\mathcal{PL}_\infty$. Suppose now that $L$ is such that $\Sigma  L\in\Gamma\Lambda\mathcal{PL}_\infty$. We prove that $L$ is a $\Gamma(\mathcal{P}re\mathcal{L}ie_\infty,-)$-algebra, or equivalently, that $\Sigma  L$ is a $\Gamma(B^c(\text{\normalfont Perm}^\vee),-)$-algebra (see Proposition \ref{gammashift}). We first define
    $$l:\bigoplus_{n\geq 0}(\Sigma ^{-1}\overline{\text{\normalfont Perm}^\vee}(n)\otimes(\Sigma L)^{\otimes n})^{\Sigma _n}\longrightarrow\Sigma  L$$

    \noindent by setting, for every basis elements $x,y_1,\ldots,y_n\in\Sigma L$,
     $$l(\mathcal{O}(\Sigma^{-1}  e_1^{\sum_ir_i+1}\otimes x\otimes y_1^{\otimes r_1}\otimes \cdots\otimes y_n^{\otimes r_n})):=x\llbrace y_1,\ldots, y_n\rrbrace_{r_1,\ldots,r_n}.$$

     \noindent We then extend $l:\Gamma(B^c(\text{\normalfont Perm}^\vee),\Sigma  L)\longrightarrow\Sigma  L$ by Lemma \ref{compoprelie}. By the same identities as before, we can show that $l$ preserves the differentials, giving a structure of a $\Gamma(\Lambda\mathcal{P}re\mathcal{L}ie_\infty,-)$-algebra on $\Sigma  L$.
\end{proof}

\begin{remarque}
    This theorem implies that the category of the $\Gamma(\Lambda\mathcal{P}re\mathcal{L}ie_\infty,-)$-algebras is a full subcategory of $\Gamma\Lambda\mathcal{PL}_\infty$. However, a morphism in $\Gamma\Lambda\mathcal{PL}_\infty$ does not necessarily preserves the monadic structure of $\Gamma(\Lambda\mathcal{P}re\mathcal{L}ie_\infty,-)$.
\end{remarque}


\begin{cor}
    For every complete $\Gamma(\mathcal{P}re\mathcal{L}ie,-)$-algebra $L$, the dg $\mathbb{K}$-module $\Sigma L$ is endowed with the structure of a $\widehat{\Gamma\Lambda\mathcal{P}\mathcal{L}}_\infty$-algebra such that $\mathcal{MC}(L)$ is in bijective correspondence with $\mathcal{MC}(\Sigma L)$. 
\end{cor}

\begin{proof}
    The operad morphism $\mathcal{P}re\mathcal{L}ie_\infty\longrightarrow\mathcal{P}re\mathcal{L}ie$ given in Remark \ref{prelieinftoprelie} gives rise to a monad morphism $\Gamma(\mathcal{P}re\mathcal{L}ie_\infty,-)\longrightarrow\Gamma(\mathcal{P}re\mathcal{L}ie,-)$. Then, every $\Gamma(\mathcal{P}re\mathcal{L}ie,-)$-algebra $L$ is endowed with the structure of a $\Gamma(\mathcal{P}re\mathcal{L}ie_\infty,-)$-algebra. By Theorem \ref{thmfonda}, the dg $\mathbb{K}$-module $\Sigma L$ is a $\Gamma\Lambda\mathcal{PL}_\infty$-algebra. If we denote by $-\{-,\ldots,-\}_{r_1,\ldots,r_n}$ the weighted brace operations given by the $\Gamma(\mathcal{P}re\mathcal{L}ie,-)$-algebra structure on $L$, then, by definition of the weighted brace operations given in the proof of Theorem \ref{thmfonda}, we have
    $$\Sigma x\llbrace \Sigma  y_1,\ldots,\Sigma y_n\rrbrace_{r_1,\ldots,r_n}=\left\{\begin{array}{cl}
        -\Sigma d(x) & \text{if }r_1+\cdots+ r_n=0 \\
        (-1)^{|x|}\Sigma x\{y_1\}_1 & \text{if }r_1=1, r_2=\ldots=r_n=0\\
        0 & \text{if } r_1+\cdots+r_n\geq 2
    \end{array}\right.$$

    \noindent for every $x,y_1,\ldots,y_n\in L$. Since the operations $-\{-,\ldots,-\}_{r_1,\ldots,r_n}$ preserve the filtration on $L$, the operations $-\llbrace -,\ldots,-\rrbrace_{r_1,\ldots,r_n}$ also preserve the filtration on $\Sigma L$ so that $\Sigma L\in\widehat{\Gamma\Lambda\mathcal{PL}}_\infty$. Moreover, the previous computation of the braces shows that $x\in\mathcal{MC}(L)$ if and only if $\Sigma x\in\mathcal{MC}(\Sigma L)$, which proves the corollary.
\end{proof}

\begin{cor}
    Suppose that $char(\mathbb{K})=0$. Then every $\Gamma\Lambda\mathcal{PL}_\infty$-algebra $V$ is endowed with the structure of a $\Lambda\mathcal{PL}_\infty$-algebra such that
    $$x\llbrace y_1,\ldots,y_n\rrbrace_{r_1,\ldots,r_n}=\frac{1}{\prod_i r_i!}x\llbrace\underbrace{y_1,\ldots,y_1}_{r_1},\ldots,\underbrace{y_n,\ldots,y_n}_{r_n}\rrbrace$$

    \noindent for every $x\in V$ and $y_1,\ldots,y_n\in V$ with associated weights $r_1,\ldots,r_n\geq 0$, where we consider the symmetric braces $-\llbrace -,\ldots,-\rrbrace$ defined in Proposition \ref{caracprelieinf}.
\end{cor}

\begin{proof}
    Let $V$ be a $\Gamma\Lambda\mathcal{PL}_\infty$-algebra. By Theorem \ref{thmfonda}, we have that $\Sigma^{-1}V$ is a $\Gamma(\mathcal{P}re\mathcal{L}ie_\infty,-)$-algebra. Since we have a morphism of monads given by the trace map
    $$Tr:\mathcal{S}(\mathcal{P}re\mathcal{L}ie_\infty,-)\longrightarrow\Gamma(\mathcal{P}re\mathcal{L}ie_\infty,-),$$

    \noindent $\Sigma^{-1}V$ is endowed with a $\mathcal{P}re\mathcal{L}ie_\infty$-algebra structure, so that $V$ is a $\Lambda\mathcal{PL}_\infty$-algebra. By using the definition of weighted brace operations, we obtain the desired relation.
\end{proof}

\section{A morphism from $\mathcal{P}re\mathcal{L}ie_\infty$ to $\mathcal{B}race\underset{\text{\normalfont H}}{\otimes}\mathcal{E}$}

In this section, we construct an operad morphism from $\mathcal{P}re\mathcal{L}ie_\infty$ to $\mathcal{B}race\underset{\text{\normalfont H}}{\otimes}\mathcal{E}$. We will define this morphism as a composite of the form $\mathcal{P}re\mathcal{L}ie_\infty\longrightarrow \mathcal{B}race\underset{\text{\normalfont H}}{\otimes}B^c(\Lambda^{-1}\mathcal{B}race^\vee)\longrightarrow\mathcal{B}race\underset{\text{\normalfont H}}{\otimes}\mathcal{E}$.\\

In $\mathsection$\ref{sec:231}, we construct an operad morphism $B^c(\Lambda^{-1}\mathcal{B}race^\vee)\longrightarrow\mathcal{E}$, which will be given as a lift of some diagram.\\

In $\mathsection$\ref{sec:232}, we give a computation of the twisted coderivation on the free brace coalgebra $\mathcal{B}race^c(\Sigma^{-1} N_\ast(\Delta^n))$ induced by the morphism constructed in $\mathsection$\ref{sec:231} and the $\mathcal{E}$-coalgebra structure on $N_*(\Delta^n)$ by induction on $n\geq 0$.\\

In $\mathsection$\ref{sec:233}, we construct the morphism $\mathcal{P}re\mathcal{L}ie_\infty\longrightarrow \mathcal{B}race\underset{\text{\normalfont H}}{\otimes}B^c(\Lambda^{-1}\mathcal{B}race^\vee)$ and deduce Theorem \ref{theoremD}.

\subsection{A morphism from $B^c(\Lambda^{-1}\mathcal{B}race^\vee)$ to $\mathcal{E}$}\label{sec:231}

Let $\mathcal{C}om$ be the commutative operad. Recall that if we consider the model structure on the category of symmetric operads $\mathcal{P}$ such that $\mathcal{P}(0)=0$ (see \cite[$\mathsection$3.3]{hinicherratum}), then we have an acyclic fibration $\mathcal{E}\overset{\sim}{\longtwoheadrightarrow}\mathcal{C}om$. We thus have there exists a lift of the following diagram:
\[\begin{tikzcd}
	& {\mathcal{E}} \\
	{B^c(\Lambda^{-1}\mathcal{B}race^\vee)} & {\mathcal{C}om}
	\arrow["\sim", two heads, from=1-2, to=2-2]
	\arrow["\exists", dashed, from=2-1, to=1-2]
	\arrow[from=2-1, to=2-2]
\end{tikzcd}.\]

\noindent The goal of this subsection is to give an explicit choice of such a lift. Equivalently, we are searching for elements $\mu_T\in\mathcal{E}(|T|)_{|T|-2}$ for every $T\in\mathcal{PRT}$ such that
$$d(\Lambda\mu_T)+\sum_{S\subset T}\Lambda\mu_{T/S}\circ_S\Lambda\mu_{S}=0.$$

For every $\sigma \in\Sigma _r$, let $h_\mathcal{E}^\sigma :\mathcal{E}(r)_d\longrightarrow\mathcal{E}(r)_{d+1}$ be the morphism such that, for every $w_0,\ldots,w_d\in\Sigma _r$, 
$$h_{\mathcal{E}}^{\sigma }(w_0,\ldots,w_r)=(\sigma ,w_0,\ldots,w_r).$$

\noindent We can check that this morphism is a homotopy between the identity map of $\mathcal{E}(r)_d$ and the morphism $\varphi^\sigma_\mathcal{E} :\mathcal{E}(r)_d\longrightarrow\mathcal{E}(r)_d$ defined by
$$\varphi_{\mathcal{E}}^\sigma (w_0,\ldots,w_d)=\left\{\begin{array}{ll}
    \sigma  & \text{if }d=0 \\
    0 & \text{else}
\end{array}\right..$$

\noindent Accordingly,
$$d h_\mathcal{E}^\sigma +h_\mathcal{E}^\sigma  d=id_{\mathcal{E}}-\varphi^\sigma_\mathcal{E} .$$

By functoriality, we have that $\Lambda h_\mathcal{E}^\sigma $ is a homotopy between the identity map and $\Lambda\varphi^\sigma_\mathcal{E} $.

\begin{cons}\label{consmu}
    We set $\mu_{\begin{tikzpicture}[baseline={([yshift=-.5ex]current bounding box.center)},scale=0.3]
    \node[draw,circle,scale=0.3] (i) at (0,0) {$1$};
    \end{tikzpicture}}=0$, $\mu_{\begin{tikzpicture}[baseline={([yshift=-.5ex]current bounding box.center)},scale=0.3]
    \node[draw,circle,scale=0.3] (i) at (0,0) {$1$};
    \node[draw,circle,scale=0.3] (1) at (0,1) {$2$};
    \draw (i) -- (1);
    \end{tikzpicture}}=(12)$ and $\mu_{\begin{tikzpicture}[baseline={([yshift=-.5ex]current bounding box.center)},scale=0.3]
    \node[draw,circle,scale=0.3] (i) at (0,0) {$2$};
    \node[draw,circle,scale=0.3] (1) at (0,1) {$1$};
    \draw (i) -- (1);
    \end{tikzpicture}}=(21)$. For every $T\in\mathcal{PRT}$, we define $\mu_T$ by induction on $|T|$ by
    $$\Lambda\mu_T=-\Lambda h_{\mathcal{E}}^{\sigma _T}\left(\sum_{S\subset T}\Lambda\mu_{T/S}\circ_S\Lambda\mu_{S}\right),$$

    \noindent where we take $\mu_{T/S}\in\mathcal{E}(V_{T/S})$ and $\mu_{S}\in\mathcal{E}(V_S)$ (see Remark \ref{remlabel}).
\end{cons}

\begin{example}\label{exmu} Let us make explicit the $\mu_T$'s for $T$ a tree with $3$ vertices. 
\begin{itemize}
    \item If $T=\begin{tikzpicture}[baseline={([yshift=-.5ex]current bounding box.center)},scale=0.6]
    \node[draw,circle,scale=0.6] (i) at (0,0) {$1$};
    \node[draw,circle,scale=0.6] (1) at (-0.5,1) {$2$};
    \node[draw,circle,scale=0.6] (3) at (0.5,1) {$3$};
    \draw (i) -- (1);
    \draw (i) -- (3);
    \end{tikzpicture}$, then we have two non trivial trees $S_1=\begin{tikzpicture}[baseline={([yshift=-.5ex]current bounding box.center)},scale=0.6]
    \node[draw,circle,scale=0.6] (i) at (0,0) {$1$};
    \node[draw,circle,scale=0.6] (1) at (0,1) {$2$};
    \draw (i) -- (1);
    \end{tikzpicture}$ and $S_2=\begin{tikzpicture}[baseline={([yshift=-.5ex]current bounding box.center)},scale=0.6]
    \node[draw,circle,scale=0.6] (i) at (0,0) {$1$};
    \node[draw,circle,scale=0.6] (1) at (0,1) {$3$};
    \draw (i) -- (1);
    \end{tikzpicture}$, which are such that $T/S_1=\begin{tikzpicture}[baseline={([yshift=-.5ex]current bounding box.center)},scale=0.6]
    \node[draw,circle,scale=0.6] (i) at (0,0) {$S_1$};
    \node[draw,circle,scale=0.6] (1) at (0,1) {$3$};
    \draw (i) -- (1);
    \end{tikzpicture}$ and $T/S_2=\begin{tikzpicture}[baseline={([yshift=-.5ex]current bounding box.center)},scale=0.6]
    \node[draw,circle,scale=0.6] (i) at (0,0) {$S_2$};
    \node[draw,circle,scale=0.6] (1) at (0,1) {$2$};
    \draw (i) -- (1);
    \end{tikzpicture}$. We thus have
    \begin{center}
    $\begin{array}{lll}
        \displaystyle\sum_{S\subset T}\Lambda\mu_{T/S}\circ_S\Lambda\mu_S & = & \Sigma^{-1} (S_13)\circ_{S_1}\Sigma^{-1} (12)+\Sigma^{-1} (S_2 2)\circ_{S_2}\Sigma^{-1} (13)\\
        & = & \Sigma^{-1} (123)-\Sigma^{-1} (132).\\
    \end{array}$
    \end{center}

    \noindent We then deduce $\Lambda\mu_T=\Sigma^{-1} (123,132)$.
    
    \item If $T=\begin{tikzpicture}[baseline={([yshift=-.5ex]current bounding box.center)},scale=0.6]
    \node[draw,circle,scale=0.6] (i) at (0,0) {$1$};
    \node[draw,circle,scale=0.6] (1) at (0,1) {$2$};
    \node[draw,circle,scale=0.6] (3) at (0,2) {$3$};
    \draw (i) -- (1);
    \draw (1) -- (3);
    \end{tikzpicture}$, then we have two non trivial trees $S_1=\begin{tikzpicture}[baseline={([yshift=-.5ex]current bounding box.center)},scale=0.6]
    \node[draw,circle,scale=0.6] (i) at (0,0) {$1$};
    \node[draw,circle,scale=0.6] (1) at (0,1) {$2$};
    \draw (i) -- (1);
    \end{tikzpicture}$ and $S_2=\begin{tikzpicture}[baseline={([yshift=-.5ex]current bounding box.center)},scale=0.6]
    \node[draw,circle,scale=0.6] (i) at (0,0) {$2$};
    \node[draw,circle,scale=0.6] (1) at (0,1) {$3$};
    \draw (i) -- (1);
    \end{tikzpicture}$, which are such that $T/S_1=\begin{tikzpicture}[baseline={([yshift=-.5ex]current bounding box.center)},scale=0.6]
    \node[draw,circle,scale=0.6] (i) at (0,0) {$S_1$};
    \node[draw,circle,scale=0.6] (1) at (0,1) {$3$};
    \draw (i) -- (1);
    \end{tikzpicture}$ and $T/S_2=\begin{tikzpicture}[baseline={([yshift=-.5ex]current bounding box.center)},scale=0.6]
    \node[draw,circle,scale=0.6] (i) at (0,0) {$1$};
    \node[draw,circle,scale=0.6] (1) at (0,1) {$S_2$};
    \draw (i) -- (1);
    \end{tikzpicture}$. We thus have
    \begin{center}
    $\begin{array}{lll}
        \displaystyle\sum_{S\subset T}\Lambda\mu_{T/S}\circ_S\Lambda\mu_S & = & \Sigma^{-1} (S_13)\circ_{S_1}\Sigma^{-1} (12)+\Sigma^{-1} (1 S_2)\circ_{S_2}\Sigma^{-1} (23)\\
        & = & \Sigma^{-1} (123)-\Sigma^{-1} (123)\\
        & = & 0.
    \end{array}$
    \end{center}

    \noindent We then deduce $\Lambda\mu_T=0$.
\end{itemize}

\end{example}

\begin{thm}\label{mu}
    For every tree $T\in\mathcal{PRT}$, we have
    $$d(\Lambda\mu_T)+\sum_{S\subset T}\Lambda\mu_{T/S}\circ_S\Lambda\mu_{S}=0$$

    \noindent where we consider the elements $\mu_T$'s defined in Construction \ref{consmu}. Then we have an explicit lift
\[\begin{tikzcd}
	& {\mathcal{E}} \\
	{B^c(\Lambda^{-1}\mathcal{B}race^\vee)} & {\mathcal{C}om}
	\arrow["\sim", two heads, from=1-2, to=2-2]
	\arrow[dashed, from=2-1, to=1-2]
	\arrow[from=2-1, to=2-2]
\end{tikzcd}\]

\noindent given by the morphism
\begin{center}
    $\begin{array}{ccc}
        B^c(\Lambda^{-1}\mathcal{B}race^\vee) & \longrightarrow & \mathcal{E} \\
        \Sigma^{-1}\Lambda^{-1}T^\vee & \longmapsto & \mu_T
    \end{array}.$
\end{center}
\end{thm}

\begin{proof}
    The theorem obviously holds if $|T|=1,2$, and also if $|T|=3$ by Example \ref{exmu}. We now suppose that $n:=|T|\geq 4$. We use the identity $d\Lambda h_{\mathcal{E}}^{\sigma_T} +\Lambda h_{\mathcal{E}}^{\sigma_T}  d=id_{\Lambda\mathcal{E}}-\Lambda\varphi_\mathcal{E}^{\sigma_T} $:
    \begin{multline*}
        d(\Lambda\mu_T)=\Lambda h_{\mathcal{E}}^{\sigma _T}d\left(\sum_{S\subset T}\Lambda\mu_{T/S}\circ_S\Lambda\mu_{S}\right)+\Lambda\varphi_\mathcal{E}^{\sigma _T}\left(\sum_{S\subset T}\Lambda\mu_{T/S}\circ_S\Lambda\mu_{S}\right)-\sum_{S\subset T}\Lambda\mu_{T/S}\circ_S\Lambda\mu_{S}.
    \end{multline*}

    \noindent By an immediate induction, for every non trivial tree $T$, we have that $\Lambda\mu_T\in\Lambda\mathcal{E}(n)_{-1}$, or equivalently $\mu_T\in\mathcal{E}(n)_{n-2}$. This then gives $|\mu_{T/S}\circ_S\mu_S|=n-3>0$. Thus, by definition of $\varphi_\mathcal{E}^{\sigma _T}$, we have that 
    $$\Lambda\varphi_\mathcal{E}^{\sigma _T}\left(\sum_{S\subset T}\Lambda\mu_{T/S}\circ_S\Lambda\mu_{S}\right)=0.$$

    \noindent We now prove that
    $$d\left(\sum_{S\subset T}\Lambda\mu_{T/S}\circ_S\Lambda\mu_{S}\right)=0.$$

    \noindent We compute:
    $$d\left(\sum_{S\subset T}\Lambda\mu_{T/S}\circ_S\Lambda\mu_{S}\right)=\sum_{S\subset T}d(\Lambda\mu_{T/S})\circ_S\Lambda\mu_{S}-\sum_{S\subset T}\Lambda\mu_{T/S}\circ_Sd(\Lambda\mu_{S}),$$

    \noindent because the differential of the Barratt-Eccles operad is compatible with its operad structure. Now, because $\mu_{S}=0$ if $S$ has only one vertex, we can consider subtrees of $T$ with at most $n-1$ vertices. We can then use the induction hypothesis. 
\begin{multline*}
    d\left(\sum_{S\subset T}\Lambda\mu_{T/S}\circ_S\Lambda\mu_{S}\right) = -\sum_{S\subset T}\sum_{U\subset T/S}(\Lambda\mu_{(T/S)/U}\circ_U\Lambda\mu_{U})\circ_S\Lambda\mu_{S}\\+\sum_{S\subset T}\sum_{U\subset S}\Lambda\mu_{T/S}\circ_S(\Lambda\mu_{S/U}\circ_U\Lambda\mu_{U})
\end{multline*}

\noindent We have two types of subtrees $U$ of $T/S$: either $U$ does not contain the vertex $S$, so that $U$ can be canonicaly seen as a subtree of $T$ such that $V_U\cap V_S=\emptyset$, or $U$ contains the vertex $S$, so that $U$ can be seen as a subtree of $T$ such that $S\subset U$. We thus have
\begin{multline*}
    d\left(\sum_{S\subset T}\Lambda\mu_{T/S}\circ_S\Lambda\mu_{S}\right) =-\sum_{S\subset T}\sum_{S\subset U\subset T}(\Lambda\mu_{T/U}\circ_{U/S}\Lambda\mu_{U/S})\circ_S\Lambda\mu_{S}\\ -\sum_{\substack{S,U\subset T\\ V_S\cap V_U=\emptyset}}(\Lambda\mu_{(T/S)/U}\circ_U\Lambda\mu_{U})\circ_S\Lambda\mu_{S}\\+\sum_{S\subset T}\sum_{U\subset S}\Lambda\mu_{T/S}\circ_S(\Lambda\mu_{S/U}\circ_U\Lambda\mu_{U}).
\end{multline*}

\noindent In the second line, by exchanging the roles of $S$ and $U$, we have a sum of terms of the form
$$(\Lambda\mu_{(T/S)/U}\circ_U\Lambda\mu_{U})\circ_S\Lambda\mu_{S}+(\Lambda\mu_{(T/U)/S}\circ_S\Lambda\mu_{S})\circ_U\Lambda\mu_{U}$$

\noindent which is $0$, because of the operadic axioms and because $|\Lambda\mu_{U}|=|\Lambda\mu_{S}|=-1$. We thus obtain

\begin{center}
$\begin{array}{lll}
    \displaystyle d\left(\sum_{S\subset T}\Lambda\mu_{T/S}\circ_S\Lambda\mu_{S}\right) & = & \displaystyle-\sum_{U\subset T}\sum_{S\subset U}(\Lambda\mu_{(T/U)}\circ_{U/S}\Lambda\mu_{(U/S)})\circ_S\Lambda\mu_{S}\\
    & &\displaystyle+\sum_{S\subset T}\sum_{U\subset S}\Lambda\mu_{T/S}\circ_S(\Lambda\mu_{(S/U)}\circ_U\Lambda\mu_{U})\\
    & = & \displaystyle-\sum_{S\subset T}\sum_{U\subset S}(\Lambda\mu_{T/S}\circ_{S/U}\Lambda\mu_{(S/U)})\circ_U\Lambda\mu_{U}\\
    & & \displaystyle+\sum_{S\subset T}\sum_{U\subset S}\Lambda\mu_{T/S}\circ_S(\Lambda\mu_{(S/U)}\circ_U\Lambda\mu_{U})\\
    & = & 0,
    \end{array}$
    \end{center}

    \noindent using again the associativity of the operadic composition.
\end{proof}

\subsection{On the twisted coderivation of ${\mathcal{B}race}^c(\Sigma N^\ast(\Delta^n))$}\label{sec:232}

Every $\mathcal{E}$-algebra $E$ inherits a $B^c(\Lambda^{-1}\mathcal{B}race^\vee)$-algebra structure induced by the morphism $B^c(\Lambda^{-1}\mathcal{B}race^\vee)\longrightarrow\mathcal{E}$ given by Theorem \ref{mu}. This algebraic structure is equivalent to giving a twisted coderivation on the free brace coalgebra $\mathcal{B}race^c(\Sigma E)$ generated by $\Sigma E$. Recall that 
    $$\mathcal{B}race^c(\Sigma E)=\bigoplus_{k\geq 1}\bigoplus_{T\in\mathcal{PRT}(k)}(T^\vee\otimes(\Sigma E)^{\otimes k})^{\Sigma_{k}},$$

    \noindent where we equalize the action of $\Sigma_{k}$ on $T^\vee$ with the action of $\Sigma_k$ by permutation on $(\Sigma E)^{\otimes k}$. In the following, we identify the $k=1$ component with $\Sigma E$.\\ If $E$ is finite dimensional, giving such a coderivation is equivalent to giving a twisting morphism $\partial^E$ on the free complete brace algebra 
    $$\widehat{\mathcal{B}race}(\Sigma^{-1}E^\vee)=\prod_{k\geq 1}\bigoplus_{T\in\mathcal{PRT}(k)}(T^\vee\otimes(\Sigma^{-1}E^\vee)^{\otimes k})_{\Sigma_{k}}.$$
    
    \noindent This completion is obtained from the free brace algebra $\mathcal{B}race(\Sigma^{-1} E^\vee)$ endowed with the filtration
    $$F_p\mathcal{B}race(\Sigma^{-1} E^\vee)=\bigoplus_{k\geq p+1}\bigoplus_{T\in\mathcal{PRT}(k)}(T^\vee\otimes(\Sigma^{-1}E^\vee)^{\otimes k})_{\Sigma_{k}}.$$

    \noindent One can check that the brace algebra structure of $\mathcal{B}race(\Sigma^{-1}E^\vee)$ preserves this filtration so that the completion $\widehat{\mathcal{B}race}(\Sigma^{-1}E^\vee)$ is endowed with a brace algebra structure. The derivation $\partial^E$ is thus given by a generating function
    $$\partial^E=\sum_{\substack{T\in\mathcal{PRT}\\T\ \text{\normalfont canonical}}}T\otimes\Lambda\mu_T^{E^\vee}$$

    \noindent where $\mu_T^{E^\vee}:E^\vee\longrightarrow (E^\vee)^{\otimes |T|}$ is the map induced by $\mu_T\in\mathcal{E}(|T|)$ given by the $\mathcal{E}$-algebra structure of $E$. Note also that the definition of $\partial^E$ is natural with respect to $E$. Namely, if $F$ is an other finite dimensional $\mathcal{E}$-algebra and $f:F\longrightarrow E$ a morphism of $\mathcal{E}$-algebras, then we have a commutative diagram
\[\begin{tikzcd}
	{\widehat{\mathcal{B}race}(\Sigma^{-1}E^\vee)} && {\widehat{\mathcal{B}race}(\Sigma^{-1}F^\vee)} \\
	{\widehat{\mathcal{B}race}(\Sigma^{-1}E^\vee)} && {\widehat{\mathcal{B}race}(\Sigma^{-1}F^\vee)}
	\arrow["{\widehat{\mathcal{B}race}(\Sigma^{-1}f^\vee)}", from=1-1, to=1-3]
	\arrow["{\partial^E}"', from=1-1, to=2-1]
	\arrow["{\partial^F}", from=1-3, to=2-3]
	\arrow["{\widehat{\mathcal{B}race}(\Sigma^{-1}f^\vee)}"', from=2-1, to=2-3]
\end{tikzcd}.\]

    The goal of this subsection is to give a computation of the differentials $\partial^n:=\partial^{N^\ast(\Delta^n)}$ by induction on $n\geq 0$. To achieve this, we use Construction \ref{consmu} which defines the $\Lambda\mu_T$'s, and analyze the coaction of $TR(\Lambda\mu_T)$ on $\Sigma^{-1} \underline{0\cdots n}\in\Sigma^{-1}N_*(\Delta^n)$, where $TR$ is the table reduction morphism (see Proposition \ref{tablered}).\\

Let $r,d\geq 0$ and $1\leq k\leq r$. We denote by
$$\pi_k:\chi(r)_d\longrightarrow\chi(\llbracket k,r\rrbracket)_d$$

\noindent the morphism obtained by forgetting $1, \ldots, k-1$. If the degree does not match, we send the surjection on $0$. Note that $\pi_1=id$.

\begin{lm}\label{indmu}
    Let $w\in\mathcal{E}(r)_d$. Then
    $$TR(h_{\mathcal{E}}^{id}(w))=\sum_{k=1}^r1\cdots k\cdot\pi_k(TR(w)),$$

    \noindent where we consider the concatenation of a surjection in $\chi(k)$ with a surjection of $\chi(\llbracket k,r\rrbracket)$, giving a surjection in $\chi(r)$.
\end{lm}

\begin{proof}
    Let $w=(w_0,\ldots,w_d)\in\mathcal{E}(r)_d$. On one hand, we have
    $$TR(h_\mathcal{E}^{id}(w))=\sum_{k=1}^r\sum_{\substack{r_0+\cdots+r_d=r+d-k+1\\ 1\leq r_i\leq r}} \left|\begin{array}{ccc}
         1 & \cdots & k \\
         w_0'(1) & \cdots & w_0'(r_0)\\
         \vdots & & \vdots\\
         w_d'(1) & \cdots & w_d'(r_d)\\
    \end{array}\right.$$

    \noindent where each $w_k'(1),\ldots, w_k'(r_k)$ are obtained from $w_k(1),\ldots,w_k(r)$ by taking the first $r_k$ terms which are not among
    \begin{center}
        $\begin{array}{ccc}
             1 & \cdots & k-1  \\
             w_0'(1) & \cdots & w_0'(r_0-1)\\
             \vdots & &\vdots \\
             w_{k-1}'(1) & \cdots & w_{k-1}'(r_{k-1}-1)
        \end{array}.$
    \end{center}

    \noindent On the other hand, if we write
    $$TR(w)=\sum_{\substack{r_0'+\cdots+r_d'=r+d\\1\leq r_i\leq r}}\left|\begin{array}{ccc}
             w_0''(1) & \ldots & w_0''(r_0')\\
             \vdots & & \vdots\\
             w_{d}''(1) & \ldots & w_{d}''(r_{d}')
        \end{array}\right.$$

    \noindent where each $w_k''(1),\ldots,w_k''(r_k')$ are obtained from $w_k(1),\ldots,w_k(r)$ by taking the first $r_k'$ terms which are not among
    \begin{center}
        $\begin{array}{ccc}
             w_0''(1) & \cdots & w_0''(r_0'-1)\\
             \vdots & &\vdots \\
             w_{k-1}''(1) & \cdots & w_{k-1}''(r_{k-1}'-1)
        \end{array},$
    \end{center}
\noindent then
    \begin{center}
        $\displaystyle \pi_k(TR(w))=\sum_{\substack{r_0+\cdots+r_d=r+d-k+1\\1\leq r_i\leq r}}\left|\begin{array}{ccc}
             w_0'(1) & \ldots & w_0'(r_0)\\
             \vdots & & \vdots\\
             w_{d}'(1) & \ldots & w_{d}'(r_{d})
        \end{array}\right.$
    \end{center}

    \noindent where, as above, each $w_k'(1),\ldots,w_k'(r_k)$ are obtained from $w_k(1),\ldots,w_k(r)$ by taking the first $r_k$ terms which are not among
    \begin{center}
        $\begin{array}{ccc}
             1 & \cdots & k-1  \\
             w_0'(1) & \cdots & w_0'(r_0-1)\\
             \vdots & &\vdots \\
             w_{k-1}'(1) & \cdots & w_{k-1}'(r_{k-1}-1)
        \end{array}.$
    \end{center}
    
    \noindent We then deduce
    $$\sum_{k=1}^r1\cdots k\cdot\pi_k(TR(w))=\sum_{k=1}^r\sum_{\substack{r_0+\cdots+r_d=r+d-k+1\\ 1\leq r_i\leq r}} \left|\begin{array}{ccc}
         1 & \cdots & k \\
         w_0'(1) & \cdots & w_0'(r_0)\\
         \vdots & & \vdots\\
         w_d'(1) & \cdots & w_d'(r_d)\\
    \end{array}\right.$$

    \noindent which proves the lemma.
\end{proof}

\begin{lm}\label{surj}
    Let $X$ be a totally ordered finite set and $T\in\mathcal{PRT}(X)$ with $|T|\geq 3$. We let $b_T$ to be the number of vertices in the first branch of $T$ (without the root).
    \begin{itemize}
        \item If $b_T\geq 2$, then $TR(\mu_T)=0$.
        \item If $b_T=1$, then, if we denote by $r$ the root of $T$ and $s$ the second element of $V_T$, then there exists ${u}_{T}\in\chi(X\setminus\{r\})$ such that $TR(\mu_T)=rs\cdot {u}_{T}$.
    \end{itemize}
\end{lm}

\begin{proof}
   It is sufficient to prove the lemma for a canonical tree $T\in\chi(1<\cdots<|T|)$. We prove the lemma by induction on $|T|$. If $|T|=3$, then Example \ref{exmu} implies that that the assertion of the lemma is true. We now suppose that $|T|\geq 4$. Suppose first that $b_T\geq 2$.  By Lemma \ref{indmu},
    $$TR(\mu_T)=-\sum_{k=2}^{|T|}\sum_{S\subset T}\pm1\cdots k\cdot\pi_k(TR(\mu_{T/S})\circ_STR(\mu_{S})).$$

    \noindent Note that the sum begins at $k=2$, since that, by an immediate induction, the elements $TR(\mu_T)$'s begins at $1$. Let $S\subset T$. Our goal is to prove that, for every $2\leq k\leq|T|$,
    $$1\cdots k\cdot \pi_k(TR(\mu_{T/S})\circ_S TR(\mu_{S}))=0$$

    \noindent We distinguish several cases. If either $b_S\geq 2$ or $b_{T/S}\geq 2$, then the identity holds by induction hypothesis. Suppose now that $b_S=b_{T/S}=1$. Since $b_T\geq 2$, we have these different cases.
    
    \begin{itemize}
        \item If $r(S)$ is one of the vertex of the first branch of $T$, then, because $b_{T/S}=1$, the tree $S$ is the full first branch of $T$. In this situation, the first permutation occurring in $\mu_{T/S}\circ_S\mu_{S}$ is the identity permutation, so that taking the image of this element under $h_\mathcal{E}^{id}$ gives $0$.
        \item If $r(S)=1$, since $b_S=b_{T/S}=1$, we have $b_T=2$, and the second vertex of $S$ is $2$. Then, by induction hypothesis,
        $$TR(\mu_{S})=12\cdot u_{S};$$
        $$TR(\mu_{T/S})=S3\cdot u_{T/S},$$

        \noindent where $u_{S}\in\chi(V_S\setminus\{1\})$ and $u_{T/S}\in\chi(V_{T/S}\setminus\{S\})$. We thus have
        $$TR(\mu_{T/S})\circ_S TR(\mu_{S})=12\cdot u_{S}\cdot 3\cdot u_{T/S}.$$

        \noindent Since $2$ occurs in the surjection $u_{S}$, taking the image of such an element under $\pi_k$ will gives $0$ for every $k\geq 3$. If $k=2$, then the corresponding term is
        $$12\cdot\pi_2(12\cdot u_{S}\cdot 3\cdot u_{T/S})=122\cdot u_{S}\cdot 3\cdot u_{T/S}=0.$$

        \item If $r(S)$ is neither $1$ nor an element of the first branch of $T$, then we cannot have $b_{T/S}=1$, since we have supposed $b_T\geq 2$.
    \end{itemize}
    
    \noindent This concludes the case $b_T\geq 2$. We now suppose that $b_T=1$. We use again the identity up to signs
    $$TR(\mu_T)=-\sum_{k=2}^{|T|}\sum_{S\subset T}\pm 1\cdots k\cdot\pi_k(TR(\mu_{T/S})\circ_S TR(\mu_{S})),$$

    \noindent given by Lemma \ref{indmu}. Our goal is to prove that the terms of this sum with $k\geq 3$ are $0$. Let $S\subset T$ be such that $b_S=b_{T/S}=1$. We have three cases.

\begin{itemize}
    \item If $|S|,|T/S|\neq 2$, then, since $b_T=1$, we cannot have $r(S)=2$ so that $2$ is the second occurrence of either the surjection $TR(\mu_{S})$ or the surjection $TR(\mu_{T/S})$. By induction hypothesis, in any case, taking the image of $\pi_k$ for every $k\geq 3$ of the corresponding element gives $0$.

    \item Suppose now that $|S|=2$ and $|T/S|\neq 2$. If $S$ does not contain $2$, then the above argument gives a resulting element equal to $0$ in the sum, for $k\geq 3$. Else, we have $r(S)=1$ so that $TR(\mu_{S})=12$. We thus have, by induction hypothesis, that $TR(\mu_{T/S})\circ_STR(\mu_{S})$ is of the form $123 \cdot {u}_{T/S}$ where $u_{T/S}=0$ or ${u}_{T/S}\in\chi(V_{T/S}\setminus\{S\})$. The image of this element under $\pi_k$ is $0$ for every $k\geq 4$. If $k=3$, then we have $123\cdot\pi_3(123\cdot u_{T/S})=1233\cdot\pi_3(u_{T/S})=0$.

    \item Suppose now that $|T/S|=2$ and $|S|\neq 2$. Because $b_T=1$, we have $r(S)=1$. We then obtain that $TR(\mu_{T/S})=S\alpha$ for some $1\leq\alpha\leq |T|$. Therefore, by induction hypothesis, the composite $TR(\mu_{T/S})\circ_S TR(\mu_{S})$ is given by $1\Sigma\cdot u_{S}\cdot\alpha$, where $\Sigma$ is the second vertex of $S$. If $\alpha\neq 2$, then $\Sigma=2$ and $u_{S}$ contains an occurrence of $2$. Taking the image of such element under $\pi_k$ for every $k\geq 3$ gives $0$. If $\alpha=2$, then $\Sigma=3$. The image under $\pi_k$ of the resulting composite gives $0$ for every $k\geq 4$. If $k=3$, then the resulting term in the sum is $123\cdot\pi_3(13\cdot{u}_{S}\cdot2)=1233\cdot\pi_3({u}_{S})=0$.
\end{itemize}

    The lemma is proved.
\end{proof}

This lemma allows us to compute $\partial^0$.

\begin{lm}\label{firstdiff}
    We have the following identity:
    $$\partial^0(\Sigma^{-1} \underline{0})=
    \begin{tikzpicture}[baseline={([yshift=-.5ex]current bounding box.center)},scale=1]
    \node[scale=1] (i) at (0,0) {$\Sigma^{-1} \underline{0}$};
    \node[scale=1] (2) at (0,1) 
    {$\Sigma^{-1} \underline{0}$};
    \draw (i) -- (2);
    \end{tikzpicture}.$$
\end{lm}

\begin{proof}
    For every canonical $n$-tree $T$ with $n\geq 3$, we have that $\Lambda\mu_T(\Sigma^{-1} \underline{0})=0$. Indeed, by Lemma \ref{surj}, in this case, the number $2$ occurs at least two times in the surjection $TR(\mu_T)$, so that its coevaluation on $\underline{0}$ gives $0$ by definition of the interval cuts operations. There only remains the case $n=2$. The associated canonical tree is $T=\begin{tikzpicture}[baseline={([yshift=-.5ex]current bounding box.center)},scale=0.5]
    \node[draw,circle,scale=0.5] (i) at (0,0) {$1$};
    \node[draw,circle,scale=0.5] (1) at (0,1) {$2$};
    \draw (i) -- (1);
    \end{tikzpicture}$, which gives $TR(\mu_T)=(12)$ by definition. By definition of the interval cuts operations, the coevaluation of the surjection $(12)$ on $\underline{0}$ is $\underline{0}\otimes\underline{0}$, which gives the result.
\end{proof}

For every $n,k\geq 0$, let $\partial_{(k)}^n$ be the composite of $\partial^n$ with the projection on trees with $k+1$ vertices. Our goal is to compute $\partial_{(k)}^n$ by induction on $n,k\geq 0$. Recall, from after Proposition \ref{homonstar}, the two morphisms $\varphi^0_n:N_*(\Delta^n)\longrightarrow N_*(\Delta^n)$ and $h_n^0:N_*(\Delta^n)\longrightarrow N_{*+1}(\Delta^n)$, which satisfy
$$d h_n^0+h_n^0d=id_ {N_*(\Delta^n)}-\varphi_n^0.$$

\noindent We keep the notations $\varphi_n^0$ and $h_n^0$ for the two induced morphisms on $\Sigma ^{-1}N_*(\Delta^n)$, which also satisfy the same homotopy relation. We extend such morphisms on the tensor algebra of $\Sigma ^{-1}N_*(\Delta^n)$ by
$$(\varphi_n^0)(x_1\otimes\cdots\otimes x_p)=\varphi_n^0(x_1)\otimes\cdots\otimes\varphi_n^0(x_p);$$
$$(h_n^0)(x_1\otimes\cdots\otimes x_p)=\sum_{i=1}^p\pm \varphi_n^0(x_1)\otimes \cdots\otimes \varphi_n^0(x_{i-1})\otimes h_n^0(x_i)\otimes x_{i+1}\otimes \cdots\otimes x_p,$$

\noindent for every $x_1,\ldots,x_p\in\Sigma ^{-1}N_*(\Delta^n)$. We extend $\varphi_n^0$ and $h_n^0$ on $\mathcal{B}race(\Sigma ^{-1}N_*(\Delta^n))$ by setting, for every canonical tree $T$ and $X\in(\Sigma ^{-1}N_*(\Delta^n))^{\otimes |T|}$,
$$\Phi_n^0(T\otimes X)=T\otimes\varphi_n^0(X);$$
$$H_n^0(T\otimes X)=T\otimes h_n^0(X).$$

\noindent These definitions are extended to any tree $T$ by symmetry. We then obtain the homotopy relation:
\medskip
$$\partial_{(0)}^n H_n^0+H_n^0\partial_{(0)}^n=id-\Phi_n^0.$$
\medskip

\begin{thm}\label{rec}
    Let $n,k\geq 1$. Then
    \medskip
    $$\partial_{(k)}^n(\Sigma^{-1} \underline{0\cdots n})=-H_n^0\left(\sum_{\substack{p+q=k\\ p\neq 0}}\partial_{(p)}^n\partial_{(q)}^n(\Sigma^{-1} \underline{0\cdots n})\right).$$
    \medskip

    In particular, we have the induction relation:
    \medskip
    $$\partial_{(k)}^n(\Sigma^{-1} \underline{0\cdots n})=H_n^0d^0\partial_{(k)}^{n-1}(\Sigma^{-1} \underline{0\cdots (n-1)})-H_n^0\left(\sum_{\substack{p+q=k\\ p,q\neq 0}}\partial_{(p)}^n\partial_{(q)}^n(\Sigma^{-1} \underline{0\cdots n})\right).$$
    \medskip
\end{thm}

Starting from Lemma \ref{firstdiff}, this theorem allows us to compute the elements $\partial_{(k)}^n(\Sigma^{-1}\underline{0\cdots n})$ by induction on $n,k\geq 0$.\\

\begin{proof}
    We have
    \medskip
    $$\partial_{(0)}^n\partial_{(k)}^n(\Sigma^{-1} \underline{0\cdots n})=-\sum_{\substack{p+q=k\\p\neq 0}}\partial_{(p)}^n\partial_{(q)}^n(\Sigma^{-1} \underline{0\cdots n}).$$
    \medskip

\noindent Applying $H_n^0$ on this equality gives
\begin{multline*}
    \partial_{(k)}^n(\Sigma^{-1} \underline{0\cdots n})=-H_n^0\left(\sum_{\substack{p+q=k\\ p\neq 0}}\partial_{(p)}^n\partial_{(q)}^n(\Sigma^{-1} \underline{0\cdots n})\right)\\+\Phi_n^0\left(\sum_{\substack{p+q=k\\ p\neq 0}}\partial_{(p)}^n\partial_{(q)}^n(\Sigma^{-1} \underline{0\cdots n})\right)-\partial_{(0)}^n H_n^0\partial_{(k)}^n(\Sigma^{-1} \underline{0\cdots n})
\end{multline*}

\noindent The second term on the right hand-side vanishes, since the differentials $\partial^n$ are compatible with the simplicial structure of $\widehat{\mathcal{B}race}(\Sigma ^{-1}N_*(\Delta^\bullet))$ defined tensor-wise and $\Phi_n^0(\Sigma^{-1} \underline{0\cdots n})=0$ since $n\geq 1$. We now deal with the last term. Let $T$ be a canonical tree with $|T|=k+1$ and $b_T=1$. By Lemma \ref{surj}, there exists ${u}_{T}\in\chi(\llbracket 2,k+1\rrbracket)$ such that
$$TR(\mu_T)=12\cdot{u}_{T}.$$

\noindent Then, up to a sign, the elements occurring in each terms of $TR(\mu_T)(\underline{0\cdots n})$ are on the form
$$\underline{0\cdots k}\otimes \underline{x}\otimes X,$$

\noindent where $\underline{x}\in N_*(\Delta^n)$ has a length which is at least $2$, and $X$ is a tensor product of elements in $N_*(\Delta^n)$. The image of such elements by $h_n^0$ is $0$. We thus deduce that $H_n^0\partial_{(k)}^n(\Sigma^{-1} \underline{0\cdots n})=0$, which proves the first formula.\\

 We now prove the induction relation. For every $n,k\geq 1$, we have
$$\partial_{(k)}^n(\Sigma^{-1} \underline{0\cdots n})=-H_n^0\partial_{(k)}^n\partial_{(0)}^n(\Sigma^{-1} \underline{0\cdots n})-H_n^0\left(\sum_{\substack{p+q=k\\ p,q\neq 0}}\partial_{(p)}^n\partial_{(q)}^n(\Sigma^{-1} \underline{0\cdots n})\right).$$

\noindent We only need to compute the first term. We have
$$\partial_{(k)}^n\partial_{(0)}^n(\Sigma^{-1} \underline{0\cdots n})=-\sum_{i=0}^n(-1)^id^i\partial_{(k)}^{n-1}(\Sigma^{-1} \underline{0\cdots (n-1)}).$$

\noindent For every $1\leq i\leq n$ we have $H_n^0 d^i=d^i H_{n-1}^0$. From the first formula that we have proved, we deduce that the element $\partial_{(k)}^{n-1}(\Sigma^{-1}\underline{0\cdots (n-1)})$ is in the image of $H_{n-1}^0$. Since $H_{n-1}^0H_{n-1}^0=0$, we obtain at the end
$$H_n^0\partial_{(k)}^n\partial_{(0)}^n(\Sigma^{-1} \underline{0\cdots n})=H_n^0d^0\partial_{(k)}^{n-1}(\Sigma^{-1}\underline{0\cdots (n-1)}),$$

\noindent which proves the theorem.
\end{proof}



\begin{cor}\label{diff1}
    Let $n\geq 0$. Then
    $$\partial^n_{(1)}(\Sigma^{-1} \underline{0\cdots n})=
    \sum_{k=0}^n(-1)^k\begin{tikzpicture}[baseline={([yshift=-.5ex]current bounding box.center)},scale=1]
    \node[scale=1] (i) at (0,0) {$\Sigma^{-1} \underline{0\cdots k}$};
    \node[scale=1] (2) at (0,1) 
    {$\Sigma^{-1} \underline{k\cdots n}$};
    \draw (i) -- (2);
    \end{tikzpicture}.$$
\end{cor}

\begin{proof}
    We prove the corollary by induction on $n\geq 0$. Lemma \ref{firstdiff} proves the case $n=0$. We now suppose that $n\geq 1$. Theorem \ref{rec} gives
    $$\partial_{(1)}^n(\Sigma^{-1}\underline{0\cdots n})=H_n^0d^0\partial_{(1)}^{n-1}(\Sigma^{-1}\underline{0\cdots (n-1)}).$$

    \noindent By induction hypothesis on $n$, we have
    $$d^0\partial^{n-1}_{(1)}(\Sigma^{-1} \underline{0\cdots (n-1)})=
    -\sum_{k=1}^{n}(-1)^k\begin{tikzpicture}[baseline={([yshift=-.5ex]current bounding box.center)},scale=1]
    \node[scale=1] (i) at (0,0) {$\Sigma^{-1} \underline{1\cdots k}$};
    \node[scale=1] (2) at (0,1) 
    {$\Sigma^{-1} \underline{k\cdots n}$};
    \draw (i) -- (2);
    \end{tikzpicture}.$$

    \noindent This gives
    $$\partial^{n}_{(1)}(\Sigma^{-1} \underline{0\cdots n})=
    \sum_{k=1}^{n}(-1)^k\begin{tikzpicture}[baseline={([yshift=-.5ex]current bounding box.center)},scale=1]
    \node[scale=1] (i) at (0,0) {$\Sigma^{-1} \underline{0\cdots k}$};
    \node[scale=1] (2) at (0,1) 
    {$\Sigma^{-1} \underline{k\cdots n}$};
    \draw (i) -- (2);
    \end{tikzpicture}+\begin{tikzpicture}[baseline={([yshift=-.5ex]current bounding box.center)},scale=1]
    \node[scale=1] (i) at (0,0) {$\Sigma^{-1} \underline{0}$};
    \node[scale=1] (2) at (0,1) 
    {$\Sigma^{-1} \underline{0\cdots n}$};
    \draw (i) -- (2);
    \end{tikzpicture},$$

    \noindent which gives the result.
\end{proof}

We can also compute the differentials $\partial^1,\partial^2$ and $\partial^3$. In the following corollary, we set
$$\Sigma^{-1}\underline{x}\overline\circledcirc\Sigma^{-1}\underline{y}=\Sigma^{-1}\underline{x}+\Sigma^{-1}\underline{y}+\sum_{k\geq 1}\begin{tikzpicture}[baseline={([yshift=-.5ex]current bounding box.center)},scale=1]
    \node (ix) at (0,0) {$\Sigma^{-1} \underline{x}$};
    \node (1x) at (-1,1) {$\Sigma^{-1} \underline{y}$};
    \node (2x) at (0,1) {$\overset{k}{\cdots}$};
    \node (3x) at (1,1) {$\Sigma^{-1} \underline{y}$};
    \draw (ix) -- (1x);
    \draw (ix) -- (3x);
    \end{tikzpicture},$$

\noindent for every degree $1$ elements $\underline{x},\underline{y}\in N_*(\Delta^n)$. Note that the operation $\overline\circledcirc$ corresponds to the circular product $\circledcirc$ defined in \cite[Remark 2.20]{moi} in the brace algebra $\widehat{\mathcal{B}race}(\Sigma^{-1}N_*(\Delta^n))$. In particular, the product $\overline\circledcirc$ is associative. This operation is reviewed in details in the beginning of $\mathsection$\ref{sec:242}. In order to write shorter formulas, we also put a weight on the arrows of our trees. We precisely set, for every $\underline{x}\in N_*(\Delta^n)$, for every tree $T_1,\ldots,T_k\in\widehat{\mathcal{B}race}(\Sigma^{-1}N_*(\Delta^n))$ and for every integers $r_1,\ldots,r_k\geq 1$,
$$\begin{tikzpicture}[baseline={([yshift=-.5ex]current bounding box.center)},scale=1]
    \node (i) at (0,0) {$\Sigma^{-1} \underline{x}$};
    \node (1) at (-1,1) {$T_1$};
    \node (2) at (0,1) {${\cdots}$};
    \node (3) at (1,1) {$T_k$};
    \draw (i) -- node[midway, draw, circle, fill=white, inner sep=1pt, scale=0.5] {$r_1$} (1);
    \draw (i) -- node[midway, draw, circle, fill=white, inner sep=1pt, scale=0.5] {$r_k$} (3);
    \end{tikzpicture}=\begin{tikzpicture}[baseline={([yshift=-.5ex]current bounding box.center)},scale=1]
    \node (i) at (0,0) {$\Sigma^{-1} \underline{x}$};
    \node (1) at (-3,1) {$T_1$};
    \node (1b) at (-2,1) {$\overset{r_1}{\cdots}$};
    \node (1bb) at (-1,1) {$T_1$};
    \node (2) at (0,1) {${\cdots}$};
    \node (3) at (1,1) {$T_k$};
    \node (3b) at (2,1) {$\overset{r_k}{\cdots}$};
    \node (3bb) at (3,1) {$T_k$};
    \draw (i) -- (1);
    \draw (i) -- (1bb);
    \draw (i) -- (3);
    \draw (i) -- (3bb);
    \end{tikzpicture}.$$

\noindent If $r_i=1$ for some $i$, we remove the weight from the arrow.\\

In the following corollary, we drop the desuspension $\Sigma^{-1}$ on basis elements of $\Sigma^{-1}N_*(\Delta^n)$.

\begin{cor}\label{calculdiff}
    We have the following formulas in $\widehat{\mathcal{B}race}(\Sigma ^{-1}N_*(\Delta^n))$:
    \begin{flalign*}
    &\indent\bullet\ \partial^1( \underline{01})= \underline{0}- \underline{1}-
    \begin{tikzpicture}[baseline={([yshift=-.5ex]current bounding box.center)},scale=1]
    \node[scale=1] (i) at (0,0) {$ \underline{01}$};
    \node[scale=1] (2) at (0,1) {$ \underline{1}$};
    \draw (i) -- (2);
    \end{tikzpicture}+\sum_{k\geq 1}\begin{tikzpicture}[baseline={([yshift=-.5ex]current bounding box.center)},scale=1]
    \node[scale=1] (i) at (0,0) {$ \underline{0}$};
    \node[scale=1] (1) at (0,1) {$ \underline{01}$};
    \draw (i) -- node[midway, draw, circle, fill=white, inner sep=1pt, scale=0.5] {$k$} (1);
    \end{tikzpicture};&
\end{flalign*}
\begin{flalign*}
    &\indent\bullet\ \partial^2( \underline{012})= \underline{02}- \underline{01}- \underline{12}+
    \begin{tikzpicture}[baseline={([yshift=-.5ex]current bounding box.center)},scale=1]
    \node[scale=1] (i) at (0,0) {$ \underline{012}$};
    \node[scale=1] (2) at (0,1) {$ \underline{2}$};
    \draw (i) -- (2);
    \end{tikzpicture}-\sum_{k\geq 1}\begin{tikzpicture}[baseline={([yshift=-.5ex]current bounding box.center)},scale=1]
    \node[scale=1] (i) at (0,0) {$ \underline{01}$};
    \node[scale=1] (1) at (0,1) {$ \underline{12}$};
    \draw (i) -- node[midway, draw, circle, fill=white, inner sep=1pt, scale=0.5] {$k$} (1);
    \end{tikzpicture}+\sum_{i,j\geq 0}\begin{tikzpicture}[baseline={([yshift=-.5ex]current bounding box.center)},scale=1]
    \node (i) at (0,0) {$ \underline{0}$};
    \node (1) at (-1.5,1) {$ \underline{02}$};
    \node (2) at (0,1) {$ \underline{012}$};
    \node (3) at (2,1) {$\underline{01}\overline\circledcirc\underline{12}$};
    \draw (i) -- node[midway, draw, circle, fill=white, inner sep=1pt, scale=0.5] {$i$} (1);
    \draw (i) -- (2);
    \draw (i) -- node[midway, draw, circle, fill=white, inner sep=1pt, scale=0.5] {$j$} (3);
    \end{tikzpicture};
\end{flalign*}
\medskip
\begin{multline*}
    \indent\bullet\ \partial^3( \underline{0123})= \underline{023}- \underline{123}+ \underline{013}-\underline{012}-
    \begin{tikzpicture}[baseline={([yshift=-.5ex]current bounding box.center)},scale=1]
    \node[scale=1] (i) at (0,0) {$ \underline{0123}$};
    \node[scale=1] (2) at (0,1) {$ \underline{3}$};
    \draw (i) -- (2);
    \end{tikzpicture}\\+\sum_{k\geq 1}\begin{tikzpicture}[baseline={([yshift=-.5ex]current bounding box.center)},scale=1]
    \node[scale=1] (i) at (0,0) {$ \underline{012}$};
    \node[scale=1] (1) at (0,1) {$ \underline{23}$};
    \draw (i) -- node[midway, draw, circle, fill=white, inner sep=1pt, scale=0.5] {$k$} (1);
    \end{tikzpicture}-\sum_{i,j\geq 0}\begin{tikzpicture}[baseline={([yshift=-.5ex]current bounding box.center)},scale=1]
    \node (i) at (0,0) {$ \underline{01}$};
    \node (1) at (-1.5,1) {$ \underline{13}$};
    \node (2) at (0,1) {$ \underline{123}$};
    \node (3) at (2,1) {$\underline{12}\overline\circledcirc\underline{23}$};
    \draw (i) -- node[midway, draw, circle, fill=white, inner sep=1pt, scale=0.5] {$i$} (1);
    \draw (i) -- (2);
    \draw (i) -- node[midway, draw, circle, fill=white, inner sep=1pt, scale=0.5] {$j$} (3);
    \end{tikzpicture}\\-\sum_{i,j,k\geq 0}\begin{tikzpicture}[baseline={([yshift=-.5ex]current bounding box.center)},scale=1]
    \node (i) at (0,0) {$ \underline{0}$};
    \node (1) at (-3.5,1.5) {$ \underline{03}$};
    \node (2) at (-2,1.5) {$ \underline{023}$};
    \node (3) at (0,1.5) {$\underline{02}\overline\circledcirc\underline{23}$};
    \node (4) at (3,1.5) {$\underline{012}\circledcirc(1+\underline{23})$};
    \node (5) at (6.5,1.5) {$\underline{01}\overline\circledcirc\underline{12}\overline\circledcirc\underline{23}$};
    \draw (i) -- node[midway, draw, circle, fill=white, inner sep=1pt, scale=0.5] {$i$} (1);
    \draw (i) -- (2);
    \draw (i) -- node[midway, draw, circle, fill=white, inner sep=1pt, scale=0.5] {$j$} (3);
    \draw (i) -- (4);
    \draw (i) -- node[midway, draw, circle, fill=white, inner sep=1pt, scale=0.5] {$k$} (5);
    \end{tikzpicture}\\+\sum_{i,j,k,l,m\geq 0}\begin{tikzpicture}[baseline={([yshift=-.5ex]current bounding box.center)},scale=1]
    \node (i) at (0,0) {$ \underline{0}$};
    \node (1) at (-3.5,1.5) {$ \underline{03}$};
    \node (2) at (-2,1.5) {$ \underline{013}$};
    \node (3) at (0,1.5) {$\underline{01}\overline\circledcirc\underline{13}$};
    \node (4) at (2,1.5) {$\underline{01}$};
    \node (4b) at (0.5,2.5) {$\underline{13}$};
    \node (4bb) at (2,2.5) {$\underline{123}$};
    \node (4bbb) at (4,2.5) {$\underline{12}\overline\circledcirc\underline{23}$};
    \node (5) at (4.5,1.5) {$\underline{01}\overline\circledcirc\underline{12}\overline\circledcirc\underline{23}$};
    \draw (i) -- node[midway, draw, circle, fill=white, inner sep=1pt, scale=0.5] {$i$} (1);
    \draw (i) -- (2);
    \draw (i) -- node[midway, draw, circle, fill=white, inner sep=1pt, scale=0.5] {$j$} (3);
    \draw (i) -- (4);
    \draw (4) -- (4bb);
    \draw (i) -- node[midway, draw, circle, fill=white, inner sep=1pt, scale=0.5] {$k$} (5);
    \draw (4) -- node[midway, draw, circle, fill=white, inner sep=1pt, scale=0.5] {$l$} (4b);
    \draw (4) -- node[midway, draw, circle, fill=white, inner sep=1pt, scale=0.5] {$m$} (4bbb);
    \end{tikzpicture}\\+\sum_{i,j,k\geq 0}\begin{tikzpicture}[baseline={([yshift=-.5ex]current bounding box.center)},scale=1]
    \node (i) at (0,0) {$ \underline{0}$};
    \node (1) at (-3.5,1.5) {$ \underline{03}$};
    \node (2) at (-2,1.5) {$ \underline{013}$};
    \node (3) at (0,1.5) {$\underline{01}\overline\circledcirc\underline{13}$};
    \node (4) at (2,1.5) {$\underline{123}$};
    \node (5) at (4.5,1.5) {$\underline{01}\overline\circledcirc\underline{12}\overline\circledcirc\underline{23}$};
    \draw (i) -- node[midway, draw, circle, fill=white, inner sep=1pt, scale=0.5] {$i$} (1);
    \draw (i) -- (2);
    \draw (i) -- node[midway, draw, circle, fill=white, inner sep=1pt, scale=0.5] {$j$} (3);
    \draw (i) -- (4);
    \draw (i) -- node[midway, draw, circle, fill=white, inner sep=1pt, scale=0.5] {$k$} (5);
    \end{tikzpicture}\\+\sum_{i,j\geq 0}\begin{tikzpicture}[baseline={([yshift=-.5ex]current bounding box.center)},scale=1]
    \node (i) at (0,0) {$ \underline{0}$};
    \node (1) at (-1.5,1) {$ \underline{03}$};
    \node (2) at (0,1) {$ \underline{0123}$};
    \node (3) at (2.5,1) {$\underline{01}\overline\circledcirc\underline{12}\overline\circledcirc\underline{23}$};
    \draw (i) -- node[midway, draw, circle, fill=white, inner sep=1pt, scale=0.5] {$i$} (1);
    \draw (i) -- (2);
    \draw (i) -- node[midway, draw, circle, fill=white, inner sep=1pt, scale=0.5] {$j$} (3);
    \end{tikzpicture}.
\end{multline*}
\medskip
\end{cor}

\begin{proof}
    We first compute $\partial^1$. Lemma \ref{firstdiff} and Corollary \ref{diff1} give
    $$\partial^1( \underline{0})=
    \begin{tikzpicture}[baseline={([yshift=-.5ex]current bounding box.center)},scale=1]
    \node[scale=1] (i) at (0,0) {$ \underline{0}$};
    \node[scale=1] (2) at (0,1) 
    {$ \underline{0}$};
    \draw (i) -- (2);
    \end{tikzpicture}\ ;\ \partial^1( \underline{1})=
    \begin{tikzpicture}[baseline={([yshift=-.5ex]current bounding box.center)},scale=1]
    \node[scale=1] (i) at (0,0) {$ \underline{1}$};
    \node[scale=1] (2) at (0,1) 
    {$ \underline{1}$};
    \draw (i) -- (2);
    \end{tikzpicture}\ ;\ \partial^1_{(1)}( \underline{01})=
    \begin{tikzpicture}[baseline={([yshift=-.5ex]current bounding box.center)},scale=1]
    \node[scale=1] (i) at (0,0) {$ \underline{0}$};
    \node[scale=1] (2) at (0,1) 
    {$ \underline{01}$};
    \draw (i) -- (2);
    \end{tikzpicture}-\begin{tikzpicture}[baseline={([yshift=-.5ex]current bounding box.center)},scale=1]
    \node[scale=1] (i) at (0,0) {$ \underline{01}$};
    \node[scale=1] (2) at (0,1) 
    {$ \underline{1}$};
    \draw (i) -- (2);
    \end{tikzpicture}.$$

    \noindent We compute $\partial_{(k)}^1( \underline{01})$ by induction on $k\geq 1$. We give the details for the case $k=2$. By Theorem \ref{rec}, we have
    $$\partial_{(2)}^1( \underline{01})=-H_1^0\partial_{(2)}^1\partial_{(0)}^1( \underline{01})-H_1^0\partial_{(1)}^1\partial_{(1)}^1( \underline{01}).$$

    \noindent Because the differentials $\partial^n$'s preserve the face maps, we have
    \begin{center}
        $\begin{array}{lll}
        \partial_{(2)}^1\partial_{(0)}^1( \underline{01}) & = & d^0\partial_{(2)}^0( \underline{0})-d^1\partial_{(2)}^0( \underline{0})\\
        & = & 0,
        \end{array}$
    \end{center}

    \noindent since $\partial_{(2)}^0=0$. Now, by the Leibniz rule in $\mathcal{B}race(\Sigma^{-1}N_*(\Delta^1))$,
    $$\partial_{(1)}^1\partial_{(1)}^1( \underline{01})=\begin{tikzpicture}[baseline={([yshift=-.5ex]current bounding box.center)},scale=1]
    \node[scale=1] (i) at (0,0) {$\partial_{(1)}^1( \underline{0})$};
    \node[scale=1] (2) at (0,1) 
    {$ \underline{01}$};
    \draw (i) -- (2);
    \end{tikzpicture}-\begin{tikzpicture}[baseline={([yshift=-.5ex]current bounding box.center)},scale=1]
    \node[scale=1] (i) at (0,0) {$ \underline{0}$};
    \node[scale=1] (2) at (0,1) 
    {$\partial_{(2)}^1( \underline{01})$};
    \draw (i) -- (2);
    \end{tikzpicture}-\begin{tikzpicture}[baseline={([yshift=-.5ex]current bounding box.center)},scale=1]
    \node[scale=1] (i) at (0,0) {$\partial_{(2)}^1( \underline{01})$};
    \node[scale=1] (2) at (0,1) 
    {$ \underline{1}$};
    \draw (i) -- (2);
    \end{tikzpicture}-\begin{tikzpicture}[baseline={([yshift=-.5ex]current bounding box.center)},scale=1]
    \node[scale=1] (i) at (0,0) {$ \underline{01}$};
    \node[scale=1] (2) at (0,1) 
    {$\partial_{(2)}^1( \underline{1})$};
    \draw (i) -- (2);
    \end{tikzpicture}.$$

    \noindent These terms give
    \begin{center}
        $\begin{array}{lll}
        \begin{tikzpicture}[baseline={([yshift=-.5ex]current bounding box.center)},scale=1]
    \node[scale=1] (i) at (0,0) {$\partial_{(1)}^1( \underline{0})$};
    \node[scale=1] (2) at (0,1) 
    {$ \underline{01}$};
    \draw (i) -- (2);
    \end{tikzpicture} & = & \begin{tikzpicture}[baseline={([yshift=-4ex]current bounding box.center)},scale=1]
    \node[scale=1] (i) at (0,0) {$ \underline{0}$};
    \node[scale=1] (2) at (0,1) 
    {$ \underline{0}$};
    \node (3) at (0,2) {$ \underline{01}$};
    \draw (i) -- (2);
    \draw (2) -- (3);
    \end{tikzpicture}+\begin{tikzpicture}[baseline={([yshift=-.5ex]current bounding box.center)},scale=1]
    \node[scale=1] (i) at (0,0) {$ \underline{0}$};
    \node[scale=1] (1) at (-1,1) {$ \underline{0}$};
    \node[scale=1] (3) at (1,1) {$ \underline{01}$};
    \draw (i) -- (1);
    \draw (i) -- (3);
    \end{tikzpicture}+\begin{tikzpicture}[baseline={([yshift=-.5ex]current bounding box.center)},scale=1]
    \node[scale=1] (i) at (0,0) {$ \underline{0}$};
    \node[scale=1] (1) at (-1,1) {$ \underline{01}$};
    \node[scale=1] (3) at (1,1) {$ \underline{0}$};
    \draw (i) -- (1);
    \draw (i) -- (3);
    \end{tikzpicture};\\
    -\begin{tikzpicture}[baseline={([yshift=-.5ex]current bounding box.center)},scale=1]
    \node[scale=1] (i) at (0,0) {$ \underline{0}$};
    \node[scale=1] (2) at (0,1) 
    {$\partial_{(2)}^1( \underline{01})$};
    \draw (i) -- (2);
    \end{tikzpicture} & = & -\begin{tikzpicture}[baseline={([yshift=-4ex]current bounding box.center)},scale=1]
    \node[scale=1] (i) at (0,0) {$ \underline{0}$};
    \node[scale=1] (2) at (0,1) 
    {$ \underline{0}$};
    \node (3) at (0,2) {$ \underline{01}$};
    \draw (i) -- (2);
    \draw (2) -- (3);
    \end{tikzpicture}+\begin{tikzpicture}[baseline={([yshift=-4ex]current bounding box.center)},scale=1]
    \node[scale=1] (i) at (0,0) {$ \underline{0}$};
    \node[scale=1] (2) at (0,1) 
    {$ \underline{01}$};
    \node (3) at (0,2) {$ \underline{1}$};
    \draw (i) -- (2);
    \draw (2) -- (3);
    \end{tikzpicture};\\
    -\begin{tikzpicture}[baseline={([yshift=-.5ex]current bounding box.center)},scale=1]
    \node[scale=1] (i) at (0,0) {$\partial_{(2)}^1( \underline{01})$};
    \node[scale=1] (2) at (0,1) 
    {$ \underline{1}$};
    \draw (i) -- (2);
    \end{tikzpicture} & = & -\begin{tikzpicture}[baseline={([yshift=-4ex]current bounding box.center)},scale=1]
    \node[scale=1] (i) at (0,0) {$ \underline{0}$};
    \node[scale=1] (2) at (0,1) 
    {$ \underline{01}$};
    \node (3) at (0,2) {$ \underline{1}$};
    \draw (i) -- (2);
    \draw (2) -- (3);
    \end{tikzpicture}+\begin{tikzpicture}[baseline={([yshift=-4ex]current bounding box.center)},scale=1]
    \node[scale=1] (i) at (0,0) {$ \underline{01}$};
    \node[scale=1] (2) at (0,1) 
    {$ \underline{1}$};
    \node (3) at (0,2) {$ \underline{1}$};
    \draw (i) -- (2);
    \draw (2) -- (3);
    \end{tikzpicture}\boxed{-\begin{tikzpicture}[baseline={([yshift=-.5ex]current bounding box.center)},scale=1]
    \node[scale=1] (i) at (0,0) {$ \underline{0}$};
    \node[scale=1] (1) at (-1,1) {$ \underline{1}$};
    \node[scale=1] (3) at (1,1) {$ \underline{01}$};
    \draw (i) -- (1);
    \draw (i) -- (3);
    \end{tikzpicture}}-\begin{tikzpicture}[baseline={([yshift=-.5ex]current bounding box.center)},scale=1]
    \node[scale=1] (i) at (0,0) {$ \underline{0}$};
    \node[scale=1] (1) at (-1,1) {$ \underline{01}$};
    \node[scale=1] (3) at (1,1) {$ \underline{1}$};
    \draw (i) -- (1);
    \draw (i) -- (3);
    \end{tikzpicture};\\
    & & +\begin{tikzpicture}[baseline={([yshift=-.5ex]current bounding box.center)},scale=1]
    \node[scale=1] (i) at (0,0) {$ \underline{01}$};
    \node[scale=1] (1) at (-1,1) {$ \underline{1}$};
    \node[scale=1] (3) at (1,1) {$ \underline{1}$};
    \draw (i) -- (1);
    \draw (i) -- (3);
    \end{tikzpicture}-\begin{tikzpicture}[baseline={([yshift=-.5ex]current bounding box.center)},scale=1]
    \node[scale=1] (i) at (0,0) {$ \underline{01}$};
    \node[scale=1] (1) at (-1,1) {$ \underline{1}$};
    \node[scale=1] (3) at (1,1) {$ \underline{1}$};
    \draw (i) -- (1);
    \draw (i) -- (3);
    \end{tikzpicture};\\
    -\begin{tikzpicture}[baseline={([yshift=-.5ex]current bounding box.center)},scale=1]
    \node[scale=1] (i) at (0,0) {$ \underline{01}$};
    \node[scale=1] (2) at (0,1) 
    {$\partial_{(2)}^1( \underline{1})$};
    \draw (i) -- (2);
    \end{tikzpicture} & = & -\begin{tikzpicture}[baseline={([yshift=-4ex]current bounding box.center)},scale=1]
    \node[scale=1] (i) at (0,0) {$ \underline{01}$};
    \node[scale=1] (2) at (0,1) 
    {$ \underline{1}$};
    \node (3) at (0,2) {$ \underline{1}$};
    \draw (i) -- (2);
    \draw (2) -- (3);
    \end{tikzpicture}.
        \end{array}$
    \end{center}
    
    \noindent The boxed tree is the only tree which gives a non-zero element when applying $H_1^0$. This gives
    $$\partial_{(2)}^1( \underline{01})=\begin{tikzpicture}[baseline={([yshift=-.5ex]current bounding box.center)},scale=1]
    \node[scale=1] (i) at (0,0) {$ \underline{0}$};
    \node[scale=1] (1) at (0,1) {$ \underline{01}$};
    \draw (i) -- node[midway, draw, circle, fill=white, inner sep=1pt, scale=0.5] {$2$} (1);
    \end{tikzpicture}.$$

    \noindent We now suppose that $k\geq 3$. By definition, and since $\partial_{(k-1)}^1\partial_{(0)}^1(\underline{01})=0$ because $\partial_{(k-1)}^0=0$, we have
    $$\partial_{(k)}^1( \underline{01})=-H_1^0\partial_{(k-1)}^1\partial_{(1)}^1( \underline{01})-H_1^0\partial_{(1)}^1\partial_{(k-1)}^1(\underline{01})-\sum_{\substack{p+q=k\\ p,q\neq 0,1}}H_1^0\partial_{(p)}^1\partial_{(q)}^1(\underline{01}).$$

    \noindent By induction hypothesis, for every $p,q\neq 0,1$ such that $p+q=k$, the term $\partial_{(p)}^1\partial_{(q)}^1(\underline{01})$ will only give trees with as vertices $\underline{0}$ or $\underline{01}$, so that $H_1^0\partial_{(p)}^1\partial_{(q)}^1(\underline{01})=0$. We now look at the remaining terms. Since $k-1\geq 2$, we have by induction hypothesis and by the Leibniz rule
    \begin{multline*}
        \partial_{(1)}^1\partial_{(k-1)}^1(\underline{01})=\sum_{p+q+r=k-1}\begin{tikzpicture}[baseline={([yshift=-.5ex]current bounding box.center)},scale=1]
    \node[scale=1] (i) at (0,0) {$ \underline{0}$};
    \node[scale=1] (1bb) at (-1,1) {$\underline{01}$};
    \node[scale=1] (2) at (0,1) {$ \underline{0}$};
    \node[scale=1] (2bb) at (0,2) {$\underline{01}$};
    \node[scale=1] (3b) at (1,1) {$\underline{01}$};
    \draw (i) -- node[midway, draw, circle, fill=white, inner sep=1pt, scale=0.5] {$p$} (1bb);
    \draw (i) -- (2);
    \draw (2) -- node[midway, draw, circle, fill=white, inner sep=1pt, scale=0.5] {$q$} (2bb);
    \draw (i) -- node[midway, draw, circle, fill=white, inner sep=1pt, scale=0.5] {$r$} (3b);
    \end{tikzpicture}\\
    +\sum_{p+q=k-1}\begin{tikzpicture}[baseline={([yshift=-.5ex]current bounding box.center)},scale=1]
    \node[scale=1] (i) at (0,0) {$ \underline{0}$};
    \node[scale=1] (1bb) at (-1,1) {$\underline{01}$};
    \node[scale=1] (2) at (0,1) {$ \underline{0}$};
    \node[scale=1] (2bb) at (0,2) {$\underline{01}$};
    \node[scale=1] (3b) at (1,1) {$\underline{01}$};
    \draw (i) -- node[midway, draw, circle, fill=white, inner sep=1pt, scale=0.5] {$p$} (1bb);
    \draw (2) -- node[midway, draw, circle, fill=white, inner sep=1pt, scale=0.5] {$q$} (2bb);
    \draw (i) -- node[midway, draw, circle, fill=white, inner sep=1pt, scale=0.5] {$r$} (3b);
    \draw (i) -- (2);
    \end{tikzpicture}-\sum_{p+q=k-1}\begin{tikzpicture}[baseline={([yshift=-.5ex]current bounding box.center)},scale=1]
    \node[scale=1] (i) at (0,0) {$ \underline{0}$};
    \node[scale=1] (1bb) at (-1,1) {$\underline{01}$};
    \node[scale=1] (2) at (0,1) {$ \underline{01}$};
    \node[scale=1] (2bb) at (0,2) {$\underline{1}$};
    \node[scale=1] (3b) at (1,1) {$\underline{01}$};
    \draw (i) -- node[midway, draw, circle, fill=white, inner sep=1pt, scale=0.5] {$p$} (1bb);
    \draw (i) -- (2);
    \draw (2) -- (2bb);
    \draw (i) -- node[midway, draw, circle, fill=white, inner sep=1pt, scale=0.5] {$q$} (3b);
    \end{tikzpicture},
    \end{multline*}

    \noindent which gives $H_1^0\partial_{(1)}^1\partial_{(k-1)}^1(\underline{01})=0$. We finally have
    \begin{multline*}
        \partial_{(k-1)}^1\partial_{(1)}^1(\underline{01})=-\begin{tikzpicture}[baseline={([yshift=-.5ex]current bounding box.center)},scale=1]
    \node[scale=1] (i) at (0,0) {$ \underline{0}$};
    \node (2) at (0,1) {$\underline{0}$};
    \node[scale=1] (2bb) at (0,2.5) {$\underline{01}$};
    \draw (i) -- (2);
    \draw (2) -- node[midway, draw, circle, fill=white, inner sep=1pt, scale=0.4] {$k-1$} (2bb);
    \end{tikzpicture}-\sum_{p+q=k-2}\begin{tikzpicture}[baseline={([yshift=-.5ex]current bounding box.center)},scale=1]
    \node[scale=1] (i) at (0,0) {$ \underline{0}$};
    \node[scale=1] (1bb) at (-1,1) {$\underline{01}$};
    \node[scale=1] (2) at (0,1) {$ \underline{01}$};
    \node[scale=1] (2bb) at (0,2) {$\underline{1}$};
    \node[scale=1] (3b) at (1,1) {$\underline{01}$};
    \draw (i) -- node[midway, draw, circle, fill=white, inner sep=1pt, scale=0.5] {$p$} (1bb);
    \draw (i) -- (2);
    \draw (2) -- (2bb);
    \draw (i) -- node[midway, draw, circle, fill=white, inner sep=1pt, scale=0.5] {$q$} (3b);
    \end{tikzpicture}-\sum_{p+q=k-1}\begin{tikzpicture}[baseline={([yshift=-.5ex]current bounding box.center)},scale=1]
    \node[scale=1] (i) at (0,0) {$ \underline{0}$};
    \node[scale=1] (1bb) at (-1,1) {$\underline{01}$};
    \node[scale=1] (2) at (0,1) {$ \underline{1}$};
    \node[scale=1] (3b) at (1,1) {$\underline{01}$};
    \draw (i) -- node[midway, draw, circle, fill=white, inner sep=1pt, scale=0.5] {$p$} (1bb);
    \draw (i) -- (2);
    \draw (i) -- node[midway, draw, circle, fill=white, inner sep=1pt, scale=0.5] {$q$} (3b);
    \end{tikzpicture}.
    \end{multline*}

    \noindent All the terms occurring in the right hand-side give elements in the kernel of $H_1^0$, except for the last sum with $p=0$ and $q=k-1$, which gives
    $$\partial_{(k)}^1( \underline{01})=\begin{tikzpicture}[baseline={([yshift=-.5ex]current bounding box.center)},scale=1]
    \node[scale=1] (i) at (0,0) {$ \underline{0}$};
    \node[scale=1] (1) at (0,1) {$ \underline{01}$};
    \draw (i) -- node[midway, draw, circle, fill=white, inner sep=1pt, scale=0.5] {$k$} (1);
    \end{tikzpicture}.$$

    \noindent The computation of $\partial^1$ is proved. We now compute $\partial^2$. By Lemma \ref{firstdiff}, we have
    $$\partial^2( \underline{0})=
    \begin{tikzpicture}[baseline={([yshift=-.5ex]current bounding box.center)},scale=1]
    \node[scale=1] (i) at (0,0) {$ \underline{0}$};
    \node[scale=1] (2) at (0,1) 
    {$ \underline{0}$};
    \draw (i) -- (2);
    \end{tikzpicture}\ ;\ \partial^2( \underline{1})=
    \begin{tikzpicture}[baseline={([yshift=-.5ex]current bounding box.center)},scale=1]
    \node[scale=1] (i) at (0,0) {$ \underline{1}$};
    \node[scale=1] (2) at (0,1) 
    {$ \underline{1}$};
    \draw (i) -- (2);
    \end{tikzpicture}\ ;\ \partial^2( \underline{2})=
    \begin{tikzpicture}[baseline={([yshift=-.5ex]current bounding box.center)},scale=1]
    \node[scale=1] (i) at (0,0) {$ \underline{2}$};
    \node[scale=1] (2) at (0,1) 
    {$ \underline{2}$};
    \draw (i) -- (2);
    \end{tikzpicture}.$$

    \noindent By Corollary \ref{diff1}, we have
    $$\partial^2_{(1)}( \underline{01})=
    \begin{tikzpicture}[baseline={([yshift=-.5ex]current bounding box.center)},scale=1]
    \node[scale=1] (i) at (0,0) {$ \underline{0}$};
    \node[scale=1] (2) at (0,1) 
    {$ \underline{01}$};
    \draw (i) -- (2);
    \end{tikzpicture}-\begin{tikzpicture}[baseline={([yshift=-.5ex]current bounding box.center)},scale=1]
    \node[scale=1] (i) at (0,0) {$ \underline{01}$};
    \node[scale=1] (2) at (0,1) 
    {$ \underline{1}$};
    \draw (i) -- (2);
    \end{tikzpicture}\ ;\ \partial^2_{(1)}( \underline{02})=
    \begin{tikzpicture}[baseline={([yshift=-.5ex]current bounding box.center)},scale=1]
    \node[scale=1] (i) at (0,0) {$ \underline{0}$};
    \node[scale=1] (2) at (0,1) 
    {$ \underline{02}$};
    \draw (i) -- (2);
    \end{tikzpicture}-\begin{tikzpicture}[baseline={([yshift=-.5ex]current bounding box.center)},scale=1]
    \node[scale=1] (i) at (0,0) {$ \underline{02}$};
    \node[scale=1] (2) at (0,1) 
    {$ \underline{2}$};
    \draw (i) -- (2);
    \end{tikzpicture}\ ;\ \partial^2_{(1)}( \underline{12})=
    \begin{tikzpicture}[baseline={([yshift=-.5ex]current bounding box.center)},scale=1]
    \node[scale=1] (i) at (0,0) {$ \underline{1}$};
    \node[scale=1] (2) at (0,1) 
    {$ \underline{12}$};
    \draw (i) -- (2);
    \end{tikzpicture}-\begin{tikzpicture}[baseline={([yshift=-.5ex]current bounding box.center)},scale=1]
    \node[scale=1] (i) at (0,0) {$ \underline{12}$};
    \node[scale=1] (2) at (0,1) 
    {$ \underline{2}$};
    \draw (i) -- (2);
    \end{tikzpicture},$$

    \noindent and
    $$\partial^2_{(1)}( \underline{012})=
    \begin{tikzpicture}[baseline={([yshift=-.5ex]current bounding box.center)},scale=1]
    \node[scale=1] (i) at (0,0) {$ \underline{0}$};
    \node[scale=1] (2) at (0,1) 
    {$ \underline{012}$};
    \draw (i) -- (2);
    \end{tikzpicture}-\begin{tikzpicture}[baseline={([yshift=-.5ex]current bounding box.center)},scale=1]
    \node[scale=1] (i) at (0,0) {$ \underline{01}$};
    \node[scale=1] (2) at (0,1) 
    {$ \underline{12}$};
    \draw (i) -- (2);
    \end{tikzpicture}+\begin{tikzpicture}[baseline={([yshift=-.5ex]current bounding box.center)},scale=1]
    \node[scale=1] (i) at (0,0) {$ \underline{012}$};
    \node[scale=1] (2) at (0,1) 
    {$ \underline{2}$};
    \draw (i) -- (2);
    \end{tikzpicture}.$$

    \noindent As before, we can compute by hand $\partial_{(2)}^2( \underline{012})$ by using Theorem \ref{rec}. We obtain
    $$\partial_{(2)}^1( \underline{012})=-\begin{tikzpicture}[baseline={([yshift=-.5ex]current bounding box.center)},scale=1]
    \node[scale=1] (i) at (0,0) {$ \underline{01}$};
    \node[scale=1] (1) at (-1,1) {$ \underline{12}$};
    \node[scale=1] (3) at (1,1) {$ \underline{12}$};
    \draw (i) -- (1);
    \draw (i) -- (3);
    \end{tikzpicture}+\begin{tikzpicture}[baseline={([yshift=-.5ex]current bounding box.center)},scale=1]
    \node[scale=1] (i) at (0,0) {$ \underline{0}$};
    \node[scale=1] (1) at (-1,1) {$ \underline{012}$};
    \node[scale=1] (3) at (1,1) {$( \underline{01}+ \underline{12})$};
    \draw (i) -- (1);
    \draw (i) -- (3);
    \end{tikzpicture}+\begin{tikzpicture}[baseline={([yshift=-.5ex]current bounding box.center)},scale=1]
    \node[scale=1] (i) at (0,0) {$ \underline{0}$};
    \node[scale=1] (1) at (-1,1) {$ \underline{02}$};
    \node[scale=1] (3) at (1,1) {$ \underline{012}$};
    \draw (i) -- (1);
    \draw (i) -- (3);
    \end{tikzpicture}.$$

    \noindent We now compute $\partial_{(k)}^2( \underline{012})$ for every $k\geq 3$. We use that
    $$\partial_{(k)}^2( \underline{012})=-H_2^0\partial_{(k)}^2\partial_{(0)}^2( \underline{012})-H_2^0\partial_{(k-1)}^2\partial_{(1)}^2( \underline{012})-\sum_{\substack{p+q=k\\ q\neq 0,1\\p\neq 0}}H_2^0\partial_{(p)}^2\partial_{(q)}^2( \underline{012}).$$

    \noindent We have
    $$-H_2^0\partial_{(k)}^2\partial_{(0)}^2( \underline{012})=-\begin{tikzpicture}[baseline={([yshift=-.5ex]current bounding box.center)},scale=1]
    \node[scale=1] (i) at (0,0) {$ \underline{01}$};
    \node[scale=1] (1) at (0,1) {$ \underline{12}$};
    \draw (i) -- node[midway, draw, circle, fill=white, inner sep=1pt, scale=0.5] {$k$} (1);
    \end{tikzpicture}+\begin{tikzpicture}[baseline={([yshift=-.5ex]current bounding box.center)},scale=1]
    \node[scale=1] (i) at (-0,0) {$ \underline{0}$};
    \node[scale=1] (1) at (-0.75,1.25) {$ \underline{012}$};
    \node[scale=1] (2) at (0.75,1.25) {$\underline{12}$};
    \draw (i) -- (1);
    \draw (i) -- node[midway, draw, circle, fill=white, inner sep=1pt, scale=0.4] {$k-1$} (2);
    \end{tikzpicture}$$

    \noindent by induction hypothesis. We now compute $H_2^0\partial_{(k-1)}^2\partial_{(1)}^2( \underline{012})$. We have
    \begin{center}
        $\begin{array}{ccc}
        -H_2^0\partial_{(k-1)}^2\left(\begin{tikzpicture}[baseline={([yshift=-.5ex]current bounding box.center)},scale=1]
    \node[scale=1] (i) at (0,0) {$ \underline{0}$};
    \node[scale=1] (2) at (0,1) 
    {$ \underline{012}$};
    \draw (i) -- (2);
    \end{tikzpicture}\right) & = & 0,\\
    $$H_2^0\partial_{(k-1)}^2\left(\begin{tikzpicture}[baseline={([yshift=-.5ex]current bounding box.center)},scale=1]
    \node[scale=1] (i) at (0,0) {$ \underline{01}$};
    \node[scale=1] (2) at (0,1) 
    {$ \underline{12}$};
    \draw (i) -- (2);
    \end{tikzpicture}\right) & = & \begin{tikzpicture}[baseline={([yshift=-.5ex]current bounding box.center)},scale=1]
    \node[scale=1] (i) at (-0,0) {$ \underline{0}$};
    \node[scale=1] (1) at (-0.75,1.25) {$ \underline{012}$};
    \node[scale=1] (2) at (0.75,1.25) {$\underline{01}$};
    \draw (i) -- (1);
    \draw (i) -- node[midway, draw, circle, fill=white, inner sep=1pt, scale=0.4] {$k-1$} (2);
    \end{tikzpicture},
        \end{array}$
    \end{center}

    \noindent and
    \begin{multline*}
        -H_2^0\partial_{(k-1)}^2\left(\begin{tikzpicture}[baseline={([yshift=-.5ex]current bounding box.center)},scale=1]
    \node[scale=1] (i) at (0,0) {$ \underline{012}$};
    \node[scale=1] (2) at (0,1) 
    {$ \underline{2}$};
    \draw (i) -- (2);
    \end{tikzpicture}\right)=\\
    \sum\begin{tikzpicture}[baseline={([yshift=-.5ex]current bounding box.center)},scale=1]
    \node (i) at (0,0) {$ \underline{0}$};
    \node (1) at (-2,1.5) {$ \underline{02}$};
    \node (1bis) at (-3.25,1.5) {${\cdots}$};
    \node (1bbis) at (-4.5,1.5) {$ \underline{02}$};
    \node (2) at (0,1.5) {$ \underline{012}$};
    \node[scale=0.5] (3) at (3,1.5) {{\biduledroitebis}};
    \node (3b) at (5.5,1.5) {${\cdots}$};
    \node[scale=0.5] (3bb) at (8,1.5) {{\biduledroitebis}};
    \draw (i) -- (1);
    \draw (i) -- (1bbis);
    \draw (i) -- (2);
    \draw (i) -- (3);
    \draw (i) -- (3bb);
    \end{tikzpicture}
    \end{multline*}

    \noindent where we sum over all such trees with $k+1$ vertices which contain at least one element $ \underline{02}$. Finally, for every $p,q$ such that $p+q=k,q\neq 0,1$ and $p\neq 0$, we have
    $$-H_2^0\partial_{(p)}^2\partial_{(q)}^2( \underline{012})=\sum \begin{tikzpicture}[baseline={([yshift=-.5ex]current bounding box.center)},scale=1]
    \node (i) at (0,0) {$ \underline{0}$};
    \node (2) at (-3,1.5) {$ \underline{012}$};
    \node[scale=0.5] (3) at (0,1.5) {{\biduledroitebis}};
    \node (3b) at (2.5,1.5) {${\cdots}$};
    \node[scale=0.5] (3bb) at (5,1.5) {{\biduledroitebis}};
    \draw (i) -- (2);
    \draw (i) -- (3);
    \draw (i) -- (3bb);
    \end{tikzpicture}$$

    \noindent where we sum over all such trees with $p$ elements $ \underline{01}$ and $q-1$ elements $ \underline{12}$. This concludes the proof of the computation of $\partial^2$.\\

    We only give the ideas for the computation of $\partial^3$. By induction on $k\geq 0$, we first two lines are given by $\partial_{(k)}^3\partial_{(0)}^3(\underline{0123})$. The tree
    $$\begin{tikzpicture}[baseline={([yshift=-.5ex]current bounding box.center)},scale=1]
    \node[scale=1] (i) at (0,0) {$ \underline{0123}$};
    \node[scale=1] (2) at (0,1) {$ \underline{3}$};
    \draw (i) -- (2);
    \end{tikzpicture}$$

    \noindent will give trees of $\partial_{(k-1)}^3(\underline{0123})$ with $\underline{0}$ as root on which we add a branch linked to the root with $\underline{03}$ as non-root vertex. We thus can focus to trees with no elements $\underline{03}$. The trees with no element $\underline{03}$ of the third line are obtained from the differentiation of the trees of the form
    $$\begin{tikzpicture}[baseline={([yshift=-.5ex]current bounding box.center)},scale=1]
    \node[scale=1] (i) at (0,0) {$ \underline{012}$};
    \node[scale=1] (1) at (0,1) {$ \underline{23}$};
    \draw (i) -- node[midway, draw, circle, fill=white, inner sep=1pt, scale=0.5] {$i$} (1);
    \end{tikzpicture}$$

    \noindent for some $1\leq i\leq k-1$. The trees of the three last lines with no element $\underline{03}$ is obtained by the differentiation of trees of the form
    $$\begin{tikzpicture}[baseline={([yshift=-.5ex]current bounding box.center)},scale=1]
    \node (i) at (0,0) {$ \underline{01}$};
    \node (1) at (-1.5,1) {$ \underline{13}$};
    \node (2) at (0,1) {$ \underline{123}$};
    \node (3) at (2,1) {$\underline{12}\overline\circledcirc\underline{23}$};
    \draw (i) -- node[midway, draw, circle, fill=white, inner sep=1pt, scale=0.5] {$i$} (1);
    \draw (i) -- (2);
    \draw (i) -- node[midway, draw, circle, fill=white, inner sep=1pt, scale=0.5] {$j$} (3);
    \end{tikzpicture}$$

    \noindent for some $i,j\geq 0$ such that $1\leq i+j\leq k-2$.\\

    The lemma is proved.
\end{proof}

\subsection{A morphism from $\mathcal{P}re\mathcal{L}ie_\infty$ to $\mathcal{B}race\underset{\text{\normalfont H}}{\otimes}B^c(\Lambda^{-1}\mathcal{B}race^\vee)$}\label{sec:233}

In this subsection, we construct a morphism $\mathcal{P}re\mathcal{L}ie_\infty\longrightarrow\mathcal{B}race\underset{\text{\normalfont H}}{\otimes}B^c(\Lambda^{-1}\mathcal{B}race^\vee)$ which will give, together with Theorem \ref{mu}, a morphism of operads $\mathcal{P}re\mathcal{L}ie_\infty\longrightarrow\mathcal{B}race\underset{\text{\normalfont H}}{\otimes}\mathcal{E}$.\\

Recall that, for every operad $\mathcal{P}$ such that $\mathcal{P}(n)$ is finite dimensional for every $n\geq 0$, we have a morphism of operads
\begin{center}
    $\begin{array}{ccc}
        \mathcal{L}ie_\infty & \longrightarrow & \mathcal{P}\otimes B^c(\Lambda^{-1}\mathcal{P}^\vee)   \\
        \Sigma^{n-1} & \longmapsto & \sum_{x\in\mathcal{B}(n)} x\otimes \Sigma^{n-1}x^\vee  
    \end{array}$
\end{center}

\noindent where $\mathcal{B}(n)$ denotes a basis of $\mathcal{P}(n)$, and where $\mathcal{L}ie_\infty=B^c(\Lambda^{-1}\mathcal{C}om^\vee)$ is the operad which governs Lie algebras up to homotopy. Recall also that we have a morphism of operads $\mathcal{L}ie_\infty\longrightarrow\mathcal{P}re\mathcal{L}ie_\infty$ given by the morphism of symmetric sequences
\begin{center}
    $\begin{array}{ccc}
        \text{\normalfont Perm} & \longrightarrow & \mathcal{C}om  \\
        e_1^n & \longmapsto & 1 
    \end{array}$.
\end{center}

\begin{lm}\label{morph2}
    There exists an explicit lift of the diagram
\[\begin{tikzcd}
	{\mathcal{L}ie_\infty} & {\mathcal{B}race\underset{\text{\normalfont H}}{\otimes}B^c(\Lambda^{-1}\mathcal{B}race^\vee)} \\
	{\mathcal{P}re\mathcal{L}ie_\infty}
	\arrow[from=1-1, to=1-2]
	\arrow[from=1-1, to=2-1]
	\arrow["\exists"', dashed, from=2-1, to=1-2]
\end{tikzcd}.\]
\end{lm}

\begin{proof}
    Giving a morphism $\mathcal{P}re\mathcal{L}ie_\infty\longrightarrow\mathcal{B}race\underset{\text{\normalfont H}}{\otimes} B^c(\Lambda^{-1}\mathcal{B}race^\vee)$ is equivalent to giving a Maurer-Cartan element $f$ in the pre-Lie algebra $\text{\normalfont Hom}_{\Sigma\text{\normalfont Seq}_\mathbb{K}}(\overline{\text{\normalfont Perm}}^\vee,\mathcal{B}race\underset{\text{\normalfont H}}{\otimes} B^c(\mathcal{B}race^\vee))$ (see for instance \cite[Theorem 6.5.10]{loday}). By symmetry, it is sufficient to give the image of $e_1^n$ for every $n\geq 1$. We set
    $$f(e_1^n)=-\sum_{\substack{T\in\mathcal{PRT}(n)\\ r(T)=1}}T\otimes \Sigma^{-1} T^\vee.$$

    \noindent We check that $d(f)(e_1^n)+(f\star f)(e_1^n)=0$, where $\star$ denotes the pre-Lie product of $\text{\normalfont Hom}_{\Sigma\text{\normalfont Seq}_\mathbb{K}}(\overline{\text{\normalfont Perm}}^\vee,\mathcal{B}race\underset{\text{\normalfont H}}{\otimes} B^c(\mathcal{B}race^\vee))$. We first have
    $$d(f)(e_1^n)=d(f(e_1^n))=-\sum_{\substack{T\in \mathcal{PRT}(n)\\ r(T)=1}}\sum_{S\subset T}T\otimes(\Sigma^{-1} (T/S)^\vee\circ_S \Sigma^{-1} S^\vee).$$

    \noindent We now compute $(f\star f)(e_1^n)$. Recall that
    $$\Delta_1(e_1^n)=\sum_{\substack{p+q=n+1\\p,q\geq 2}}e_1^p\otimes e_1^q;$$
    $$\forall k\neq 1, \Delta_k(e_1^n)=\sum_{\substack{p+q=n+1\\p,q\geq 2}}\sum_{i=1}^q e_1^p\otimes e_i^q.$$

    \noindent The $\Delta_1$ part gives
    $$\sum_{\substack{p+q=n+1\\ p,q\geq 2}}\sum_{\omega\in Sh_\ast(q,1,\ldots,1)}\sum_{\substack{U\in \mathcal{PRT}(p)\\ r(U)=1}}\sum_{\substack{V\in \mathcal{PRT}(q)\\ r(V)=1}}\omega\cdot(U\circ_1 V)\otimes\omega\cdot(\Sigma^{-1} U^\vee\circ_1 \Sigma^{-1} V^\vee),$$

    \noindent and the $\Delta_k$'s part for $k\neq 1$ gives
    $$\sum_{\substack{p+q=n+1\\ p,q\geq 2}}\sum_{k=2}^p\sum_{\omega\in Sh_\ast(1,\ldots,\underset{k}{q},\ldots,1)}\sum_{\substack{U\in \mathcal{PRT}(p)\\r(U)=1}}\sum_{V\in \mathcal{PRT}(q)}\omega\cdot(U\circ_{k} V)\otimes\omega\cdot(\Sigma^{-1} U^\vee\circ_{k}\Sigma^{-1} V^\vee).$$

    \noindent Let $T\in\mathcal{PRT}(n)$ be such that $r(T)=1$. Consider a term $\omega\cdot(U\circ_k V)$ occurring in one of the two previous sums, with $k\geq 1$, $\omega\in Sh_\ast(1,\ldots,\underset{k}{q},\ldots,1), U\in\mathcal{PRT}(p),V\in\mathcal{PRT}(q)$, such that $T$ occurs in the expansion of $\omega\cdot(U\circ_k V)$. We see $U$ and $V$ as $U\in\mathcal{PRT}(1<\cdots<k-1<V<k+q<\cdots<p)$ and $V\in\mathcal{PRT}(k<\cdots<k+q-1)$. Because $\omega\in Sh_\ast(1,\ldots,\underset{k}{q},\ldots,1)$, the composite $\omega\cdot(U\circ_k V)$ is equal to the composite $(\omega\cdot U)\circ_V (\omega\cdot V)$ where $\omega\cdot U$ is $U$ seen in $\mathcal{PRT}(1<\cdots<k-1<V<\omega(k+q)<\cdots<\omega(p))$ and $\omega\cdot V$ is $V$ seen in $\mathcal{PRT}(k=\omega(k)<\cdots<\omega(k+q-1))$. Thus, by definition of the operadic composition in $\mathcal{B}race$, the tree $\omega\cdot V$ can be seen as a subtree $S\subset T$ such that $T/S=\omega\cdot U$.\\

    \noindent In the converse direction, let $S\subset T$. Let $k=min(V_S)$ and $q=|S|$. Let $\omega_S:\llbracket k,k+q-1\rrbracket\longrightarrow V_S$ and $\omega_{T/S}:\llbracket 1,n\rrbracket\setminus\llbracket k+1,k+q-1\rrbracket\longrightarrow V_{T/S}$ be the unique order preserving maps between the two considered finite sets. Then, by definition, $\omega=\omega_{T/S}\circ_k\omega_S\in Sh_\ast(1,\ldots,\underset{k}{q},\ldots,1)$. We finally set $U=\omega^{-1}\cdot(T/S)$ and $V=\omega^{-1}\cdot S$. Because $T/S\circ_S S$ obviously contains the tree $T$, we have that $T$ occurs in the composite $\omega\cdot(U\circ_k V)$.\\ 

    \noindent We thus have proved that
    $$(f\star f)(e_1^n)=\sum_{\substack{T\in\mathcal{PRT}(n)\\ r(T)=1}}T\otimes\left(\sum_{S\subset T}\Sigma^{-1} (T/S)^\vee\circ_S\Sigma^{-1} S^\vee\right).$$

    \noindent The identity $d(f)(e_1^n)+(f\star f)(e_1^n)=0$ follows.
\end{proof}

We now prove Theorem \ref{theoremD}.

\begin{thm}\label{morph}
    There exists an operad morphism $\mathcal{P}re\mathcal{L}ie_\infty\longrightarrow \mathcal{B}race\underset{\text{\normalfont H}}{\otimes} \mathcal{E}$ which fits in a commutative square
\[\begin{tikzcd}
	{\mathcal{P}re\mathcal{L}ie_\infty} & {\mathcal{B}race\underset{\text{\normalfont H}}{\otimes}\mathcal{E}} \\
	{\mathcal{P}re\mathcal{L}ie} & {\mathcal{B}race}
	\arrow[from=1-1, to=1-2]
	\arrow[from=1-1, to=2-1]
	\arrow[from=1-2, to=2-2]
	\arrow[from=2-1, to=2-2]
\end{tikzcd}.\]
\end{thm}

\begin{proof}
    The morphism $\mathcal{P}re\mathcal{L}ie_\infty\longrightarrow\mathcal{B}race\underset{\text{\normalfont H}}{\otimes}\mathcal{E}$ is given by the composite of the morphism $\mathcal{P}re\mathcal{L}ie_\infty\longrightarrow\mathcal{B}race\underset{\text{\normalfont H}}{\otimes}B^c(\Lambda^{-1}\mathcal{B}race^\vee)$ given by Lemma \ref{morph2} with the morphism $\mathcal{B}race\underset{\text{\normalfont H}}{\otimes}B^c(\Lambda^{-1}\mathcal{B}race^\vee)\longrightarrow\mathcal{B}race\underset{\text{\normalfont H}}{\otimes}\mathcal{E}$ given by applying the morphism $B^c(\Lambda^{-1}\mathcal{B}race^\vee)\longrightarrow\mathcal{E}$ defined in Theorem \ref{mu} on the second tensor. The commutative diagram is an immediate check.
\end{proof}

\begin{cor}\label{bracee}
    Every $\mathcal{B}race\underset{\text{\normalfont H}}{\otimes} \mathcal{E}$-algebra $L$ admits the structure of a $\Gamma(\mathcal{P}re\mathcal{L}ie_\infty,-)$-algebra.
\end{cor}

\begin{proof}
    Using that the action of $\Sigma _n$ on $(\mathcal{B}race\underset{\text{\normalfont H}}{\otimes} \mathcal{E})(n)$ is free and the previous theorem, we define, for every $\mathcal{B}race\underset{\text{\normalfont H}}{\otimes} \mathcal{E}$-algebras $L$, the composite
    \begin{center}
\begin{tikzcd}
\Gamma(\mathcal{P}re\mathcal{L}ie_\infty,L)\arrow[r] & \Gamma(\mathcal{B}race\underset{\text{\normalfont H}}{\otimes} \mathcal{E},L) & \mathcal{S}(\mathcal{B}race\underset{\text{\normalfont H}}{\otimes} \mathcal{E},L)\arrow[l, "\simeq"']\arrow[r] & L.
\end{tikzcd}
\end{center}

\noindent This gives a $\Gamma(\mathcal{P}re\mathcal{L}ie_\infty,-)$-algebra structure on $L$.
\end{proof}

In particular, if $L=A\otimes \Sigma E$ where $A$ is a brace algebra and $E$ a $\mathcal{E}$-algebra, then $A\otimes\Sigma E$ is a $\Gamma\Lambda\mathcal{PL}_\infty$-algebra. We can compute the weighted braces of $A\otimes\Sigma E$ as follows. Let $l:\mathcal{S}(\mathcal{B}race,A)\longrightarrow A$ be the brace algebra structure on $A$, and let $\partial^E$ be the twisting morphism on $\mathcal{B}race^c(\Sigma E)$ induced by the $B^c(\Lambda^{-1}\mathcal{B}race^\vee)$-algebra structure of $E$ (see Theorem \ref{mu}). Then, for every $a,b_1,\ldots,b_n\in A, x,y_1,\ldots,y_n\in\Sigma E$ and $r_1,\ldots,r_n\geq 0$, we have
\begin{multline*}
    a\otimes x\llbrace b_1\otimes  y_1,\ldots,b_n\otimes  y_n\rrbrace_{r_1,\ldots,r_n}\\
    =\sum_{\sigma\in Sh(r_1,\ldots,r_n)}\sum_{\substack{T\in\mathcal{PRT}(r+1)\\T\text{ canonical}}}\pm l(T\otimes a\otimes c_{\sigma(1)}\otimes\cdots\otimes c_{\sigma(r)})\otimes\partial^E(T^\vee\otimes x\otimes  {z_{\sigma(1)}}\otimes\cdots\otimes {z_{\sigma(r)}}),
\end{multline*}

\noindent where we have set $r=r_1+\cdots+r_n$, $c_1,\ldots,c_r=\underbrace{b_1,\ldots,b_1}_{r_1},\ldots,\underbrace{b_n,\dots,b_n}_{r_n}$ and ${z_1},\ldots,{z_r}=\underbrace{{x_1},\ldots,{x_1}}_{r_1},\ldots,\underbrace{{x_n},\ldots,{x_n}}_{r_n}$. The sign is given by the permutation of the $c_i$'s with $ {x}$ and the $ {z_i}$'s, and the permutation of $\partial^E$ with $a$ and the $c_i$'s.

\section{The simplicial Maurer-Cartan set of a complete brace algebra}

The goal of this section is to define the notion of a simplicial Maurer-Cartan set $\mathcal{MC}_\bullet(A)$ associated to a brace algebra $A$, and to study the homotopy type of this simplicial set. Explicitly, the $n$-component $\mathcal{MC}_n(A)$ will be defined as the Maurer-Cartan set of $A\otimes\Sigma N^*(\Delta^n)$ for the $\widehat{\Gamma\Lambda\mathcal{PL}_\infty}$-algebra structure given by Corollary \ref{bracee}.\\

In $\mathsection$\ref{sec:241}, we define the simplicial set $\mathcal{MC}_\bullet(A)$ and prove the first part of Theorem \ref{theoremE} which asserts that it is a Kan complex.\\

In $\mathsection$\ref{sec:242}, we prove the remaining part of Theorem \ref{theoremE}, which gives a computation of the connected components and the homotopy groups of $\mathcal{MC}_\bullet(A)$. More precisely, we first compute the connected components, whose computation is similar to \cite[Theorem 3.6]{moi}, before computing the $\pi_1,\pi_2$ and then the $\pi_n$ for $n\geq 3$.\\

In $\mathsection$\ref{sec:243}, we prove Theorem \ref{theoremF}, which is a higher version of the Goldman-Millson theorem (see \cite[$\mathsection$2.4]{goldman}). Our proof basically follows the proof found in \cite[$\mathsection$6]{rogers}, which will be adapted to our context.\\

In $\mathsection$\ref{sec:244}, we compare our simplicial notion of Maurer-Cartan set defined for complete brace algebras to the notion of simplicial Maurer-Cartan set associated to a complete Lie algebra, and prove that in fact, these two simplicial sets are weakly equivalent.

\subsection{The simplicial set $\mathcal{MC}_\bullet(A)$}\label{sec:241}

Let $A$ be a complete brace algebra. By Corollary \ref{bracee}, and using that $N^*(\Delta^n)$ is a $\mathcal{E}$-algebra, we obtain that $ A\widehat{\otimes}\Sigma N^*(\Delta^n)=A\otimes\Sigma N^*(\Delta^n)$ is a $\widehat{\Gamma\Lambda\mathcal{PL}_\infty}$-algebra with the filtration 
$$F_k(A\otimes\Sigma N^*(\Delta^n))=F_kA\otimes \Sigma N^*(\Delta^n).$$

\noindent where we denote by $(F_kA)_{k\geq 1}$ the filtration on $A$.

\begin{defi}
    Let $A$ be a complete brace algebra. Its {\normalfont simplicial Maurer-Cartan set} is the simplicial set $\mathcal{MC}_\bullet(A)$ such that
    $$\mathcal{MC}_\bullet(A)=\mathcal{MC}(A\otimes\Sigma N^*(\Delta^\bullet)).$$
\end{defi}

\begin{prop}
    The previous definition defines a functor $\mathcal{MC}_\bullet$ from the category of complete brace algebras to the category of sets.
\end{prop}

\begin{proof}
    This follows directly from the fact that every brace algebra morphism $f:A\longrightarrow B$ which preserves the filtrations gives rise to a morphism of $\mathcal{B}race\underset{\text{\normalfont H}}{\otimes}\Lambda\mathcal{E}$-algebras $ f\otimes id:A\otimes\Sigma N^*(\Delta^n)\longrightarrow B\otimes\Sigma N^*(\Delta^n)$ which preserves the filtrations, and then to a strict morphism of $\widehat{\Gamma\Lambda\mathcal{PL}_\infty}$-algebras.
\end{proof}

We aim to prove that $\mathcal{MC}_\bullet(A)$ is a Kan complex. We will basically follow the proof of the analogous theorem in \cite{felix}. Recall from Proposition \ref{homonstar} that we have morphisms $\varphi_n^i:N^*(\Delta^n)\longrightarrow N^*(\Delta^n)$ and $h_n^i:N^*(\Delta^n)\longrightarrow N^{*-1}(\Delta^n)$ which satisfy
$$d h_n^i+h_n^id=id-\varphi_n^i.$$

\noindent These relations can be carried to $A\otimes\Sigma N^*(\Delta^n)$ by setting $H_n^i=id\otimes\Sigma h_n^i$ and $\Phi_n^i=id\otimes \Sigma\varphi_n^i$. Accordingly, we have

$$d H_n^i+H_n^id=id-\Phi_n^i.$$

\noindent Note that if $x\in\mathcal{MC}(A\otimes\Sigma N^*(\Delta^n))$, then the identity

$$d(x)+\sum_{k\geq 1}x\llbrace x\rrbrace_k=0$$

\noindent gives, after composing by $H_n^i$,

$$x=\Phi_n^i(x)+d H_n^i(x)-\sum_{k\geq 1}H_n^i(x\llbrace x\rrbrace_k).$$

We give the following lemma, which is the analogue of \cite[Lemma 6.5]{felix}. For our needs, we need an analogue valid for every complete $\widehat{\Gamma\Lambda\mathcal{PL}_\infty}$-algebra structure on $A\otimes\Sigma N^*(\Delta^n)$.

\begin{lm}\label{découpe}
    Let $n\geq 0$ and $0\leq i\leq n$. Let $A\in\text{\normalfont \dgMod}$. Consider any complete $\Gamma\Lambda\mathcal{PL}_\infty$-algebra structure on $A\otimes\Sigma N^*(\Delta^n)$ such that $\Phi_n^i:A\otimes\Sigma N^*(\Delta^n)\longrightarrow A\otimes\Sigma N^*(\Delta^n)$ is a strict morphism. Then the map

    \begin{center}
        $\begin{array}{clc}
            \mathcal{MC}(A\otimes\Sigma N^*(\Delta^n)) & \longrightarrow & (\mathcal{MC}(A\otimes\Sigma N^*(\Delta^n))\cap Im(\Phi_n^i))\times Im(d H_n^i) \\
            x & \longmapsto & (\Phi_n^i(x),d H_n^i(x)) 
        \end{array}$
    \end{center}

    \noindent is a bijection.
\end{lm}

\begin{proof}
    We first note that the above map is well defined since $\Phi_n^i$ is a strict morphism by hypothesis. Now, let $e\in\mathcal{MC}(A\otimes\Sigma N^*(\Delta^n))\cap Im(\Phi_n^i)$ and $r\in Im(dH_n^i)$. We set:

    \begin{center}
        $\left\{\begin{array}{rll}
            \alpha_0 & = & e+r \\
             \forall k\geq 0, \alpha_{k+1} & = & \displaystyle e+r-\sum_{l\geq 1}H_n^i(\alpha_k\llbrace\alpha_k\rrbrace_l)
        \end{array}\right..$
    \end{center}

    \noindent This defines a Cauchy sequence $(\alpha_k)_k$. Let $\alpha$ be its limit. We then have
    $$\alpha=e+r-\sum_{l\geq 1}H_n^i(\alpha\llbrace\alpha\rrbrace_l).$$
    From this identity, we deduce $\Phi_n^i\alpha=e$ and $d H_n^i\alpha=r$. We just need to check that $\alpha\in\mathcal{MC}(A\otimes\Sigma N^*(\Delta^n))$. Using the relation
    $$d H_n^i+H_n^id=id-\Phi_n^i,$$

    \noindent we obtain
    $$d(\alpha)=d(e)+\sum_{l\geq 1}H_n^i(d(\alpha\llbrace\alpha\rrbrace_l))-\sum_{l\geq 1}\alpha\llbrace\alpha\rrbrace_l+\Phi_n^i\left(\sum_{l\geq 1}\alpha\llbrace\alpha\rrbrace_l\right),$$

    \noindent and then, because $\Phi_n^i$ is a strict morphism and that $e\in\mathcal{MC}(A\otimes\Sigma N^*(\Delta^n))$,
    $$d(\alpha)+\sum_{l\geq 1}\alpha\llbrace\alpha\rrbrace_l=\sum_{l\geq 1}H_n^i(d(\alpha\llbrace\alpha\rrbrace_l)).$$

    \noindent Let $\mathcal{R}(\alpha)=d(\alpha)+\sum_{l\geq 1}\alpha\llbrace\alpha\rrbrace_l$. We use the identity
    $$\sum_{p+q=l}\alpha\llbrace\alpha\rrbrace_p\llbrace\alpha\rrbrace_q+\sum_{p+q=l-1}\alpha\llbrace\alpha\llbrace\alpha\rrbrace_p,\alpha\rrbrace_{1,q}=0$$

    \noindent for every $l\geq 1$. We thus have
    $$d(\alpha\llbrace\alpha\rrbrace_l)=-\sum_{\substack{p+q=l\\ q\neq 0}}\alpha\llbrace\alpha\rrbrace_p\llbrace\alpha\rrbrace_q-\sum_{p+q=l-1}\alpha\llbrace\alpha\llbrace\alpha\rrbrace_p,\alpha\rrbrace_{1,q}$$
    \noindent so that
    $$\sum_{l\geq 1}d(\alpha\llbrace\alpha\rrbrace_l)=-\sum_{q\geq 1}\mathcal{R}(\alpha)\llbrace\alpha\rrbrace_q-\sum_{q\geq 0}\alpha\llbrace\mathcal{R}(\alpha),\alpha\rrbrace_q.$$

    \noindent This leads finally to the identity

    $$\mathcal{R}(\alpha)=-\sum_{q\geq 1}H_n^i(\mathcal{R}(\alpha)\llbrace\alpha\rrbrace_q)-\sum_{q\geq 0}H_n^i(\alpha\llbrace\mathcal{R}(\alpha),\alpha\rrbrace_{1,q}).$$

    \noindent It follows from this identity that if $\mathcal{R}(\alpha)\in F_k(A)\otimes\Sigma N^*(\Delta^n)$ for some $k\geq 1$, then $\mathcal{R}(\alpha)\in F_{k+1}(A)\otimes\Sigma N^*(\Delta^n)$. We thus have $\mathcal{R}(\alpha)=0$ so that $\alpha\in\mathcal{MC}(A\otimes\Sigma N^*(\Delta^n))$. We then have a bijection which has as inverse this previous construction.
\end{proof}

\begin{defi}
    A simplicial $\widehat{\Gamma\Lambda\mathcal{PL}_\infty}$-algebra is a simplicial object in the category $\widehat{\Gamma\Lambda\mathcal{PL}_\infty}$. A simplicial $\widehat{\Gamma\Lambda\mathcal{PL}_\infty}$-algebra $A$ is {\normalfont strict} if the face and degeneracy maps of $A$ are strict morphisms of $\widehat{\Gamma\Lambda\mathcal{PL}_\infty}$-algebras. A morphism of simplicial $\widehat{\Gamma\Lambda\mathcal{PL}_\infty}$-algebras $\phi:A\longrightarrow B$ is {\normalfont strict} if, for every $n\geq 0$, the map $\phi_n:A_n\longrightarrow B_n$ is a strict morphism of $\widehat{\Gamma\Lambda\mathcal{PL}_\infty}$-algebras.
\end{defi}

\begin{thm}\label{kan}
    Let $A,B\in\text{\normalfont\dgMod}$ be such that $A\otimes\Sigma N^*(\Delta^\bullet)$ and $B\otimes\Sigma N^*(\Delta^\bullet)$ are endowed with the structure of simplicial $\widehat{\Gamma\Lambda\mathcal{PL}_\infty}$-algebras. Let $f:A\longrightarrow B$ be a surjective morphism in \normalfont{\dgMod} such that $f\otimes id:A\otimes\Sigma N^*(\Delta^\bullet)\longrightarrow B\otimes\Sigma N^*(\Delta^\bullet)$ is a strict morphism.\\ Then $\mathcal{MC}(f\otimes id):\mathcal{MC}(A\otimes\Sigma N^*(\Delta^\bullet))\longrightarrow\mathcal{MC}(B\otimes\Sigma N^*(\Delta^\bullet))$ is a Kan fibration.
\end{thm}

\begin{proof}
    Since the $\Gamma\Lambda\mathcal{PL}_\infty$-algebra structures are compatible with the simplicial structures on $A\otimes\Sigma N^*(\Delta^\bullet)$ and $B\otimes\Sigma N^*(\Delta^\bullet)$, we can follow the same proof as in \cite[Proposition 6.6]{felix} to obtain the result.
\end{proof}

Applying this result to $B=0$ thus gives the following corollary.

\begin{cor}
    For every $A\in\text{\normalfont \dgMod}$ such that $A\otimes\Sigma N^*(\Delta^\bullet)$ is a strict simplicial $\widehat{\Gamma\Lambda\mathcal{PL}_\infty}$-algebra, the simplicial set $\mathcal{MC}(A\otimes\Sigma N^*(\Delta^\bullet))$ is a Kan complex.\\ In particular, for every complete brace algebra $A$, the simplicial set $\mathcal{MC}_\bullet(A)$ is a Kan complex.
\end{cor}

\subsection{Connected components and homotopy groups of $\mathcal{MC}_\bullet(A)$}\label{sec:242}

We are now able to compute the connected components and the homotopy groups of $\mathcal{MC}_\bullet(A)$ for a given complete brace algebra $A$. For this purpose, recall from \cite[Theorem 2.15]{moi} that any brace algebra $A$ is endowed with the structure of a $\Gamma(\mathcal{P}re\mathcal{L}ie,-)$-algebra via the composite
\begin{center}
\begin{tikzcd}
\Gamma(\mathcal{P}re\mathcal{L}ie,A)\arrow[r] & \Gamma(\mathcal{B}race,A) & \mathcal{S}(\mathcal{B}race,A)\arrow[l, "\simeq"']\arrow[r] & A.
\end{tikzcd}
\end{center}

\noindent In this setting, we recall from \cite[Definition 2.19]{moi} the operation $\circledcirc$ defined by
$$x\circledcirc(1+y):=\sum_{n\geq 0}x\langle\underbrace{y,\ldots,y}_n\rangle$$

\noindent for every $x\in A$ and $y\in A_0$. By \cite[Theorem 2.24]{moi}, we have that this operation induces a group structure on the set $G=1+A_0$ with the product
$$(1+x)\circledcirc (1+y)=1+x+y+\sum_{n\geq 1}x\langle \underbrace{y,\ldots,y}_n\rangle.$$

\noindent This group is called the \textit{gauge group} associated to the brace algebra $A$. In the following, we use the operation $\overline\circledcirc$ defined by
$$x\overline\circledcirc y:=x+y+\sum_{n\geq 1}x\langle\underbrace{y,\ldots,y}_n\rangle,$$

\noindent for every $x\in A$ and $y\in A_0$. Note that the group $(1+A_0,\circledcirc,1)$ is isomorphic to the group $(A_0,\overline{\circledcirc},0)$. Using this identification and \cite[Theorem 2.29]{moi}, we have an action of $(A_0,\overline\circledcirc,0)$ on $\mathcal{MC}(A)$ by
$$x\cdot \tau=(\tau+x\langle \tau\rangle-d(x))\overline{\circledcirc} x^{\overline{\circledcirc} -1}$$

\noindent for every $x\in A_0$ and $\tau\in\mathcal{MC}(A)$.\\

By Corollary \ref{calculdiff}, we have an obvious identification $\mathcal{MC}_0(A)=\mathcal{MC}(A)$, using the Maurer-Cartan set of a $\Gamma(\mathcal{P}re\mathcal{L}ie,-)$-algebra (see \cite[Definition 2.17]{moi}). This identification is given by sending $\tau\in\mathcal{MC}(A)$ to $- \tau\otimes\Sigma\underline{0}^\vee\in\mathcal{MC}_0(A)$.\\

In this subsection, in order to write easier formulas, for every $\underline{x}\in N_*(\Delta^n)$, we drop the desuspension $\Sigma^{-1}$ on the element $\Sigma^{-1}\underline{x}\in\Sigma^{-1} N_*(\Delta^n)$. Analogously, we drop the suspension $\Sigma$ on elements of $\Sigma N^*(\Delta^n)$. 

\subsubsection{Connected components}

We first compute the $\pi_0$. We begin by the following lemma.

\begin{lm}\label{mc1}
    Let $\tau_0,\tau_1\in\mathcal{MC}(A)$. Then every element $\alpha\in\mathcal{MC}_1(A)$ such that $d_0\alpha=\tau_0$ and $d_1\alpha=\tau_1$ are written
    $$\alpha=-\tau_1\otimes  \underline{0}^\vee-\tau_0\otimes  \underline{1}^\vee-h\otimes  \underline{01}^\vee$$
    \noindent where $h\in A_0$ is such that
    $$d(h)=\tau_0+h\langle\tau_0\rangle-\tau_1\circledcirc(1+h).$$
\end{lm}

\begin{proof}
    Let $\alpha\in\mathcal{MC}_1(A)$ be such that $d_0\alpha=\tau_0$ and $d_1\alpha=\tau_1$. We write 

    $$\alpha=-\tau_1\otimes  \underline{0}^\vee-\tau_0\otimes  \underline{1}^\vee-h\otimes  \underline{01}^\vee$$
    
    \noindent for some $h\in A_0$. We make explicit the Maurer-Cartan condition on $\alpha$. We have
    $$d(\alpha)=-d(\tau_1)\otimes  \underline{0}^\vee-d(\tau_0)\otimes  \underline{1}^\vee+(-d(h)+\tau_0-\tau_1)\otimes  \underline{01}^\vee.$$

    \noindent Let $p\geq 1$. By formula $(v)$ of Theorem \ref{relationsprelieinfinite}, we have that
    \begin{multline*}
        \alpha\llbrace \alpha\rrbrace_p = \sum_{p_1+p_2+p_3=p} -\tau_1\otimes  \underline{0}^\vee\llbrace -\tau_1\otimes  \underline{0}^\vee, -\tau_0\otimes  \underline{1}^\vee, -h\otimes  \underline{01}^\vee\rrbrace_{p_1,p_2,p_3}\\
    -\sum_{p_1+p_2+p_3=p}\tau_0\otimes  \underline{1}^\vee\llbrace -\tau_1\otimes  \underline{0}^\vee, -\tau_0\otimes  \underline{1}^\vee, -h\otimes  \underline{01}^\vee\rrbrace_{p_1,p_2,p_3}\\
    -\sum_{p_1+p_2+p_3=p}h\otimes  \underline{01}^\vee\llbrace -\tau_1\otimes  \underline{0}^\vee, -\tau_0\otimes  \underline{1}^\vee, -h\otimes  \underline{01}^\vee\rrbrace_{p_1,p_2,p_3}.
    \end{multline*}

    \noindent By the computation of $\partial^1$ in Corollary \ref{calculdiff}, the first sum gives non-zero elements only for the case $p_1=1$ and $p_2=p_3=0$, and the case $p_1=p_2=0$. This then gives
    $$-\tau_1\langle\tau_1\rangle\otimes  \underline{0}^\vee-\sum_{n\geq 1}\tau_1\langle\underbrace{h,\ldots,h}_n\rangle\otimes  \underline{01}^\vee.$$

    \noindent The second sum gives non-zero elements only for the case $p_2=1$ and $p_1=p_3=0$. We obtain the term
    $$-\tau_0\langle\tau_0\rangle\otimes  \underline{1}^\vee.$$

    \noindent Finally, the third sum gives non-zero elements only for the case $p_2=1$ and $p_1=p_3=0$. We have the term
    $$h\langle\tau_0\rangle\otimes  \underline{01}^\vee.$$

    \noindent At the end, we have
    $$\sum_{p\geq 0}\alpha\llbrace \alpha\rrbrace_p=\left(-d(h)+\tau_0+h\langle\tau_0\rangle-\sum_{n\geq 0}\tau_1\langle\underbrace{h,\ldots,h}_n\rangle\right)\otimes  \underline{01}^\vee.$$

    \noindent Then, the Maurer-Cartan condition on $\alpha$ is equivalent to the equation
    $$d(h)=\tau_0+h\langle\tau_0\rangle-\tau_1\circledcirc(1+h)$$

    \noindent which proves the lemma.
\end{proof}

Recall that the Deligne groupoid associated to $A$ is the category formed by Maurer-Cartan elements with as morphisms the elements of the gauge group (see \cite[Proposition-Definition 2.22]{moi}).

\begin{thm}\label{pi0}
    Let $A$ be a complete brace algebra. We have a bijection
    $$\pi_0(\mathcal{MC}_\bullet(A))\simeq\pi_0\text{\normalfont Deligne}(A),$$

    \noindent where we denote by $\pi_0\text{\normalfont Deligne}(A)$ the set of objects in $\text{\normalfont Deligne}(A)$ up to isomorphisms.
\end{thm}

\begin{proof}
    Recall that
    $$\pi_0(\mathcal{MC}_\bullet(A))=\mathcal{MC}_0(A)/\sim,$$

    \noindent where $\sim$ denotes the homotopy relation in $\text{\normalfont sSet}$. Consider the projection $f$ of $\mathcal{MC}_0(A)$ on $\mathcal{MC}(A)/G$, where $G$ denotes the gauge group of the $\Gamma(\mathcal{P}re\mathcal{L}ie,-)$-algebra $A$. Let $\tau_0,\tau_1\in\mathcal{MC}(A)$. By Lemma \ref{mc1}, the elements $- \tau_1\otimes\underline{0}^\vee$ and $-\tau_0\otimes\underline{1}^\vee$ are homotopic in $\mathcal{MC}_0(A)$ if and only if there exists $h\in A_0$ such that $h\cdot\tau_0=\tau_1$, which proves that $f$ induces a bijection $\overline{f}:\pi_0(\mathcal{MC}_\bullet(A))\longrightarrow\pi_0\text{\normalfont Deligne}(A)$.
\end{proof}

\subsubsection{The group $\pi_1(\mathcal{MC}_\bullet(A),\tau)$}

We now compute $\pi_1(\mathcal{MC}_\bullet(A),\tau)$ for a given $\tau\in\mathcal{MC}(A)$. Let $\text{\normalfont Aut}_{\text{\normalfont Deligne}(A)}(\tau)=\{h\in A_0\ |\ d(h)=\tau+h\langle\tau\rangle-\tau\circledcirc (1+h)\}$. We have the following lemma.

\begin{lm}
    Let $A$ be a complete brace algebra and $\tau\in\mathcal{MC}(A)$. For every $h,h'\in\text{\normalfont Aut}_{\text{\normalfont Deligne}(A)}(\tau)$, we write $h\sim_\tau h'$ if there exists $\psi\in A_{1}$ such that
    
    $$h-h'=d(\psi)+\psi\langle\tau\rangle+\sum_{p,q\geq 0}\tau\langle\underbrace{h,\ldots,h}_{p},\psi,\underbrace{h',\ldots,h'}_q\rangle.$$


\noindent Then $\sim_\tau$ is an equivalence relation on the set $\text{\normalfont Aut}_{\text{\normalfont Deligne}(A)}(\tau)$. Moreover, the circular product $\circledcirc$ is compatible with $\sim_\tau$, so that the triple $(\text{\normalfont Aut}_{\text{\normalfont Deligne}(A)}(\tau)/\sim_\tau,\overline{\circledcirc},0)$ is a group.
\end{lm}

\begin{proof}
    The relation $\sim_\tau$ is reflexive (just take $\psi=0$ so that $h\sim_\tau h$ for all $h\in \text{\normalfont Aut}_{\text{\normalfont Deligne}(A)}(\tau)$). We prove that this relation is transitive. Let $h,h',h''\in A_0$ be such that $h\sim_\tau h'$ and $h'\sim_\tau h''$. Then there exist $\psi,\psi'\in A_1$ such that
    \begin{equation}h-h'=d(\psi)+\psi\langle\tau\rangle+\sum_{p,q\geq 0}\tau\langle\underbrace{h,\ldots,h}_q,\psi,\underbrace{h',\ldots,h'}_q\rangle;\end{equation}
    \begin{equation}h'-h''=d(\psi')+\psi'\langle\tau\rangle+\sum_{p,q\geq 0}\tau\langle\underbrace{h',\ldots,h'}_q,\psi',\underbrace{h'',\ldots,h''}_q\rangle.\end{equation}
    
    \noindent We set $\psi'':=\psi+\psi'+\sum_{p,q,r\geq 0}\tau\langle\underbrace{h,\ldots,h}_p,\psi,\underbrace{h',\ldots,h'}_q,\psi',\underbrace{h'',\ldots,h''}_r\rangle$, and prove that

$$h-h''=d(\psi'')+\psi''\langle\tau\rangle+\sum_{p,q\geq 0}\tau\langle\underbrace{h,\ldots,h}_p,\psi'',\underbrace{h'',\ldots,h''}_q\rangle.$$

\noindent Let us analyze the right hand-side. We analyze the terms given by $d(\psi'')$ and compare it with the others given either by $\psi''\langle\tau\rangle$ or by the terms of the form $\tau\langle h,\ldots,h,\psi'',h'',\ldots,h''\rangle$. We first have
$$d(\psi)+d(\psi')=h-h''-\psi\langle\tau\rangle-\psi'\langle\tau\rangle-\sum_{p,q\geq 0}\tau\langle\underbrace{h,\ldots,h}_p,\psi,\underbrace{h',\ldots,h'}_q\rangle-\sum_{q,r\geq 0}\tau\langle\underbrace{h',\ldots,h'}_q,\psi',\underbrace{h'',\ldots,h''}_r\rangle.$$

\noindent We now differentiate the sum which occurs in the definition of $\psi''$. By the Leibniz rule in the brace algebra $A$, and by applying the differential on $\tau\in\mathcal{MC}(A)$, we, in particular, obtain the sum
$$-\sum_{p,q,r\geq 0}\tau\langle\tau\rangle\langle \underbrace{h,\ldots,h}_p,\psi,\underbrace{h',\ldots,h'}_q,\psi',\underbrace{h'',\ldots,h''}_r\rangle.$$

\noindent This can be computed by using the brace algebra structure of $A$:
\begin{multline*}
    -\sum_{p,q,r\geq 0}\tau\langle\tau\rangle\langle \underbrace{h,\ldots,h}_p,\psi,\underbrace{h',\ldots,h'}_q,\psi',\underbrace{h'',\ldots,h''}_r\rangle=\\
    -\sum_{p_1,p_2,q,r\geq 0}\tau\langle\underbrace{h,\ldots,h}_{p_1},\tau\circledcirc(1+h),\underbrace{h,\ldots,h}_{p_2},\psi,\underbrace{h',\ldots,h'}_q,\psi',\underbrace{h'',\ldots,h''}_r\rangle\\
    -\sum_{p,q,r,s,t\geq 0}\tau\langle\underbrace{h,\ldots,h}_{p},\tau\langle\underbrace{h,\ldots,h}_{s},\psi,\underbrace{h',\ldots,h'}_{t}\rangle,\underbrace{h',\ldots,h'}_{q},\psi',\underbrace{h'',\ldots,h''}_r\rangle\\
    +\sum_{p,q_1,q_2,r\geq 0}\tau\langle\underbrace{h,\ldots,h}_p,\psi,\underbrace{h',\ldots,h'}_{q_1},\tau\circledcirc(1+h'),\underbrace{h',\ldots,h'}_{q_2},\psi',\underbrace{h'',\ldots,h''}_r\rangle\\
    +\sum_{p,q,r,s,t\geq 0}\tau\langle\underbrace{h,\ldots,h}_p,\psi,\underbrace{h',\ldots,h'}_{q},\tau\langle\underbrace{h',\ldots,h'}_{s},\psi',\underbrace{h'',\ldots,h''}_{t}\rangle,\underbrace{h'',\ldots,h''}_{r}\rangle\\
    -\sum_{p,q,r_1,r_2\geq 0}\tau\langle\underbrace{h,\ldots,h}_p,\psi,\underbrace{h',\ldots,h'}_r,\psi',\underbrace{h'',\ldots,h''}_{r_1},\tau\circledcirc(1+h''),\underbrace{h'',\ldots,h''}_{r_2}\rangle\\
    -\sum_{p,q,r,s,t\geq 0}\tau\langle\underbrace{h,\ldots,h}_{p},\tau\langle\underbrace{h,\ldots,h}_{q},\psi,\underbrace{h',\ldots,h'}_s,\psi',\underbrace{h'',\ldots,h''}_{t}\rangle,\underbrace{h'',\ldots,h''}_{r}\rangle.
\end{multline*}

\noindent The remaining terms obtained by the Leibniz rule in the sum occurring in the definition of $\psi''$ are
\medskip
\begin{center}
$\begin{array}{ll}
- & \displaystyle\sum_{p_1,p_2,q,r\geq 0}\tau\langle\underbrace{h,\ldots,h}_{p_1},d(h),\underbrace{h,\ldots,h}_{p_2},\psi,\underbrace{h',\ldots,h'}_q,\psi',\underbrace{h'',\ldots,h''}_r\rangle\\
- & \displaystyle\sum_{p,q,r\geq 0}\tau\langle\underbrace{h,\ldots,h}_{p},d(\psi),\underbrace{h',\ldots,h'}_{q},\psi',\underbrace{h'',\ldots,h''}_r\rangle\\
+ & \displaystyle\sum_{p,q_1,q_2,r\geq 0}\tau\langle\underbrace{h,\ldots,h}_p,\psi,\underbrace{h',\ldots,h'}_{q_1},d(h'),\underbrace{h',\ldots,h'}_{q_2},\psi',\underbrace{h'',\ldots,h''}_r\rangle\\
+ & \displaystyle\sum_{p,q,r\geq 0}\tau\langle\underbrace{h,\ldots,h}_p,\psi,\underbrace{h',\ldots,h'}_{q},d(\psi'),\underbrace{h'',\ldots,h''}_{r}\rangle\\
- & \displaystyle\sum_{p,q,r_1,r_2\geq 0}\tau\langle\underbrace{h,\ldots,h}_p,\psi,\underbrace{h',\ldots,h'}_r,\psi',\underbrace{h'',\ldots,h''}_{r_1},d(h''),\underbrace{h'',\ldots,h''}_{r_2}\rangle.
\end{array}$
\end{center}
\medskip

\noindent By using equations (1) and (2), the definition of $\psi''$ and that $h,h',h''\in\text{\normalfont Aut}_{\text{\normalfont Deligne}(A)}(\tau)$, we obtain
\medskip
\begin{multline*}
    d(\psi'')=h-h''-\psi'\langle\tau\rangle-\psi'\langle\tau\rangle-\sum_{p,q\geq 0}\tau\langle\underbrace{h,\ldots,h}_p,\psi,\underbrace{h',\ldots,h'}_q\rangle\\
    -\sum_{q,r\geq 0}\tau\langle\underbrace{h',\ldots,h'}_q,\psi',\underbrace{h'',\ldots,h''}_r\rangle\\
    -\sum_{p_1,p_2,q,r\geq 0}\tau\langle\underbrace{h,\ldots,h}_{p_1},\tau+h\langle\tau\rangle,\underbrace{h,\ldots,h}_{p_2},\psi,\underbrace{h',\ldots,h'}_q,\psi',\underbrace{h'',\ldots,h''}_r\rangle\\
    -\sum_{p,q,r\geq 0}\tau\langle\underbrace{h,\ldots,h}_{p},h-h'-\psi\langle\tau\rangle,\underbrace{h',\ldots,h'}_q,\psi',\underbrace{h'',\ldots,h''}_r\rangle\\
    +\sum_{p,q_1,q_2,r\geq 0}\tau\langle\underbrace{h,\ldots,h}_p,\psi,\underbrace{h',\ldots,h'}_{q_1},\tau+h'\langle\tau\rangle,\underbrace{h',\ldots,h'}_{q_2},\psi',\underbrace{h'',\ldots,h''}_r\rangle\\
    +\sum_{p,q,r\geq 0}\tau\langle\underbrace{h,\ldots,h}_p,\psi,\underbrace{h',\ldots,h'}_{q},h'-h''-\psi'\langle\tau\rangle,\underbrace{h'',\ldots,h''}_{r}\rangle\\
    -\sum_{p,q,r_1,r_2\geq 0}\tau\langle\underbrace{h,\ldots,h}_p,\psi,\underbrace{h',\ldots,h'}_r,\psi',\underbrace{h'',\ldots,h''}_{r_1},\tau+h''\langle\tau\rangle,\underbrace{h'',\ldots,h''}_{r_2}\rangle\\
    -\sum_{p,r\geq 0}\tau\langle\underbrace{h,\ldots,h}_{p},\psi''-\psi-\psi',\underbrace{h'',\ldots,h''}_{r}\rangle.
    \end{multline*}
    \medskip

\noindent Now, by some variable substitutions, note that we have the identities
\medskip
\begin{multline*}
    \sum_{p,q,r\geq 0}\tau\langle\underbrace{h,\ldots,h}_{p},h-h',\underbrace{h',\ldots,h'}_q,\psi',\underbrace{h'',\ldots,h''}_r\rangle\\=\sum_{p,r\geq 0}\tau\langle\underbrace{h,\ldots,h}_p,\psi',\underbrace{h'',\ldots,h''}_r\rangle-\sum_{q,r\geq 0}\tau\langle\underbrace{h',\ldots,h'}_q,\psi',\underbrace{h'',\ldots,h''}_r\rangle;
\end{multline*}
\medskip
\begin{multline*}
    \sum_{p,q,r\geq 0}\tau\langle\underbrace{h,\ldots,h}_p,\psi,\underbrace{h',\ldots,h'}_{q},h'-h'',\underbrace{h'',\ldots,h''}_{r}\rangle\\=\sum_{p,q\geq 0}\tau\langle\underbrace{h,\ldots,h}_p,\psi,\underbrace{h',\ldots,h'}_q\rangle-\sum_{p,r\geq 0}\tau\langle\underbrace{h,\ldots,h}_p,\psi,\underbrace{h'',\ldots,h''}_r\rangle.
\end{multline*}
\medskip

\noindent This finally gives
\medskip
\begin{multline*}
    d(\psi'')=h-h''-\psi'\langle\tau\rangle-\psi'\langle\tau\rangle\\
    -\sum_{p_1,p_2,q,r\geq 0}\tau\langle\underbrace{h,\ldots,h}_{p_1},\tau+h\langle\tau\rangle,\underbrace{h,\ldots,h}_{p_2},\psi,\underbrace{h',\ldots,h'}_q,\psi',\underbrace{h'',\ldots,h''}_r\rangle\\
    +\sum_{p,q,r\geq 0}\tau\langle\underbrace{h,\ldots,h}_{p},\psi\langle\tau\rangle,\underbrace{h',\ldots,h'}_q,\psi',\underbrace{h'',\ldots,h''}_r\rangle\\
    +\sum_{p,q_1,q_2,r\geq 0}\tau\langle\underbrace{h,\ldots,h}_p,\psi,\underbrace{h',\ldots,h'}_{q_1},\tau+h'\langle\tau\rangle,\underbrace{h',\ldots,h'}_{q_2},\psi',\underbrace{h'',\ldots,h''}_r\rangle\\
    -\sum_{p,q,r\geq 0}\tau\langle\underbrace{h,\ldots,h}_p,\psi,\underbrace{h',\ldots,h'}_{q},\psi'\langle\tau\rangle,\underbrace{h'',\ldots,h''}_{r}\rangle\\
    -\sum_{p,q,r_1,r_2\geq 0}\tau\langle\underbrace{h,\ldots,h}_p,\psi,\underbrace{h',\ldots,h'}_r,\psi',\underbrace{h'',\ldots,h''}_{r_1},\tau+h''\langle\tau\rangle,\underbrace{h'',\ldots,h''}_{r_2}\rangle\\
    -\sum_{p,r\geq 0}\tau\langle\underbrace{h,\ldots,h}_{p},\psi'',\underbrace{h'',\ldots,h''}_{r}\rangle.
    \end{multline*}
    \medskip

\noindent We also have
\begin{multline*}
    \psi''\langle\tau\rangle=\psi\langle\tau\rangle+\psi'\langle\tau\rangle+\sum_{p_1,p_2,q,r\geq 0}\tau\langle\underbrace{h,\ldots,h}_{p_1},\tau+h\langle\tau\rangle,\underbrace{h,\ldots,h}_{p_2},\psi,\underbrace{h',\ldots,h'}_q,\psi',\underbrace{h'',\ldots,h''}_r\rangle\\
    -\sum_{p,q,r\geq 0}\tau\langle\underbrace{h,\ldots,h}_p,\psi\langle\tau\rangle,\underbrace{h',\ldots,h'}_q,\psi',\underbrace{h'',\ldots,h''}_r\rangle\\
    -\sum_{p,q_1,q_2,r\geq 0}\tau\langle\underbrace{h,\ldots,h}_p,\psi,\underbrace{h',\ldots,h'}_{q_1},\tau+h'\langle\tau\rangle,\underbrace{h',\ldots,h'}_{q_2},\psi',\underbrace{h'',\ldots,h''}_r\rangle\\
    +\sum_{p,q,r\geq 0}\tau\langle\underbrace{h,\ldots,h}_p,\psi,\underbrace{h',\ldots,h'}_q,\psi'\langle\tau\rangle,\underbrace{h'',\ldots,h''}_r\rangle\\
    +\sum_{p,q,r_1,r_2\geq 0}\tau\langle\underbrace{h,\ldots,h}_p,\psi,\underbrace{h',\ldots,h'}_q,\psi',\underbrace{h'',\ldots,h''}_{r_1},\tau+h''\langle\tau\rangle,\underbrace{h'',\ldots,h''}_{r_2}\rangle.
\end{multline*}

\noindent At the end, we obtain
$$d(\psi'')+\psi''\langle\tau\rangle =h-h''-\sum_{p,q\geq 0}\tau\langle\underbrace{h,\ldots,h}_p,\psi'',\underbrace{h'',\ldots,h''}_q\rangle$$

\noindent which proves that $h\sim_\tau h''$.\\

We now prove that if $h\sim_{\tau} h'$, then $h'\sim_\tau h$. We use the previous construction. More precisely, let $\psi\in A_{1}$ be such that
$$h-h'=d(\psi)+\psi\langle\tau\rangle+\sum_{p,q\geq 0}\tau\langle\underbrace{h,\ldots,h}_{p},\psi,\underbrace{h',\ldots,h'}_q\rangle.$$

\noindent We search some element $\psi'$ such that the associated $\psi''$ previously constructed for the transitivity is $0$. We set $\psi_0'=-\psi$ and, for all $n\geq 0$, 

$$\psi_{n+1}'=-\psi-\sum_{p,q,r\geq 0}\tau\langle\underbrace{h,\ldots,h}_q,\psi,\underbrace{h',\ldots,h'}_q,\psi_n',\underbrace{h,\ldots,h}_r\rangle.$$


\noindent We obtain a Cauchy sequence $(\psi_n')_n$. Let $\psi'$ be its limit, which satisfies

$$\psi'=-\psi-\sum_{p,q,r\geq 0}\tau\langle\underbrace{h,\ldots,h}_p,\psi,\underbrace{h',\ldots,h'}_q,\psi',\underbrace{h,\ldots,h}_r\rangle.$$


\noindent By the same computations as for the proof of the transitivity, we can check that $\psi'$ satisfies the equation


$$h'-h=d(\psi')+\psi'\langle\tau\rangle+\sum_{p,q\geq 0}\tau\langle\underbrace{h',\ldots,h'}_p,\psi',\underbrace{h,\ldots,h}_q\rangle,$$

\noindent which proves that $h'\sim_\tau h$.\\

We thus have proved that $\sim_\tau$ is an equivalence relation. We now prove that the circular product $\circledcirc$ is compatible with $\sim_\tau$. Let $h,h_1,h_2\in A_0$ be such that $h_1\sim_\tau h_2$. Let $\psi\in A_{1}$ be such that
\begin{equation}
    h_1-h_2=d(\psi)+\psi\langle\tau\rangle+\sum_{p,q\geq 0}\tau\langle\underbrace{h_1,\ldots,h_1}_p,\psi,\underbrace{h_2,\ldots,h_2}_q\rangle.
\end{equation}

We prove first that $h_1\overline{\circledcirc}{h}\sim_\tau h_2\overline{\circledcirc}{h}$. Let $\psi':=\psi\circledcirc(1+{h})$. We need to show that
$$h_1\overline\circledcirc{h}-h_2\overline\circledcirc{h}=d(\psi')+\psi'\langle\tau\rangle+\sum_{p,q\geq 0}\tau\langle\underbrace{h_1\overline\circledcirc{h},\ldots,h_1\overline\circledcirc{h}}_p,\psi',\underbrace{h_2\overline\circledcirc{h},\ldots,h_2\overline\circledcirc{h}}_q\rangle.$$

\noindent We first compute $d(\psi')$. From \cite[Lemma 2.28]{moi}, we have
$$d(\psi')=d(\psi)\circledcirc(1+{h})-\psi\circledcirc(1+{h};d({h})),$$

\noindent where we have set, following \cite[Definition 2.27]{moi} in the case of a complete brace algebra,
$$a\circledcirc(1+b;c):=\sum_{n\geq 0}\sum_{k=0}^na\langle b,\ldots,b,\underset{k}{c},b,\ldots,b\rangle.$$

\noindent for every $a\in A$ and $b,c\in A_0$. We have
$$d(\psi)\circledcirc(1+{h})=h_1\overline\circledcirc{h}-h_2\overline\circledcirc{h}-\psi\langle\tau\rangle\circledcirc(1+{h})-\sum_{p,q\geq 0}\tau\langle\underbrace{h_1,\ldots,h_1}_p,\psi,\underbrace{h_2,\ldots,h_2}_q\rangle\circledcirc(1+{h}).$$

\noindent By the second formula of \cite[Lemma 2.28]{moi}, we have
$$\psi\langle\tau\rangle\circledcirc(1+{h})=\psi\circledcirc(1+{h};\tau\circledcirc(1+{h})).$$

\noindent Finally, we have
\medskip
$$\sum_{p,q\geq 0}\tau\langle\underbrace{h_1,\ldots,h_1}_p,\psi,\underbrace{h_2,\ldots,h_2}_q\rangle\circledcirc(1+{h})=\sum_{p,q\geq 0}\tau\langle \underbrace{h_1\overline\circledcirc{h},\ldots,h_1\overline\circledcirc{h}}_p,\psi\circledcirc(1+{h}),\underbrace{h_2\overline\circledcirc{h},\ldots,h_2\overline\circledcirc{h}}_q\rangle.$$
\medskip

\noindent We thus have

\medskip
\begin{multline*}
    d(\psi')=h_1\overline\circledcirc{h}-h_2\overline\circledcirc{h}-\psi\circledcirc(1+{h};\tau\circledcirc(1+{h}))\\-\sum_{p,q\geq 0}\tau\langle \underbrace{h_1\overline\circledcirc{h},\ldots,h_1\overline\circledcirc{h}}_p,\psi',\underbrace{h_2\overline\circledcirc{h},\ldots,h_2\overline\circledcirc{h}}_q\rangle-\psi\circledcirc(1+{h};d({h})).
\end{multline*}
\medskip

\noindent Since we have, by \cite[Lemma 2.28]{moi}, that
$$\psi'\langle\tau\rangle=\psi\circledcirc(1+{h};\tau+{h}\langle\tau\rangle).$$
\medskip
\noindent we obtain at the end
\medskip
$$d(\psi')+\psi'\langle\tau\rangle+\sum_{p,q\geq 0}\tau\langle\underbrace{h_1\overline\circledcirc{h},\ldots,h_1\overline\circledcirc{h}}_p,\psi',\underbrace{h_2\overline\circledcirc{h},\ldots,h_2\overline\circledcirc{h}}_q\rangle=h_1\overline\circledcirc{h}-h_2\overline{\circledcirc}{h}$$
\medskip

\noindent which proves that $h_1\overline\circledcirc{h}\sim_\tau h_2\overline\circledcirc{h}$.\\

We now prove that $h\overline\circledcirc{h}_1\sim_\tau h\overline\circledcirc{h}_2$. Let $\psi'=\psi+\sum_{p,q\geq 0}h\langle\underbrace{{h}_1,\ldots,{h}_1}_p,\psi,\underbrace{{h}_2,\ldots,{h}_2}_q\rangle$. We show that
\medskip
$$h\overline\circledcirc{h}_1-h\overline\circledcirc{h}_2=d(\psi')+\psi'\langle\tau\rangle+\sum_{p,q\geq 0}\tau\langle\underbrace{h\overline\circledcirc{h}_1,\ldots,h\overline\circledcirc{h}_1}_p,\psi',\underbrace{h\overline\circledcirc{h}_2,\ldots,h\overline\circledcirc{h}_2}_q\rangle.$$
\medskip

\noindent We first compute the sum $\sum_{p,q\geq 0}d(h)\langle\underbrace{{h}_1,\ldots,{h}_1}_p,\psi,\underbrace{{h}_2,\ldots,{h}_2}_q\rangle$. We use that $d(h)=\tau+h\langle\tau\rangle-\tau\circledcirc(1+h)$ to get
\medskip
\begin{multline*}\sum_{p,q\geq 0}d(h)\langle\underbrace{{h}_1,\ldots,{h}_1}_p,\psi,\underbrace{{h}_2,\ldots,{h}_2}_q\rangle=\sum_{p,q\geq 0}\tau\langle\underbrace{{h}_1,\ldots,{h}_1}_p,\psi,\underbrace{{h}_2,\ldots,{h}_2}_q\rangle\\
+\sum_{p_1,p_2,q\geq 0}h\langle\underbrace{{h}_1,\ldots,{h}_1}_{p_1},\tau\circledcirc(1+ h_1),\underbrace{h_1,\ldots,h_1}_{p_2},\psi,\underbrace{{h}_2,\ldots,{h}_2}_q\rangle\\
+\sum_{p,q,s,t\geq 0}h\langle\underbrace{h_1,\ldots,h_1}_p,\tau\langle\underbrace{h_1,\ldots,h_1}_s,\psi,\underbrace{h_2,\ldots,h_2}_t\rangle,\underbrace{h_2,\ldots,h_2}_q\rangle\\
-\sum_{p,q_1,q_2\geq 0}h\langle\underbrace{{h}_1,\ldots,{h}_1}_p,\psi,\underbrace{{h}_2,\ldots,{h}_2}_{q_1},\tau\circledcirc(1+ h_2),\underbrace{h_2,\ldots,h_2}_{q_1}\rangle\\
-\sum_{p,q,s,t\geq 0}\tau\langle\underbrace{h\overline\circledcirc{h}_1,\ldots,h\overline\circledcirc{h}_1}_p,\psi+h\langle\underbrace{h_1,\ldots,h_1}_s,\psi,\underbrace{h_2,\ldots,h_2}_t\rangle,\underbrace{h\overline\circledcirc{h}_2,\ldots,h\overline\circledcirc{h}_2}_q\rangle.
\end{multline*}
\medskip

\noindent Using the Leibniz rule, we obtain
\medskip
\begin{multline*}
d(\psi')=d(\psi)+\sum_{p,q\geq 0}\tau\langle\underbrace{{h}_1,\ldots,{h}_1}_p,\psi,\underbrace{{h}_2,\ldots,{h}_2}_q\rangle\\
+\sum_{p_1,p_2,q\geq 0}h\langle\underbrace{{h}_1,\ldots,{h}_1}_{p_1},\tau\circledcirc(1+ h_1),\underbrace{h_1,\ldots,h_1}_{p_2},\psi,\underbrace{{h}_2,\ldots,{h}_2}_q\rangle\\
+\sum_{p,q,s,t\geq 0}h\langle\underbrace{h_1,\ldots,h_1}_p,\tau\langle\underbrace{h_1,\ldots,h_1}_s,\psi,\underbrace{h_2,\ldots,h_2}_t\rangle,\underbrace{h_2,\ldots,h_2}_q\rangle\\
-\sum_{p,q_1,q_2\geq 0}h\langle\underbrace{{h}_1,\ldots,{h}_1}_p,\psi,\underbrace{{h}_2,\ldots,{h}_2}_{q_1},\tau\circledcirc(1+ h_2),\underbrace{h_2,\ldots,h_2}_{q_1}\rangle\\
-\sum_{p,q,s,t\geq 0}\tau\langle\underbrace{h\overline\circledcirc{h}_1,\ldots,h\overline\circledcirc{h}_1}_p,\psi+h\langle\underbrace{h_1,\ldots,h_1}_s,\psi,\underbrace{h_2,\ldots,h_2}_t\rangle,\underbrace{h\overline\circledcirc{h}_2,\ldots,h\overline\circledcirc{h}_2}_q\rangle\\
+\sum_{p_1,p_2,q\geq 0}h\langle\underbrace{{h}_1,\ldots,{h}_1}_{p_1},d(h_1),\underbrace{h_1,\ldots,h_1}_{p_2},\psi,\underbrace{{h}_2,\ldots,{h}_2}_q\rangle\\
+\sum_{p,q\geq 0}h\langle\underbrace{h_1,\ldots,h_1}_p,d(\psi),\underbrace{h_2,\ldots,h_2}_q\rangle\\
-\sum_{p,q_1,q_2\geq 0}h\langle\underbrace{{h}_1,\ldots,{h}_1}_p,\psi,\underbrace{{h}_2,\ldots,{h}_2}_{q_1},d(h_2),\underbrace{h_2,\ldots,h_2}_{q_1}\rangle.
\end{multline*}
\medskip

\noindent Using equation $(3)$ and that $h_1,h_2\in\text{\normalfont Aut}_{\text{\normalfont Deligne}(A)}(\tau)$, we deduce
\medskip
\begin{multline*}
d(\psi')=h_1-h_2-\psi\langle\tau\rangle+\sum_{p_1,p_2,q\geq 0}h\langle\underbrace{{h}_1,\ldots,{h}_1}_{p_1},\tau+h_1\langle\tau\rangle,\underbrace{h_1,\ldots,h_1}_{p_2},\psi,\underbrace{{h}_2,\ldots,{h}_2}_q\rangle\\
-\sum_{p,q_1,q_2\geq 0}h\langle\underbrace{{h}_1,\ldots,{h}_1}_p,\psi,\underbrace{{h}_2,\ldots,{h}_2}_{q_1},\tau+h_2\langle\tau\rangle,\underbrace{h_2,\ldots,h_2}_{q_1}\rangle\\
+\sum_{p,q\geq 0}h\langle\underbrace{h_1,\ldots,h_1}_p,h_1-h_2-\psi\langle\tau\rangle,\underbrace{h_2,\ldots,h_2}_q\rangle\\
-\sum_{p,q,\geq 0}\tau\langle\underbrace{h\overline\circledcirc{h}_1,\ldots,h\overline\circledcirc{h}_1}_p,\psi',\underbrace{h\overline\circledcirc{h}_2,\ldots,h\overline\circledcirc{h}_2}_q\rangle.
\end{multline*}
\medskip

\noindent Since we have
\medskip
$$\sum_{p,q\geq 0}h\langle\underbrace{h_1,\ldots,h_1}_p,h_1-h_2,\underbrace{h_2,\ldots,h_2}_q\rangle=\sum_{p\geq 0}h\langle\underbrace{h_1,\ldots,h_1}_p\rangle-\sum_{r\geq 0}h\langle \underbrace{h_2,\ldots,h_2}_r\rangle,$$
\medskip

\noindent we finally obtain
\medskip
\begin{multline*}
d(\psi')=h\overline\circledcirc h_1-h\overline\circledcirc h_2-\psi\langle\tau\rangle+\sum_{p_1,p_2,q\geq 0}h\langle\underbrace{{h}_1,\ldots,{h}_1}_{p_1},\tau+h_1\langle\tau\rangle,\underbrace{h_1,\ldots,h_1}_{p_2},\psi,\underbrace{{h}_2,\ldots,{h}_2}_q\rangle\\
-\sum_{p,q_1,q_2\geq 0}h\langle\underbrace{{h}_1,\ldots,{h}_1}_p,\psi,\underbrace{{h}_2,\ldots,{h}_2}_{q_1},\tau+h_2\langle\tau\rangle,\underbrace{h_2,\ldots,h_2}_{q_1}\rangle\\
-\sum_{p,q\geq 0}h\langle\underbrace{h_1,\ldots,h_1}_p,\psi\langle\tau\rangle,\underbrace{h_2,\ldots,h_2}_q\rangle\\
-\sum_{p,q,\geq 0}\tau\langle\underbrace{h\overline\circledcirc{h}_1,\ldots,h\overline\circledcirc{h}_1}_p,\psi',\underbrace{h\overline\circledcirc{h}_2,\ldots,h\overline\circledcirc{h}_2}_q\rangle.
\end{multline*}
\medskip

\noindent We also have
\medskip
\begin{multline*}
    \psi'\langle\tau\rangle=\psi\langle\tau\rangle-\sum_{p_1,p_2,q\geq 0}h\langle\underbrace{h_1,\ldots,h_1}_{p_1},\tau+h_1\langle\tau\rangle,\underbrace{h_1,\ldots,h_1}_{p_2},\psi,\underbrace{h_2,\ldots,h_2}_q\rangle\\
    +\sum_{p,q\geq 0}h\langle\underbrace{h_1,\ldots,h_1}_p,\psi\langle\tau\rangle,\underbrace{h_2,\ldots,h_2}_q\rangle\\
    +\sum_{p,q_1,q_2\geq 0}h\langle\underbrace{h_1,\ldots,h_1}_p,\psi,\underbrace{h_2,\ldots,h_2}_{q_1},\tau+h_2\langle\tau\rangle,\underbrace{h_2,\ldots,h_2}_{q_2}\rangle.
\end{multline*}
\medskip

\noindent At the end, we obtain
$$d(\psi')+\psi'\langle\tau\rangle=h\overline\circledcirc h_1-h\overline\circledcirc h_2-\sum_{p,q,\geq 0}\tau\langle\underbrace{h\overline\circledcirc{h}_1,\ldots,h\overline\circledcirc{h}_1}_p,\psi',\underbrace{h\overline\circledcirc{h}_2,\ldots,h\overline\circledcirc{h}_2}_q\rangle$$

\noindent which proves that $h\overline\circledcirc{h}_1\sim_\tau h\overline\circledcirc{h}_2$. The lemma is proved.
\end{proof}

\begin{thm}
    Let $A$ be a complete brace algebra and $\tau\in\mathcal{MC}(A)$. Then
    $$\pi_1(\mathcal{MC}_\bullet(A),\tau)\simeq\text{\normalfont Aut}_{\text{\normalfont Deligne}(A)}(\tau)/\sim_\tau.$$
\end{thm}

\begin{proof}
    Recall that $\text{\normalfont Aut}_{\text{\normalfont Deligne(A)}}(\tau)=\{h\in A_0\ |\ d(h)=\tau+h\langle\tau\rangle-\tau\circledcirc(1+h)\}$. Let $h\in A_0$. By Lemma \ref{mc1}, we have that $h\in\text{\normalfont Aut}_{\text{\normalfont Deligne(A)}}(\tau)$ if and only if
    $$-\tau\otimes( \underline{0}^\vee+ \underline{1}^\vee)-h\otimes  \underline{01}^\vee\in\mathcal{MC}(A\otimes \Sigma N^*(\Delta^1)).$$

    \noindent We thus have a bijection
    $$f:\text{\normalfont Aut}_{\text{\normalfont Deligne(A)}}(\tau)\longrightarrow\mathcal{MC}_1(A)_\tau$$

    \noindent where we denote by $\mathcal{MC}_1(A)_\tau$ the subset of $\mathcal{MC}_1(A)$ given by elements whose $0$ and $1$ vertices are given by $ \tau$. Consider now $h,h'\in A_0$ such that
    $$d(h)=\tau+h\langle\tau\rangle-\tau\circledcirc(1+h);$$
    $$d(h')=\tau+h'\langle\tau\rangle-\tau\circledcirc(1+h').$$

    \noindent Let $\xi\in\mathcal{MC}_2(A)$ be such that $d_1\xi=f(h)$ and $d_2\xi=f(h')$. We write $\xi$ as
    $$\xi=-\tau\otimes ( \underline{0}^\vee+ \underline{1}^\vee+ \underline{2}^\vee)-h'\otimes  \underline{01}^\vee-h\otimes  \underline{02}^\vee+\psi\otimes  \underline{012}^\vee$$

    \noindent for some $\psi\in A_1$. We make precise the Maurer-Cartan condition on $\xi$. We first have
    $$d(\xi)=-d(\tau)\otimes ( \underline{0}^\vee+ \underline{1}^\vee+ \underline{2}^\vee)-d(h')\otimes  \underline{01}^\vee-d(h)\otimes  \underline{02}^\vee+(d(\psi)-h+h')\otimes  \underline{012}^\vee.$$

    \noindent By Lemma \ref{diff1}, we have
    \begin{multline*}
    \xi\llbrace\xi\rrbrace_1=-\tau\langle\tau\rangle\otimes ( \underline{0}^\vee+ \underline{1}^\vee+ \underline{2}^\vee)-(\tau\langle h'\rangle-h'\langle\tau\rangle)\otimes  \underline{01}^\vee\\
    -(\tau\langle h\rangle+h\langle\tau\rangle)\otimes  \underline{02}^\vee+(\tau\langle\psi\rangle+\psi\langle\tau\rangle)\otimes  \underline{012}^\vee.
    \end{multline*}

    \noindent Let $r\geq 2$. From the computations of $\partial^1$ and $\partial^2$ in Corollary \ref{calculdiff}, we deduce
    $$\xi\llbrace\xi\rrbrace_r=-\tau\langle\underbrace{h,\ldots,h}_r\rangle\otimes  \underline{01}^\vee-\tau\langle\underbrace{h',\ldots,h'}_r\rangle\otimes  \underline{02}^\vee+\sum_{p+q=r-1}\tau\langle\underbrace{h,\ldots,h}_p,\psi,\underbrace{h',\ldots,h'}_q\rangle\otimes  \underline{012}^\vee.$$
    
    \noindent We thus have proved that $\xi$ is a Maurer-Cartan element if and only if
    $$h-h'=d(\psi)+\psi\langle\tau\rangle+\sum_{p,q\geq 0}\tau\langle\underbrace{h,\ldots,h}_{p},\psi,\underbrace{h',\ldots,h'}_q\rangle.$$

    \noindent Equivalently, we have that $[f(h)]=[f(h')]$ if and only if $h\sim_\tau h'$. We thus have a well defined bijection
    $$\overline{f}:\text{\normalfont Aut}_{\text{\normalfont Deligne(A)}}(\tau)/\sim_\tau\longrightarrow\pi_1(\mathcal{MC}_\bullet(A),\tau),$$

    \noindent We now check that $\overline{f}$ is compatible with the group structures. Let $h,h'\in A_0$ be such that $\alpha=-\tau\otimes( \underline{0}^\vee+ \underline{1}^\vee)-h\otimes  \underline{01}^\vee$ and $\alpha'=-\tau\otimes( \underline{0}^\vee+ \underline{1}^\vee)-h'\otimes  \underline{01}^\vee$ are Maurer-Cartan elements in $\mathcal{MC}(A\otimes \Sigma N^*(\Delta^1))$. As we have seen before, by Lemma \ref{mc1}, it is equivalent to ask
    $$d(\mu)=\tau+h\langle\tau\rangle-\tau\circledcirc(1+h);$$
    $$d(h')=\tau+h'\langle\tau\rangle-\tau\circledcirc(1+h').$$
    
    \noindent By Corollary \ref{calculdiff}, we see that
    $$-\tau\otimes( \underline{0}^\vee+ \underline{1}^\vee+ \underline{2}^\vee)-h'\otimes  \underline{12}^\vee-(h\overline\circledcirc h')\otimes  \underline{02}^\vee-h\otimes  \underline{01}^\vee\in\mathcal{MC}(A\otimes \Sigma N^*(\Delta^2)).$$

    \noindent We then have
    $$[\alpha]\cdot[\alpha']=[-\tau\otimes( \underline{0}^\vee+ \underline{1}^\vee)-(h\overline\circledcirc h')\otimes  \underline{01}^\vee]$$

    \noindent in $\pi_1(\mathcal{MC}_\bullet(A),\tau)$, which gives
    $$\overline{f}([\alpha]\cdot[\alpha'])=h\overline\circledcirc h'=\overline{f}([\alpha])\circledcirc\overline{f}([\alpha']),$$

    \noindent showing that $\overline{f}$ is an isomorphism of groups.
\end{proof}

\subsubsection{The group $\pi_2(\mathcal{MC}_\bullet(A),\tau)$}

We now compute the group $\pi_2(\mathcal{MC}_\bullet(A),\tau)$. We begin by general lemmas that will also be useful for the computations of $\pi_n(\mathcal{MC}_\bullet(A),\tau)$ for $n\geq 3$.

\begin{lm}\label{degree}
    Let $T$ be a canonical tree with $|T|\geq 3$, and $n\geq 1$. If the first branch of $T$ has only one vertex, then there is no element of the form $\underline{0},\ldots,\underline{n}\in \Sigma^{-1}N_*(\Delta^n)$ among the non-root vertices in the tensor products produced by $T\otimes\Lambda\mu_T(\underline{0\cdots n})\in\mathcal{B}race^c(\Sigma ^{-1} N_*(\Delta^n))$.
\end{lm}

\begin{proof}
    For every finite set $E$ and $k\in E$, we denote by $\pi_{\{k\}}:\chi(E)\longrightarrow\chi(E\setminus\{k\})$ the morphism which forgets the element $k$. If a surjection has multiple occurrences of the element $k$, then its image by $\pi_{\{k\}}$ is $0$ by convention. Note that if $A$ and $B$ are disjoint finite sets, then for every $u\in\chi(A),v\in\chi(B)$, we have $\pi_{\{k\}}(u\cdot v)=\pi_{\{k\}}(u)\cdot\pi_{\{k\}}(v)$ in $\chi(A\sqcup B)$.\\

    Let $T$ be a canonical tree with $|T|\geq 3$. By Lemma \ref{surj}, there exists $u_{T}\in\chi(V_T\setminus\{1\})$ such that
    $$TR(\mu_T)=12\cdot u_{T}.$$

    \noindent We write uniquely $u_{T}$ as
    $$u_{T}=\sum_{i=1}^{m_T}\lambda_i^T u_{T}^i,$$

    \noindent \noindent where $\lambda_1^T,\ldots,\lambda_{m_T}^T\in\mathbb{K}$ and $u_{T}^1,\ldots,u_{T}^{m_T}$ are non degenerate surjections. We prove that, for every $1\leq i\leq m_T$ and $2\leq k\leq|T|$, 
    $$\pi_{\{k\}}(2\cdot u_{T}^i)=0.$$
    
    \noindent It is true for $k=2$, since $u_{T}^i\in\chi(2<\cdots <|T|)$ so that there are at least two occurences of $2$ in the surjection $2\cdot u_{T}^i$. Suppose now that $k\geq 3$. We prove the statement by induction on $|T|$. If $|T|=3$, the first tree of Example \ref{exmu} gives $TR(\mu_T)=\pm1232$ and $\pi_{\{3\}}(232)=22=0$. We now suppose that $|T|\geq 4$. By the proof of Lemma \ref{surj}, we have
    $$Tr(\mu_T)=-\sum_{S\subset T}\pm12\cdot\pi_2( TR(\mu_{T/S})\circ_STR(\mu_{S})).$$

    \noindent Let $S\subset T$ be such that $b_S=b_{T/S}=1$. Suppose that $|S|,|T/S|\neq 2$. By Lemma \ref{surj}, there exist $u_{S}\in\chi(V_S\setminus\{r(S)\})$ and $u_{T/S}\in\chi(V_{T/S}\setminus\{r(T/S)\})$ such that
    \begin{center}
        $\begin{array}{rll}
             TR(\mu_{S}) & = & r(S)p\cdot u_{S};\\
             TR(u_{T/S}) & = & r(T/S)q\cdot u_{T/S} 
        \end{array}$
    \end{center}

    \noindent where $p\in V_S$ and $q\in V_{T/S}$ are the second element of their respective totally ordered set. If $r(S)\neq 1$, then $r(T/S)=1$ and $q=2$, since $b_T=1$, so that
    $$TR(\mu_{T/S})\circ_S TR(\mu_{S})=12\cdot (u_{T/S}\circ_S (r(S)p\cdot u_{S}))$$

    \noindent whose associated term in the sum is $0$. Suppose now that $r(S)=1$. Then $r(T/S)=S$, so that
    $$TR(\mu_{T/S})\circ_S TR(\mu_{S})=1p\cdot u_{S}\cdot q\cdot u_{T/S}.$$
    
    \noindent We write $u_{S}$ and $u_{T/S}$ in the basis given by non degenerated surjections:
    $$TR(\mu_{T/S})\circ_S TR(\mu_{S})=\sum_{i=1}^{m_S}\sum_{j=1}^{m_{T/S}}\lambda_{i}^S\lambda_j^{T/S}1p\cdot u_{S}^i\cdot q\cdot u_{T/S}^j.$$
    
    \noindent Since $k\neq 1,2$, we have, for every $1\leq i\leq m_S$ and $1\leq j\leq m_{T/S}$,
    $$\pi_{\{k\}}(12\cdot\pi_2(1p\cdot u_{S}^i\cdot q\cdot u_{T/S}^j))=12\cdot \pi_{\{k\}}(p\cdot u_{S}^i)\cdot\pi_{\{k\}}(q\cdot u_{T/S}^j).$$
    
    \noindent By induction hypothesis (on $S$ if $k\in V_S$, on $T/S$ else), we obtain $0$. Suppose now $|S|=2$ and $|T/S|\neq 2$. By the same argument as before, we can restrict to the case $r(S)=1$ (which implies that $r(T/S)=S$), so that
    $$TR(\mu_{T/S})\circ_S TR(\mu_{S})=1pq\cdot u_{T/S}.$$

    \noindent where $p\in V_S$ and $q\in V_{T/S}$ are the second element of their respective totally ordered set. We have
    $$12\cdot\pi_2(TR(\mu_{T/S})\circ_S TR(\mu_{S}))=12pq\cdot u_{T/S}=\sum_{i=1}^{m_{T/S}}\lambda_i^{T/S}12pq\cdot u_{T/S}^i.$$

    \noindent Let $1\leq i\leq m_{T/S}$. If $p=2$, then the corresponding term in the sum is $0$. If $p\neq 2$, then $q=2$ so that we need to compute
    $$1\cdot \pi_{\{k\}}(2p2\cdot u_{T/S}^i).$$

    \noindent If $k=p$, then $\pi_{\{p\}}(2p2\cdot u_{T/S}^i)=22\cdot u_{T/S}^i=0$. If $k\neq p$, then $\pi_{\{k\}}(2p2\cdot u_{T/S}^i)=2p2\cdot\pi_{\{k\}}(u_{T/S}^i)$, which is $0$ by induction hypothesis on $T/S$. Suppose now that $|T/S|=2$ and $|S|\neq 2$. As before, we can suppose that $r(S)=1$ and $2\notin V_S$. We then have
    $$TR(\mu_{T/S})\circ_S TR(\mu_{S})=13\cdot u_{S}\cdot 2,$$

    \noindent which gives
    $$12\cdot\pi_2(TR(\mu_{T/S})\circ_S TR(\mu_{S}))=123\cdot u_{S}\cdot 2=\sum_{i=1}^{m_S}\lambda_i^S123\cdot u_{S}^i\cdot 2.$$

    \noindent Let $1\leq i\leq m_S$. Then 
    $$\pi_{\{k\}}(23\cdot u_{S}^i\cdot 2)=2\cdot\pi_{\{k\}}(3\cdot u_{S}^i)\cdot 2=0,$$

    \noindent by induction hypothesis on $S$. The case $|S|=|T/S|=2$ gives $|T|=3$ which has already be proved in the beginning of the proof.\\

    We thus have proved that $\pi_{\{k\}}(2\cdot u_{T}^i)=0$ for every canonical tree $T$ such that $|T|\geq 3$ and $2\leq k\leq |T|,1\leq i\leq m_T$. We now prove the lemma. Let $2\leq k\leq |T|$. By definition of the interval cut operations (see \cite[$\mathsection$2.2.1]{fresseberger}), the tensors with a factor of the form $\underline{0},\ldots,\underline{n}$ at position $k$ occurring in the expansion of $(T\otimes\Lambda\mu_T)( \underline{0\cdots n})$ are precisely produced by the surjections $12\cdot u_{T}^1,\ldots, 12\cdot u_{T}^{m_T}$ which contain only one occurrence of $k$. Let $1\leq i\leq m_T$ be such that $u_{T}^i$ contains only one occurrence of $k$. The tensors produced by $2\cdot u_{T}^i$ with a degree $-1$ element at position $k$ are given by the insertion of the appropriate degree $0$ vertex at position $k$ of the tensors produced by the surjection $\pi_{\{k\}}(2\cdot u_{T}^i)$. Since this surjection is $0$, the lemma is proved.
\end{proof}

\begin{lm}\label{bracedegree}
    Let $n\geq 2$. Let $a,b_1,\ldots,b_m\in A$, let $\underline{x},\underline{y_1},\ldots,\underline{y_m}\in N_*(\Delta^n)$ be basis elements and $r_1,\ldots,r_m\geq 0$. Suppose that $$|\underline{x}|+r_1|\underline{y_1}|+\cdots+r_m|\underline{y_m}|> n-2.$$

    \noindent Then $a\otimes  \underline{x}^\vee\llbrace b_1\otimes  \underline{y_1}^\vee,\ldots,b_m\otimes  \underline{y_m}^\vee\rrbrace_{r_1,\ldots,r_m}=0.$
\end{lm}

\begin{proof}
    Let $r=r_1+\cdots+r_m$. We more generally show that for every $\mu\in\Sigma^{-r}\mathcal{E}(r+1)_{r-1}$ and $\underline{z_1},\ldots,\underline{z_r}\in N_*(\Delta^n)$ such that $|\underline{x}|+|\underline{z_1}|+\cdots+|\underline{z_r}|>n-2$, the evaluation of $\mu$ on the tensor $ \underline{x}^\vee\otimes  \underline{z_1}^\vee\otimes\cdots\otimes  \underline{z_r}^\vee$ when using the $\mathcal{E}$-algebra structure of $N^*(\Delta^n)$ is $0$. One one hand, the evaluation of $\mu$ on the tensor $ \underline{x}^\vee\otimes  \underline{z_1}^\vee\otimes\cdots\otimes  \underline{z_r}^\vee$ is an element with degree $-1-|\underline{x}|-|\underline{z_1}|-\cdots-|\underline{z_r}|<1-n$. On the other hand, since the result is an element of $\Sigma N^*(\Delta^n)$, its degree is equal or greater than $| \underline{0\cdots n}^\vee|=1-n$. The evaluation of $\mu$ on the tensor $ \underline{x}^\vee\otimes  \underline{z_1}^\vee\otimes\cdots\otimes  \underline{z_r}^\vee$ must then be $0$.\\

    To obtain the lemma, we apply this result to $\mu=\Lambda\mu_T$ where $T\in\mathcal{PRT}(r+1)$ is a canonical tree, and $\underline{z_1},\ldots,\underline{z_r}=\underbrace{\underline{y_1},\ldots,\underline{y_1}}_{r_1},\ldots,\underbrace{\underline{y_m},\ldots,\underline{y_m}}_{r_m}$ up to a shuffle permutation in $Sh(r_1,\ldots,r_m)$.
\end{proof}

Before stating the next lemma, recall that if $A$ is a brace algebra and $\tau\in\mathcal{MC}(A)$, then we have a differential defined by
$$d_\tau(x)=d(x)+\tau\langle x\rangle-(-1)^{|x|} x\langle\tau\rangle.$$

\noindent We denote by $A^\tau$ the underlying dg $\mathbb{K}$-module.

\begin{lm}\label{mcn}
    Let $\tau\in\mathcal{MC}(A)$ and $n\geq 1$. We denote by $\mathcal{MC}_{n+1}(A)_\tau$ the set given by elements $\xi\in\mathcal{MC}_{n+1}(A)$ with faces given by $\tau$. Then we have a bijection
    $$f:Z_n(A^\tau)\longrightarrow\mathcal{MC}_{n+1}(A)_\tau$$

    \noindent given by
    $$f(h)=-\tau\otimes\left(\sum_{k=0}^{n+1} \underline{k}^\vee\right)-h\otimes  \underline{0\cdots (n+1)}^\vee.$$
\end{lm}

\begin{proof}
    Let $\xi\in\mathcal{MC}_{n+1}(A)_\tau$. Then there exists $h\in A_n$ such that
    $$\xi=-\tau\otimes \left(\sum_{k=0}^{n+1} \underline{k}^\vee\right)-h\otimes  \underline{0\cdots(n+1)}^\vee.$$

    \noindent We make precise the Maurer-Cartan condition on $\xi$. Let $p\geq 2$. By Lemma \ref{degree}, we have
    \begin{multline*}
        \xi\llbrace\xi\rrbrace_p=\sum_{k=0}^{n+1}-(-1)^{p}\tau\otimes  \underline{k}^\vee\llbrace h\otimes  \underline{0\cdots (n+1)}^\vee\rrbrace_p\\+(-1)^{p+1}h\otimes  \underline{0\cdots (n+1)}^\vee\llbrace h\otimes  \underline{0\cdots (n+1)}^\vee\rrbrace_p.
    \end{multline*}

    \noindent By Lemma \ref{bracedegree}, and since we have $n+np>-1+np>n-1$ because $p\geq 2$, we deduce that $\xi\llbrace\xi\rrbrace_p=0$. If $p=1$, then, by Corollary \ref{diff1},
    $$\xi\llbrace\xi\rrbrace_1=-\tau\langle\tau\rangle\otimes\left(\sum_{k=0}^{n+1} \underline{k}^\vee\right)-(\tau\langle h\rangle-(-1)^nh\langle\tau\rangle)\otimes  \underline{0\cdots (n+1)}^\vee.$$

    \noindent We also have
    $$d(\xi)=-d(\tau)\otimes\left(\sum_{k=0}^{n+1} \underline{k}^\vee\right)-d(h)\otimes  \underline{0\cdots (n+1)}^\vee.$$

    \noindent The Maurer-Cartan condition on $\xi$ is then equivalent to
    $$d(h)+\tau\langle h\rangle-(-1)^nh\langle\tau\rangle=0.$$
    
    \noindent which gives our desired bijection
    $$f:Z_n(A^\tau)\longrightarrow\mathcal{MC}(A\otimes \Sigma N^*(\Delta^{n+1}))_\tau.$$
\end{proof}

We now consider $n=2$. The computation of $\pi_2$ will emphasize a group structure on $H_1(A^\tau)$ given by the following lemma.

\begin{lm}\label{equivrel}
    Let $A$ be a complete brace algebra and $\tau\in\mathcal{MC}(A)$. Then $(H_1(A^\tau),\ast_\tau,0)$ is an abelian group with the product $\ast_\tau$ defined by
    $$[\mu]\ast_\tau[\mu']=[\mu+\mu'+\tau\langle \mu,\mu'\rangle].$$
\end{lm}

\begin{proof}
    We first prove that if $\mu,\mu'\in Z_1(A^\tau)$ then $\mu'':=\mu+\mu'+\tau\langle \mu,\mu'\rangle\in Z_1(A^\tau)$. We have
    \begin{multline*}
        d(\mu'')=d(\mu)+d(\mu')-\tau\langle\tau,\mu,\mu'\rangle+\tau\langle\mu,\tau,\mu'\rangle-\tau\langle\mu,\mu',\tau\rangle\\
        -\tau\langle\tau\langle\mu\rangle,\mu'\rangle+\tau\langle\mu,\tau\langle\mu'\rangle\rangle-\tau\langle\tau\langle\mu,\mu'\rangle\rangle-\tau\langle d(\mu),\mu'\rangle+\tau\langle\mu,d(\mu')\rangle.
    \end{multline*}

    \noindent We also have
    \begin{multline*}
        \mu''\langle\tau\rangle=\mu\langle\tau\rangle+\mu'\langle\tau\rangle+\tau\langle\tau,\mu,\mu'\rangle-\tau\langle\mu,\tau,\mu'\rangle+\tau\langle\mu,\mu'\langle\tau\rangle\rangle-\tau\langle\mu\langle\tau\rangle,\mu'\rangle+\tau\langle\mu,\mu'\langle\tau\rangle\rangle
    \end{multline*}

    \noindent and

    $$\tau\langle\mu''\rangle=\tau\langle\mu\rangle+\tau\langle\mu'\rangle+\tau\langle\tau\langle\mu,\mu'\rangle\rangle.$$

    \noindent At the end, we obtain that $d_\tau(\mu'')=0$, which proves that $\mu''\in Z_1(A^\tau)$. \\
    
    We now show that the product $\ast_\tau$ is well defined on $H_1(A^\tau)$. Let $\mu,\mu_1,\mu_2\in Z_1(A^\tau)$ and $\psi\in A_2$ be such that
    $$\mu_1-\mu_2=d(\psi)+\tau\langle\psi\rangle-\psi\langle \tau\rangle.$$

    \noindent Let $\psi':=\psi+\tau\langle \mu,\psi\rangle$. We show that
    $$\mu_1-\mu_2+\tau\langle \mu,\mu_1-\mu_2\rangle=d(\psi')+\tau\langle\psi'\rangle-\psi'\langle\tau\rangle.$$

    \noindent We first compute $d(\psi')$. We have
    \begin{multline*}
        d(\psi')=d(\psi)-\tau\langle\tau,\mu,\psi\rangle+\tau\langle\mu,\tau,\psi\rangle+\tau\langle\mu,\psi,\tau\rangle-\tau\langle\tau\langle\mu\rangle,\psi\rangle+\tau\langle\mu,\tau\langle\psi\rangle\rangle\\
        -\tau\langle\tau\langle\mu,\psi\rangle\rangle-\tau\langle d(\mu),\psi\rangle+\tau\langle\mu,d(\psi)\rangle.
    \end{multline*}

    \noindent We also have
    $$\psi'\langle\tau\rangle=\psi\langle\tau\rangle-\tau\langle\tau,\mu,\psi\rangle+\tau\langle\mu,\tau,\psi\rangle+\tau\langle\mu,\psi,\tau\rangle+\tau\langle\mu\langle\tau\rangle,\psi\rangle+\tau\langle\mu,\psi\langle\tau\rangle\rangle.$$

    \noindent At the end, we obtain
    $$d(\psi')+\tau\langle\psi'\rangle-\psi'\langle\tau\rangle=\mu_1-\mu_2+\tau\langle\mu,\mu_1-\mu_2\rangle$$

    \noindent so that
    $$[\mu+\mu_1+\tau\langle\mu,\mu_1\rangle]=[\mu+\mu_2+\tau\langle\mu,\mu_2\rangle].$$

    \noindent By the same computations with $\psi':=\psi-\tau\langle\psi,\mu\rangle$, we can show that
    $$[\mu+\mu_1+\tau\langle\mu_1,\mu\rangle]=[\mu+\mu_2+\tau\langle\mu_2,\mu\rangle].$$

    \noindent The product $\ast_\tau$ is thus well defined on $H_1(A^\tau)$. We now prove that it endows $H_1(A^\tau)$ with an abelian group structure. We prove the associativity of the operation $\ast_\tau$. We have
    \begin{center}
        $\begin{array}{rll}
             ([\mu]\ast_\tau[\mu'])\ast_\tau[\mu''] & = & [\mu+\mu'+\mu''+\tau\langle \mu,\mu'\rangle+\tau\langle \mu+\mu'+\tau\langle \mu,\mu'\rangle,\mu''\rangle];\\
             {}[\mu]\ast_\tau([\mu']\ast_\tau[\mu'']) & = & [\mu+\mu'+\mu''+\tau\langle \mu',\mu''\rangle+\tau\langle \mu,\mu'+\mu''+\tau\langle \mu',\mu''\rangle\rangle].\\
        \end{array}$
    \end{center}

    \noindent The difference between the two representatives is
    $$\tau\langle \mu,\tau\langle \mu',\mu''\rangle\rangle-\tau\langle\tau\langle \mu,\mu'\rangle,\mu''\rangle.$$

    \noindent We show that this element is the image of $\psi:=\tau\langle \mu,\mu',\mu''\rangle\in A_2$ under $d_\tau$. First, using that $d(\tau)=-\tau\langle\tau\rangle$ and the brace algebra structure on $A$, we have
    \begin{multline*}
        d(\tau)\langle\mu,\mu',\mu''\rangle=-\tau\langle\tau,\mu,\mu',\mu''\rangle+\tau\langle\mu,\tau,\mu',\mu''\rangle-\tau\langle\mu,\mu',\tau,\mu''\rangle+\tau\langle\mu,\mu',\mu'',\tau\rangle\\
        -\tau\langle\tau\langle\mu\rangle,\mu',\mu''\rangle+\tau\langle\mu,\tau\langle\mu'\rangle,\mu''\rangle-\tau\langle\mu,\mu',\tau\langle\mu''\rangle\rangle\\
        -\tau\langle\tau\langle\mu,\mu'\rangle,\mu''\rangle+\tau\langle\mu,\tau\langle\mu',\mu''\rangle\rangle-\tau\langle\tau\langle\mu,\mu',\mu''\rangle\rangle.
    \end{multline*}

    \noindent This gives
    \begin{multline*}
        d(\psi)=-\tau\langle\tau,\mu,\mu',\mu''\rangle+\tau\langle\mu,\tau,\mu',\mu''\rangle-\tau\langle\mu,\mu',\tau,\mu''\rangle+\tau\langle\mu,\mu',\mu'',\tau\rangle\\
        -\tau\langle\tau\langle\mu\rangle,\mu',\mu''\rangle+\tau\langle\mu,\tau\langle\mu'\rangle,\mu''\rangle-\tau\langle\mu,\mu',\tau\langle\mu''\rangle\rangle\\
        -\tau\langle\tau\langle\mu,\mu'\rangle,\mu''\rangle+\tau\langle\mu,\tau\langle\mu',\mu''\rangle\rangle-\tau\langle\tau\langle\mu,\mu',\mu''\rangle\rangle\\
        -\tau\langle d(\mu),\mu',\mu''\rangle+\tau\langle\mu,d(\mu'),\mu''\rangle-\tau\langle\mu,\mu',d(\mu'')\rangle
    \end{multline*}

    \noindent Using that $\mu,\mu',\mu''\in Z_1(A^\tau)$, we obtain
    \begin{multline*}
        d(\psi)=-\tau\langle\tau,\mu,\mu',\mu''\rangle+\tau\langle\mu,\tau,\mu',\mu''\rangle-\tau\langle\mu,\mu',\tau,\mu''\rangle+\tau\langle\mu,\mu',\mu'',\tau\rangle\\
        +\tau\langle\mu\langle\tau\rangle,\mu',\mu''\rangle-\tau\langle\mu,\mu'\langle\tau\rangle,\mu''\rangle+\tau\langle\mu,\mu',\mu''\langle\tau\rangle\rangle\\
        -\tau\langle\tau\langle\mu,\mu'\rangle,\mu''\rangle+\tau\langle\mu,\tau\langle\mu',\mu''\rangle\rangle-\tau\langle\tau\langle\mu,\mu',\mu''\rangle\rangle.
    \end{multline*}

    \noindent We also have
    \begin{multline*}
        \psi\langle\tau\rangle=-\tau\langle\tau,\mu,\mu',\mu''\rangle+\tau\langle\mu,\tau,\mu',\mu''\rangle-\tau\langle\mu,\mu',\tau,\mu''\rangle+\tau\langle\mu,\mu',\mu'',\tau\rangle\\
        +\tau\langle\mu\langle\tau\rangle,\mu',\mu''\rangle-\tau\langle\mu,\mu'\langle\tau\rangle,\mu''\rangle+\tau\langle\mu,\mu',\mu''\langle\tau\rangle\rangle.
    \end{multline*}

    \noindent which finally gives
    $$d(\psi)+\tau\langle\psi\rangle-\psi\langle\tau\rangle=\tau\langle\mu,\tau\langle\mu',\mu''\rangle\rangle-\tau\langle\tau\langle\mu,\mu'\rangle,\mu''\rangle$$
    
    \noindent so that we have the associativity.\\

    We now prove that every element $[\mu]$ has an inverse under $\ast_\tau$. We set $\mu'_0=-\mu$ and, for every $n\geq 0$,
    $$\mu'_{n+1}=-\mu-\tau\langle \mu,\mu'_n\rangle.$$

    \noindent We obtain a Cauchy sequence in $A$. Because $A$ is complete, this sequence has a limit denoted by $\mu'$ which satisfies
    $$\mu+\mu'+\tau\langle \mu,\mu'\rangle=0$$

    \noindent so that $[\mu']$ is the inverse of $[\mu]$ under $\ast_\tau$. We thus have proved that $\ast_\tau$ endows $H_1(A^\tau)$ with a group structure.\\
    
    We now prove that $\ast_\tau$ is abelian. Let $\mu,\mu'\in A_1$. We set $\psi:=\mu\langle\mu'\rangle$, and prove that
    $$d(\psi)+\tau\langle\psi\rangle-\psi\langle\tau\rangle=\tau\langle \mu',\mu\rangle-\tau\langle \mu,\mu'\rangle.$$

    \noindent We have
    \begin{center}
        $\begin{array}{lll}
            d(\psi) & = & d(\mu)\langle\mu'\rangle-\mu\langle d(\mu')\rangle \\
            & = & -\tau\langle\mu\rangle\langle\mu'\rangle-\mu\langle\tau\rangle\langle\mu'\rangle+\mu\langle\tau\langle\mu'\rangle\rangle+\mu\langle\mu'\langle\tau\rangle\rangle\\
            & = & -\tau\langle\mu\langle\mu'\rangle\rangle-\tau\langle\mu,\mu'\rangle+\tau\langle\mu',\mu\rangle-\mu\langle\tau\langle\mu'\rangle\rangle\\
            & & -\mu\langle\tau,\mu'\rangle+\mu\langle\mu',\tau\rangle+\mu\langle\tau\langle\mu'\rangle\rangle+\mu\langle\mu'\langle\tau\rangle\rangle\\
            & = & -\tau\langle\mu\langle\mu'\rangle\rangle-\tau\langle\mu,\mu'\rangle+\tau\langle\mu',\mu\rangle\\
            & & -\mu\langle\tau,\mu'\rangle+\mu\langle\mu',\tau\rangle+\mu\langle\mu'\langle\tau\rangle\rangle\\
        \end{array}$
    \end{center}

    \noindent and
    $$\psi\langle\tau\rangle=\mu\langle\mu',\tau\rangle-\mu\langle\tau,\mu'\rangle+\mu\langle\mu'\langle\tau\rangle\rangle.$$

    \noindent We then have
    $$d(\psi)+\tau\langle\psi\rangle-\psi\langle\tau\rangle=\tau\langle\mu',\mu\rangle-\tau\langle\mu,\mu'\rangle$$

    \noindent which proves that
    $$[\mu+\mu'+\tau\langle\mu,\mu'\rangle]=[\mu'+\mu+\tau\langle\mu',\mu\rangle].$$

    \noindent The operation $\ast_\tau$ is then commutative.\\

    The lemma is proved.
\end{proof}

\begin{thm}
    Let $A$ be a complete brace algebra and $\tau\in\mathcal{MC}(A)$. Then
    $$\pi_2(\mathcal{MC}_\bullet(A),\tau)\simeq (H_1(A^\tau),\ast_\tau,0).$$
\end{thm}

\begin{proof}
    By Lemma \ref{mcn}, we have a bijection
    $$f:Z_1(A^\tau)\longrightarrow\mathcal{MC}_2(A)_\tau.$$ 
    \noindent We consider its composite $\widetilde{f}:Z_1(A^{\tau})\longrightarrow\pi_{2}(\mathcal{MC}_\bullet(A),\tau)$ with the projection of $\mathcal{MC}_2(A)_\tau$ onto $\pi_2(\mathcal{MC}_\bullet(A),\tau)$. We show that $\widetilde{f}$ is compatible with the equivalence relation on $H_1(A^\tau)$ given by Lemma \ref{equivrel}. Let $\mu,\mu'\in Z_1(A^\tau)$ be such that there exists $\psi\in A_2$ with $\mu-\mu'=d_\tau(\psi)$. Namely,

    $$d(\psi)+\tau\langle\psi\rangle-\psi\langle\tau\rangle=\mu-\mu'.$$

    \noindent By Corollary \ref{diff1}, Corollary \ref{calculdiff} and Lemma \ref{bracedegree}, we have
    $$-\tau\otimes\left(\underline{0}^\vee+\underline{1}^\vee+\underline{2}^\vee\right)-\mu\otimes\underline{123}^\vee-\mu'\otimes \underline{023}^\vee+\psi\otimes\underline{0123}^\vee\in\mathcal{MC}(A\otimes \Sigma N^*(\Delta^{3})),$$

    \noindent which shows that $\widetilde{f}(\mu)=\widetilde{f}(\mu')$. We thus have a well defined map
    $$\overline{f}:H_1(A^\tau)\longrightarrow\pi_2(\mathcal{MC}_\bullet(A),\tau).$$
    
    \noindent We prove that $\overline{f}$ preserves the group structures. Let $\mu,\mu'\in Z_1(A^\tau)$. Recall that
    $$f(\mu)=-\tau\otimes( \underline{0}^\vee+ \underline{1}^\vee+ \underline{2}^\vee)-\mu\otimes  \underline{012}^\vee;$$
    $$f(\mu')=-\tau\otimes( \underline{0}^\vee+ \underline{1}^\vee+ \underline{2}^\vee)-\mu'\otimes  \underline{012}^\vee.$$

    \noindent We search for $\mu''\in A_1$ and $\psi\in A_2$ such that
    \begin{multline*}
        \omega:=-\tau\otimes( \underline{0}^\vee+ \underline{1}^\vee+ \underline{2}^\vee+ \underline{3}^\vee)-\mu'\otimes  \underline{123}^\vee\\-\mu''\otimes  \underline{023}^\vee-\mu\otimes  \underline{013}^\vee+\psi\otimes  \underline{0123}^\vee\in\mathcal{MC}(A\otimes\Sigma N^*(\Delta^3)).
    \end{multline*}

    \noindent We have
    $$d(\omega)=-d(\tau)\otimes( \underline{0}^\vee+ \underline{1}^\vee+ \underline{2}^\vee+ \underline{3}^\vee)-d(\mu')\otimes  \underline{123}^\vee-d(\mu'')\otimes  \underline{023}^\vee-d(\mu)\otimes  \underline{013}^\vee$$
    $$+(d(\psi)-\mu''+\mu+\mu')\otimes  \underline{0123}^\vee.$$

    \noindent By Corollary \ref{diff1}, we also have
    $$\omega\llbrace\omega\rrbrace_1=-\tau\langle\tau\rangle\otimes ( \underline{0}^\vee+ \underline{1}^\vee+ \underline{2}^\vee+ \underline{3}^\vee)-(\tau\langle\mu'\rangle+\mu'\langle\tau\rangle)\otimes  \underline{123}^\vee-(\tau\langle \mu''\rangle+\mu''\langle\tau\rangle)\otimes  \underline{023}^\vee$$
    $$-(\tau\langle\mu\rangle+\mu\langle\tau\rangle)\otimes  \underline{013}^\vee+(\tau\langle\psi\rangle-\psi\langle\tau\rangle)\otimes  \underline{0123}^\vee.$$

    \noindent By Corollary \ref{calculdiff}, we have
    $$\omega\llbrace\omega\rrbrace_2=\tau\langle\mu,\mu'\rangle\otimes  \underline{0123}^\vee.$$

    \noindent Finally, for every $p>2$ and by Lemma \ref{degree},
    \begin{center}
        $\begin{array}{lll}
            \omega\llbrace\omega\rrbrace_p & = & -\tau\otimes ( \underline{0}^\vee+ \underline{1}^\vee+ \underline{2}^\vee+ \underline{3}^\vee)\llbrace-\mu'\otimes  \underline{123}^\vee-\mu''\otimes  \underline{023}^\vee-\mu\otimes  \underline{013}^\vee+\psi\otimes  \underline{0123}^\vee\rrbrace_p\\
            & = & \displaystyle\sum_{s+t=p}\tau\otimes ( \underline{0}^\vee+ \underline{1}^\vee+ \underline{2}^\vee+ \underline{3}^\vee)\llbrace-\mu'\otimes  \underline{123}^\vee\\
            & & \ \ \ \ \ \ \ \ \ \ \ \ \ \ \  \ \ \ \ \ \ \ \ \ \ \ \ \ \ \ \ \ \ \ \ \ \ \ \ \ \ \ \ \ \ -\mu''\otimes  \underline{023}^\vee-\mu\otimes  \underline{013}^\vee,\psi\otimes  \underline{0123}^\vee\rrbrace_{s,t}.
        \end{array}$
    \end{center}

    \noindent Since $p>2$, for every $s,t\geq 0$ such that $s+t=p$, we have $2s+3t>p+2$. From Lemma \ref{bracedegree}, we deduce $\omega\llbrace\omega\rrbrace_p=0$. We then see that $\omega$ is a Maurer-Cartan element if and only if
    $$d_\tau(\psi)-\mu''+\mu+\mu'+\tau\langle \mu,\mu'\rangle=0.$$

    \noindent If we set $\psi=0$ and $\mu''=\mu+\mu'+\tau\langle \mu,\mu'\rangle$, this shows that
    $$[f(\mu)]\cdot [f(\mu')]=[-\tau\otimes( \underline{0}^\vee+ \underline{1}^\vee+ \underline{2}^\vee)-(\mu+\mu'+\tau\langle \mu,\mu'\rangle)\otimes  \underline{012}^\vee]$$

    \noindent in $\pi_2(\mathcal{MC}_\bullet( A),\tau)$. We thus have proved
    $$\overline{f}([\mu])\cdot\overline{f}([\mu'])=\overline{f}([\mu]\ast_\tau[\mu']).$$

    \noindent The morphism $\overline{f}$ is surjective, since $\widetilde{f}$ is bijective. It is also injective. Indeed, the equation $\overline{f}([\mu])=0$ is equivalent to $[\mu]=0$, according to the beginning of the proof of this theorem with $\mu'=0$. The map $\overline{f}$ is thus an isomorphism, which proves the theorem.
\end{proof}

\subsubsection{Computation of $\pi_n(\mathcal{MC}_\bullet(A),\tau)$ for $n\geq 3$}

We finally compute the groups $\pi_n(\mathcal{MC}_\bullet(A),\tau)$ for every $n\geq 3$.

\begin{thm}
    Let $A$ be a complete brace algebra and $\tau\in\mathcal{MC}(A)$. Then, for all $n\geq 3$, we have an isomorphism of groups
    $$\pi_{n+1}(\mathcal{MC}_\bullet(A),\tau)\simeq H_{n}(A^\tau).$$
\end{thm}

\begin{proof} 
By Lemma \ref{mcn}, we have a bijection $f:Z_n(A^\tau)\longrightarrow\mathcal{MC}_{n+1}(A)_\tau$. Consider its composite $\widetilde{f}:Z_n(A^{\tau})\longrightarrow\pi_{n+1}(\mathcal{MC}_\bullet(A),\tau)$ with the projection of $\mathcal{MC}_{n+1}(A)_\tau$ onto $\pi_{n+1}(\mathcal{MC}_\bullet(A),\tau)$. We show that $\widetilde{f}$ is a morphism of groups. Let $\mu,\mu'\in Z_n(A^\tau)$. We set
    $$\omega=-\tau\otimes\left(\sum_{k=0}^n \underline{k}^\vee\right)-\mu\otimes  \underline{0\cdots(n+1)}^\vee;$$
    $$\omega'=-\tau\otimes\left(\sum_{k=0}^n \underline{k}^\vee\right)-\mu'\otimes  \underline{0\cdots(n+1)}^\vee.$$

    \noindent We compute $[\omega]+[\omega']$ in $\pi_{n+1}(\mathcal{MC}_\bullet(A),\tau)$. This is equivalent to searching $\mu''\in Z_{n}(A^\tau)$ and $\psi\in A_{n+1}$ such that
    \begin{multline*}\xi:=-\tau\otimes\left(\sum_{k=0}^{n+2} \underline{k}^\vee\right)-\mu\otimes  \underline{1\cdots (n+2)}^\vee-\mu''\otimes  \underline{02\cdots(n+2)}^\vee\\
    -\mu'\otimes  \underline{013\cdots(n+2)}^\vee+\psi\otimes  \underline{0\cdots(n+2)}^\vee\in\mathcal{MC}(A\otimes \Sigma N^*(\Delta^{n+2})).\end{multline*}

    \noindent We make precise the Maurer-Cartan condition on $\xi$. We first compute $d(\xi)$. Note that we have, for every $0\leq k\leq n+2$,
    $$d(\underline{k}^\vee)=\sum_{k<j\leq n+2}\underline{kj}^\vee-\sum_{0\leq i<k}\underline{ik}^\vee,$$

    \noindent which implies
    $$\sum_{k=0}^{n+2}d(\underline{k}^\vee)=0$$

    \noindent by a variable substitution. We then have
    \begin{multline*}
        d(\xi)=-d(\tau)\otimes\left(\sum_{k=0}^{n+2} \underline{k}^\vee\right)-d(\mu)\otimes  \underline{1\cdots (n+2)}^\vee-d(\mu'')\otimes  \underline{02\cdots (n+2)}^\vee\\-d(\mu')\otimes  \underline{013\cdots (n+2)}^\vee
        +(d(\psi)+\mu''-\mu-\mu')\otimes  \underline{0\cdots (n+2)}^\vee.
    \end{multline*}

    \noindent We now compute $\xi\llbrace\xi\rrbrace_1$. By Corollary \ref{diff1}, we have
    \begin{multline*}
        \xi\llbrace\xi\rrbrace_1=-\tau\langle\tau\rangle\otimes\left(\sum_{k=0}^{n+2} \underline{k}^\vee\right)-(\tau\langle\mu\rangle-(-1)^n\mu\langle\tau\rangle)\otimes  \underline{1\cdots (n+2)}^\vee\\-(\tau\langle \mu''\rangle-(-1)^n\mu''\langle\tau\rangle)\otimes  \underline{02\cdots (n+2)}^\vee\\
        -(\tau\langle\mu'\rangle-(-1)^n\mu'\langle\tau\rangle)\otimes  \underline{013\cdots (n+2)}^\vee\\+(\tau\langle\psi\rangle-(-1)^{n+1}\psi\langle\tau\rangle)\otimes  \underline{0\cdots (n+2)}^\vee.
    \end{multline*}

    \noindent We now show that $\xi\llbrace\xi\rrbrace_p=0$ for every $p\geq 2$. By Lemma \ref{degree}, we have
    \begin{multline*}
        \xi\llbrace\xi\rrbrace_p=-\sum_{s+t=p}\tau\otimes\left(\sum_{k=0}^{n+2} \underline{k}^\vee\right)\llbrace\mu\otimes  \underline{1\cdots (n+2)}^\vee-\mu''\otimes  \underline{02\cdots (n+2)}^\vee\\
    -\mu'\otimes  \underline{013\cdots (n+2)}^\vee,\psi\otimes  \underline{0\cdots (n+2)}^\vee\rrbrace_{s,t}\\
        +\sum_{s+t=p}(\mu\otimes  \underline{1\cdots (n+2)}^\vee-\mu''\otimes  \underline{02\cdots (n+2)}^\vee-\mu'\otimes  \underline{013\cdots (n+2)}^\vee)\llbrace \mu\otimes  \underline{1\cdots (n+2)}^\vee\\
        -\mu''\otimes  \underline{02\cdots (n+2)}^\vee-\mu'\otimes  \underline{013\cdots (n+2)}^\vee,\psi\otimes  \underline{0\cdots (n+2)}^\vee\rrbrace_{s,t}\\
        +\sum_{s+t=p}\psi\otimes  \underline{0\cdots (n+2)}^\vee\llbrace\mu\otimes  \underline{1\cdots (n+2)}^\vee-\mu''\otimes  \underline{02\cdots (n+2)}^\vee\\
        -\mu'\otimes  \underline{013\cdots (n+2)}^\vee,\psi\otimes  \underline{0\cdots (n+2)}^\vee\rrbrace_{s,t}
    \end{multline*}
    
    \noindent Since $n\geq 3$ and $p\geq 2$, we can apply Lemma \ref{bracedegree} to obtain $\xi\llbrace\xi\rrbrace_p=0$. At the end, since $\mu,\mu',h\in Z_n(A^\tau)$, we have
    $$d(\xi)+\sum_{p\geq 1}\xi\llbrace\xi\rrbrace_p=(d(\psi)+\tau\langle\psi\rangle-(-1)^{n+1}\psi\langle\tau\rangle-\mu-\mu'+\mu'')\otimes  \underline{0\cdots (n+2)}^\vee.$$

    \noindent If we set $\mu'':=\mu+\mu'\in Z_n(A^\tau)$ and $\psi:=0$, we then obtain that $\xi\in\mathcal{MC}(A\otimes\Sigma N^*(\Delta^{n+2}))$. We thus have proved that
    $$[\omega]+[\omega']=\left[-\tau\otimes\left(\sum_{k=1}^{n+1} \underline{k}^\vee\right)-(\mu+\mu')\otimes  \underline{0\cdots(n+1)}^\vee\right],$$

    \noindent which gives $\widetilde{f}(\mu+\mu')=\widetilde{f}(\mu)+\widetilde{f}(\mu')$. Now, because $f$ is a bijection, we only need to prove that the kernel of $\widetilde{f}$ is exactly given by $d_\tau(A_{n+1})$. Let $\mu\in Z_n(A^\tau)$ and $\psi\in A_{n+1}$. By the previous computations, we see that the equation
    $$d(\psi)+\tau\langle\psi\rangle-(-1)^{n+1}\psi\langle\tau\rangle=\mu$$

    \noindent is equivalent to the assumption
    $$-\tau\otimes\left(\sum_{k=0}^{n+2}\underline{k}^\vee\right)-\mu\otimes\underline{1\cdots (n+2)}^\vee+\psi\otimes\underline{0\cdots(n+2)}^\vee\in\mathcal{MC}(A\otimes \Sigma N^*(\Delta^{n+2})),$$

    \noindent which shows that $\widetilde{f}(\mu)=0$ if and only if $\mu=d_\tau(\psi)$ for some $\psi\in A_{n+1}$. We thus have an isomorphism
    $$\overline{f}:H_n(A^{\tau})\overset{\simeq}{\longrightarrow}\pi_{n+1}(\mathcal{MC}_\bullet(A),\tau).$$
\end{proof}

\subsection{Remarks: interpretation of the low dimensional twisting coderivations}\label{sec:242b}

In this subsection, we give an interpretation of the differentials $\partial^0,\partial^1,\partial^2$ and $\partial^3$ computed in Lemma \ref{firstdiff} and Corollary \ref{calculdiff}. This interpretation will be obtained by the study of the first simplices associated to the Maurer-Cartan simplicial set of $\text{\normalfont Hom}(\Lambda^{-1}\mathcal{A}s^\vee,\text{\normalfont End}_A)$ for some $A\in\text{\normalfont dgMod}_\mathbb{K}$.\\

Recall that for every non-symmetric cooperad $\mathcal{C}$ and non-symmetric operad $\mathcal{P}$ such that $\mathcal{C}(0)=\mathcal{P}(0)=0$, the sequence $\text{\normalfont Hom}(\mathcal{C},\mathcal{P})$ is endowed with the structure of an operad such that, for every $f\in\text{\normalfont Hom}(\mathcal{C},\mathcal{P})(k),g_1\in\text{\normalfont Hom}(\mathcal{C},\mathcal{P})(i_1),\ldots,g_k\in\text{\normalfont Hom}(\mathcal{C},\mathcal{P})(i_k)$ with $n=i_1+\cdots+i_k$, the composition $\gamma(f\otimes g_1\otimes\cdots\otimes g_k)$ is given by the composite
\[\begin{tikzcd}
	{\mathcal{C}(n)} & {\mathcal{C}\circ\mathcal{C}(n)} & {\mathcal{C}(k)\otimes\mathcal{C}(i_{1})\otimes \cdots\otimes\mathcal{C}(i_{k})} \\
	&& {\mathcal{P}(k)\otimes \mathcal{P}(i_{1})\otimes \cdots\otimes \mathcal{P}(i_{k})} & {\mathcal{P}\circ\mathcal{P}(n)} & {\mathcal{P}(n)}
	\arrow["\Delta", from=1-1, to=1-2]
	\arrow[two heads, from=1-2, to=1-3]
	\arrow["{f\otimes g_{1}\otimes\cdots\otimes g_{k}}", from=1-3, to=2-3]
	\arrow[hook, from=2-3, to=2-4]
	\arrow["\gamma", from=2-4, to=2-5]
\end{tikzcd}.\]

\noindent From \cite[Proposition 1]{gerstenhaber2}, we deduce that $\bigoplus_{n\geq 1}\text{\normalfont Hom}(\mathcal{C}(n),\mathcal{P}(n))$ is endowed with the structure of a brace algebra. The braces are given by
$$f\langle g_1,\ldots,g_n\rangle=\sum_{1\leq i_1<\cdots <i_n\leq r}\gamma(f\otimes id\otimes\cdots\otimes \underset{i_1}{g_1}\otimes\cdots\otimes\underset{i_n}{g_n}\otimes\cdots\otimes id)$$

\noindent where $f\in\text{\normalfont Hom}(\mathcal{C}(r),\mathcal{P}(r)),g_1\in\text{\normalfont Hom}(\mathcal{C}(m_1),\mathcal{P}(m_1)),\ldots,g_n\in\text{\normalfont Hom}(\mathcal{C}(m_n),\mathcal{P}(m_n))$. We immediatly see that $\bigoplus_{n\geq 2}\text{\normalfont Hom}(\mathcal{C}(n),\mathcal{P}(n))$ is a sub brace algebra of $\bigoplus_{n\geq 1}\text{\normalfont Hom}(\mathcal{C}(n),\mathcal{P}(n))$. Since $\prod_{n\geq 2}\text{\normalfont Hom}(\mathcal{C}(n),\mathcal{P}(n))$ is the completion of $\bigoplus_{n\geq 2}\text{\normalfont Hom}(\mathcal{C}(n),\mathcal{P}(n))$ under the filtration defined by
$$F_p(\bigoplus_{n\geq 2}\text{\normalfont Hom}(\mathcal{C}(n),\mathcal{P}(n))):=\bigoplus_{n\geq p+1}\text{\normalfont Hom}(\mathcal{C}(n),\mathcal{P}(n)),$$

\noindent we have that the above brace algebra structure on $\bigoplus_{n\geq 1}\text{\normalfont Hom}(\mathcal{C}(n),\mathcal{P}(n))$ induces a complete brace algebra structure on $\prod_{n\geq 2}\text{\normalfont Hom}(\mathcal{C}(n),\mathcal{P}(n))$.\\

We now consider the non-symmetric operad $\mathcal{A}s$ such that $\mathcal{A}s(0)=0$ and $\mathcal{A}s(n)=\mathbb{K}$ for every $n\geq 1$ with trivial operadic compositions. Since $\mathcal{A}s$ is self-dual for Koszul duality (see for instance \cite[Proposition 9.1.9]{loday}), the operad $\mathcal{A}s_\infty=B^c(\Lambda^{-1}\mathcal{A}s^\vee)$ encodes associative algebras up to homotopy. We apply the above analysis with $\mathcal{C}=\Lambda^{-1}\mathcal{A}s^\vee$ and $\mathcal{P}=\text{\normalfont End}_A$ for some $A\in\text{\normalfont dgMod}_\mathbb{K}$ in order to study morphisms from $\mathcal{A}s_\infty$ to $\text{\normalfont End}_A$, or equivalently associative up to homotopy algebra structures on $A$. Note that we have an isomorphism of operads
$$\text{\normalfont Hom}(\Lambda^{-1}\mathcal{A}s^\vee,\text{\normalfont End}_A)\simeq\text{\normalfont End}_{\Sigma A}.$$

\noindent We set
$$\overline B(A)=\bigoplus_{n\geq 1}(\Sigma A)^{\otimes n}\ ;\ B_{\geq 2}(A)=\bigoplus_{n\geq 2}(\Sigma A)^{\otimes n}$$

\noindent so that $\overline B(A)=\Sigma A\oplus B_{\geq 2}(A)$. Let $d$ be the differential of $\overline B(A)$ obtained from the internal differential of $A$ by the Leibniz rule. Recall that $\overline B(A)$ is a coalgebra with as coproduct
$$\Delta(a_1\otimes\cdots\otimes a_n)=\sum_{k=1}^{n-1} (a_1\otimes\cdots\otimes a_k)\otimes (a_{k+1}\otimes\cdots\otimes a_n)$$

\noindent for every $n\geq 2$ and $a_1,\ldots,a_n\in\Sigma A$. The above isomorphism of operads provides a complete brace algebra structure on $\text{\normalfont Hom}(B_{\geq 2}(A),\Sigma A)\simeq\prod_{n\geq 2}\text{\normalfont End}_{\Sigma A}(n)$. Note that we have the isomorphism $\text{\normalfont Hom}(\overline B(A),\Sigma A)\simeq\text{\normalfont Hom}(\Sigma A,\Sigma A)\oplus\text{\normalfont Hom}(B_{\geq 2}(A),\Sigma A)$. In the following, we denote by $1\in\text{\normalfont Hom}(\Sigma A,\Sigma A)$ the identity morphism so that we have a natural inclusion $\mathbb{K}1\oplus\text{\normalfont Hom}(B_{\geq 2}(A),\Sigma A)\subset\text{\normalfont Hom}(\overline B(A),\Sigma A)$.

\begin{prop}
    Giving a Maurer-Cartan element $\phi\in\mathcal{MC}(\text{\normalfont Hom}(B_{\geq 2}(A),\Sigma A))$ is equivalent to giving a coderivation of coalgebra of the form $d+\partial_\phi$ on $\overline B(A)$, where $\partial_\phi$ is the morphism obtained from $\phi$ by the Leibniz rule in the coalgebra $\overline B(A)$.
\end{prop}

\begin{proof}
    Let $\phi\in\text{\normalfont Hom}(B_{\geq 2}(A),\Sigma A)$ be a degree $-1$ morphism. Then $(d+\partial_\phi)^2=0$ if and only if $d(\phi)+\phi\partial_\phi=0$. By definition of $\partial_\phi$, we have, for every $a_1,\ldots,a_n\in \Sigma A$, 
    $$\partial_\phi(a_1\otimes\cdots\otimes a_n)=\sum_{i=1}^n\sum_{j=1}^{n-i}\pm a_1\otimes\cdots\otimes a_{i-1}\otimes\phi(a_i\otimes\cdots\otimes a_{i+j})\otimes a_{i+j+1}\otimes\cdots\otimes a_n,$$
    \noindent which gives $\phi\partial_\phi=\phi\langle\phi\rangle$. We thus have obtained that $d+\partial_\phi$ is a derivation of coalgebra if and only if $\phi\in\mathcal{MC}(\text{\normalfont Hom}(\overline B(A),\Sigma A))$.
\end{proof}

Since giving a morphism of operads $\mathcal{A}s_\infty\longrightarrow\text{\normalfont End}_A$ is equivalent to giving a Maurer-Cartan element in $\text{\normalfont Hom}(B_{\geq 2}(A),\Sigma A)$, we have the following classical definition.

\begin{defi}
    An {\normalfont associative algebra up to homotopy} is a pair $(A,\phi)$ where $A$ is a dg $\mathbb{K}$-module and $\phi\in\mathcal{MC}(\text{\normalfont Hom}(B_{\geq 2}(A),\Sigma A))$.
\end{defi}

For every $\phi\in\mathcal{MC}(\text{\normalfont Hom}(B_{\geq 2}(A),\Sigma A))$, we denote by $\overline B(A,\phi)$ the dg $\mathbb{K}$-module $\overline B(A)$ endowed with the coderivation $d+\partial_\phi$.

\begin{defi}
    Let $(A_1,\phi_1)$ and $(A_2,\phi_2)$ be two associative algebras up to homotopy. An $\infty$-morphism $f:(A_1,\phi_1)\rightarrow (A_2,\phi_2)$ is a morphism of coalgebras $f:\overline B(A,\phi_1)\longrightarrow \overline B(A,\phi_2)$ which commutes with the coderivations.
\end{defi}

In the following, we consider the category of associative algebras up to homotopy with set morphisms the $\infty$-morphisms.

\begin{remarque}
    Note that since $B(A_2)$ is cofree, giving a morphism of coalgebras $\overline B(A_1)\longrightarrow \overline B(A_2)$ is equivalent to giving a morphism $\overline B(A_1)\longrightarrow\Sigma A_2$. 
\end{remarque}

\begin{prop}\label{mc1ass}
    Let $\phi_0,\phi_1\in\mathcal{MC}(\text{\normalfont Hom}(\overline{ \overline B}(A),\Sigma A))$. Then giving
    $$-\phi_0\otimes\underline{0}^\vee-\phi_1\otimes\underline{1}^\vee-\phi_{01}\otimes\underline{01}^\vee\in\mathcal{MC}_1(\text{\normalfont Hom}(B_{\geq 2}(A),\Sigma A))$$

    \noindent is equivalent to giving a morphism of coalgebras
    $$\Phi_{01}:\overline B(A,\phi_1)\longrightarrow \overline B(A,\phi_0)$$

    \noindent which is the identity on $\Sigma A\subset \overline B(A)$.
\end{prop}

\begin{proof}
    Let $\omega:=-\phi_0\otimes\underline{0}^\vee-\phi_1\otimes\underline{1}^\vee-\phi_{01}\otimes\underline{01}^\vee\in\text{\normalfont Hom}(B_{\geq 2}(A),\Sigma A)\otimes\Sigma N^*(\Delta^1)$. Let $\Phi_{01}:\overline B(A)\longrightarrow \overline B(A)$ be the unique morphism of coalgebras such that its composite with the projection $\pi_{\Sigma A}:\overline B(A)\longrightarrow\Sigma A$ is $1+\phi_{01}$. We characterize the equation
    $$(d+\partial_{\phi_0})\Phi_{01}=\Phi_{01}(d+\partial_{\phi_1}).$$

    \noindent Since $\overline B(A)$ is cofree, this identity is equivalent to 
    $$\pi_{\Sigma A}(d+\partial_{\phi_0})\Phi_{01}=\pi_{\Sigma A}\Phi_{01}(d+\partial_{\phi_1}),$$

    \noindent and then to
    $$d(\phi_{01})=\phi_{1}+\phi_{01}\partial_{\phi_1}-\phi_0\Phi_{01}.$$

    \noindent We precisely have $\phi_{01}\partial_{\phi_1}=\phi_{01}\langle\phi_1\rangle$ and $\phi_0\Phi_{01}=\phi_0\circledcirc(1+\phi_{01})$ by definition of $\partial_{\phi_1}$ and $\Phi_{01}$. We thus have obtained that $\Phi_{01}$ commutes with the differentials if and only if 
    $$d(\phi_{01})=\phi_1+\phi_{01}\langle\phi_1\rangle-\phi_0\circledcirc(1+\phi_{01}).$$

    \noindent By Lemma \ref{mc1}, this identity is equivalent to ask $\omega\in\mathcal{MC}_1(\text{\normalfont Hom}(B_{\geq 2}(A),\Sigma A))$, which proves the proposition.
\end{proof}


We now characterize elements of $\mathcal{MC}_2(\text{\normalfont Hom}(B_{\geq 2}(A),\Sigma A))$. First, note that for every associative algebra $E$, every $\phi\in\mathcal{MC}(\text{\normalfont Hom}(B_{\geq 2}(A),\Sigma A))$ induces an element in $\mathcal{MC}(\text{\normalfont Hom}(B_{\geq 2}(A\otimes E),\Sigma(A\otimes E)))$, which we still denote by $\phi$, and which is defined by applying $\phi$ on the left, and the algebra structure of $E$ on the right. In particular, for every morphism of associative algebras $f:E\longrightarrow E'$, we have an $\infty$-morphism $id\otimes f:(A\otimes E,\phi)\longrightarrow (A\otimes E',\phi)$.\\

\noindent Next, recall that, for every $n\geq 0$, the dg $\mathbb{K}$-module $N^*(\Delta^n)$ is endowed with the structure of an associative algebra. This associative algebra structure is obtained by the dualization of the coassociative coalgebra structure on $N_*(\Delta^n)$ given by the Alexander-Whitney diagonal $AW:N_*(\Delta^n)\longrightarrow N_*(\Delta^n)\otimes N_*(\Delta^n)$ which is the operation given by the permutation $(12)\in\mathcal{E}(2)_0$. Explicitly, we have
$$AW(\underline{a_0\cdots a_d})=\sum_{k=0}^d\underline{a_0\cdots a_k}\otimes\underline{a_k\cdots a_d},$$

\noindent for every $0\leq a_0<\cdots<a_d\leq n$.

\begin{prop}\label{mcass2}
    Giving a Maurer-Cartan element in $\mathcal{MC}_2(\text{\normalfont Hom}(B_{\geq 2}(A),\Sigma A))$ is equivalent to giving Maurer-Cartan elements $\phi_0,\phi_1,\phi_2\in\mathcal{MC}(\text{\normalfont Hom}(B_{\geq 2}(A),\Sigma A))$ and a diagram of the form
\[\begin{tikzcd}
	& {(A,\phi_1)} \\
	{(A,\phi_2)} && {(A,\phi_0)}
	\arrow["{\Phi_{01}}", from=1-2, to=2-3]
	\arrow["{\Phi_{12}}", from=2-1, to=1-2]
	\arrow[""{name=0, anchor=center, inner sep=0}, "{\Phi_{02}}"', from=2-1, to=2-3]
	\arrow["{\Phi_{012}}"', shorten >=3pt, Rightarrow, from=1-2, to=0]
\end{tikzcd}\]

\noindent in the category of $A_\infty$-algebras, where $\Phi_{012}:(A,\phi_2)\longrightarrow (A\otimes N^*(\Delta^1),\phi_{0})$ is a homotopy 
from $\Phi_{01}\Phi_{12}$ to $\Phi_{02}$.
\end{prop}

\begin{proof}
    We consider
    $$\omega:=-\phi_0\otimes\underline{0}^\vee-\phi_1\otimes\underline{1}^\vee-\phi_2\otimes\underline{2}^\vee-\phi_{01}\otimes\underline{01}^\vee-\phi_{02}\otimes\underline{02}^\vee-\phi_{12}\otimes\underline{02}^\vee-\phi_{012}\otimes\underline{012}^\vee.$$

    \noindent We characterize the Maurer-Cartan condition on $\omega$. By definition of the $\Gamma\Lambda\mathcal{PL}_\infty$-algebra structure on $\Sigma\text{\normalfont Hom}(B_{\geq 2}(A),\Sigma A)\otimes N^*(\Delta^2)$, and by Corollary \ref{calculdiff}, looking at the vertices of $d(\omega)+\sum_{n\geq 1}\omega\llbrace\omega\rrbrace_n$ gives the Maurer-Cartan condition on $\phi_0,\phi_1,\phi_2\in\text{\normalfont Hom}(B_{\geq 2}(A),\Sigma A)$. Looking at the components given by $\underline{01}^\vee,\underline{02}^\vee$ and $\underline{12}^\vee$ also give the Maurer-Cartan condition on the elements
    $$-\phi_0\otimes\underline{0}^\vee-\phi_1\otimes\underline{1}^\vee-\phi_{01}\otimes\underline{01}^\vee,$$
    $$-\phi_0\otimes\underline{0}^\vee-\phi_2\otimes\underline{1}^\vee-\phi_{02}\otimes\underline{01}^\vee,$$
    $$-\phi_1\otimes\underline{0}^\vee-\phi_2\otimes\underline{1}^\vee-\phi_{12}\otimes\underline{01}^\vee.$$

    \noindent In particular, by Proposition \ref{mc1ass}, such datas are equivalent to giving three $\infty$-morphisms $\Phi_{01}:(A,\phi_1)\longrightarrow (A,\phi_0),\Phi_{02}:(A,\phi_2)\longrightarrow  (A,\phi_0)$ and $\Phi_{12}:(A,\phi_2)\longrightarrow  (A,\phi_1)$ which reduce to the identity on $\Sigma A$. We now analyze the $\underline{012}^\vee$ component of $d(\omega)+\sum_{n\geq 1}\omega\llbrace\omega\rrbrace_n$. By Corollary \ref{calculdiff}, the Maurer-Cartan condition on $\omega$ gives, when looking at the $\underline{012}^\vee$ component,
    $$d(\phi_{012})-\phi_{01}\overline\circledcirc\phi_{12}+\phi_{02}+\phi_{012}\langle\phi_2\rangle+\sum_{i,j\geq 0}\phi_0\langle\underbrace{\phi_{02},\ldots,\phi_{02}}_i,\phi_{012},\underbrace{\phi_{01}\overline\circledcirc\phi_{12},\ldots,\phi_{01}\overline\circledcirc\phi_{12}}_j\rangle=0.$$

    \noindent Now let $\Phi_{012}:\overline B(A)\longrightarrow \overline B(A\otimes N^*(\Delta^1))$ be the unique morphism of coalgebras such that its composite with the projection on $\Sigma A\otimes N^*(\Delta^1)$ is
    $$(1+\phi_{02})\otimes\underline{0}^\vee+\phi_{012}\otimes\underline{01}^\vee+(1+\phi_{01}\overline\circledcirc\phi_{12})\otimes\underline{1}^\vee.$$

    \noindent We characterize the equation
    $$\pi_{\Sigma A\otimes N^*(\Delta^1)}(d+\partial_{\phi_0})\Phi_{012}=\pi_{\Sigma A\otimes N^*(\Delta^1)}\Phi_{012}(d+\partial_{\phi_2}).$$

    \noindent On one hand, we have
    \begin{multline*}\pi_{\Sigma A\otimes N^*(\Delta^1)}(d+\partial_{\phi_0})\Phi_{012}=(d_{\Sigma A}+d\phi_{02}+\phi_0\circledcirc(1+\phi_{02}))\otimes\underline{0}^\vee\\+(d_{\Sigma A}+d\phi_{01}\overline\circledcirc\phi_{12}+\phi_0\circledcirc(1+\phi_{01}\overline\circledcirc\phi_{12}))\otimes\underline{1}^\vee\\ +(d\phi_{012}-\phi_{01}\overline\circledcirc\phi_{12}+\phi_{02}+\sum_{i,j\geq 0}\phi_0\langle\underbrace{\phi_{02},\ldots,\phi_{02}}_i,\phi_{012},\underbrace{\phi_{01}\overline\circledcirc\phi_{12},\ldots,\phi_{01}\overline\circledcirc\phi_{12}}_j\rangle)\otimes\underline{01}^\vee.\end{multline*}

    \noindent On the other hand, we have
    \begin{multline*}
        \pi_{\Sigma A\otimes N^*(\Delta^1)}\Phi_{012}(d+\partial_{\phi_2})=(d_{\Sigma A}+\phi_{02}d+\phi_2+\phi_{02}\langle\phi_2\rangle)\otimes\underline{0}^\vee\\
        +(d_{\Sigma A}+\phi_{01}\overline\circledcirc\phi_{12}d+(\phi_{01}\overline\circledcirc\phi_{12})\langle\phi_2\rangle)\otimes\underline{1}^\vee\\
        -(\phi_{012}d+\phi_{012}\langle\phi_2\rangle)\otimes\underline{01}^\vee,
    \end{multline*}

\noindent which proves the proposition.
\end{proof}

We now characterize $\mathcal{MC}_3(\text{\normalfont Hom}(B_{\geq 2}(A),\Sigma A))$. We first show how to compose homotopies from $(A,\phi)$ to $(A\otimes N^*(\Delta^1),\phi')$ for some Maurer-Cartan elements $\phi,\phi'\in\mathcal{MC}(\text{\normalfont Hom}(B_{\geq 2}(A),\Sigma A))$. Let $f,g,h:(A,\phi)\longrightarrow (A,\phi')$. Let $H_1:(A,\phi)\longrightarrow(A\otimes N^*(\Delta^1),\phi')$ be a homotopy from $f$ to $g$, and $H_2:(A,\phi)\longrightarrow(A\otimes N^*(\Delta^1),\phi')$ be a homotopy from $g$ to $h$. We consider the pullback
\[\begin{tikzcd}
	{N^*(\Delta^1)\underset{\mathbb{K}}{\times}N^*(\Delta^1)} && {N^*(\Delta^1)} \\
	\\
	{N^*(\Delta^1)} && {\mathbb{K}}
	\arrow["{\pi_1}", dashed, from=1-1, to=1-3]
	\arrow["{\pi_2}"', dashed, from=1-1, to=3-1]
	\arrow["{d_0}", from=1-3, to=3-3]
	\arrow[""{name=0, anchor=center, inner sep=0}, "{d_1}"', from=3-1, to=3-3]
	\arrow["\lrcorner"{anchor=center, pos=0.125}, draw=none, from=1-1, to=0]
\end{tikzcd},\]

\noindent where we identify $N^*(\Delta^0)$ with $\mathbb{K}$. Explicitly, we have $N^*(\Delta^1)\underset{\mathbb{K}}{\times}N^*(\Delta^1)=(N^*(\Delta^1)\times N^*(\Delta^1))/((\underline{1}^\vee,0)\sim (0,\underline{0}^\vee))$. One can see that the algebra structure of $N^*(\Delta^1)\times N^*(\Delta^1)$ preserves the equivalence relation $\sim$ so that $A\otimes(N^*(\Delta^1)\underset{\mathbb{K}}{\times}N^*(\Delta^1))$ is a path object for $A$ in the category of $A_\infty$-algebras. We thus obtain a homotopy $H:=H_2\times H_1:(A,\phi)\rightarrow (A\otimes (N^*(\Delta^1)\underset{\mathbb{K}}{\times}N^*(\Delta^1)),\phi')$ from $f$ to $h$.\\

Now let $G_1,G_2:(A,\phi)\rightarrow (A\otimes(N^*(\Delta^1)\underset{\mathbb{K}}{\times}N^*(\Delta^1)),\phi')$ be two homotopies from $f$ to $h$ obtained as above. In the next proposition, we use a particular way to compose $G_1$ with $G_2$. This composition is defined as follows. Let $N_\square^*(\square^2)=N^*(\Delta^1)\otimes N^*(\Delta^1)$ and $N_\square^*(\partial\square^2)=N_\square^*(\square^2)/(\mathbb{K}\cdot \underline{01}^\vee\otimes\underline{01}^\vee)$. We consider the morphisms of algebras $\Gamma_1,\Gamma_2:N_\square^*(\partial\square^2)\longrightarrow N^*(\Delta^1)\underset{\mathbb{K}}{\times}N^*(\Delta^1)$ defined by

$$\begin{array}{ccc}
    \begin{array}{cccc}
        \Gamma_1: &  N_\square^*(\partial\square^2) & \longrightarrow &  N^*(\Delta^1)\underset{\mathbb{K}}{\times}N^*(\Delta^1)\\
        & \underline{0}^\vee\otimes\underline{0}^\vee & \longmapsto & (\underline{0}^\vee,0)\\
        & \underline{0}^\vee\otimes\underline{1}^\vee & \longmapsto & (\underline{1}^\vee,0)\\
       & \underline{1}^\vee\otimes\underline{1}^\vee  & \longmapsto & (0,\underline{1}^\vee)\\
       & \underline{0}^\vee\otimes\underline{01}^\vee & \longmapsto & (\underline{01}^\vee,0)\\
       & \underline{01}^\vee\otimes\underline{1}^\vee  & \longmapsto & (0,\underline{01}^\vee)
    \end{array} & \text{\normalfont and} & \begin{array}{cccc}
        \Gamma_2: & N_\square^*(\partial\square^2) & \longrightarrow &  N^*(\Delta^1)\underset{\mathbb{K}}{\times}N^*(\Delta^1)\\
        & \underline{0}^\vee\otimes\underline{0}^\vee & \longmapsto & (\underline{0}^\vee,0)\\
        & \underline{1}^\vee\otimes\underline{0}^\vee & \longmapsto & (\underline{1}^\vee,0)\\
       &  \underline{1}^\vee\otimes\underline{1}^\vee & \longmapsto & (0,\underline{1}^\vee)\\
        & \underline{01}^\vee\otimes\underline{0}^\vee & \longmapsto & (\underline{01}^\vee,0)\\
       &  \underline{1}^\vee\otimes\underline{01}^\vee & \longmapsto & (0,\underline{01}^\vee)
    \end{array}
\end{array}.$$

\noindent From a geometrical point of view, the morphism $\Gamma_1$ allows us to see the product $N^*(\Delta^1)\underset{\mathbb{K}}{\times}N^*(\Delta^1)$ as the top left corner of $N_\square^*(\square^2)$, while $\Gamma_2$ allows us to see it at the bottom right corner of $N_\square^*(\square^2)$. In particular, one can check that $N_\square^*(\partial\square^2)$ is the pullback of the diagram
\[\begin{tikzcd}
	{N_\square^*(\partial\square^2)} & {N^*(\Delta^1)\underset{\mathbb{K}}{\times}N^*(\Delta^1)} \\
	{N^*(\Delta^1)\underset{\mathbb{K}}{\times}N^*(\Delta^1)} & {\mathbb{K}\cdot(\underline{0}^\vee,0)\oplus\mathbb{K}\cdot(0,\underline{1}^\vee)}
	\arrow["{\Gamma_1}", dashed, from=1-1, to=1-2]
	\arrow["{\Gamma_2}"', dashed, from=1-1, to=2-1]
	\arrow[two heads, from=1-2, to=2-2]
	\arrow[""{name=0, anchor=center, inner sep=0}, two heads, from=2-1, to=2-2]
	\arrow["\lrcorner"{anchor=center, pos=0.125}, draw=none, from=1-1, to=0]
\end{tikzcd}.\]

\noindent Since $G_1$ and $G_2$ are homotopies from $f$ to $h$, their projection on $\Sigma A\otimes\mathbb{K}\cdot (\underline{0}^\vee,0)$ (respectively $\Sigma A\otimes\mathbb{K}\cdot (0,\underline{1}^\vee)$) agree and are given by $h$ (respectively $f$). Therefore, the morphisms $G_1$ and $G_2$ induce an $\infty$-morphism $G_1\square G_2:(A,\phi)\rightarrow(A\otimes N_\square^*(\partial\square^2),\phi')$ given by the following pullback square diagram:
\[\begin{tikzcd}
	{(A,\phi)} & {(A\otimes N_\square^*(\partial\square^2),\phi')} & {(A\otimes (N^*(\Delta^1)\underset{\mathbb{K}}{\times}N^*(\Delta^1)),\phi')} \\
	& {(A\otimes (N^*(\Delta^1)\underset{\mathbb{K}}{\times}N^*(\Delta^1)),\phi')} & {(A\otimes (\mathbb{K}\cdot(\underline{0}^\vee,0)\oplus\mathbb{K}\cdot(0,\underline{1}^\vee)),\phi')}
	\arrow["{G_1\square G_2}", dashed, from=1-1, to=1-2]
	\arrow["{G_1}", curve={height=-30pt}, from=1-1, to=1-3]
	\arrow["{G_2}"', from=1-1, to=2-2]
	\arrow["{id\otimes\Gamma_1}", from=1-2, to=1-3]
	\arrow["{id\otimes\Gamma_2}"', from=1-2, to=2-2]
	\arrow[two heads, from=1-3, to=2-3]
	\arrow[""{name=0, anchor=center, inner sep=0}, two heads, from=2-2, to=2-3]
	\arrow["\lrcorner"{anchor=center, pos=0.125}, draw=none, from=1-2, to=0]
\end{tikzcd}.\]

\begin{prop}
    Giving a Maurer-Cartan element in $\text{\normalfont Hom}(B_{\geq 2}(A),\Sigma A)\otimes\Sigma N^*(\Delta^3)$ is equivalent to giving $\phi_0,\phi_1,\phi_2,\phi_3\in\mathcal{MC}(\text{\normalfont Hom}(B_{\geq 2}(A),\Sigma A))$, two homotopy diagrams of the form
    \medskip
\[\begin{tikzcd}
	{(A,\phi_2)} && {(A,\phi_1)} & {(A,\phi_2)} && {(A,\phi_1)} \\
	\\
	{(A,\phi_3)} && {(A,\phi_0)} & {(A,\phi_3)} && {(A,\phi_0)}
	\arrow["{\Phi_{12}}", from=1-1, to=1-3]
	\arrow["{\Phi_{01}}", from=1-3, to=3-3]
	\arrow["{\Phi_{12}}", from=1-4, to=1-6]
	\arrow[""{name=0, anchor=center, inner sep=0}, "{\Phi_{02}}"{description}, from=1-4, to=3-6]
	\arrow["{\Phi_{01}}", from=1-6, to=3-6]
	\arrow["{\Phi_{23}}", from=3-1, to=1-1]
	\arrow[""{name=1, anchor=center, inner sep=0}, "{\Phi_{13}}"{description}, from=3-1, to=1-3]
	\arrow[""{name=2, anchor=center, inner sep=0}, "{\Phi_{03}}"', from=3-1, to=3-3]
	\arrow["{\Phi_{23}}", from=3-4, to=1-4]
	\arrow[""{name=3, anchor=center, inner sep=0}, "{\Phi_{03}}"', from=3-4, to=3-6]
	\arrow["{\Phi_{123}}", shorten >=6pt, Rightarrow, from=1-1, to=1]
	\arrow["{\Phi_{013}}", shorten >=9pt, Rightarrow, from=1-3, to=2]
	\arrow["{\Phi_{023}}"', shorten >=9pt, Rightarrow, from=1-4, to=3]
	\arrow["{\Phi_{012}}"', shorten >=6pt, Rightarrow, from=1-6, to=0]
\end{tikzcd}\]
\medskip

\noindent and a lifting diagram
\medskip
\[\begin{tikzcd}
	& {(A\otimes N_\square^*(\square^2),\phi_0)} \\
	{(A,\phi_3)} & {(A\otimes N_\square^*(\partial\square^2),\phi_0)}
	\arrow[two heads, from=1-2, to=2-2]
	\arrow["{\exists\Phi_{0123}}", dashed, from=2-1, to=1-2]
	\arrow["{H_1\square H_2}"', from=2-1, to=2-2]
\end{tikzcd},\]
\medskip

\noindent where we denote by $H_1,H_2:(A,\phi_3)\longrightarrow  (A\otimes (N^*(\Delta^1)\underset{\mathbb{K}}{\times}N^*(\Delta^1)),\phi_0)$ the homotopies from $\Phi_{01}\Phi_{12}\Phi_{23}$ to $\Phi_{03}$ given by the homotopy diagrams.
\end{prop}

\begin{proof}
    Let 
    \medskip
    \begin{multline*}
        \omega:=-\phi_0\otimes\underline{0}^\vee-\phi_1\otimes\underline{1}^\vee-\phi_2\otimes\underline{2}^\vee-\phi_3\otimes\underline{3}^\vee\\
        -\phi_{01}\otimes\underline{01}^\vee-\phi_{02}\otimes\underline{02}^\vee-\phi_{12}\otimes\underline{12}^\vee-\phi_{03}\otimes\underline{03}^\vee-\phi_{13}\otimes\underline{13}^\vee-\phi_{23}\otimes\underline{23}^\vee\\
        -\phi_{012}\otimes\underline{012}^\vee-\phi_{013}\otimes\underline{013}^\vee-\phi_{023}\otimes\underline{023}^\vee-\phi_{123}\otimes\underline{123}^\vee-\phi_{0123}\otimes\underline{0123}^\vee
    \end{multline*}
    \medskip
    
    \noindent be an element of $\text{\normalfont Hom}(B_{\geq 2}(A),\Sigma A)\otimes\Sigma N^*(\Delta^3)$. By Proposition \ref{mcass2}, the Maurer-Cartan condition on the four faces of $\omega$ is precisely equivalent to giving the first two diagrams given in the assertion of the proposition, since the $\Gamma\Lambda\mathcal{PL}_\infty$-algebra structure of $\text{\normalfont Hom}(B_{\geq 2}(A),\Sigma A)\otimes N^*(\Delta^3)$ is compatible with the simplicial structures. From now on, we suppose that $d_0\omega,d_1\omega,d_2\omega$ and $d_3\omega$ are elements of $\mathcal{MC}_2(\text{\normalfont Hom}(B_{\geq 2}(A),\Sigma A))$. Then the only possibly non-zero component of $d(\omega)+\sum_{n\geq 1}\omega\llbrace\omega\rrbrace_n$ is the $\underline{0123}^\vee$ component. By Corollary \ref{calculdiff}, this component is $0$ if and only if we have the identity
    \begin{multline*}
        d(\phi_{0123})+\phi_{023}-\phi_{123}+\phi_{013}-\phi_{012}-\phi_{0123}\langle\phi_3\rangle\\
        +\sum_{k\geq 1}\phi_{012}\langle\underbrace{\phi_{23},\ldots,\phi_{23}}_k\rangle-\sum_{i,j\geq 0}\phi_{01}\langle\underbrace{\phi_{13},\ldots,\phi_{13}}_i,\phi_{123},\underbrace{\phi_{12}\overline\circledcirc\phi_{23},\ldots,\phi_{12}\overline\circledcirc\phi_{23}}_j\rangle\\
        +\sum_{i,j,k\geq 0}\phi_0\langle\underbrace{\phi_{03},\ldots,\phi_{03}}_i,\phi_{023},\underbrace{\phi_{02}\overline\circledcirc\phi_{23}}_j,\phi_{012}\circledcirc(1+\phi_{23}),\underbrace{\phi_{01}\overline\circledcirc\phi_{12}\overline\circledcirc\phi_{23},\ldots,\phi_{01}\overline\circledcirc\phi_{12}\overline\circledcirc\phi_{23}}_m\rangle\\
        +\sum_{i,j,m\geq 0}\phi_0\langle\underbrace{\phi_{03},\ldots,\phi_{03}}_i,\phi_{013},\underbrace{\phi_{01}\overline\circledcirc\phi_{13},\ldots,\phi_{01}\overline\circledcirc\phi_{13}}_j,\phi_{123},\underbrace{\phi_{01}\overline\circledcirc\phi_{12}\overline\circledcirc\phi_{23},\ldots,\phi_{01}\overline\circledcirc\phi_{12}\overline\circledcirc\phi_{23}}_k\rangle\\
        +\sum_{i,j,k,l,m\geq 0}\phi_0\langle\underbrace{\phi_{03},\ldots,\phi_{03}}_i,\phi_{013},\underbrace{\phi_{01}\overline\circledcirc\phi_{13},\ldots,\phi_{01}\overline\circledcirc\phi_{13}}_j,\\\phi_{01}\langle\underbrace{\phi_{12},\ldots,\phi_{12}}_k,\phi_{123},\underbrace{\phi_{12}\overline\circledcirc\phi_{23},\ldots,\phi_{12}\overline\circledcirc\phi_{23}}_l\rangle,\underbrace{\phi_{01}\overline\circledcirc\phi_{12}\overline\circledcirc\phi_{23},\ldots,\phi_{01}\overline\circledcirc\phi_{12}\overline\circledcirc\phi_{23}}_m\rangle\\
        +\sum_{i,j\geq 0}\phi_0\langle\underbrace{\phi_{03},\ldots,\phi_{03}}_i,\phi_{0123},\underbrace{\phi_{01}\overline\circledcirc\phi_{12}\overline\circledcirc\phi_{23},\ldots,\phi_{01}\overline\circledcirc\phi_{12}\overline\circledcirc\phi_{23}}_j\rangle=0.
    \end{multline*}

\noindent We let $\Phi_{0123}:\overline B(A)\longrightarrow\overline B(A\otimes N_\square^*(\square^2))$ to be the morphism of coalgebras whose composite with the projection on $\Sigma A\otimes N_\square^*(\square^2)$ is
\begin{multline*}
    (1+\phi_{03})\otimes\underline{0}^\vee\otimes\underline{0}^\vee\\+(1+\phi_{01})\circledcirc (1+\phi_{13})\otimes\underline{0}^\vee\otimes\underline{1}^\vee\\+(1+\phi_{02})\circledcirc (1+\phi_{23})\otimes\underline{1}^\vee\otimes\underline{0}^\vee\\
    (1+\phi_{01})\circledcirc (1+\phi_{12})\circledcirc (1+\phi_{23})\otimes\underline{1}^\vee\otimes\underline{1}^\vee\\+\phi_{023}\otimes\underline{01}^\vee\otimes\underline{0}^\vee\\+\phi_{013}\otimes\underline{0}^\vee\otimes\underline{01}^\vee\\
    +\phi_{012}\circledcirc(1+\phi_{23})\otimes\underline{1}^\vee\otimes\underline{01}^\vee\\+\left(\phi_{123}+\sum_{i,j\geq 0}\phi_{01}\langle\underbrace{\phi_{12},\ldots,\phi_{12}}_i,\phi_{123},\underbrace{\phi_{12}\overline\circledcirc\phi_{23},\ldots,\phi_{12}\overline\circledcirc\phi_{23}}_j\rangle\right)\otimes\underline{01}^\vee\otimes\underline{1}^\vee\\+\phi_{0123}\otimes\underline{01}^\vee\otimes\underline{01}^\vee.
\end{multline*}

\noindent We check that $\omega$ is a Maurer-Cartan element if and only if $\Phi_{0123}$ commutes with the differentials. The latter condition is expressed by the identity
$$\pi_{\Sigma A\otimes N_\square^*(\square^2)}\Phi_{0123}(d+\partial_{\phi_3})-\pi_{\Sigma A\otimes N_\square^*(\square^2)}(d+\partial_{\phi_0})\Phi_{0123})=0.$$

\noindent Since the morphisms $\Phi_{03},\Phi_{01}\Phi_{13},\Phi_{02}\Phi_{23}$ and $\Phi_{01}\Phi_{12}\Phi_{23}$ commute with the differentials, the components given by $\underline{0}^\vee\otimes\underline{0}^\vee,\underline{0}^\vee\otimes\underline{1}^\vee,\underline{1}^\vee\otimes\underline{0}^\vee$ and $\underline{1}^\vee\otimes\underline{1}^\vee$ are indeed $0$. Since $\Phi_{023}$ and $\Phi_{013}$ commute with the differentials, the components given by $\underline{01}^\vee\otimes\underline{0}^\vee$ and $\underline{0}^\vee\otimes\underline{01}^\vee$ are also $0$. We now look at the component given by $\underline{1}^\vee\otimes\underline{01}^\vee$. Since the algebra structure of $N_\square^*(\square^2)$ is compatible with its underlying simplicial structure, it is equivalent to check that the element
$$(1+\phi_{02}\overline\circledcirc\phi_{23})\otimes\underline{0}^\vee+\phi_{012}\circledcirc(1+\phi_{23})\otimes\underline{01}^\vee+(1+\phi_{01}\overline\circledcirc\phi_{12}\overline\circledcirc\phi_{23})\otimes\underline{1}^\vee$$

\noindent is a Maurer-Cartan element in $\text{\normalfont Hom}(B_{\geq 2}(A),A)\otimes N^*(\Delta^1)$. From Proposition \ref{mcass2}, one can see that it is equivalent to check that the composite $\Phi_{012}\Phi_{23}:(\overline B(A),\phi_3)\longrightarrow (\overline B(A\otimes N^*(\Delta^1),\phi_0)$ commutes with the differentials, which is the case since $\Phi_{012}$ and $\Phi_{23}$ commute with the differentials. Analogously, we $\underline{01}^\vee\otimes\underline{1}^\vee$ is also $0$, since the composite $\Phi_{01}\Phi_{123}$ commutes with the differentials.\\

\noindent We now look at the $\underline{01}^\vee\otimes\underline{01}^\vee$ component. The composite $\pi_{\Sigma A\otimes N_\square^*(\square^2)}\Phi_{0123}(d+\partial_{\phi_3})$ gives 
$$(\phi_{0123}d+\phi_{0123}\langle\phi_3\rangle)\otimes\underline{01}^\vee\otimes\underline{01}^\vee$$
\noindent as $\underline{01}^\vee\otimes\underline{01}^\vee$ component. We now compute the $\underline{01}^\vee\otimes\underline{01}^\vee$ component given by the composite $\pi_{\Sigma A\otimes N_\square^*(\square^2)}(d+\partial_{\phi_0})\Phi_{0123}$. Computing $\pi_{\Sigma A\otimes N_\square^*(\square^2)}d\Phi_{0123}$ gives the terms
\begin{multline*}
    (d\phi_{0123}-\phi_{013}+\phi_{023}+\phi_{012}\circledcirc(1+\phi_{23})\\-\phi_{123}-\sum_{i,j\geq 0}\phi_{01}\langle\underbrace{\phi_{12},\ldots,\phi_{12}}_i,\phi_{123},\underbrace{\phi_{12}\overline\circledcirc\phi_{23},\ldots,\phi_{12}\overline\circledcirc\phi_{23}}_j\rangle)\otimes\underline{01}^\vee\otimes\underline{01}^\vee.
\end{multline*}
\noindent We now compute $\pi_{\Sigma A\otimes N_\square^*(\square^2)}\partial_{\phi_0}\Phi_{0123}$. Note that the only way to write $\underline{01}^\vee\otimes\underline{01}^\vee$ as a product in $N^*(\Delta^1)\otimes N^*(\Delta^1)$ are given by one of the three following products:
$$(\underline{0}^\vee\otimes\underline{0}^\vee)\overset{i}{\cdots}(\underline{0}^\vee\otimes\underline{0}^\vee)\cdot(\underline{01}^\vee\otimes\underline{0}^\vee)\cdot(\underline{1}^\vee\otimes\underline{0}^\vee)\overset{j}{\cdots}(\underline{1}^\vee\otimes\underline{0}^\vee)\cdot(\underline{1}^\vee\otimes\underline{01}^\vee)\cdot(\underline{1}^\vee\otimes\underline{1}^\vee)\overset{k}{\cdots}(\underline{1}^\vee\otimes\underline{1}^\vee);$$
$$(\underline{0}^\vee\otimes\underline{0}^\vee)\overset{i}{\cdots}(\underline{0}^\vee\otimes\underline{0}^\vee)\cdot(\underline{0}^\vee\otimes\underline{01}^\vee)\cdot(\underline{0}^\vee\otimes\underline{1}^\vee)\overset{j}{\cdots}(\underline{0}^\vee\otimes\underline{1}^\vee)\cdot(\underline{01}^\vee\otimes\underline{1}^\vee)\cdot(\underline{1}^\vee\otimes\underline{1}^\vee)\overset{k}{\cdots}(\underline{1}^\vee\otimes\underline{1}^\vee);$$
$$(\underline{0}^\vee\otimes\underline{0}^\vee)\overset{i}{\cdots}(\underline{0}^\vee\otimes\underline{0}^\vee)\cdot(\underline{01}^\vee\otimes\underline{01}^\vee)\cdot(\underline{1}^\vee\otimes\underline{1}^\vee)\overset{j}{\cdots}(\underline{1}^\vee\otimes\underline{1}^\vee).$$

\noindent for every $i,j,k\geq 0$. These type of products give respectively\\
$\displaystyle-\sum_{i,j,k\geq 0}\phi_0\langle\underbrace{\phi_{03},\ldots,\phi_{03}}_i,\phi_{023},\underbrace{\phi_{02}\overline\circledcirc\phi_{23},\ldots,\phi_{02}\overline\circledcirc\phi_{23}}_j,\phi_{012}\circledcirc(1+\phi_{23}),$\\
\begin{flushright}
$\underbrace{\phi_{01}\overline\circledcirc\phi_{12}\overline\circledcirc\phi_{23},\ldots,\phi_{01}\overline\circledcirc\phi_{12}\overline\circledcirc\phi_{23}}_k\rangle\otimes\underline{01}^\vee\otimes\underline{01}^\vee;$
\end{flushright}
\begin{multline*}
    \sum_{i,j,k,l,m\geq 0}\phi_0\langle\underbrace{\phi_{03},\ldots,\phi_{03}}_i,\phi_{013},\underbrace{\phi_{01}\overline\circledcirc\phi_{13},\ldots,\phi_{01}\overline\circledcirc\phi_{13}}_j,\\\phi_{123}+\phi_{01}\langle\underbrace{\phi_{12},\ldots,\phi_{12}}_k,\phi_{123},\underbrace{\phi_{12}\overline\circledcirc\phi_{23},\ldots,\phi_{12}\overline\circledcirc\phi_{23}}_l\rangle,\underbrace{\phi_{01}\overline\circledcirc\phi_{12}\overline\circledcirc\phi_{23},\ldots,\phi_{01}\overline\circledcirc\phi_{12}\overline\circledcirc\phi_{23}}_m\rangle\otimes\underline{01}^\vee\otimes\underline{01}^\vee;\\
\end{multline*}
$$\sum_{i,j\geq 0}\phi_0\langle\underbrace{\phi_{03},\ldots,\phi_{03}}_i,\phi_{0123},\underbrace{\phi_{01}\overline\circledcirc\phi_{12}\overline\circledcirc\phi_{23},\ldots,\phi_{01}\overline\circledcirc\phi_{12}\overline\circledcirc\phi_{23}}_j\rangle\otimes\underline{01}^\vee\otimes\underline{01}^\vee,$$

\noindent as $\underline{01}^\vee\otimes\underline{01}^\vee$. We thus have obtained that $\Phi_{0123}$ commutes with the differentials if and only if $\omega$ is a Maurer-Cartan element, which proves the lemma.
\end{proof}

\subsection{A Goldman-Millson theorem}\label{sec:243}

Our next goal is to prove an extension of the classical Goldman-Millson theorem for Lie-algebras (see \cite[$\mathsection$2.4]{goldman}). The proof of our analogue will be adapted from the proof given in \cite[$\mathsection$6]{rogers} in the setting of associative algebras up to homotopy.\\

We first prove that the category $\Gamma\Lambda\mathcal{PL}_\infty$ admits finite products.

\begin{lm}
    Let $(V_1,Q_{V_1}),(V_2,Q_{V_2})\in\Gamma\Lambda\mathcal{PL}_\infty$. Then there exists a $\Gamma\Lambda\mathcal{PL}_\infty$-algebra structure on $V_1\times V_2$ such that the morphisms $\pi_{V_1}:V_1\times V_2\longrightarrow V_1$ and $\pi_{V_2}:V_1\times V_2\longrightarrow V_2$ are strict morphisms of $\Gamma\Lambda\mathcal{PL}_\infty$-algebras.\\ Moreover, for every $\phi_1: W\rightsquigarrow V_1$ and $\phi_2:W\rightsquigarrow V_2$, there exists a unique $\infty$-morphism denoted by $\phi_1\times\phi_2:W\rightsquigarrow V_1\times V_2$ such that $\pi_{V_1}(\phi_1\times\phi_2)=\phi_1$ and $\pi_{V_2}(\phi_1\times\phi_2)=\phi_2$.
\end{lm}

\begin{proof}
We let $Q_{V_1\times V_2}$ be the coderivation produced by the morphism
\[\begin{tikzcd}
	{Q^0_{V_1\times V_2}:\Gamma\text{\normalfont Perm}^c(V_1\times V_2)} & {\Gamma\text{\normalfont Perm}^c(V_1)\times\Gamma\text{\normalfont Perm}^c(V_2)} & {V_1\times V_2}
	\arrow[two heads, from=1-1, to=1-2]
	\arrow["{Q_{V_1}^0\times Q_{V_2}^0}", from=1-2, to=1-3]
\end{tikzcd}.\]
    \noindent Recall that the coderivation $Q_{V_1\times V_2}$ is obtained from $Q^0_{V_1\times V_2}$ by $Q_{V_1\times V_2}=\widetilde{\Psi}_1(Q_{V_1\times V_2}^0)+\widetilde{\Psi}_2(Q_{V_1\times V_2}^0)$ (see the proof of Proposition \ref{bijendo}). We check that $Q_{V_1\times V_2}Q_{V_1\times V_2}=0$. By definition of $\widetilde{\Psi}_1$, we have the following commutative diagram:
\[\begin{tikzcd}
	{(V_1\times V_2)\otimes\Gamma(V_1\times V_2)} & {(V_1\times V_2)\otimes\Gamma(V_1\times V_2)} & {V_1\times V_2} \\
	{(V_1\otimes\Gamma(V_1))\times (V_2\otimes\Gamma(V_2))} && {(V_1\otimes\Gamma(V_1))\times (V_2\otimes\Gamma(V_2))}
	\arrow["{\widetilde{\Psi}_1(Q^0_{V_1\times V_2})}", from=1-1, to=1-2]
	\arrow[two heads, from=1-1, to=2-1]
	\arrow["{Q^0_{V_1\times V_2}}", from=1-2, to=1-3]
	\arrow["{\widetilde{\Psi}_1(Q^0_{V_1})\times\widetilde{\Psi}_1(Q^0_{V_2})}"', from=2-1, to=2-3]
	\arrow["{Q^0_{V_1}\times Q^0_{V_2}}"', from=2-3, to=1-3]
\end{tikzcd}.\]

\noindent We also have the commutative diagram, by definition of $\widetilde{\Psi}_2$:
\[\begin{tikzcd}
	{(V_1\times V_2)\otimes\Gamma(V_1\times V_2)} && {(V_1\times V_2)\otimes\Gamma(V_1\times V_2)} & {V_1\times V_2} \\
	{(V_1\otimes\Gamma(V_1))\times (V_2\otimes\Gamma(V_2))} &&& {(V_1\otimes\Gamma(V_1))\times (V_2\otimes\Gamma(V_2))}
	\arrow["{\widetilde{\Psi}_2(Q_{V_1\times V_2}^0)}", from=1-1, to=1-3]
	\arrow[two heads, from=1-1, to=2-1]
	\arrow["{Q^0_{V_1\times V_2}}", from=1-3, to=1-4]
	\arrow["{\widetilde{\Psi}_2(Q_{V_1}^0)\times\widetilde{\Psi}_2(Q_{V_2}^0)}"', from=2-1, to=2-4]
	\arrow["{Q^0_{V_1}\times Q^0_{V_2}}"', from=2-4, to=1-4]
\end{tikzcd}.\]

\noindent Finally, we have proved that $Q_{V_1\times V_2}^0\widetilde{\Psi}(Q_{V_1\times V_2}^0)$ fits in the following commutative diagram
\[\begin{tikzcd}
	{(V_1\times V_2)\otimes\Gamma(V_1\times V_2)} && {(V_1\times V_2)\otimes\Gamma(V_1\times V_2)} & {V_1\times V_2} \\
	{(V_1\otimes\Gamma(V_1))\times (V_2\otimes\Gamma(V_2))} &&& {(V_1\otimes\Gamma(V_1))\times (V_2\otimes\Gamma(V_2))}
	\arrow["{\widetilde{\Psi}(Q_{V_1\times V_2}^0)}", from=1-1, to=1-3]
	\arrow[two heads, from=1-1, to=2-1]
	\arrow["{Q^0_{V_1\times V_2}}", from=1-3, to=1-4]
	\arrow["0"{description}, from=2-1, to=1-4]
	\arrow["{\widetilde{\Psi}(Q_{V_1}^0)\times\widetilde{\Psi}(Q_{V_2}^0)}"', from=2-1, to=2-4]
	\arrow["{Q^0_{V_1}\times Q^0_{V_2}}"', from=2-4, to=1-4]
\end{tikzcd}\]

\noindent which proves that $(V_1\times V_2,Q_{V_1\times V_2})\in\Gamma\Lambda\mathcal{PL}_\infty$. Now let $\phi_1:W\rightsquigarrow V_1$ and $\phi_2:W\rightsquigarrow V_2$ be two $\infty$-morphisms. We define $\phi:W\rightsquigarrow V_1\times V_2$ by $\phi^0=\phi_1^0\times\phi_2^0$. We prove that this gives an $\infty$-morphism i.e. $\phi^0Q=(Q_{V_1}^0\times Q_{V_2}^0)\phi.$ This can be proved with the following commutative diagram:
\[\begin{tikzcd}
	{\Gamma\text{\normalfont Perm}^c(W)} && {\Gamma\text{\normalfont Perm}^c(W)} \\
	{\Gamma\text{\normalfont Perm}^c(V_1\times V_2)} & {\Gamma\text{\normalfont Perm}^c(V_1)\times \Gamma\text{\normalfont Perm}^c(V_2)} & {V_1\times V_2}
	\arrow["Q", from=1-1, to=1-3]
	\arrow["{\phi}"', from=1-1, to=2-1]
	\arrow["{\phi_1^0\times\phi_2^0}", from=1-3, to=2-3]
	\arrow["{\pi_{V_1}\times\pi_{V_2}}", two heads, from=2-1, to=2-2]
	\arrow["{Q^0_{V_1\times V_2}}"', curve={height=18pt}, from=2-1, to=2-3]
	\arrow["{Q_{V_1}^0\times Q_{V_2}^0}", from=2-2, to=2-3]
\end{tikzcd}.\]

\noindent The identities $\pi_{V_1}(\phi_1\times\phi_2)=\phi_1$ and $\pi_{V_2}(\phi_1\times\phi_2)=\phi_2$ follow by immediate computations.
\end{proof}

\begin{remarque}
    By an immediate check, the above definitions extend to the category $\widehat{\Gamma\Lambda\mathcal{PL}_\infty}$ of complete $\Gamma\Lambda\mathcal{PL}_\infty$-algebras. Explicitly, if $(V_1,Q_{V_1})$ and $(V_2,Q_{V_2})$ are complete with respect to some filtrations, then $(V_1\times V_2,Q_{V_1\times V_2})$ is also complete with the filtration
    $$F_n(V_1\times V_2)=F_nV_1\times F_nV_2.$$

    \noindent In this setting, we deduce immediately from the definition of $Q^0_{V_1\times V_2}$ that we have a bijection
    $$\mathcal{MC}(V_1\times V_2)\simeq\mathcal{MC}(V_1)\times\mathcal{MC}(V_2).$$
\end{remarque}

We give an analogue of \cite[Proposition 5.2]{rogers}.

\begin{lm}\label{decoupage}
    Let $A,B\in\text{\normalfont \dgMod}$ be such that $A\otimes\Sigma N^*(\Delta^\bullet)$ and $B\otimes\Sigma N^*(\Delta^\bullet)$ are endowed with the structure of simplicial $\widehat{\Gamma\Lambda\mathcal{PL}_\infty}$-algebras. Let $\Theta:A\longrightarrow B$ be a morphism in {\normalfont\dgMod} such that $\Theta\otimes id$ is a strict morphism of $\widehat{\Gamma\Lambda\mathcal{PL}_\infty}$-algebras for every $n\geq 0$. Suppose that $\Theta$ is an acyclic fibration of dg $\mathbb{K}$-modules.\\
    Then the map
    $$\Theta\otimes id:\mathcal{MC}(A\otimes\Sigma N^*(\Delta^\bullet))\longrightarrow \mathcal{MC}(B\otimes\Sigma N^*(\Delta^\bullet))$$

    \noindent is a weak equivalence of simplicial sets.
\end{lm}

\begin{proof}
    Since the two simplicial sets $\mathcal{MC}(A\otimes\Sigma N^*(\Delta^\bullet))$ and $\mathcal{MC}(B\otimes\Sigma N^*(\Delta^\bullet))$ are Kan complexes by Theorem \ref{kan}, it suffices to show that $\Theta\otimes id$ induces a bijection on the sets of connected components and an isomorphism on every homotopy groups. Let $\tau:B\longrightarrow A$ and $h:A\longrightarrow A$ be such that
    $$\Theta\tau=id\ ;\ id-\tau\Theta=dh+hd.$$

    \noindent We endow $\text{\normalfont Ker}(\Theta)$ with the brace algebra structure defined by $a\langle b_1,\ldots,b_r\rangle=0$ for every $r\geq 1$ and $a,b_1,\ldots,b_r\in\text{\normalfont Ker}(\Theta)$. Our first goal is to define a morphism $\Psi_\bullet:A\otimes\Sigma N^*(\Delta^\bullet)\rightsquigarrow  \text{\normalfont Ker}(\Theta)\otimes\Sigma N^*(\Delta^\bullet)$ of simplicial $\widehat{\Gamma\Lambda\mathcal{PL}_\infty}$-algebras. For every $n\geq 0$, let $(\Psi_n)_0^0=(id-\tau\Theta)\otimes id$. We set, for every $k\neq 0$,
    $$(\Psi_n)_k^0=(\Psi_n)_0^0 (h\otimes id)Q_k^0,$$

    \noindent where we denote by $Q$ the coderivation on $\Gamma\text{\normalfont Perm}^c(A\otimes\Sigma N^\ast(\Delta^n))$ given by the $\Gamma\Lambda\mathcal{PL}_\infty$-algebra structure on $A\otimes\Sigma N^*(\Delta^n)$. We check that $\Psi_n:A\otimes\Sigma N^*(\Delta^n)\rightsquigarrow \text{\normalfont Ker}(\Theta)\otimes\Sigma N^*(\Delta^n)$:
    \begin{center}
        $\begin{array}{lll}
            d(\Psi_n)_k^0 & = & \displaystyle d(\Psi_n)_0^0 (h\otimes id) Q_k^0\\
            & = & \displaystyle(\Psi_n)_0^0 d (h\otimes id) Q_k^0\\
            & = & \displaystyle(\Psi_n)^0_0 Q_k^0-(\Psi_n)_0^0 (h\otimes id) d Q_k^0\\
            & = & \displaystyle(\Psi_n)^0_0 Q_k^0+\sum_{i=1}^k (\Psi_n)^0_0 (h\otimes id) Q_i^0 Q_k^i\\
            & = & \displaystyle(\Psi_n)^0_0 Q_k^0+\sum_{i=1}^k (\Psi_n)^0_i Q_k^i\\
            & = & \displaystyle\sum_{i=0}^k(\Psi_n)^0_iQ^i_k,
        \end{array}$
    \end{center}

\noindent which proves that $\Psi_n:A\otimes\Sigma N^*(\Delta^n)\rightsquigarrow  \text{\normalfont Ker}(\Theta)\otimes\Sigma N^*(\Delta^n)$. Since $\Psi_\bullet$ is defined in terms of morphisms which are compatible with the simplicial structure, the morphism $\Psi_\bullet$ is a morphism of simplicial $\widehat{\Gamma\Lambda\mathcal{PL}_\infty}$-algebras. Consider now, for every $n\geq 0$, the $\widehat{\Gamma\Lambda\mathcal{PL}_\infty}$-algebra structure on $(B\times\text{\normalfont Ker}(\Theta))\otimes\Sigma N^*(\Delta^\bullet)$ given by the isomorphism
$$(B\times\text{\normalfont Ker}(\Theta))\otimes\Sigma N^*(\Delta^\bullet)\simeq (B\otimes\Sigma N^*(\Delta^\bullet))\times (\text{\normalfont Ker}(\Theta)\otimes\Sigma N^*(\Delta^\bullet)).$$

\noindent Let $g_n=(\Theta\otimes id)\times \Psi_n:A\otimes\Sigma N^*(\Delta^n)\rightsquigarrow (B\times\text{\normalfont Ker}(\Theta))\otimes \Sigma N^*(\Delta^n)$. Then $g_\bullet$ is a morphism of simplicial $\widehat{\Gamma\Lambda\mathcal{PL}_\infty}$-algebras. We also have that $\mathcal{MC}(g_\bullet)$ is an isomorphism of simplicial sets. Indeed, we have that $(g_\bullet)_0^0$ is an isomorphism of simplicial sets, with as inverse $(b,k)\otimes\underline{x}\longmapsto (\tau(b)+k)\otimes \underline{x}$, for every $b\in B, k\in\text{\normalfont Ker}(\Theta)$ and $\underline{x}\in\Sigma N^*(\Delta^\bullet)$. By looking at the Maurer-Cartan spaces degree wise, we obtain the following commutative diagram:
\[\begin{tikzcd}
	{\mathcal{MC}(A\otimes\Sigma N^*(\Delta^\bullet))} && {\mathcal{MC}((B\times\text{\normalfont Ker}(\Theta))\otimes\Sigma N^*(\Delta^\bullet))} \\
	\\
	{\mathcal{MC}(B\otimes\Sigma N^*(\Delta^\bullet))} && {\mathcal{MC}(B\otimes\Sigma N^*(\Delta^\bullet))\times\mathcal{MC}(\text{\normalfont Ker}(\Theta)\otimes\Sigma N^*(\Delta^\bullet))}
	\arrow["\mathcal{MC}(g_\bullet)", "\simeq"', from=1-1, to=1-3]
	\arrow["\mathcal{MC}(\Theta\otimes id)"', from=1-1, to=3-1]
	\arrow["\simeq", from=1-3, to=3-3]
	\arrow[two heads, from=3-3, to=3-1]
\end{tikzcd}\]

\noindent It is then sufficient to prove that the projection $\mathcal{MC}(B\otimes\Sigma N^*(\Delta^\bullet))\times\mathcal{MC}(\text{\normalfont Ker}(\Theta)\otimes\Sigma N^*(\Delta^\bullet))\longrightarrow\mathcal{MC}(B\otimes\Sigma N^*(\Delta^\bullet))$ is a weak equivalence of simplicial sets, which is true because 
$\mathcal{MC}(\text{\normalfont Ker}(\Theta)\otimes\Sigma N^*(\Delta^\bullet))=\mathcal{MC}_\bullet(\text{\normalfont Ker}(\Theta))$, and this simplicial set has trivial $\pi_0$ and homotopy groups according to the computations of the connected components and the homotopy groups made in $\mathsection$\ref{sec:242}. We then have the result.
\end{proof}

The next lemma is an analogue of \cite[Proposition 5.5]{rogers}.

\begin{lm}\label{pback}
    Let $A,B,C$ be $\widehat{\Gamma\Lambda\mathcal{PL}_\infty}$-algebras. Let $\Theta: A\rightsquigarrow C$ and $\Phi:B\rightsquigarrow C$ be two $\infty$-morphisms of $\widehat{\Gamma\Lambda\mathcal{PL}_\infty}$-algebras. We suppose that $\Phi$ is strict, and that $\Phi_0^0$ is surjective. Then there exists a $\widehat{\Gamma\Lambda\mathcal{PL}_\infty}$-algebra structure on $A\times\text{\normalfont Ker}(\Phi_0^0)$ and $H:A\times \text{\normalfont Ker}(\Phi_0^0)\rightsquigarrow A\times B$ such that the following diagram is a pullback square diagram in $\widehat{\Gamma\Lambda\mathcal{PL}_\infty}$:
\[\begin{tikzcd}
	{A\times\text{\normalfont Ker}(\Phi^0_0)} && B \\
	\\
	A && C
	\arrow["{\pi_BH}", dashed, from=1-1, to=1-3]
	\arrow["{\pi_{A}}"', dashed, from=1-1, to=3-1]
	\arrow["\lrcorner"{anchor=center, pos=0.125}, draw=none, from=1-1, to=3-3]
	\arrow["\Phi", dashed, from=1-3, to=3-3]
	\arrow["\Theta"', dashed, from=3-1, to=3-3]
\end{tikzcd}.\]
\end{lm}

\begin{proof}
    We follow the proof of \cite[Proposition 5.5]{rogers}. Let $\sigma : C\longrightarrow B$ be a morphism of $\mathbb{K}$-modules such that $\Phi_0^0\sigma=id$. We define two morphisms $J_0^0,H_0^0:A\times B\longrightarrow A\times B$ by $J_0^0(a,b)=(a,b-\sigma \Theta_0^0(a))$ and $H_0^0(a,b)=(a,\sigma \Theta_0^0(a)+b)$. We set
    $$H_{n}^0=(0,\sigma \Theta_{n}^0\pi_{A});$$
    $$J_n^0=(0,-\sigma \Theta_n^0\pi_{A}),$$

    \noindent An immediate computation gives $HJ=JH=id$. Therefore, if we denote by $Q$ the $\widehat{\Gamma\Lambda\mathcal{PL}_\infty}$-algebra structure on $A\times B$, then $\widetilde{Q}=JQH$ is a degree $-1$ coderivation on $\Gamma\text{\normalfont Perm}^c(A\times B)$. We note that $\widetilde{Q}$ preserves $\Gamma\text{\normalfont Perm}^c(A\times \text{\normalfont Ker}(\Phi_0^0))$ and the filtrations, so that $A\times\text{\normalfont Ker}(\Phi_0^0)$ is a $\widehat{\Gamma\Lambda\mathcal{PL}_\infty}$-algebra such that $H:A\times \text{\normalfont Ker}(\Phi_0^0)\rightsquigarrow A\times B$. Consider now a diagram of $\infty$-morphisms:
\[\begin{tikzcd}
	D && B \\
	\\
	A && C
	\arrow["{\Psi_2}", from=1-1, to=1-3]
	\arrow["{\Psi_1}"', from=1-1, to=3-1]
	\arrow["\Phi", from=1-3, to=3-3]
	\arrow["\Theta"', from=3-1, to=3-3]
\end{tikzcd}.\]

\noindent Then, we have the commutative diagram 
\[\begin{tikzcd}
	D \\
	& {A\times \text{\normalfont Ker}(\Phi_0^0)} && B \\
	\\
	& A && C
	\arrow["{J(\Psi_1\times\Psi_2)}", from=1-1, to=2-2]
	\arrow["{\Psi_2}", curve={height=-24pt}, from=1-1, to=2-4]
	\arrow["{\Psi_1}"', curve={height=24pt}, from=1-1, to=4-2]
	\arrow["{\pi_BH}", from=2-2, to=2-4]
	\arrow["{\pi_{A}}"', from=2-2, to=4-2]
	\arrow["\Phi", from=2-4, to=4-4]
	\arrow["\Theta"', from=4-2, to=4-4]
\end{tikzcd},\]

\noindent which proves the result.
\end{proof}

We now prove Theorem \ref{theoremF}.

\begin{thm}\label{gm}
    Let $\Theta:A\longrightarrow B$ be a morphism of complete brace algebras such that $\Theta$ is a weak equivalence in $\text{\normalfont dgMod}_\mathbb{K}$. Then $\mathcal{MC}_\bullet(\Theta):\mathcal{MC}_\bullet(A)\longrightarrow\mathcal{MC}_\bullet(B)$ is a weak equivalence.
\end{thm}

Before giving the proof, note that if $B$ is a complete brace algebra, and if we set $B^I=B\otimes N^*(\Delta^1)$, then we have the following decomposition of the diagonal map in the category of $\widehat{\Gamma\Lambda\mathcal{PL}_\infty}$-algebras
\[\begin{tikzcd}
	{ B\otimes\Sigma N^*(\Delta^n)} & { B^I\otimes\Sigma N^*(\Delta^n)} & { (B\times B)\otimes\Sigma N^*(\Delta^n)}
	\arrow["{s_0\otimes id}"', tail, from=1-1, to=1-2]
	\arrow["\Delta", curve={height=-12pt}, from=1-1, to=1-3]
	\arrow["{(d_0\times d_1)\otimes id}"', two heads, from=1-2, to=1-3]
\end{tikzcd}\]

\noindent \noindent for every $n\geq 0$. This decomposition comes from Proposition \ref{pathobject} with $\mathcal{P}=\Lambda\mathcal{B}race$ and $R=B\otimes\Sigma N^*(\Delta^n)$. The map $s_0:B\longrightarrow B^I$ is given by the simplicial map $s_0:N^*(\Delta^0)\longrightarrow N^*(\Delta^1)$ and the maps $d_0,d_1:B^I\longrightarrow B$ are given by $d_0,d_1:N^*(\Delta^1)\longrightarrow N^*(\Delta^0)$. In particular, the morphisms $s_0,d_0$ and $d_1$ induce strict morphisms of $\widehat{\Gamma\Lambda\mathcal{PL}_\infty}$-algebras, since the action of $\mathcal{E}$ on $N^*(\Delta^n)$ is compatible with its underlying simplicial structure.

\begin{proof}
    Lemma \ref{decoupage} proves the theorem in the case of an acyclic fibration $\Theta$. Consider now the general case. Since $d_0\otimes id:B^I\otimes\Sigma N^*(\Delta^n)\longrightarrow B\otimes\Sigma N^*(\Delta^n)$ is a strict morphism and surjective, we can apply Lemma \ref{pback}:
\[\begin{tikzcd}
	{ A\otimes\Sigma N^*(\Delta^n)} & {(A\times \text{\normalfont Ker}(d_0))\otimes\Sigma N^*(\Delta^n)} &&& { B^I\otimes\Sigma N^*(\Delta^n)} \\
	\\
	\\
	& { A\otimes\Sigma N^*(\Delta^n)} &&& { B\otimes\Sigma N^*(\Delta^n)}
	\arrow["{\exists\Psi_n}", dashed, from=1-1, to=1-2]
	\arrow["{(s_0\Theta)\otimes id}", curve={height=-30pt}, from=1-1, to=1-5]
	\arrow["{=}"{description}, curve={height=12pt}, no head, from=1-1, to=4-2]
	\arrow["{(\pi_{\text{\normalfont Ker}(d_0)}\otimes id)H_n}", from=1-2, to=1-5]
	\arrow["{\pi_{A}\otimes id}"', two heads, from=1-2, to=4-2]
	\arrow["\lrcorner"{anchor=center, pos=0.125}, draw=none, from=1-2, to=4-5]
	\arrow["{d_0\otimes id}", two heads, from=1-5, to=4-5]
	\arrow["{\Theta\otimes id}"', from=4-2, to=4-5]
\end{tikzcd}\]

\noindent where $H_n:(A\times \text{\normalfont Ker}(d_0))\otimes\Sigma N^*(\Delta^n)\rightsquigarrow (A\times B^I)\otimes\Sigma N^*(\Delta^n)$ is given by Lemma \ref{pback}, and $\Psi_n:A\otimes\Sigma N^*(\Delta^n)\rightsquigarrow (A\times \text{\normalfont Ker}(d_0))\otimes\Sigma N^*(\Delta^n)$ is the unique $\infty$-morphism which makes the previous diagram commutative.\\

\noindent We recall the $\widehat{\Gamma\Lambda\mathcal{PL}_\infty}$-algebra structure on $(A\times \text{\normalfont Ker}(d_0))\otimes\Sigma N^*(\Delta^n)$. Let $Q_n$ be the coderivation on $\Gamma\text{\normalfont Perm}^c((A\times B^I)\otimes\Sigma N^*(\Delta^n))$ given by the product 
$$(A\times B^I)\otimes\Sigma N^*(\Delta^n)\simeq (A\otimes\Sigma N^*(\Delta^n))\times (B^I\otimes\Sigma N^*(\Delta^n)).$$

\noindent Consider the morphisms $H_n,J_n:(A\times B^I)\otimes\Sigma N^*(\Delta^n)\rightsquigarrow (A\times B^I)\otimes\Sigma N^*(\Delta^n)$ defined in the proof of Lemma \ref{pback}. We note that these morphisms are strict, as the morphism $\Theta\otimes id:A\otimes\Sigma N^*(\Delta^n)\longrightarrow B\otimes\Sigma N^*(\Delta^n)$ is strict, and are defined by
$$(H_n)_0^0((a,b)\otimes \underline{x})=(a,\sigma \Theta(a)+b)\otimes\underline{x};$$
$$(J_n)_0^0((a,b)\otimes\underline{x})=(a,b-\sigma\Theta(a))\otimes\underline{x},$$

\noindent for every $a\in A,b\in B$ and $\underline{x}\in\Sigma N^*(\Delta^n)$, and where $\sigma : B\longrightarrow B^I$ 
is a splitting of $d_0:B^I\longrightarrow B$. Then, the $\widehat{\Gamma\Lambda\mathcal{PL}_\infty}$-algebra structure on $(A\times \text{\normalfont Ker}(d_0))\otimes\Sigma N^*(\Delta^n)$ is given by
$$(\widetilde{Q}_n)_p^q=(J_n)_q^q(Q_n)_p^q(H_n)_p^p.$$

\noindent We thus see that the $\widehat{\Gamma\Lambda\mathcal{PL}_\infty}$-algebra structures on $(A\times \text{\normalfont Ker}(d_0))\otimes\Sigma N^*(\Delta^n)$ for all $n\geq 0$ endow the simplicial set $(A\times \text{\normalfont Ker}(d_0))\otimes\Sigma N^*(\Delta^\bullet)$ with the structure of a strict simplicial object in $\widehat{\Gamma\Lambda\mathcal{PL}_\infty}$. Moreover, the map $\pi_A:A\times \text{\normalfont Ker}(d_0)\longrightarrow A$ is an acyclic fibration, and a simple computation shows that $\pi_A\otimes id$ is a strict morphism. By Lemma \ref{decoupage}, we deduce that $\mathcal{MC}(\pi_{A}\otimes id)$ is a weak equivalence. By the $2$ out of $3$ axiom in $\text{\normalfont sSet}$, we also have that $\mathcal{MC}(\Psi_\bullet)$ is a weak equivalence of simplicial sets.\\

Let $h:A\times B^I\longrightarrow A\times B^I$ be the morphism such that $(H_n)_0^0= h\otimes id$. For every $n\geq 0$, we set
$$P_n=d_1\pi_{B^I}h\otimes id:(A\times\text{\normalfont Ker}(d_0))\otimes\Sigma N^*(\Delta^n)\longrightarrow  B\otimes\Sigma N^*(\Delta^n)$$

\noindent We show that $P_n$ is a strict acyclic fibration. First, for every $n\geq 0$, the morphism $P_n$ is strict as it is the composite of strict morphisms. Moreover, we have the identity $\Theta\otimes id=\Psi_\bullet P_\bullet$, which shows that $P_n$ is acyclic for every $n\geq 0$. We now prove that $P_n$ is surjective. For every $b\in B$ and $\underline{x}\in \Sigma N^*(\Delta^n)$, we have $P_n(0,b\otimes\underline{0}^\vee\otimes\underline{x})=b\otimes\underline{x}$ which proves that $P_n$ is surjective for every $n\geq 0$. By Lemma \ref{decoupage}, we have that $\mathcal{MC}(P_\bullet)$ is a weak equivalence. Finally, since we have $\Theta\otimes id=\Psi_\bullet P_\bullet$, it follows that $\mathcal{MC}(\Theta\otimes id)$ is also a weak equivalence, which proves the theorem.
\end{proof}

\subsection{Comparison with the deformation theory of shifted $\mathcal{L}ie_\infty$-algebras}\label{sec:244}

Let $\mathcal{L}ie_\infty$ be an operad which encodes Lie algebras up to homotopy, for instance $\mathcal{L}ie_\infty=B^c(\Lambda^{-1}\mathcal{C}om^\vee)$. We call a $\Lambda\mathcal{L}_\infty$-algebra any algebra over the operad $\Lambda\mathcal{L}ie_\infty$. These algebras have been widely studied in the literature. Recall (for instance from \cite[$\mathsection$2]{rogerslieinf}, or \cite{berglund} for the non-shifted analogue) that giving a $\Lambda\mathcal{L}_\infty$-algebra structure on a graded $\mathbb{K}$-module $V$ is equivalent to giving degree $-1$ brackets
$$[\underbrace{-,\ldots,-}_n]:(V^{\otimes n})_{\Sigma_n}\longrightarrow V$$

\noindent for every $n\geq 0$ such that we have the higher Jacobi relations:
$$\sum_{k=1}^n\sum_{\sigma\in\text{\normalfont Sh}(k,n-k)}\pm[[x_{\sigma(1)},\ldots,x_{\sigma(k)}],x_{\sigma(k+1)},\ldots,x_{\sigma(n)}]=0$$

\noindent for every $x_1,\ldots,x_n\in V$. In particular, the $0$-bracket $d:=[-]$ is a differential.\\ 

\begin{prop}\label{prelielie}
    There exists an operad morphism $\mathcal{L}ie_\infty\longrightarrow\mathcal{P}re\mathcal{L}ie_\infty$ which fits in the following commutative diagram:
\[\begin{tikzcd}
	{\mathcal{L}ie_\infty} & {\mathcal{P}re\mathcal{L}ie_\infty} \\
	{\mathcal{L}ie} & {\mathcal{P}re\mathcal{L}ie}
	\arrow[from=1-1, to=1-2]
	\arrow[from=1-1, to=2-1]
	\arrow[from=1-2, to=2-2]
	\arrow[from=2-1, to=2-2]
\end{tikzcd}.\]

In particular, every $\Lambda\mathcal{PL}_\infty$-algebra is a $\Lambda\mathcal{L}_\infty$-algebra with the brackets
$$[x_1,\ldots,x_n]=\sum_{i=1}^n\pm x_i\llbrace x_1,\ldots,\widehat{x_i},\ldots, x_n\rrbrace.$$
\end{prop}

\begin{proof}
    We have an operad morphism $\text{\normalfont Perm}\longrightarrow\mathcal{C}om$ defined by $e_i^n\longrightarrow 1$ for every $n\geq 1$ and $1\leq i\leq n$. By duzalization, this gives a cooperad morphism $\mathcal{C}om^\vee\longrightarrow\text{\normalfont Perm}^\vee$ defined by $1\longmapsto\sum_{i=1}^n(e_i^n)^\vee$ for every $n\geq 1$. Taking the cobar construction then gives a well-defined morphism $\mathcal{L}ie_\infty\longrightarrow\mathcal{P}re\mathcal{L}ie_\infty$. The commutativity of the square comes from immediate computation. The relation between the $\Lambda\mathcal{PL}_\infty$-algebra structure and its induced $\Lambda\mathcal{L}_\infty$-algebra structure comes from the morphism $\mathcal{C}om^\vee\longrightarrow\text{\normalfont Perm}^\vee$.
\end{proof}


Proposition \ref{prelielie} implies that every complete $\Lambda\mathcal{PL}_\infty$-algebra is endowed with the structure of a complete $\Lambda\mathcal{L}_\infty$-algebra. For every $\Lambda\mathcal{PL}_\infty$-algebra $A$, we denote by $L(A)$ the underlying $\Lambda\mathcal{L}_\infty$-algebra structure on $A$.\\

From now, we work over a field $\mathbb{K}$ with $char(\mathbb{K})=0$. Using Theorem \ref{morph}, we can use the deformation theory developed in \cite{rogerslie} for $\Lambda\mathcal{L}_\infty$-algebras (called $L[1]_\infty$-algebras in this reference). Following \cite[$\mathsection$5.6]{rogerslie}, for every Lie algebra $L$, we set
$$\mathcal{MC}_\bullet(L)=\mathcal{MC}( L\widehat{\otimes}\Sigma\Omega^\ast(\Delta^\bullet)),$$

\noindent where $\Omega^\ast(\Delta^n)$ denotes the dg associative and commutative algebra of polynomial De Rham forms on the simplex $\Delta^n$, and where we consider, on the right hand-side, the Maurer-Cartan set of the $\Lambda\mathcal{L}_\infty$-algebra $ L\otimes\Sigma\Omega^\ast(\Delta^\bullet)$ (see \cite[$\mathsection$5.4]{rogerslie}).\\

Note that since $char(\mathbb{K})=0$, the category of $\Gamma(\mathcal{P}re\mathcal{L}ie_\infty,-)$-algebras is equivalent to the category of $\mathcal{P}re\mathcal{L}ie_\infty$-algebras, so that $\Gamma\Lambda\mathcal{PL}_\infty=\Lambda\mathcal{PL}_\infty$. The goal of this subsection is to prove that, for every complete brace algebra $A$, the simplicial sets $\mathcal{MC}_\bullet(A)$ and $\mathcal{MC}_\bullet(L(A))$ are weakly equivalent.\\

In the following, we distinguish the $\Lambda\mathcal{PL}_\infty$-algebra structure with the $\Lambda\mathcal{L}_\infty$-algebra structure. More precisely, for every complete $\Lambda\mathcal{PL}_\infty$-algebra $V$, we set:
    $$\mathcal{MC}^{\Lambda\mathcal{PL}_\infty}(V) = \mathcal{MC}(V)\ ;\ 
        \mathcal{MC}^{\Lambda\mathcal{L}_\infty}(V)=\mathcal{MC}(L(V)).$$

\noindent We also set, for every complete brace algebra $A$,
$$\mathcal{MC}_\bullet^{\Lambda\mathcal{PL}_\infty}(A)=\mathcal{MC}_\bullet(A)\ ;\ \mathcal{MC}^{\Lambda\mathcal{L}_\infty}_\bullet(A)=\mathcal{MC}_\bullet(L(A)),$$

\noindent where $L(A)$ is the Lie algebra endowed with the bracket $[x,y]=x\langle y\rangle-(-1)^{|x||y|}y\langle x\rangle$.

\begin{lm}
    Let $V$ be a $\widehat{\Lambda\mathcal{PL}_\infty}$-algebra. Then
    $$\mathcal{MC}^{\Lambda\mathcal{L}_\infty}(V)=\mathcal{MC}^{\Lambda\mathcal{PL}_\infty}(V).$$
\end{lm}

\begin{proof}
    By Proposition \ref{prelielie}, we have, for every $x_1,\ldots,x_n\in V$,
    $$[x_1,\ldots,x_n]=\sum_{k=1}^n\pm x_k\llbrace x_1,\ldots,\widehat{x_k},\ldots,x_n\rrbrace.$$

    \noindent Then, the Maurer-Cartan equation
    $$d(x)+\sum_{n\geq 1}\frac{1}{n!}x\llbrace\underbrace{x,\ldots,x}_n\rrbrace=0$$

    \noindent is equivalent to the equation
    $$d(x)+\sum_{n\geq 2}\frac{1}{n!}[\underbrace{x,\ldots,x}_n]=0,$$

    \noindent which is precisely the Maurer-Cartan equation in complete $\Lambda\mathcal{L}ie_\infty$-algebras.
\end{proof}

Let $A$ be a complete brace algebra, $B$ a dg commutative and associative algebra, and $E$ be a $\mathcal{E}$-algebra. Then the tensor products $( A\widehat{\otimes} B)\widehat{\otimes}\Sigma E$ and $(A\widehat{\otimes} E)\widehat{\otimes}\Sigma  B$ are endowed with a complete $\Lambda\mathcal{L}_\infty$-algebra structure. Indeed, the first one is induced by the composite
\[\begin{tikzcd}
	{\Lambda\mathcal{L}ie_\infty} & {\Lambda\mathcal{B}race\underset{\text{\normalfont H}}{\otimes}\mathcal{E}} & {(\mathcal{B}race\underset{\text{\normalfont H}}\otimes\mathcal{C}om)\underset{\text{\normalfont H}}{\otimes}\Lambda\mathcal{E}}
	\arrow[from=1-1, to=1-2]
	\arrow["\simeq", from=1-2, to=1-3]
\end{tikzcd},\]

\noindent while the second one is induced by the composite
\[\begin{tikzcd}
	{\Lambda\mathcal{L}ie_\infty} & {\Lambda\mathcal{B}race\underset{\text{\normalfont H}}{\otimes}\mathcal{E}} & {(\mathcal{B}race\underset{\text{\normalfont H}}\otimes\mathcal{E})\underset{\text{\normalfont H}}{\otimes}\Lambda\mathcal{C}om}
	\arrow[from=1-1, to=1-2]
	\arrow["\simeq", from=1-2, to=1-3]
\end{tikzcd}.\]

\begin{lm}
    The isomorphism
\[\begin{tikzcd}
	{( A\widehat{\otimes} B)\widehat{\otimes}\Sigma E} & {(A\widehat{\otimes} E)\widehat{\otimes} \Sigma B}
	\arrow["\simeq", from=1-1, to=1-2]
\end{tikzcd}\]
\noindent which exchanges $E$ and $B$ is an isomorphism of complete $\Lambda\mathcal{L}_\infty$-algebras.
\end{lm}

\begin{proof}
    Straightforward computations.
\end{proof}

We thus obtain the following theorem.

\begin{thm}
    Let $A$ be a complete brace algebra. Then there exists a simplicial set $\mathcal{S}_\bullet^A$ and a zig-zag of weak equivalences in simplicial sets:
\[\begin{tikzcd}
	{\mathcal{MC}_\bullet^{\Lambda\mathcal{PL}_\infty}(A)} & {\mathcal{S}_\bullet^A} & {\mathcal{MC}_\bullet^{\Lambda\mathcal{L}_\infty}(A)}
	\arrow["\sim", from=1-1, to=1-2]
	\arrow["\sim"', from=1-3, to=1-2]
\end{tikzcd}.\]
\end{thm}

One major consequence of this theorem is that the homotopy groups that we have computed are isomorphic to the one's found in \cite{berglund} if the field is of characteristic $0$.

\begin{proof}
    We first remark that, for any $n\geq 0$, we have a morphism of complete brace algebras which is a weak equivalence:
\[\begin{tikzcd}
	A & {A\widehat{\otimes}\Omega^\ast(\Delta^n).}
	\arrow["\sim", from=1-1, to=1-2]
\end{tikzcd}\]

\noindent By Theorem \ref{gm}, we obtain a weak equivalence
\[\begin{tikzcd}
	{\mathcal{MC}^{\Lambda\mathcal{PL}_\infty}_\bullet(A)} & {\mathcal{MC}_\bullet^{\Lambda\mathcal{PL}_\infty}(A\widehat{\otimes}\Omega^\ast(\Delta^n))=\mathcal{MC}^{\Lambda\mathcal{PL}_\infty}((A\widehat{\otimes}\Omega^\ast(\Delta^n))\widehat{\otimes}\Sigma  N^*(\Delta^\bullet)).}
	\arrow["\sim", from=1-1, to=1-2]
\end{tikzcd}\]

\noindent We now apply \cite[Chapter IV, Proposition 1.9]{goerss}. Recall that the \textit{diagonal} of a bisimplicial set $X$ (see \cite[Chapter IV, $\mathsection$1]{goerss}) is the simplicial set $Diag(X)$ defined by
$$Diag(X)_n=X_{nn}$$

\noindent for every $n\geq 0$. Since we have a point-wise weak equivalence, this extends to the following weak-equivalence of simplicial sets
\[\begin{tikzcd}
	{\mathcal{MC}^{\Lambda\mathcal{PL}_\infty}_\bullet(A)} & {Diag(\mathcal{MC}^{\Lambda\mathcal{PL}_\infty}(( A\widehat{\otimes}\Omega^\ast(\Delta^\bullet))\widehat{\otimes} \Sigma N^*(\Delta^\bullet)))}
	\arrow["\sim", from=1-1, to=1-2]
\end{tikzcd},\]

\noindent Similarly, by \cite[Theorem 1.1]{rogers2}, we have a weak equivalence
\[\begin{tikzcd}
	{Diag(\mathcal{MC}^{\Lambda\mathcal{L}_\infty}(( A\otimes N^\ast(\Delta^\bullet))\otimes \Sigma\Omega^*(\Delta^\bullet)))} & {\mathcal{MC}_\bullet^{\Lambda\mathcal{L}_\infty}(A)}
	\arrow["\sim"', from=1-2, to=1-1]
\end{tikzcd}.\]

\noindent By combining the above weak equivalences with the two previous lemmas, we obtain the following diagram:
\[\begin{tikzcd}
	{\mathcal{MC}_\bullet^{\Lambda\mathcal{PL}_\infty}(A)} & {Diag(\mathcal{MC}^{\Lambda\mathcal{PL}_\infty}((A\otimes\Omega^\ast(\Delta^\bullet))\otimes\Sigma N^*(\Delta^\bullet)))} \\
	& {Diag(\mathcal{MC}^{\Lambda\mathcal{L}_\infty}((A\otimes\Omega^\ast(\Delta^\bullet))\otimes\Sigma N^*(\Delta^\bullet)))} \\
	& {Diag(\mathcal{MC}^{\Lambda\mathcal{L}_\infty}(( A\otimes N^\ast(\Delta^\bullet))\otimes\Sigma \Omega^*(\Delta^\bullet)))} & {\mathcal{MC}_\bullet^{\Lambda\mathcal{L}_\infty}(L(A))}
	\arrow["\sim", from=1-1, to=1-2]
	\arrow["{=}", no head, from=1-2, to=2-2]
	\arrow["\simeq", from=2-2, to=3-2]
	\arrow["\sim"', from=3-3, to=3-2]
\end{tikzcd},\]

\noindent which proves the theorem.
\end{proof}

\section{A mapping space in the category of non-symmetric operads}

In this section, we give an explicit construction of a mapping space $\text{\normalfont Map}_{\mathcal{O}p}(B^c(\mathcal{C}),\mathcal{P})$ in the category of non symmetric operads in terms of $\Gamma\Lambda\mathcal{PL}_\infty$ operations. Explicitly, we give a construction of a mapping space as the simplicial Maurer-Cartan set associated to the complete brace algebra $\text{\normalfont Hom}_{\text{\normalfont Seq}_\mathbb{K}}(\overline{\mathcal{C}},\overline{\mathcal{P}})$.\\

In $\mathsection$\ref{sec:251}, we make recollections on the construction of the free operad functor and on the model structure used for operads in this memoir. In this memoir, we use an explicit description of the free operad functor in terms of trees with inputs, which we define in this section.\\

In $\mathsection$\ref{sec:253}, we give an explicit construction of a cosimplicial frame associated to the cobar construction $B^c(\mathcal{C})$ of a coaugmented non symmetric cooperad as a sequence.\\

In $\mathsection$\ref{sec:254}, we finally prove Theorem \ref{theoremG}, which gives a description of a mapping space $\text{\normalfont Map}_{\mathcal{O}p}(B^c(\mathcal{C}),\mathcal{P})$ in the category of non symmetric operads as the simplicial Maurer-Cartan set associated to the brace algebra $\text{\normalfont Hom}_{\text{\normalfont Seq}_\mathbb{K}}(\overline{\mathcal{C}},\overline{\mathcal{P}})$. This gives a computation of the connected components and the homotopy groups of $\text{\normalfont Map}_{\mathcal{O}p}(B^c(\mathcal{C}),\mathcal{P})$ by using Theorem \ref{theoremE}.

\subsection{The free operad functor and the model structure on $\mathcal{O}p$}\label{sec:251}

We first recall the definition of the free operad functor and the model structure on operads. We will mostly follow conventions of \cite{muro}. Let $\text{\normalfont Seq}_\mathbb{K}$ be the category of sequences in $\text{dgMod}_\mathbb{K}$. Recall that we have an obvious model structure on $\text{\normalfont Seq}_\mathbb{K}$ which is defined arity wise, using the standard model structure on $\text{dgMod}_\mathbb{K}$.\\

\noindent The model structure on the category of non symmetric operads $\mathcal{O}p$ is obtained by transferring the model structure of $\text{\normalfont Seq}_\mathbb{K}$ from an adjunction
\[\begin{tikzcd}
	{\mathcal{F}:\text{\normalfont Seq}_\mathbb{K}} & {\mathcal{O}p:\omega}
	\arrow[shift left, from=1-1, to=1-2]
	\arrow[shift left, from=1-2, to=1-1]
\end{tikzcd},\]
    
\noindent where $\omega:\mathcal{O}p\longrightarrow\text{\normalfont Seq}_\mathbb{K}$ is the functor which forgets the operad structure. The left adjoint $\mathcal{F}:\text{\normalfont Seq}_\mathbb{K}\longrightarrow\mathcal{O}p$ is the \textit{free operad} functor, for which we recall the construction.\\

We define the notion of tree with inputs, which is analogue to the notion of "planted planar tree with inputs" given in \cite[Definition 3.4]{muro}.

\begin{defi}\label{treeleave}
    Let $n\geq 0$. A (planar) {\normalfont tree with inputs} is the data of a tree ${T}\in\mathcal{PRT}(n)$ and, for each vertex of $T$, an integer which represents the number of ingoing arrows, which may includes some edges of $T$. We also add an outgoing arrow on the root of $T$. The ingoing arrows with only one vertex of $T$ are called the {\normalfont inputs} of the tree. 
    \begin{center}
        \begin{tikzpicture}[baseline={([yshift=-.5ex]current bounding box.center)},scale=0.8]
    \node (0) at (0,-1) {};
    \node[draw,circle,scale=0.65] (i) at (0,0) {$3$};
    \node[draw,circle,scale=0.65] (1) at (-1.5,1) {$1$};
    \node[draw,circle,scale=0.65] (1b) at (-2.5,2) {$4$};
    \node (1bc) at (-3,3) {};
    \node (1bcc) at (-2.5,3) {};
    \node (1bccc) at (-2,3) {};
    \node[draw,circle,scale=0.65] (1bb) at (-0.5,2) {$5$};
    \node (1bbc) at (-1,3) {};
    \node (1bbcc) at (0,3) {};
    \node (2) at (0,1) {};
    \node[draw,circle,scale=0.65] (3) at (1.25,1) {$2$};
    \node (3b) at (0.25,2) {};
    \node (3bb) at (2.25,2) {};
    \draw[stealth-] (0) -- (i);
    \draw[-stealth] (2) -- (i);
    \draw[stealth-] (1) -- (1b);
    \draw[stealth-] (1) -- (1bb);
    \draw[stealth-] (i) -- (1);
    \draw[stealth-] (i) -- (3);
    \draw[stealth-] (1b) -- (1bc);
    \draw[stealth-] (1b) -- (1bcc);
    \draw[stealth-] (1b) -- (1bccc);
    \draw[stealth-] (1bb) -- (1bbcc);
    \draw[stealth-] (1bb) -- (1bbc);
    \end{tikzpicture}.
\end{center}
    We usually denote by $\underline{T}$ any tree with inputs with underlying tree $T\in\mathcal{PRT}$. We call $T$ the {\normalfont shape} of $\underline{T}$, and set $\text{\normalfont Shape}(\underline{T})=T$. We also set $V_{\underline{T}}=V_T$. For every vertex $v\in V_{\underline{T}}$, we denote by $\text{\normalfont val}_{\underline{T}}(v)$ the number of ingoing arrows which go to $v$. We denote by $\underline{\mathcal{PRT}_k}(n)$ the set of trees with $n$ vertices and $k$ inputs and $\underline{\mathcal{T}ree_k}(n)=\mathbb{K}[\underline{\mathcal{PRT}_k}(n)]$.
\end{defi}

As in Definition \ref{canonicaltree}, we can consider trees with inputs $\underline{T}\in\underline{\mathcal{PRT}_k}(a_1<\cdots <a_n)$ in a general totally ordered finite set $a_1<\cdots <a_n$. We say that $\underline{T}$ is \textit{canonical} (or \textit{in the canonical order}) if its shape $\text{\normalfont Shape}(\underline{T})\in\mathcal{PRT}(a_1<\cdots <a_n)$ is canonical.\\

For every tree $\underline{T}\in\underline{\mathcal{PRT}_k}(n)$, we endow the inputs with the canonical labeling from $1$ to $k$ obtained by following the canonical order of $T$. For instance, the tree given in the definition is seen as
\begin{center}
        \begin{tikzpicture}[baseline={([yshift=-.5ex]current bounding box.center)},scale=0.8]
    \node (0) at (0,-1) {};
    \node[draw,circle,scale=0.65] (i) at (0,0) {$3$};
    \node[draw,circle,scale=0.65] (1) at (-1.5,1) {$1$};
    \node[draw,circle,scale=0.65] (1b) at (-2.5,2) {$4$};
    \node[scale=0.8] (1bc) at (-3,3) {$1$};
    \node[scale=0.8] (1bcc) at (-2.5,3) {$2$};
    \node[scale=0.8] (1bccc) at (-2,3) {$3$};
    \node[draw,circle,scale=0.65] (1bb) at (-0.5,2) {$5$};
    \node[scale=0.8] (1bbc) at (-1,3) {$4$};
    \node[scale=0.8] (1bbcc) at (0,3) {$5$};
    \node[scale=0.8] (2) at (0,1) {$6$};
    \node[draw,circle,scale=0.65] (3) at (1.25,1) {$2$};
    \draw[stealth-] (0) -- (i);
    \draw[-stealth] (2) -- (i);
    \draw[stealth-] (1) -- (1b);
    \draw[stealth-] (1) -- (1bb);
    \draw[stealth-] (i) -- (1);
    \draw[stealth-] (i) -- (3);
    \draw[stealth-] (1b) -- (1bc);
    \draw[stealth-] (1b) -- (1bcc);
    \draw[stealth-] (1b) -- (1bccc);
    \draw[stealth-] (1bb) -- (1bbcc);
    \draw[stealth-] (1bb) -- (1bbc);
    \end{tikzpicture}\ .
\end{center}

As for trees in $\mathcal{PRT}$, we have the following definitions.

\begin{defi}
    Let $\underline{T}$ be a tree with inputs and with underlying shape $T\in\mathcal{PRT}$.
    
    \begin{itemize}
        \item A {\normalfont subtree} of $\underline{T}$ is the data of a subtree $S$ of $T$, endowed with the unique choice of arrows such that, for every $v\in V_{\underline{S}}\subset V_{\underline{T}}$, we have $\text{\normalfont val}_{\underline{S}}(v)=\text{\normalfont val}_{\underline{T}}(v)$.
        \item If $\underline{S}$ is a subtree of $\underline{T}$, we denote by $\underline{T}/\underline{S}$ the tree of shape $T/S$ obtained by contracting the tree $\underline{S}$ on the tree with only one vertex, denoted by $S$, with the same number of inputs as $\underline{S}$.
    \end{itemize}
\end{defi}

For instance, if we consider the above tree with inputs $\underline{T}$, then the following tree with inputs
    $$\underline{S}=\begin{tikzpicture}[baseline={([yshift=-.5ex]current bounding box.center)},scale=0.8]
    \node (0) at (0,-1) {};
    \node[draw,circle,scale=0.65] (i) at (0,0) {$3$};
    \node[draw,circle,scale=0.65] (1) at (-1.5,1) {$1$};
    \node[scale=0.8] (1b) at (-2.5,2) {$1$};
    \node[scale=0.8] (1bb) at (-0.5,2) {$2$};
    \node[scale=0.8] (2) at (0,1) {$3$};
    \node[scale=0.8] (3) at (1.25,1) {$4$};
    \draw[stealth-] (0) -- (i);
    \draw[-stealth] (2) -- (i);
    \draw[stealth-] (1) -- (1b);
    \draw[stealth-] (1) -- (1bb);
    \draw[stealth-] (i) -- (1);
    \draw[stealth-] (i) -- (3);
    \end{tikzpicture}$$

    \noindent is a subtree of $\underline{T}$ such that
    $$\underline{T}/\underline{S}=\begin{tikzpicture}[baseline={([yshift=-.5ex]current bounding box.center)},scale=0.8]
    \node (0) at (0,-1) {};
    \node[draw,circle,scale=0.65] (i) at (0,0) {$S$};
    \node[draw,circle,scale=0.65] (1b) at (-1.5,1) {$4$};
    \node[scale=0.8] (1bc) at (-2,2) {$1$};
    \node[scale=0.8] (1bcc) at (-1.5,2) {$2$};
    \node[scale=0.8] (1bccc) at (-1,2) {$3$};
    \node[draw,circle,scale=0.65] (1bb) at (-0.5,1) {$5$};
    \node[scale=0.8] (1bbc) at (-0.75,2) {$4$};
    \node[scale=0.8] (1bbcc) at (-0.25,2) {$5$};
    \node[scale=0.8] (2) at (0.5,1) {$6$};
    \node[draw,circle,scale=0.65] (3) at (1.5,1) {$2$};
    \draw[stealth-] (0) -- (i);
    \draw[-stealth] (2) -- (i);
    \draw[stealth-] (i) -- (3);
    \draw[stealth-] (i) -- (1b);
    \draw[stealth-] (i) -- (1bb);
    \draw[stealth-] (1b) -- (1bc);
    \draw[stealth-] (1b) -- (1bcc);
    \draw[stealth-] (1b) -- (1bccc);
    \draw[stealth-] (1bb) -- (1bbcc);
    \draw[stealth-] (1bb) -- (1bbc);
    \end{tikzpicture}.$$

Let $p,q,n,m\geq 0,1\leq i\leq p$ and $\underline{U}\in \underline{\mathcal{T}ree_p}(n),\underline{V}\in \underline{\mathcal{T}ree_q}(m)$. We let $\underline{U}\circ_i\underline{V}$ to be the tree in $\underline{\mathcal{T}ree_{p+q-1}}(n+m)$ given by the attachment of the unique outgoing arrow of $\underline{V}$ to the $i$-th ingoing arrow of $\underline{U}$. This defines a morphism
$$\circ_i:\underline{\mathcal{T}ree_p}(n)\otimes \underline{\mathcal{T}ree_q}(m)\longrightarrow \underline{\mathcal{T}ree_{p+q-1}}(n+m).$$

\begin{lm}\label{freeop}
    Let $\underline{\mathcal{T}ree}$ be the sequence defined by
    $$\underline{\mathcal{T}ree}(k)=\bigoplus_{n\geq 0}\underline{\mathcal{T}ree_k}(n).$$

    \noindent Then the morphisms $\circ_i:\underline{\mathcal{T}ree_p}\otimes \underline{\mathcal{T}ree_q}\longrightarrow \underline{\mathcal{T}ree_{p+q-1}}$ endow the sequence $\underline{\mathcal{T}ree}$ with the structure of an operad.
\end{lm}

Using this notion of tree, we set
$$\mathcal{F}(M)(k)=\bigoplus_{n\geq 0}\left(\bigoplus_{\underline{T}\in \underline{\mathcal{PRT}_k}(n)}\underline{T}\otimes\bigotimes_{i=1}^n M(\text{\normalfont val}_{\underline{T}}(i))\right)_{\Sigma _{n}}$$

\noindent where, in the sum, we consider the action of $\Sigma_n$ on $\underline{\mathcal{PRT}_k}(n)$ by the permutation of the labels of the vertices, and the action of $\Sigma _{n}$ on $\bigotimes_{i=1}^nM(\text{\normalfont val}_{\underline{T}}(i))$ by permutations. The operadic structure of $\mathcal{F}(M)$ is given by the operadic structure of $\underline{\mathcal{T}ree}$ and the concatenation of the elements in $M$. We denote by $\mathcal{F}_{(\underline{T})}(M)=\bigotimes_{i=1}^nM(\text{\normalfont val}_{\underline{T}}(i))$ the $\underline{T}$-component of $\mathcal{F}(M)$ associated to some tree $\underline{T}\in \underline{\mathcal{PRT}_k}(n)$.\\


We can check that the functor $\mathcal{F}:\text{\normalfont Seq}_\mathbb{K}\longrightarrow\mathcal{O}p$ is left adjoint to the forgetful functor $\omega:\mathcal{O}p\longrightarrow\text{\normalfont Seq}_\mathbb{K}$ which forgets the operad structure:
\[\begin{tikzcd}
	{\mathcal{F}:\text{\normalfont Seq}_\mathbb{K}} & {\mathcal{O}p:\omega}
	\arrow[shift left, from=1-1, to=1-2]
	\arrow[shift left, from=1-2, to=1-1]
\end{tikzcd}.\]

\noindent This adjunction implies the following result.

\begin{prop}[{see \cite[Theorem 1.1]{muro}}]
    The category $\mathcal{O}p$ is endowed with a cofibrantly generated model structure such that the forgetful functor $\omega:\mathcal{O}p\longrightarrow\text{\normalfont Seq}_\mathbb{K}$ creates weak-equivalences and fibrations. Cofibrations are given by the left lifting property with respect to acyclic fibrations.
\end{prop}

\begin{remarque}
In the following subsections, we also use the notion of cofree cooperad generated by a sequence $M$ such that $M(0)=0$. For every $k\geq 0$ and $n\geq 1$, let $\underline{\mathcal{PRT}_k}^0(n)$ be the subset of $\underline{\mathcal{PRT}_k}(n)$ given by trees $\underline{T}$ such that $\text{\normalfont val}_{\underline{T}}(v)\neq 0$ for every $v\in V_{\underline{T}}$. Let $\underline{\mathcal{T}ree_k}^0(n):=\mathbb{K}[\underline{\mathcal{PRT}_k}^0(n)]$. Then the sequence $\underline{\mathcal{T}ree}^0$ defined by
$$\underline{\mathcal{T}ree}^0(n)=\bigoplus_{k\geq 0}\underline{\mathcal{T}ree_k}^0(n)$$

\noindent is a suboperad of $\underline{\mathcal{T}ree}$ such that, for every $n\geq 0$, the $\mathbb{K}$-module $\underline{\mathcal{T}ree}^0(n)$ is finite dimensional. By Remark \ref{coopdual}, the dual symmetric sequence $(\underline{\mathcal{T}ree}^0)^\vee$ is endowed with the structure of a cooperad. We then define
$$\mathcal{F}^c(M)(k)=\bigoplus_{n\geq 1}\left(\bigoplus_{\underline{T}\in\underline{\mathcal{PRT}_k}^0(n)}\underline{T}^\vee\otimes\bigotimes_{i=1}^nM(\text{\normalfont val}_{\underline{T}}(i))\right)^{\Sigma_{n}}$$

\noindent where we consider the action of $\Sigma_n$ on $\underline{T}^\vee$ by permutation of the vertices, and the action of $\Sigma_n$ on $\bigotimes_{i=1}^n M(\text{\normalfont val}_{\underline{T}}(i))$ by permutations. We endow $\mathcal{F}^c(M)$ with the cooperad structure given by the cooperadic structure of $(\underline{\mathcal{T}ree}^0)^\vee$, and by the deconcatenation coproduct in the tensor coalgebra of $\bigoplus_{n\geq 1}M(n)$. As for operads, we have an adjunction
\[\begin{tikzcd}
	{\omega:\mathcal{O}p^c} & {\text{\normalfont Seq}_\mathbb{K}:\mathcal{F}^c}
	\arrow[shift left, from=1-1, to=1-2]
	\arrow[shift left, from=1-2, to=1-1]
\end{tikzcd}\]

\noindent where $\omega:\mathcal{O}p^c\longrightarrow\text{\normalfont Seq}_\mathbb{K}$ is the functor which forgets the cooperad structure.
\end{remarque}


We will need to consider operadic compositions (resp. cooperadic cocompositions) shaped on trees with inputs. This can be formalized as follows. Let $\mathcal{P}$ be an augmented operad $\mathcal{P}\simeq I\oplus\overline{\mathcal{P}}$ and $\mathcal{C}$ be a coaugmented cooperad $\mathcal{C}\simeq I\oplus\overline{\mathcal{C}}$ such that $\mathcal{P}(0)=\mathcal{C}(0)=0$ and $\mathcal{P}(1)=\mathcal{C}(1)=\mathbb{K}$. By the universal property satisfied by $\mathcal{F}$, we have a unique operad morphism $\mathcal{F}(\mathcal{P})\longrightarrow\mathcal{P}$ which reduces to the identity on $\mathcal{P}\subset\mathcal{F}(\mathcal{P})$. Analogously, we have a unique cooperad morphism $\mathcal{C}\longrightarrow\mathcal{F}^c(\mathcal{C})$ whose projection on $\mathcal{C}$ is given by the identity on $\mathcal{C}$.

\begin{defi}\label{compodeltaarbres}
    Let $k\geq 1$ and $\underline{T}\in\underline{\mathcal{PRT}_k}^0$. We define $\gamma_{(\underline{T})}:\mathcal{F}_{(\underline{T})}(\overline{\mathcal{P}})\longrightarrow\overline{\mathcal{P}}$ and $\Delta_{(\underline{T})}:\overline{\mathcal{C}}\longrightarrow\mathcal{F}^c_{(\underline{T})}(\overline{\mathcal{C}})$ by the composites
\[\begin{tikzcd}
	{\gamma_{(\underline{T})}:\mathcal{F}_{(\underline{T})}(\overline{\mathcal{P}})} & {\mathcal{F}_{(\underline{T})}(\mathcal{P})} & {\mathcal{P}} & {\overline{\mathcal{P}}}
	\arrow[hook, from=1-1, to=1-2]
	\arrow["\gamma", from=1-2, to=1-3]
	\arrow[two heads, from=1-3, to=1-4]
\end{tikzcd};\]
\[\begin{tikzcd}
	{\Delta_{(\underline{T})}:\overline{\mathcal{C}}} & {\mathcal{C}} & {\mathcal{F}^c_{(\underline{T})}(\mathcal{C})} & {\mathcal{F}^c_{(\underline{T})}(\overline{\mathcal{C}})}
	\arrow[hook, from=1-1, to=1-2]
	\arrow["\Delta", from=1-2, to=1-3]
	\arrow[two heads, from=1-3, to=1-4]
\end{tikzcd}.\]
\end{defi}

For every $p,q,n,m\geq 0$ and $1\leq i\leq p$, we define a morphism
$$\bullet_i:\underline{\mathcal{T}ree_p}(n)\otimes\underline{\mathcal{T}ree_q}(m)\longrightarrow\underline{\mathcal{T}ree_{p}}(n+m-1)$$

\noindent by the following. Let $\underline{U}\in\underline{\mathcal{PRT}_p}(n)$ and $\underline{V}\in\underline{\mathcal{PRT}_q}(m)$. If the number of arrows on the $i$-th vertex of $\underline{U}$ is not $q$, then $\underline{U}\bullet_i\underline{V}=0$. Else, we define $\underline{U}\bullet_i\underline{V}$ as the unique tree obtained by putting $\underline{V}$ in the $i$-th vertex of $\underline{U}$, and attaching the ingoing arrows of the $i$-th vertex of $\underline{U}$ into the inputs of $\underline{V}$.

\begin{lm}
    Let $\underline{T}$ be a tree with inputs and $\underline{S}$ be a subtree of $\underline{T}$. Then $\underline{T}/\underline{S}\bullet_S\underline{S}=\underline{T}$.
\end{lm}

\begin{proof}
    It is an immediate consequence of the definitions.
\end{proof}

The morphisms defined in Definition \ref{compodeltaarbres} also behave well with the compositions $\circ_i$ and $\bullet_i$.

\begin{lm}\label{compbullet}
    Let $k\geq 1$, $\underline{T}\in\underline{\mathcal{PRT}_k}^0(n)$ and $\underline{S}\subset\underline{T}$. Then
$$\gamma_{(\underline{T}/\underline{S})}\circ_{\underline{S}}\gamma_{(\underline{S})}=\gamma_{(\underline{T})}\ ;\ \Delta_{(\underline{T}/\underline{S})}\circ_{\underline{S}}\Delta_{(\underline{S})}=\Delta_{(\underline{T})}$$

    \noindent in the endomorphism operad $\text{\normalfont End}_{\bigoplus_{n\geq 2}\mathcal{P}(n)}$ and in the coendomorphism operad $\text{\normalfont CoEnd}_{\bigoplus_{n\geq 2}\mathcal{C}(n)}$ respectively.
\end{lm}

\begin{proof}
    These are direct consequences of the (co)associativity axioms in a (co)operad.
\end{proof}

\subsection{A cosimplicial frame for $B^c(\mathcal{C})$}\label{sec:253}

Let $\mathcal{C}\simeq I\oplus\overline{\mathcal{C}}$ be a coaugmented non-symmetric cooperad with $\mathcal{C}(0)=0$ and $\mathcal{C}(1)=\mathbb{K}$. The goal of this subsection is to construct a cosimplicial frame $B^c(\mathcal{C})\otimes\Delta^n$ associated to $B^c(\mathcal{C})$. We will explicitly define $B^c(\mathcal{C})\otimes\Delta^n$ as the free operad induced by a cooperad up to homotopy that will be given by $\overline{\mathcal{C}}\otimes N_*(\Delta^n)$.\\



Let $E$ be a $\mathcal{E}$-coalgebra. We endow the operad $\mathcal{F}(\overline{\mathcal{C}}\otimes\Sigma^{-1} E)$ with a general twisting morphism such that if $E=N_*(\Delta^0)\simeq\mathbb{K}$, then $\mathcal{F}(\overline{\mathcal{C}}\otimes \Sigma^{-1}E)\simeq B^c(\mathcal{C})$. Explicitly, we construct $\beta^{E}:\overline{\mathcal{C}}\otimes \Sigma^{-1} E\longrightarrow\mathcal{F}(\overline{\mathcal{C}}\otimes\Sigma^{-1} E)$ such that the morphism $\partial_{\beta^E}:\mathcal{F}(\overline{\mathcal{C}}\otimes\Sigma^{-1} E)\longrightarrow \mathcal{F}(\overline{\mathcal{C}}\otimes\Sigma^{-1} E)$ obtained from $\beta^E$ by the Leibniz rule is a twisting morphism. If we denote by $d$ the differential induced by $\overline{\mathcal{C}}\otimes\Sigma^{-1} E$ on $\mathcal{F}^c(\overline{\mathcal{C}}\otimes \Sigma^{-1}E)$, then the morphism $\beta^E$ shall needs (see \cite{loday} or \cite{fressecyl} for instance) to be such that
$$d(\beta^E)+\partial_{\beta^E}\beta^E=0.$$

Let $k\geq 1$ and $\underline{T}\in\underline{\mathcal{PRT}_k}^0$ be a canonical tree with inputs with shape $T\in\mathcal{PRT}$. We define  $\beta^E_{(\underline{T})}:\overline{\mathcal{C}}\otimes \Sigma^{-1}E\longrightarrow\mathcal{F}_{(\underline{T})}(\overline{\mathcal{C}}\otimes\Sigma^{-1} E)$ by
$$\beta^E_{(\underline{T})}=\Delta_{(\underline{T})}\widetilde{\otimes}\Lambda\mu_T^E$$

\noindent where, for every $\mu\in\mathcal{E}(n)$, we denote by $\mu^E$ the morphism in $\text{\normalfont Hom}(E,E^{\otimes n})$ given by the $\mathcal{E}$-coalgebra structure $E$.\\

\noindent This gives a well defined morphism of sequences $\beta^E:\overline{\mathcal{C}}\otimes \Sigma^{-1}E\longrightarrow\mathcal{F}(\overline{\mathcal{C}}\otimes \Sigma^{-1}E)$ by summing over all canonical trees ${T}$. Note that such a sum of morphisms is will defined on $\overline{\mathcal{C}}\otimes\Sigma^{-1} E$ since $\overline{\mathcal{C}}(0)=\overline{\mathcal{C}}(1)=0$. It is also natural in $E$ by definition.

\begin{prop}
    The morphism $\beta^E$ defined above satisfies
    $$d(\beta^E)+\partial_{\beta^E}\beta^E=0.$$
    We thus have a derivation of operads $d+\partial_{\beta^E}$ on $\mathcal{F}(\overline{\mathcal{C}}\otimes \Sigma^{-1}E)$.
\end{prop}





\begin{proof}
    It is sufficient to prove the formula on $\overline{\mathcal{C}}\otimes\Sigma^{-1} E$. Let $\underline{T}$ be a tree with inputs with shape a canonical tree $T\in\mathcal{PRT}$. We show that the $\underline{T}$-component of the morphism $d(\beta^E)+\partial_{\beta^E}\beta^E$ is $0$. First, we have
    $$d(\beta_{(\underline{T})}^E)=\Delta_{(\underline{T})}\widetilde{\otimes}d(\Lambda \mu_T)^{\Sigma E},$$

    \noindent since the cooperadic structure on $\mathcal{C}$ is preserved by its differential. Next, we have by definition
    $$(\partial_{\beta^E}\beta^E)_{(\underline{T})}=\sum_{\underline{S}\subset\underline{T}}(\Delta_{(\underline{T}/\underline{S})}\circ_{\text{\normalfont Shape}(\underline{S})}\Delta_{(\underline{S})})\widetilde{\otimes}(\Lambda\mu_{T/\text{\normalfont Shape}(\underline{S})}\circ_{\text{\normalfont Shape}(\underline{S})}\Lambda\mu_{\text{\normalfont Shape}(\underline{S})})^{\Sigma E}.$$

    \noindent Note that taking a subtree $\underline{S}$ of $\underline{T}$ is equivalent to taking a subtree $S$ of $T$. We thus have
    $$(\partial_{\beta^E}\beta^E)_{(\underline{T})}=\sum_{S\subset T}\Delta_{(\underline{T})}\widetilde{\otimes}(\Lambda\mu_{T/S}\circ_{S}\Lambda\mu_{S})^{\Sigma E}.$$

    \noindent The proposition follows by Theorem \ref{mu}.
\end{proof}





We can now construct a cosimplicial frame for $B^c(\mathcal{C})$. Recall that the normalized chain complex $N_*(\Delta^n)$ admits a structure of a $\mathcal{E}$-coalgebra.

\begin{defi}
    Let $n\geq 0$. We set
    $$B^c(\mathcal{C})\otimes\Delta^n=(\mathcal{F}(\overline{\mathcal{C}}\otimes \Sigma ^{-1}N_*(\Delta^n)),\partial_{\beta^{N_*(\Delta^n)}}).$$
\end{defi}

We immediately see that $B^c(\mathcal{C})\otimes\Delta^\bullet$ defines a cosimplicial object in the category of non symmetric operads. By Corollary \ref{calculdiff}, we also have that $B^c(\mathcal{C})\otimes\Delta^0=B^c(\mathcal{C})$.\\

Recall from \cite[$\mathsection$3.2.2-$\mathsection$3.2.3]{fresselivre} that a \textit{cosimplicial frame} associated to $B^c(\mathcal{C})$ is a cosimplicial set $B^c(\mathcal{C})\otimes\Delta^\bullet$ such that, for every $n\geq 0$, the morphism $B^c(\mathcal{C})\otimes\Delta^n\longrightarrow B^c(\mathcal{C})\otimes\Delta^0$ is a weak equivalence and the morphism $B^c(\mathcal{C})\otimes\partial\Delta^n\longrightarrow B^c(\mathcal{C})\otimes\Delta^n$ is a cofibration for every $n\geq 0$.

\begin{thm}\label{reperecos}
    The cosimplicial object $B^c(\mathcal{C})\otimes\Delta^\bullet$ defines a cosimplicial frame for $B^c(\mathcal{C})$ in the category of operads.
\end{thm}

\begin{proof}
    Since the morphisms $N_*(\partial\Delta^n)\longrightarrow N_*(\Delta^n)$ are cofibrations for every $n\geq 0$, the morphisms $\overline{\mathcal{C}}\otimes\Sigma ^{-1}N_*(\partial\Delta^n)\longrightarrow\overline{\mathcal{C}}\otimes \Sigma ^{-1}N_*(\Delta^n)$ are cofibrations (see \cite[Proposition 1.4.13]{fressecyl}). We now prove that $B^c(\mathcal{C})\otimes\Delta^n\longrightarrow B^c(\mathcal{C})$ is a weak equivalence. We first note that $B^c(\mathcal{C})\otimes \Delta^n$ admits a natural filtration $(F_p (B^c(\mathcal{C})\otimes \Delta^n))_p$ defined by 
    $$F_p (B^c(\mathcal{C})\otimes \Delta^n)=\bigoplus_{k\geq 1}\bigoplus_{\substack{\underline{T}\in\underline{\mathcal{PRT}_k}^0\\\underline{T}\text{ canonical}\\ |\underline{T}|\geq p+1}}\underline{T}\otimes\mathcal{F}_{(\underline{T})}(\overline{\mathcal{C}}\otimes\Sigma ^{-1}N_\ast(\Delta^n)).$$
    
    \noindent By definition, the differential $\partial^n$ preserves this filtration. We thus have a spectral sequence which is convergent dimension-wise:
        $$E^{0}_q\Rightarrow H_\ast(B^c(\mathcal{C})\otimes \Delta^n)$$

        \noindent where we have set
        $$E_q^0=\bigoplus_{k\geq 1}\bigoplus_{\substack{\underline{T}\in\underline{\mathcal{PRT}_k}^0\\\underline{T}\text{ canonical}\\ |\underline{T}|= q+1}}\underline{T}\otimes\mathcal{F}_{(\underline{T})}(\overline{\mathcal{C}}\otimes\Sigma ^{-1}N_\ast(\Delta^n)).$$

        \noindent Because the twisting part $\partial^{N_*(\Delta^n)}$ increases the number of vertices, the differential is reduced to the internal differential $d$ on $E^0_q$. Because $\overline{\mathcal{C}}\otimes N_*(\Delta^n)\longrightarrow\overline{\mathcal{C}}$ is a weak equivalence, we have that the morphism $N_*(\Delta^n)\longrightarrow N_*(\Delta^0)$ induces a weak equivalence on $E_p^0$ for all $p$. It then induces a weak equivalence from $B^c(\mathcal{C})\otimes\Delta^n$ to $B^c(\mathcal{C})$.\\

        We then have the result.
\end{proof}

\subsection{Computation of $\text{\normalfont Map}_{\mathcal{O}p}(B^c(\mathcal{C}),\mathcal{P})$}\label{sec:254}

In this last subsection, we give an explicit description of a mapping space $\text{\normalfont Map}_{\mathcal{O}p}(B^c(\mathcal{C}),\mathcal{P})$. We know from Theorem \ref{reperecos} that we can set
$$\text{\normalfont Map}_{\mathcal{O}p}(B^c(\mathcal{C}),\mathcal{P})_n=\text{\normalfont Mor}_{\mathcal{O}p}(B^c(\mathcal{C})\otimes\Delta^n,\mathcal{P})$$

\noindent for every $n\geq 0$. The goal of this subsection is to link this object with some $\widehat{\Gamma\Lambda\mathcal{PL}_\infty}$-algebra structure on $\text{\normalfont Hom}_{\text{\normalfont Seq}_\mathbb{K}}(\overline{\mathcal{C}}\otimes\Sigma^{-1} N_*(\Delta^n),\overline{\mathcal{P}})$. We consider the dg $\mathbb{K}$-module
$$\mathcal{L}(\text{\normalfont Hom}(\overline{\mathcal{C}}\otimes N_*(\Delta^n),\overline{\mathcal{P}}))=\bigoplus_{k\geq 2}\text{\normalfont Hom}(\overline{\mathcal{C}}(k)\otimes N_*(\Delta^n),\overline{\mathcal{P}}(k)).$$

\noindent This dg $\mathbb{K}$-module is endowed with a filtration defined by
$$F_p(\mathcal{L}(\text{\normalfont Hom}(\overline{\mathcal{C}}\otimes N_*(\Delta^n),\overline{\mathcal{P}})))=\bigoplus_{k\geq p+1}\text{\normalfont Hom}(\overline{\mathcal{C}}(k)\otimes N_*(\Delta^n),\overline{\mathcal{P}}(k)),$$

\noindent so that $\text{\normalfont Hom}_{\text{\normalfont Seq}_\mathbb{K}}(\overline{\mathcal{C}}\otimes N_*(\Delta^n),\overline{\mathcal{P}})$ is the completion of $\mathcal{L}(\text{\normalfont Hom}(\overline{\mathcal{C}}\otimes N_*(\Delta^n),\overline{\mathcal{P}}))$ with respect to this filtration.

\begin{lm}\label{lembe}
    The dg $\mathbb{K}$-module $\mathcal{L}(\text{\normalfont Hom}(\overline{\mathcal{C}}\otimes N_*(\Delta^n),\overline{\mathcal{P}}))$ is endowed with the structure of a $\mathcal{B}race\underset{\text{\normalfont H}}{\otimes}\mathcal{E}$-algebra defined by
    $$(T\otimes\mu)(f_1,\ldots,f_k)=\pm\sum_{\underline{T}\in\text{\normalfont Shape}^{-1}(T)}\gamma_{(\underline{T})}\circ (f_{1}\otimes\cdots\otimes f_{k})\circ (\Delta_{(\underline{T})}\widetilde{\otimes}\mu^{N_*(\Delta^n)}),$$

    \noindent for any $T\in\mathcal{PRT}(k)$, $\mu\in\mathcal{E}(k)$ and $f_1,\ldots,f_k\in\mathcal{L}(\text{\normalfont Hom}(\overline{\mathcal{C}}\otimes N_*(\Delta^n),\overline{\mathcal{P}}))$ homogeneous, where we consider the tensor product $\widetilde{\otimes}$ (see Definition \ref{otimesmodif}). The sign is yielded by the commutation of $\mu$ with the $f_i$'s. Note that the sum is finite point-wise since we have supposed that $\mathcal{C}(0)=0$.
\end{lm}

\begin{proof}
    Let $\sigma \in\Sigma _n$. By definition of $\gamma_{(\underline{T})}$ and $\Delta_{(\underline{T})}$ for every $\underline{T}\in\underline{\mathcal{PRT}}$, we have
    $$(\sigma \cdot T\otimes\sigma \cdot\mu)(f_1,\ldots,f_k)=\pm (T\otimes\mu)(f_{\sigma^{-1}(1)},\ldots,f_{\sigma^{-1}(k)}),$$

    \noindent where we consider the action of $\sigma$ on $(\overline{\mathcal{C}}\otimes N_\ast(\Delta^n))^{\otimes k}$ by permutation of the tensors.\\

    We now prove the compatibility with the operadic structure. Let $p,q\geq 0$ and $U\in\mathcal{PRT}(p),V\in\mathcal{PRT}(q),\mu\in\mathcal{E}(p),\nu\in\mathcal{E}(q)$ and $1\leq i\leq p$. By Lemma \ref{compbullet}, we have
    \begin{multline*}
        (U\otimes\mu)(f_1,\ldots,f_{i-1},(V\otimes\nu)(f_i,\ldots,f_{i+q-1}),f_{i+q},\ldots,f_{p+q-1})\\=\pm\sum_{\substack{\underline{U}\in\text{\normalfont Shape}^{-1}(U)\\\underline{V}\in\text{\normalfont Shape}^{-1}(V)}}\gamma_{(\underline{U}\bullet_i\underline{V})}\circ (f_1\otimes\cdots\otimes f_{p+q-1})\circ (\Delta_{(\underline{U}\bullet_i\underline{V})}\widetilde{\otimes }(\mu\circ_i\nu)^{N_*(\Delta^n)}).
    \end{multline*}

    \noindent Now, write $U\circ_iV=T_1+\cdots+T_m$ for some $T_1,\ldots,T_m\in\mathcal{PRT}$. By definition of the composition product in $\mathcal{B}race$, for every $\underline{U}\in\text{\normalfont Shape}^{-1}(U)$ and $\underline{V}\in\text{\normalfont Shape}^{-1}(V)$, the tree $\underline{U}\bullet_i\underline{V}$ has shape $T_j$ for some $1\leq j\leq m$ given by the particular choice of attachments forced by $\underline{V}$. In the converse direction, for every tree $\underline{T}$ with shape $T_j$, there exists a unique subtree $\underline{V}\subset\underline{T}$ with shape $V$ such that $\underline{U}:=\underline{T}/\underline{V}$ has shape $U$. This then proves that
$$(U\otimes\mu)(f_1,\ldots,f_{i-1},(V\otimes\nu)(f_i,\ldots,f_{i+q-1}),f_{i+q},\ldots,f_{p+q-1})=\pm(U\circ_i V\otimes\mu\circ_i\nu)(f_1\otimes\cdots\otimes f_{p+q-1}).$$
\end{proof}

\begin{prop}
    The dg $\mathbb{K}$-module $\text{\normalfont Hom}_{\text{\normalfont Seq}_\mathbb{K}}(\overline{\mathcal{C}}\otimes\Sigma^{-1} N_*(\Delta^n),\overline{\mathcal{P}})$ is endowed with the structure of a $\widehat{\Gamma\Lambda\mathcal{PL}_\infty}$-algebra.
\end{prop}

\begin{proof}
    By Lemma \ref{lembe} and Theorem \ref{morph}, the dg $\mathbb{K}$-module $\Sigma \mathcal{L}(\text{\normalfont Hom}(\overline{\mathcal{C}}\otimes N_*(\Delta^n),\overline{\mathcal{P}}))$ is endowed with the structure of a $\Gamma\Lambda\mathcal{PL}_\infty$-algebra. By taking the completion, we obtain that $\text{\normalfont Hom}(\overline{\mathcal{C}}\otimes\Sigma^{-1} N_*(\Delta^n),\overline{\mathcal{P}})$ is a $\widehat{\Gamma\Lambda\mathcal{PL}_\infty}$-algebra.
\end{proof}

From the definition of the $\mathcal{B}race\underset{\text{\normalfont H}}{\otimes}\mathcal{E}$-algebra structure on $\text{\normalfont Hom}(\overline{\mathcal{C}}\otimes N_*(\Delta^n),\overline{\mathcal{P}})$, we deduce a first computation of $\text{\normalfont Map}_{\mathcal{O}p}(B^c(\mathcal{C}),\mathcal{P})$.

\begin{cor}
    We have the isomorphism of simplicial sets
    $$\text{\normalfont Map}_{\mathcal{O}p}(B^c(\mathcal{C}),\mathcal{P})\simeq\mathcal{MC}(\text{\normalfont Hom}_{\text{\normalfont Seq}_\mathbb{K}}(\overline{\mathcal{C}}\otimes\Sigma^{-1} N_*(\Delta^\bullet),\overline{\mathcal{P}})).$$
\end{cor}

Our goal is now to link this computation with the simplicial Maurer-Cartan set of $\text{\normalfont Hom}_{\text{\normalfont Seq}_\mathbb{K}}(\overline{\mathcal{C}},\overline{\mathcal{P}})$. Recall that $N_*(\Delta^n)$ has a basis given by increasing sequence of integers $0\leq a_0<\cdots<a_r\leq n$ which we denote by $\underline{a_0\cdots a_r}$. We let $\mathcal{B}_n$ to be this basis.

\begin{lm}\label{dernier}
    Let $n\geq 0$. We set
\begin{center}
    $\begin{array}{ccc}
        \phi_n:\mathcal{L}(\text{\normalfont Hom}(\overline{\mathcal{C}}\otimes N_\ast(\Delta^n),\overline{\mathcal{P}})) & \longrightarrow & \mathcal{L}(\text{\normalfont Hom}(\overline{\mathcal{C}},\overline{\mathcal{P}}))\otimes N^*(\Delta^n)\\
        f & \longmapsto & \sum_{\underline{x}\in\mathcal{B}_n}f^{\underline{x}}\otimes\underline{x}^\vee 
    \end{array}$
\end{center}

    \noindent where, for every $\underline{x}\in N_*(\Delta^n)$ and $f\in\mathcal{L}(\text{\normalfont Hom}(\overline{\mathcal{C}}\otimes N_*(\Delta^n),\overline{\mathcal{P}}))$, we denote by $f^{\underline{x}}\in\mathcal{L}(\text{\normalfont Hom}(\overline{\mathcal{C}},\overline{\mathcal{P}}))$ the map defined by $f^{\underline{x}}(c)=(-1)^{|c||\underline{x}|}f(c\otimes\underline{x})$ for every $c\in\overline{\mathcal{C}}$.\\ Then $\phi_n$ is an isomorphism of $\mathcal{B}race\underset{\text{\normalfont H}}{\otimes}\mathcal{E}$-algebras.\\
    Moreover, the sequence of isomorphisms $(\phi_n)_{n\geq 0}$ is compatible with the simplicial structures.
\end{lm}

\begin{proof}
    We first prove that $\phi_n$ commutes with the differentials. Let $f\in\mathcal{L}(\text{\normalfont Hom}(\overline{\mathcal{C}}\otimes N_*(\Delta^n),\overline{\mathcal{P}}))$. Then
    $$d(\phi_n(f))=\sum_{\underline{x}\in\mathcal{B}_n}d(f^{\underline{x}})\otimes\underline{x}^\vee+\sum_{\underline{x}\in\mathcal{B}_n}(-1)^{|f|+|\underline{x}|}f^{\underline{x}}\otimes d(\underline{x}^\vee).$$

    \noindent For every $c\in\overline{\mathcal{C}}$, we have that
    $$d(f^{\underline{x}})(c)=(-1)^{|c||\underline{x}|}d(f(c\otimes\underline{x}))-(-1)^{|f|+|\underline{x}|(|c|+1)}f(d(c)\otimes\underline{x})$$

    \noindent which gives
    $$d(f^{\underline{x}})=d(f)^{\underline{x}}+(-1)^{|f|}f^{d(\underline{x})}.$$

    \noindent We thus obtain
    $$d(\phi_n(f))=\sum_{\underline{x}\in\mathcal{B}_n}d(f)^{\underline{x}}\otimes\underline{x}^\vee+\sum_{\underline{x}\in\mathcal{B}_n}(-1)^{|f|}\left(f^{d(\underline{x})}\otimes\underline{x}^\vee+(-1)^{|\underline{x}|}f^{\underline{x}}\otimes d(\underline{x}^\vee)\right).$$

    \noindent It remains to prove that
    $$\sum_{\underline{x}\in\mathcal{B}_n}f^{d(\underline{x})}\otimes\underline{x}^\vee=-\sum_{\underline{x}\in\mathcal{B}_n}(-1)^{|\underline{x}|}f^{\underline{x}}\otimes d(\underline{x}^\vee).$$

    \noindent For every $\underline{x}\in\mathcal{B}_n$, we write $d(\underline{x})=\sum_{\underline{x}\in\mathcal{B}_n}\lambda_{\underline{x}}^{\underline{y}}\underline{y}$ where $\lambda_{\underline{y}}^{\underline{x}}\in\{-1;0;1\}$. We thus have, for every $\underline{y}\in\mathcal{B}_n$,
    $$d(\underline{y}^\vee)=-(-1)^{|\underline{y}|}\sum_{\underline{x}\in\mathcal{B}_n}\lambda_{\underline{x}}^{\underline{y}}\underline{x}^\vee.$$

    \noindent We thus have
    \begin{center}
        $\begin{array}{lll}
        \displaystyle\sum_{\underline{x}\in\mathcal{B}_n}f^{d(\underline{x})}\otimes\underline{x}^\vee & = & \displaystyle\sum_{\underline{x},\underline{y}\in\mathcal{B}_n}\lambda_{\underline{x}}^{\underline{y}}f^{\underline{y}}\otimes\underline{x}^\vee\\
        & = & \displaystyle-\sum_{\underline{y}\in\mathcal{B}_n}(-1)^{|\underline{y}|}f^{\underline{y}}\otimes d(\underline{y}^\vee).
        \end{array}$
    \end{center}

    \noindent At the end, we have obtained that
    $$d(\phi_n(f))=\sum_{\underline{x}\in\mathcal{B}_n}d(f)^{\underline{x}}\otimes\underline{x}^\vee=\phi_n(d(f))$$

    \noindent so that $\phi_n$ commutes with the differentials.\\
    
    Now, let $f_1,\ldots,f_r\in\mathcal{L}(\text{\normalfont Hom}(\overline{\mathcal{C}}\otimes N_*(\Delta^n),\overline{\mathcal{P}}))$ be homogeneous elements. Let $T\in\mathcal{PRT}$ be a canonical tree with $r$ vertices and $\mu\in\mathcal{E}(r)$. We have
    \begin{multline*}
        (T\otimes\mu)(\phi_n(f_1),\ldots,\phi_n(f_r))\\=\sum_{\underline{T}\in\text{\normalfont Shape}^{-1}(T)}\sum_{\underline{x_1},\ldots,\underline{x_r}\in\mathcal{B}_n}\pm(\gamma_{(\underline{T})}\circ (f_1^{\underline{x_1}}\otimes\cdots\otimes f_r^{\underline{x_r}})\circ\Delta_{(\underline{T})})\otimes\mu^{N^*(\Delta^n)}(\underline{x_1}^{\vee},\ldots,\underline{x_r}^{\vee}),
    \end{multline*}

    \noindent where the sign is given by
    $$\prod_{i<j}(-1)^{|\underline{x_i}|(|f_j|+|\underline{x_j}|)}\times\prod_{j=1}^r(-1)^{|\mu|(|f_j|+|\underline{x_j}|)}.$$

    \noindent Now, for every $\underline{x}\in\mathcal{B}_n$, we write
    $$\mu^{N_*(\Delta^n)}(\underline{x})=\sum_{\underline{x_1},\ldots,\underline{x_r}\in\mathcal{B}_n}\lambda_{\underline{x}}^{\underline{x_1},\ldots,\underline{x_r}}\underline{x_1}\otimes\cdots\otimes\underline{x_r}$$

    \noindent where $\lambda_{\underline{x}}^{\underline{x_1},\ldots,\underline{x_r}}\in\{-1;0;1\}$ by definition of the interval cuts operations. This gives, for every $\underline{x_1},\ldots,\underline{x_r}\in\mathcal{B}_n$, 
    $$\mu^{N^*(\Delta^n)}(\underline{x_1}^\vee,\ldots,\underline{x_r}^\vee)=\sum_{\underline{x}\in\mathcal{B}_n}\pm\underline{x}^\vee$$

    \noindent where the sign is given by
    $$\lambda_{\underline{x}}^{\underline{x_1},\ldots,\underline{x_r}}\prod_{i<j}(-1)^{|\underline{x_i}||\underline{x_j}|}\prod_{j=1}^r(-1)^{|\mu||\underline{x_j}|}.$$

    \noindent We thus have
    $$(T\otimes\mu)(\phi_n(f_1),\ldots,\phi_n(f_r))=\sum_{\underline{T}\in\text{\normalfont Shape}^{-1}(T)}\sum_{\underline{x},\underline{x_1},\ldots,\underline{x_r}\in\mathcal{B}_n}\pm (\gamma_{(\underline{T})}\circ (f_1^{\underline{x_1}}\otimes\cdots\otimes f_r^{\underline{x_r}})\circ\Delta_{(\underline{T})})\otimes\underline{x}^\vee$$

    \noindent where the sign is
    $$\lambda_{\underline{x}}^{\underline{x_1},\ldots,\underline{x_r}}\prod_{i<j}(-1)^{|\underline{x_i}||\underline{f_j}|}\prod_{j=1}^r(-1)^{|\mu||\underline{f_j}|}.$$

    \noindent We now use that, for every $c\in\overline{\mathcal{C}}$,
    $$(f_1^{\underline{x_1}}\otimes\cdots\otimes f_r^{\underline{x_r}})(\Delta_{(\underline{T})}(c))=\pm (f_1\otimes\cdots\otimes f_r)\circ (\Delta_{(\underline{T})}(c)\widetilde{\otimes}(\underline{x_1}\otimes\cdots\otimes\underline{x_r}))$$

    \noindent where the sign is given by
    $$\prod_{i<j}(-1)^{|\underline{x_i}||f_j|}\times\prod_{j=1}^r(-1)^{|c||\underline{x_j}|}.$$

    \noindent We deduce
    $$(T\otimes\mu)(\phi_n(f_1),\ldots,\phi_n(f_r))(c)=\sum_{\underline{T}\in\text{\normalfont Shape}^{-1}(T)}\sum_{\underline{x}\in\mathcal{B}_n}\pm\gamma_{(\underline{T})}((f_1\otimes\cdots\otimes f_r)(\Delta_{(\underline{T})}(c)\widetilde{\otimes}\mu^{N_*(\Delta^n)}(\underline{x})))$$

    \noindent where the sign is
    $$\prod_{j=1}^r(-1)^{|\mu||f_j|}\times (-1)^{|c|(|\underline{x}|+|\mu|)},$$

    \noindent since, for every $\underline{x},\underline{x_1},\ldots,\underline{x_r}\in\mathcal{B}_n$ such that $\lambda_{\underline{x}}^{\underline{x_1},\ldots,\underline{x_r}}\neq 0$, we have $|\mu|-\sum_{i=1}^r|\underline{x_i}|=-|\underline{x}|$. We finally obtain
    $$(T\otimes\mu)(\phi_n(f_1),\ldots,\phi_n(f_r))=\sum_{\underline{x}\in\mathcal{B}_n}((T\otimes\mu)(f_1,\ldots,f_r))^{\underline{x}}\otimes\underline{x}^\vee=\phi_n((T\otimes\mu)(f_1,\ldots,f_r)).$$
    
    \noindent We thus have proved that $\phi_n:\mathcal{L}(\text{\normalfont Hom}(\overline{\mathcal{C}}\otimes N_*(\Delta^n),\overline{\mathcal{P}}))\longrightarrow\mathcal{L}(\text{\normalfont Hom}(\overline{\mathcal{C}},\overline{\mathcal{P}}))\otimes N^*(\Delta^n)$ is a morphism of complete $\mathcal{B}race\underset{\text{\normalfont H}}{\otimes}\mathcal{E}$-algebras, and $\phi$ is obviously a bijection, with as inverse
    $$\phi_n^{-1}(f\otimes\underline{x}^{\vee})=(c\otimes\underline{x}\longmapsto (-1)^{|c||\underline{x}|}f(c))$$

    \noindent for every $f\in\mathcal{L}(\text{\normalfont Hom}(\overline{\mathcal{C}},\overline{\mathcal{P}}))$ and $\underline{x}\in \mathcal{B}_n$.\\

    The compatibility of the sequence $(\phi_n)_{n\geq 0}$ with the simplicial structures follows directly from the definition of the $\phi_n$'s.
\end{proof}

In particular, the map $\phi_n$ induces an isomorphism of $\Gamma(\mathcal{P}re\mathcal{L}ie_\infty,-)$-algebras. We thus obtain Theorem \ref{theoremG}.

\begin{thm}
    We have the identity
    $$\text{\normalfont Map}_{\mathcal{O}p}(B^c(\mathcal{C}),\mathcal{P})=\mathcal{MC}_\bullet(\text{\normalfont Hom}_{\text{\normalfont Seq}_\mathbb{K}}(\overline{\mathcal{C}},\overline{\mathcal{P}})).$$
\end{thm}

\begin{proof}
    For every $n\geq 0$, since the morphism $\phi_n$ given in Lemma \ref{dernier} preserves the filtrations, taking the completions gives an isomorphism
\[\begin{tikzcd}
	{\text{\normalfont Hom}_{\text{\normalfont Seq}_\mathbb{K}}(\overline{\mathcal{C}}\otimes\Sigma^{-1} N_*(\Delta^n),\overline{\mathcal{P}})} & {\Sigma\text{\normalfont Hom}_{\text{\normalfont Seq}_\mathbb{K}}(\overline{\mathcal{C}},\overline{\mathcal{P}})\otimes N^*(\Delta^n)}
	\arrow["\simeq", from=1-1, to=1-2]
\end{tikzcd}\]
of $\widehat{\Gamma\Lambda\mathcal{PL}_\infty}$-algebras for every $n\geq 0$. Since this isomorphism preserves the simplicial structures, we obtain the theorem.
\end{proof}

%% file: chap3.tex
\section{A mapping space in the category of symmetric connected operads}

In this last section, we show that we can describe a mapping space in the category of symmetric and connected operads as the degree-wise Maurer-Cartan set of some complete $\Gamma\Lambda\mathcal{PL}_\infty$-algebra.\\

In $\mathsection$\ref{sec:261}, we recall the construction of the free operad functor in the category of symmetric connected operads and the model structure on the latter category.\\

In $\mathsection$\ref{sec:262}, we use the surjection cooperad $\textbf{Sur}_\mathbb{K}$ to obtain a $\Sigma_\ast$-cofibrant replacement $B^c(\mathcal{C}\underset{\text{\normalfont H}}{\otimes}{\normalfont\textbf{Sur}}_\mathbb{K})$ of the cobar construction $B^c(\mathcal{C})$ associated to a symmetric cooperad $\mathcal{C}$ such that $\mathcal{C}(0)=0$. We construct an explicit cosimplicial frame associated to $B^c(\mathcal{C}\underset{\text{\normalfont H}}{\otimes}{\normalfont\textbf{Sur}}_\mathbb{K})$.\\

In $\mathsection$\ref{sec:263}, we finally deduce Theorem \ref{theoremH} which gives a computation of the mapping spaces in the category of symmetric connected operads in terms of a degree-wise simplicial Maurer-Cartan set of some $\Gamma\widehat{\Lambda\mathcal{PL}_\infty}$-algebras.

\subsection{The free symmetric operad functor and the model structure on $\Sigma\mathcal{O}p^0$}\label{sec:261}

In this subsection, we recall the construction of the free operad functor in the category $\Sigma\mathcal{O}p$ and recall the model structure on the category of symmetric connected operads $\Sigma\mathcal{O}p^0$.\\

We have a functor $-\otimes\Sigma:\text{\normalfont Seq}_\mathbb{K}\longrightarrow\Sigma\text{\normalfont Seq}_\mathbb{K}$ defined, for every $M\in\text{\normalfont Seq}_\mathbb{K}$, by
$$(M\otimes\Sigma)(n)=M(n)\otimes\mathbb{K}[\Sigma_n],$$

\noindent where $M(n)\otimes\mathbb{K}[\Sigma_n]$ is endowed with the $\Sigma_n$ action defined, for every $m\in M(n)$ and $\sigma,\tau\in\Sigma_n$, by
$$\sigma\cdot (m\otimes\tau)=m\otimes\sigma\tau.$$

\noindent The functor $-\otimes\Sigma$ fits in an adjunction
\[\begin{tikzcd}
	{-\otimes\Sigma:\text{\normalfont Seq}_\mathbb{K}} & {\Sigma\text{\normalfont Seq}_\mathbb{K}:\omega}
	\arrow[shift left, from=1-1, to=1-2]
	\arrow[shift left, from=1-2, to=1-1]
\end{tikzcd},\]

\noindent where $\omega:\Sigma\text{\normalfont Seq}_\mathbb{K}\longrightarrow\text{\normalfont Seq}_\mathbb{K}$ is the functor which forgets the symmetric groups actions. We have the following result.

\begin{prop}[{see \cite[Proposition 11.4.A]{fressemod}}] The category $\Sigma\text{\normalfont Seq}_\mathbb{K}$ is endowed with a cofibrantly generated model category structure such that the forgetful functor $\omega:\Sigma\text{\normalfont Seq}_\mathbb{K}\longrightarrow\text{\normalfont Seq}_\mathbb{K}$ creates weak-equivalences and fibrations. Cofibrations are given by the left lifting property with respect to acyclic fibrations.
\end{prop}

In fact, the category $\Sigma\mathcal{O}p$ does not have a model structure but rather a \textit{semi-model} structure (see \cite[Theorem 3]{spitz}). We instead consider the subcategory $\Sigma\mathcal{O}p^0$ of operads $\mathcal{P}$ such that $\mathcal{P}(0)=0$. Such an operad is said to be \textit{connected}. We also denote by $\Sigma\text{\normalfont Seq}_\mathbb{K}^0$ the subcategory of $\Sigma\text{\normalfont Seq}_\mathbb{K}$ given by symmetric sequences $M$ such that $M(0)=0$. As for the non symmetric context, we are searching for a convenient adjunction
\[\begin{tikzcd}
	{\mathcal{F}:\Sigma\text{\normalfont Seq}_\mathbb{K}^0} & {\Sigma\mathcal{O}p^0:\omega}
	\arrow[shift left, from=1-1, to=1-2]
	\arrow[shift left, from=1-2, to=1-1]
\end{tikzcd}\]

\noindent where $\omega:\Sigma\mathcal{O}p^0\longrightarrow\Sigma\text{\normalfont Seq}_\mathbb{K}^0$ is the functor which forgets the operad structure. To achieve this, recall that the functor $-\otimes\Sigma:\text{\normalfont Seq}_\mathbb{K}\longrightarrow\Sigma\text{\normalfont Seq}_\mathbb{K}$ restricts to a functor $-\otimes\Sigma:\mathcal{O}p\longrightarrow\Sigma\mathcal{O}p$ where, for every $\mathcal{P}\in\mathcal{O}p$, the operad structure on $\mathcal{P}\otimes\Sigma$ is defined by
$$\gamma((f\otimes\sigma)\otimes (g_1\otimes\tau_1)\otimes\cdots\otimes (g_n\otimes\tau_n))=\pm\gamma(f\otimes g_{\sigma^{-1}(1)}\otimes\cdots\otimes g_{\sigma^{-1}(n)})\otimes\sigma(\tau_1,\ldots,\tau_n),$$

\noindent where we consider the composite of $\sigma$ with $\tau_1,\ldots,\tau_n$ (see after Lemma \ref{compsymmetric}).

\begin{defi}\label{treesigma}
    Consider the operad $\underline{\mathcal{T}ree}$ defined in Lemma \ref{freeop}. We set
    $$\Sigma\underline{\mathcal{T}ree}=\underline{\mathcal{T}ree}\otimes\Sigma.$$

For every $\underline{T}\in\Sigma\underline{\mathcal{T}ree}$, we denote by $V_{\underline{T}}$ the set of vertices, and by $\text{\normalfont val}_{\underline{T}}(v)$ the number of ingoing arrows on a vertex $v\in V_{\underline{T}}$.
\end{defi}

The elements of $\Sigma\underline{\mathcal{T}ree}(n)$ can then be seen as trees with inputs endowed with a choice of labeling on the inputs. We identify such a choice with a permutation in $\Sigma_n$. For instance, if we consider the tree with inputs $\underline{T}\in\underline{\mathcal{T}ree}$ given in Definition \ref{treeleave}, then

$$\underline{T}\otimes (643215)=\begin{tikzpicture}[baseline={([yshift=-.5ex]current bounding box.center)},scale=0.8]
    \node (0) at (0,-1) {};
    \node[draw,circle,scale=0.8] (i) at (0,0) {$1$};
    \node[draw,circle,scale=0.8] (1) at (-1.5,1) {$2$};
    \node[draw,circle,scale=0.8] (1b) at (-2.5,2) {$3$};
    \node[scale=0.8] (1bc) at (-3,3) {$5$};
    \node[scale=0.8] (1bcc) at (-2.5,3) {$4$};
    \node[scale=0.8] (1bccc) at (-2,3) {$3$};
    \node[draw,circle,scale=0.8] (1bb) at (-0.5,2) {$4$};
    \node[scale=0.8] (1bbc) at (-1,3) {$2$};
    \node[scale=0.8] (1bbcc) at (0,3) {$6$};
    \node[scale=0.8] (2) at (0,1) {$1$};
    \node[draw,circle,scale=0.8] (3) at (1.25,1) {$5$};
    \draw[stealth-] (0) -- (i);
    \draw[-stealth] (2) -- (i);
    \draw[stealth-] (1) -- (1b);
    \draw[stealth-] (1) -- (1bb);
    \draw[stealth-] (i) -- (1);
    \draw[stealth-] (i) -- (3);
    \draw[stealth-] (1b) -- (1bc);
    \draw[stealth-] (1b) -- (1bcc);
    \draw[stealth-] (1b) -- (1bccc);
    \draw[stealth-] (1bb) -- (1bbcc);
    \draw[stealth-] (1bb) -- (1bbc);
    \end{tikzpicture}.$$

For every $M\in\Sigma\text{\normalfont Seq}_\mathbb{K}$, we consider the sequence
$$k\longmapsto\bigoplus_{n\geq 0}\bigoplus_{\underline{T}\in\underline{\mathcal{PRT}_k}(n)}\bigoplus_{\sigma\in\Sigma_k}(\underline{T}\otimes\sigma)\otimes\bigotimes_{i=1}^nM(\text{\normalfont val}_{\underline{T}}(i)).$$

\noindent This sequence is endowed with the structure of a non-symmetric operad given by the operadic structure of $\Sigma\underline{\mathcal{T}ree}$ and the concatenation of elements in the tensor algebra of $\bigoplus_{n\geq 0}M(n)$. In order to endow this sequence with the structure of a symmetric operad, we need to identify some elements. First, we endow this sequence with the $\Sigma_n$-action given by the left translation in $\Sigma_n$. For every $n\geq 0$, we identify the action of $\Sigma_n$ on $\underline{T}\in\underline{\mathcal{PRT}_k}(n)$ given by the permutation of the vertices with the action of $\Sigma_n$ on $\bigotimes_{i=1}^nM(\text{\normalfont val}_{\underline{T}}(i))$ given by the permutation of the factors. Next, consider a tree in $\Sigma\underline{\mathcal{T}ree}$ of the form
$$\underline{T}:=\begin{tikzpicture}[baseline={([yshift=-.5ex]current bounding box.center)},scale=0.8]
    \node (0) at (0,-1) {};
    \node[draw,circle,scale=0.8] (i) at (0,0) {$1$};
    \node (1) at (-1,1) {$\underline{T_1}$};
    \node[scale=0.8] (1b) at (-1.5,2) {$j_1^1$};
    \node[scale=0.8] (1bb) at (-1,2) {$\cdots$};
    \node[scale=0.8] (1bbb) at (-0.5,2) {$j_{k_1}^1$};
    \node (3) at (1,1) {$\underline{T_r}$};
    \node[scale=0.8] (3b) at (0.5,2) {$j_1^r$};
    \node[scale=0.8] (3bb) at (1,2) {$\cdots$};
    \node[scale=0.8] (3bbb) at (1.5,2) {$j_{k_r}^r$};
    \node (2) at (0,1) {$\cdots$};
    \draw[stealth-] (1) -- (1b);
    \draw[stealth-] (1) -- (1bbb);
    \draw[stealth-] (3) -- (3b);
    \draw[stealth-] (3) -- (3bbb);
    \draw[stealth-] (0) -- (i);
    \draw[stealth-] (i) -- (1);
    \draw[stealth-] (i) -- (3);
    \end{tikzpicture},$$

\noindent for some $r\geq 1$ and   $\underline{T_1}\in\underline{\mathcal{PRT}_{k_1}},\ldots,\underline{T_r}\in\underline{\mathcal{PRT}_{k_r}}$ with $k_1+\cdots+k_r=k$. Let $x\in M(r)$ and $Y\in\bigotimes_{i=2}^n M(\text{\normalfont val}_{\underline{T}}(i))$. For every $\mu\in\Sigma_r$, we make the identification
$$\begin{tikzpicture}[baseline={([yshift=-.5ex]current bounding box.center)},scale=0.8]
    \node (0) at (0,-1) {};
    \node[draw,circle,scale=0.8] (i) at (0,0) {$1$};
    \node (1) at (-1,1) {$\underline{T_1}$};
    \node[scale=0.8] (1b) at (-1.5,2) {$j_1^1$};
    \node[scale=0.8] (1bb) at (-1,2) {$\cdots$};
    \node[scale=0.8] (1bbb) at (-0.5,2) {$j_{k_1}^1$};
    \node (3) at (1,1) {$\underline{T_r}$};
    \node[scale=0.8] (3b) at (0.5,2) {$j_1^r$};
    \node[scale=0.8] (3bb) at (1,2) {$\cdots$};
    \node[scale=0.8] (3bbb) at (1.5,2) {$j_{k_r}^r$};
    \node (2) at (0,1) {$\cdots$};
    \draw[stealth-] (1) -- (1b);
    \draw[stealth-] (1) -- (1bbb);
    \draw[stealth-] (3) -- (3b);
    \draw[stealth-] (3) -- (3bbb);
    \draw[stealth-] (0) -- (i);
    \draw[stealth-] (i) -- (1);
    \draw[stealth-] (i) -- (3);
    \end{tikzpicture}\otimes x\otimes Y\equiv\begin{tikzpicture}[baseline={([yshift=-.5ex]current bounding box.center)},scale=0.8]
    \node (0) at (0,-1) {};
    \node[draw,circle,scale=0.8] (i) at (0,0) {$1$};
    \node (1) at (-1.5,1) {$\underline{T_{\mu^{-1}(1)}}$};
    \node[scale=0.7] (1b) at (-2.25,2) {$j_1^{\mu^{-1}(1)}$};
    \node[scale=0.8] (1bb) at (-1.5,2) {$\cdots$};
    \node[scale=0.7] (1bbb) at (-0.75,2) {$j_{k_{\mu^{-1}}(1)}^{\mu^{-1}(1)}$};
    \node (3) at (1.5,1) {$\underline{T_{\mu^{-1}(r)}}$};
    \node[scale=0.7] (3b) at (0.75,2) {$j_1^{\mu^{-1}(r)}$};
    \node[scale=0.8] (3bb) at (1.5,2) {$\cdots$};
    \node[scale=0.7] (3bbb) at (2.25,2) {$j_{k_{\mu^{-1}(r)}}^{\mu^{-1}(r)}$};
    \node (2) at (0,1) {$\cdots$};
    \draw[stealth-] (1) -- (1b);
    \draw[stealth-] (1) -- (1bbb);
    \draw[stealth-] (3) -- (3b);
    \draw[stealth-] (3) -- (3bbb);
    \draw[stealth-] (0) -- (i);
    \draw[stealth-] (i) -- (1);
    \draw[stealth-] (i) -- (3);
    \end{tikzpicture}\otimes\mu\cdot x\otimes Y.$$

    

    \noindent We iterate such identifications by induction on the number of vertices of the tree, using the operadic structure. It is an immediate check that we obtain a symmetric operad, which we denote by $\mathcal{F}(M)$, such that the functor $\mathcal{F}$ fits in the left-right adjunction
\[\begin{tikzcd}
	{\mathcal{F}:\Sigma\text{\normalfont Seq}_\mathbb{K}^0} & {\Sigma\mathcal{O}p^0:\omega}
	\arrow[shift left, from=1-1, to=1-2]
	\arrow[shift left, from=1-2, to=1-1]
\end{tikzcd}.\]

\begin{prop}[{see \cite[$\mathsection$3.3]{hinicherratum}}] The category $\Sigma\mathcal{O}p^0$ is endowed with a cofibrantly generated model structure such that the forgetful functor $\omega:\Sigma\mathcal{O}p^0\longrightarrow\Sigma\text{\normalfont Seq}_\mathbb{K}^0$ creates weak-equivalences and fibrations. Cofibrations are given by the left lifting property with respect to acyclic fibrations.
\end{prop}

Let $M\in\Sigma\text{\normalfont Seq}_\mathbb{K}^0$. In the definition of $\mathcal{F}(M)$, we can restrict to trees in $\Sigma\underline{\mathcal{T}ree}$ such that every vertex has at least one input. We denote by $\Sigma\underline{\mathcal{T}ree}^0$ the underlying sequence. In the following sections, we use an explicit choice of set of representatives for trees in $\mathcal{F}(M)$. Such a choice can be made by taking \textit{tree monomials} (see \cite[$\mathsection$3.1]{dotgrobner}), for which we recall the definition.

\begin{defi}\label{treemon}
    Let $\underline{T}\in\Sigma\underline{\mathcal{PRT}_{i_1<\cdots<i_n}}^0$ be a tree with $m$ vertices. The tree $\underline{T}$ is a {\normalfont tree monomial} if $\text{\normalfont Shape}(\underline{T})$ is in the canonical order, and if one of the three following conditions is fulfilled:

\begin{itemize}
    \item $m=0$ (so that $\underline{T}$ is the unit in the operad $\Sigma\underline{\mathcal{T}ree}$);
    \item $m=1$ and $\underline{T}$ is of the form
$$\underline{T}=\begin{tikzpicture}[baseline={([yshift=-.5ex]current bounding box.center)},scale=0.8]
    \node (0) at (0,-1) {};
    \node[draw,circle,scale=0.8] (i) at (0,0) {$a$};
    \node[scale=0.8] (1) at (-1,1) {$i_1$};
    \node[scale=0.8] (3) at (1,1) {$i_n$};
    \node (2) at (0,1) {$\cdots$};
    \draw[stealth-] (0) -- (i);
    \draw[stealth-] (i) -- (1);
    \draw[stealth-] (i) -- (3);
    \end{tikzpicture}$$
    \noindent for some vertex $a$;

\item $m\geq 2$ and $\underline{T}$ is of the form
$$\underline{T}=\begin{tikzpicture}[baseline={([yshift=-.5ex]current bounding box.center)},scale=0.8]
    \node (0) at (0,-1) {};
    \node[draw,circle,scale=0.8] (i) at (0,0) {$a$};
    \node (1) at (-1,1) {$\underline{T_1}$};
    \node[scale=0.8] (1b) at (-1.5,2) {$j_1^1$};
    \node[scale=0.8] (1bb) at (-1,2) {$\cdots$};
    \node[scale=0.8] (1bbb) at (-0.5,2) {$j_{n_1}^1$};
    \node (3) at (1,1) {$\underline{T_r}$};
    \node[scale=0.8] (3b) at (0.5,2) {$j_1^r$};
    \node[scale=0.8] (3bb) at (1,2) {$\cdots$};
    \node[scale=0.8] (3bbb) at (1.5,2) {$j_{n_r}^r$};
    \node (2) at (0,1) {$\cdots$};
    \draw[stealth-] (1) -- (1b);
    \draw[stealth-] (1) -- (1bbb);
    \draw[stealth-] (3) -- (3b);
    \draw[stealth-] (3) -- (3bbb);
    \draw[stealth-] (0) -- (i);
    \draw[stealth-] (i) -- (1);
    \draw[stealth-] (i) -- (3);
    \end{tikzpicture}$$

\noindent for some vertex $a$ and where $\underline{T_1}\in\Sigma\underline{\mathcal{PRT}_{\{j_1^1,\ldots,j_{n_1}^1\}}}^0\ldots,\underline{T_r}\in\Sigma\underline{\mathcal{PRT}_{\{j_1^r,\ldots,j_{n_r}^r\}}}^0$ are tree monomials such that $\min(j_1^1,\ldots,j_{n_1}^1)<\cdots<\min(j_1^r,\ldots,j_{n_r}^r)$.
\end{itemize}
\end{defi}

For instance, the tree
$$\begin{tikzpicture}[baseline={([yshift=-.5ex]current bounding box.center)},scale=0.8]
    \node (0) at (0,-1) {};
    \node[draw,circle,scale=0.8] (i) at (0,0) {$1$};
    \node[draw,circle,scale=0.8] (1) at (-1.5,1) {$2$};
    \node[draw,circle,scale=0.8] (1b) at (-2.5,2) {$3$};
    \node[scale=0.8] (1bc) at (-3,3) {$1$};
    \node[scale=0.8] (1bcc) at (-2.5,3) {$7$};
    \node[scale=0.8] (1bccc) at (-2,3) {$9$};
    \node[draw,circle,scale=0.8] (1bb) at (-0.5,2) {$4$};
    \node[scale=0.8] (1bbc) at (-1,3) {$2$};
    \node[scale=0.8] (1bbcc) at (0,3) {$4$};
    \node[scale=0.8] (2) at (0,1) {$3$};
    \node[draw,circle,scale=0.8] (3) at (1.25,1) {$5$};
    \node[scale=0.8] (3b) at (0.75,2) {$5$};
    \node[scale=0.8] (3bb) at (1.75,2) {$6$};
    \draw[stealth-] (0) -- (i);
    \draw[-stealth] (2) -- (i);
    \draw[stealth-] (1) -- (1b);
    \draw[stealth-] (1) -- (1bb);
    \draw[stealth-] (i) -- (1);
    \draw[stealth-] (i) -- (3);
    \draw[stealth-] (1b) -- (1bc);
    \draw[stealth-] (1b) -- (1bcc);
    \draw[stealth-] (1b) -- (1bccc);
    \draw[stealth-] (1bb) -- (1bbcc);
    \draw[stealth-] (1bb) -- (1bbc);
    \draw[stealth-] (3) -- (3b);
    \draw[stealth-] (3) -- (3bb);
    \end{tikzpicture}\in\Sigma\underline{\mathcal{PRT}_{1<\cdots<9}}$$

\noindent is a tree monomial. For every $n\geq 1$, we denote by $\mathcal{TM}_k(n)$ the set of tree monomials with $n$ vertices and $k$ inputs. We have the following result.

\begin{prop}\label{idmono}
    Let $M\in\Sigma\text{\normalfont Seq}_\mathbb{K}^0$. Then, for every $n\geq 1$,
$$\mathcal{F}(M)(k)\simeq\bigoplus_{n\geq 0}\bigoplus_{\underline{T}\in\mathcal{TM}_k(n)}\underline{T}\otimes\bigotimes_{i=1}^nM(\text{\normalfont val}_{\underline{T}}(i)).$$
\end{prop}

\begin{proof}
    The proposition is obtained by iterating the second claim of Proposition \ref{shuffle}, and by using the symmetry axioms in the operad $\mathcal{F}(M)$.
\end{proof}

As in the non-symmetric context, we have the following remark.

\begin{remarque}
    Since, for every $k\geq 0$, the $\mathbb{K}$-module $\Sigma\underline{\mathcal{T}ree_k}^0$ is finite dimensional, we have that the dual symmetric sequence $(\Sigma\underline{\mathcal{T}ree}^0)^\vee$ is endowed with the structure of a cooperad. We can then define, for every $n\geq 0$,
    $$\mathcal{F}^c(M)(k)\simeq\bigoplus_{n\geq 0}\bigoplus_{\underline{T}\in\mathcal{TM}_k(n)}\underline{T}^\vee\otimes\bigotimes_{i=1}^nM(\text{\normalfont val}_{\underline{T}}(i)).$$

    \noindent One can show that this symmetric sequence is endowed with the structure of a cooperad given by the cooperad structure in $(\Sigma\underline{\mathcal{T}ree}^0)^\vee$, and by the deconcatenation coproduct in the tensor algebra of $\bigoplus_{n\geq 0}M(n)$. As for the free operad functor, we have an adjunction
\[\begin{tikzcd}
	{\omega:(\Sigma\mathcal{O}p^c)^0} & {\Sigma\text{\normalfont Seq}_\mathbb{K}^0:\mathcal{F}^c}
	\arrow[shift left, from=1-1, to=1-2]
	\arrow[shift left, from=1-2, to=1-1]
\end{tikzcd}\]

\noindent where we denote by $(\Sigma\mathcal{O}p^c)^0$ the subcategory of $\Sigma\mathcal{O}p^c$ given by connected cooperads, and where $\omega:(\Sigma\mathcal{O}p^c)^0\longrightarrow\Sigma\text{\normalfont Seq}_\mathbb{k}^0$ is the functor which forgets the cooperad structure.
\end{remarque}

As for the non symmetric context, we consider operadic (resp. cooperadic) compositions (resp. cocompositions) shaped on trees with inputs. Let $\mathcal{P}$ be an augmented operad $\mathcal{P}\simeq I\oplus\overline{\mathcal{P}}$ and $\mathcal{C}$ be a coaugmented cooperad $\mathcal{C}\simeq I\oplus\overline{\mathcal{C}}$ such that $\mathcal{P}(0)=\mathcal{C}(0)=0$ and $\mathcal{P}(1)=\mathcal{C}(1)=\mathbb{K}$. By the universal property satisfied by $\mathcal{F}$, we have a unique operad morphism $\mathcal{F}(\mathcal{P})\longrightarrow\mathcal{P}$ which reduces to the identity on $\mathcal{P}\subset\mathcal{F}(\mathcal{P})$. Analogously, we have a unique cooperad morphism $\mathcal{C}\longrightarrow\mathcal{F}^c(\mathcal{C})$ whose projection on $\mathcal{C}$ is given by the identity on $\mathcal{C}$.

\begin{defi}\label{compodeltaarbressym}
    Let $\underline{T}$ be a tree with inputs. We define $\gamma_{(\underline{T})}:\mathcal{F}_{(\underline{T})}(\overline{\mathcal{P}})\longrightarrow\overline{\mathcal{P}}$ and $\Delta_{(\underline{T})}:\overline{\mathcal{C}}\longrightarrow\mathcal{F}^c_{(\underline{T})}(\overline{\mathcal{C}})$ by the composites
\[\begin{tikzcd}
	{\gamma_{(\underline{T})}:\mathcal{F}_{(\underline{T})}(\overline{\mathcal{P}})} & {\mathcal{F}_{(\underline{T})}(\mathcal{P})} & {\mathcal{P}} & {\overline{\mathcal{P}}}
	\arrow[hook, from=1-1, to=1-2]
	\arrow["\gamma", from=1-2, to=1-3]
	\arrow[two heads, from=1-3, to=1-4]
\end{tikzcd};\]
\[\begin{tikzcd}
	{\Delta_{(\underline{T})}:\overline{\mathcal{C}}} & {\mathcal{C}} & {\mathcal{F}^c_{(\underline{T})}(\mathcal{C})} & {\mathcal{F}^c_{(\underline{T})}(\overline{\mathcal{C}})}
	\arrow[hook, from=1-1, to=1-2]
	\arrow["\Delta", from=1-2, to=1-3]
	\arrow[two heads, from=1-3, to=1-4]
\end{tikzcd}.\]
\end{defi}

For every $p,q,n,m\geq 0$ and $1\leq i\leq p$, as for the non symmetric context, we define a morphism
$$\bullet_i:\Sigma\underline{\mathcal{T}ree_p}(n)\otimes\Sigma\underline{\mathcal{T}ree_q}(m)\longrightarrow\Sigma\underline{\mathcal{T}ree_{p}}(n+m-1)$$

\noindent defined as follows. Let $\underline{U}\in\Sigma\underline{\mathcal{T}ree_p}(n)$ and $\underline{V}\in\Sigma\underline{\mathcal{T}ree_q}(m)$. If $\text{\normalfont val}_{\underline{U}}(i)\neq q$, we set $\underline{U}\bullet_i\underline{V}=0$. Else, we define $\underline{U}\bullet_i\underline{V}$ as the tree obtained by changing the $i$-th vertex of $\underline{U}$ into the tree $\underline{V}$. The attachment of the $q$ arrows on the $i$-th vertex of $\underline{U}$ on the tree $\underline{V}$ are given following the order of the labeling in $\underline{V}$.\\

As for the non symmetric context, we have the two following lemmas.

\begin{lm}
    Let $\underline{T}$ be a tree with inputs and $\underline{S}$ be a subtree of $\underline{T}$. Then $\underline{T}/\underline{S}\bullet_S\underline{S}=\underline{T}$.
\end{lm}

\begin{lm}\label{compbulletsym}
    Let $\underline{T}\in\underline{\mathcal{PRT}}(n)$ and $\underline{S}\subset\underline{T}$. Then
$$\gamma_{(\underline{T}/\underline{S})}\circ_S\gamma_{(\underline{S})}=\gamma_{(\underline{T})}\ ;\ \Delta_{(\underline{T}/\underline{S})}\circ_S\Delta_{(\underline{S})}=\Delta_{(\underline{T})}$$

    \noindent in the endomorphism operad $\text{\normalfont End}_{\bigoplus_{n\geq 2}\mathcal{P}(n)}$ and in the coendomorphism operad $\text{\normalfont CoEnd}_{\bigoplus_{n\geq 2}\mathcal{C}(n)}$ respectively.
\end{lm}

\subsection{A cosimplicial frame for $B^c(\mathcal{C}\underset{\text{\normalfont H}}{\otimes}\textbf{Sur}_\mathbb{K})$}\label{sec:262}

Let $\textbf{Sur}_\mathbb{K}$ be the surjection cooperad defined in \cite[Theorem A.1]{pdoperads}. This cooperad is actually equal, as a symmetric sequence, to the surjection operad $\chi$ recalled in $\mathsection$\ref{sec:214}. Note however that the cooperad structure on $\textbf{Sur}_\mathbb{K}$ is not the cooperad structure obtained by dualizing the operad structure on $\chi$. We have a weak-equivalence $B^c(\mathcal{C}\underset{\text{\normalfont H}}{\otimes}{\normalfont\textbf{Sur}}_\mathbb{K})\overset{\sim}{\longrightarrow}B^c(\mathcal{C})$, which provides a $\Sigma_\ast$-cofibrant replacement of $B^c(\mathcal{C})$.\\

In this section, we construct an explicit cosimplicial frame associated to $B^c(\mathcal{C}\underset{\text{\normalfont H}}{\otimes}{\normalfont\textbf{Sur}}_\mathbb{K})$ for every symmetric coaugmented cooperad $\mathcal{C}$. To be more precise, we construct a twisted differential $\partial^n:\mathcal{F}(\overline{\mathcal{C}}\underset{\text{\normalfont H}}{\otimes}\overline{\normalfont\textbf{Sur}}_\mathbb{K}\otimes\Sigma^{-1} N_*(\Delta^n))\longrightarrow \mathcal{F}(\overline{\mathcal{C}}\underset{\text{\normalfont H}}{\otimes}\overline{\normalfont\textbf{Sur}}_\mathbb{K}\otimes\Sigma^{-1} N_*(\Delta^n))$ by an inductive process analogue to the one given in Theorem \ref{rec}.\\

For every $k\geq 1$, we define $\Phi_n^{0},H_n^{0}:(\overline{\mathcal{C}}\underset{\text{\normalfont H}}{\otimes}\overline{\normalfont\textbf{Sur}}_\mathbb{K}\otimes\Sigma^{-1}N_*(\Delta^n))^{\otimes k}\longrightarrow (\overline{\mathcal{C}}\underset{\text{\normalfont H}}{\otimes}\overline{\normalfont\textbf{Sur}}_\mathbb{K}\otimes\Sigma^{-1}N_*(\Delta^n))^{\otimes k}$ by
$$\Phi_n^0=id_{\overline{\mathcal{C}}\underset{\text{\normalfont H}}{\otimes}\overline{\normalfont\textbf{Sur}}_\mathbb{K}}^{\otimes k}\widetilde{\otimes}\phi_n^0;$$
$$H_n^0=id_{\overline{\mathcal{C}}\underset{\text{\normalfont H}}{\otimes}\overline{\normalfont\textbf{Sur}}_\mathbb{K}}^{\otimes k}\widetilde{\otimes}h_n^0,$$

\noindent where we use the tensor product $\widetilde{\otimes}$ defined in Definition \ref{otimesmodif}, and the morphisms $\phi_n^0,h_n^0:(\Sigma^{-1}N_*(\Delta^n))^{\otimes k}\longrightarrow (\Sigma^{-1}N_*(\Delta^n))^{\otimes k}$ defined after Lemma \ref{firstdiff}. We extend $\Phi_n^0$ and $H_n^0$ on $\mathcal{F}(\overline{\mathcal{C}}\underset{\text{\normalfont H}}{\otimes}\overline{\normalfont\textbf{Sur}}_\mathbb{K}\otimes\Sigma^{-1} N_*(\Delta^n))$ by using the identification given in Proposition \ref{idmono}. Note however that the morphism $H_n^0$ does not preserve the action of the symmetric groups on $\mathcal{F}(\overline{\mathcal{C}}\underset{\text{\normalfont H}}{\otimes}\overline{\normalfont\textbf{Sur}}_\mathbb{K}\otimes\Sigma^{-1} N_*(\Delta^n))$.\\

Since the action of the symmetric groups on $\mathcal{C}\underset{\text{\normalfont H}}{\otimes}\textbf{Sur}_\mathbb{K}$ is free, we can chose an explicit choice of representatives for the orbits. For every $n\geq 1$, we let $\textbf{Sur}_\mathbb{K}^{id}(n)$ to be the dg $\mathbb{K}$-module generated by surjections $u\in\textbf{Sur}_\mathbb{K}(n)$ of the form
$$u=\left|\begin{array}{cccc}
         u_0(1) & \cdots & u_0(r_0-1) & u_0(r_0) \\
         \vdots &  &\vdots \\
         u_{d-1}(1) & \cdots & u_{d-1}(r_{d-1}-1) & u_{d-1}(r_{d-1})\\
         u_d(1) & \cdots & u_d(r_d-1) & u_d(r_d)\\
    \end{array}\right.$$

\noindent with
$$u_0(1)\cdots u_0(r_0-1)\cdots u_{d-1}(1)\cdots u_{d-1}(r_{d-1}-1)u_d(1)\cdots u_d(r_d)=1\cdots n.$$

\noindent We thus have an isomorphism of graded symmetric sequences $\textbf{Sur}_\mathbb{K}\simeq\textbf{Sur}_\mathbb{K}^{id}\otimes\Sigma$. This gives an isomorphism
$$\mathcal{C}\underset{\text{\normalfont H}}{\otimes}\textbf{Sur}_\mathbb{K}\simeq (\mathcal{C}\otimes\textbf{Sur}_\mathbb{K}^{id})\otimes\Sigma$$

\noindent defined by sending $c\otimes (u\otimes\sigma)\in\mathcal{C}\underset{\text{\normalfont H}}{\otimes}(\textbf{Sur}_\mathbb{K}^{id}\otimes\Sigma)$ to $(\sigma^{-1}\cdot c\otimes u)\otimes\sigma\in (\mathcal{C}\underset{\text{\normalfont H}}{\otimes}\textbf{Sur}_\mathbb{K}^{id})\otimes\Sigma$. Note that the differential $d_\mathcal{C}$ preserves such a decomposition, but not the differential $d_{\textbf{Sur}_\mathbb{K}}$, since it does not preserve $\textbf{Sur}_\mathbb{K}^{id}$. For every $l\geq 0$, we let $F_l\textbf{Sur}_\mathbb{K}^{id}$ to be the sequence given by surjections of degree equal or less than $l$ in $\textbf{Sur}_\mathbb{K}^{id}$ and we set $F_l(\mathcal{C}\underset{\text{\normalfont H}}{\otimes}\textbf{Sur}_\mathbb{K}^{id})=\mathcal{C}\underset{\text{\normalfont H}}{\otimes} F_l\textbf{Sur}_\mathbb{K}^{id}$.\\

For every $n\geq 0$, we aim to define a derivation of operads on $\mathcal{F}(\overline{\mathcal{C}}\underset{\text{\normalfont H}}{\otimes}\overline{\normalfont\textbf{Sur}}_\mathbb{K}\otimes\Sigma^{-1} N_*(\Delta^n))$ which reduces to the internal differential of $\overline{\mathcal{C}}\underset{\text{\normalfont H}}{\otimes}\overline{\normalfont\textbf{Sur}}_\mathbb{K}\otimes\Sigma^{-1} N_*(\Delta^n)$ on trees with only one vertex. We denote by $d_{{\mathcal{C}}},d_{\textbf{Sur}_\mathbb{K}}$ and $d_{\Sigma^{-1} N_*(\Delta^n)}$ the corresponding differentials on $\overline{\mathcal{C}}\underset{\text{\normalfont H}}{\otimes}\overline{\textbf{Sur}_\mathbb{K}}\otimes\Sigma^{-1} N_*(\Delta^n)$. Let $\mathcal{C}_{ns}:=\mathcal{C}\underset{\text{\normalfont H}}{\otimes}\textbf{Sur}_\mathbb{K}^{id}\otimes id$. We construct $\beta^n:\overline{\mathcal{C}_{ns}}\otimes\Sigma^{-1} N_*(\Delta^n)\longrightarrow\mathcal{F}(\overline{\mathcal{C}}\underset{\text{\normalfont H}}{\otimes}\overline{\normalfont\textbf{Sur}}_\mathbb{K}\otimes\Sigma^{-1} N_*(\Delta^n))$ which reduces to $d_{\textbf{Sur}_\mathbb{K}}+ d_{\Sigma^{-1}N_*(\Delta^n)}$ on trees with one vertex and which is such that
$$d_{{\mathcal{C}}}(\beta^n)+\partial^n\beta^n=0,$$

\noindent where $\partial^n:\mathcal{F}(\overline{\mathcal{C}}\underset{\text{\normalfont H}}{\otimes}\overline{\normalfont\textbf{Sur}}_\mathbb{K}\otimes\Sigma^{-1} N_*(\Delta^n))\longrightarrow\mathcal{F}(\overline{\mathcal{C}}\underset{\text{\normalfont H}}{\otimes}\overline{\normalfont\textbf{Sur}}_\mathbb{K}\otimes\Sigma^{-1} N_*(\Delta^n))$ is the morphism obtained from $\beta^n$ by applying the Leibniz rule in $\mathcal{F}(\overline{\mathcal{C}}\underset{\text{\normalfont H}}{\otimes}\overline{\normalfont\textbf{Sur}}_\mathbb{K}\otimes\Sigma^{-1} N_*(\Delta^n))$.\\

In the following, we endow the sequence of operads $\mathcal{F}(\overline{\mathcal{C}}\underset{\text{\normalfont H}}{\otimes}\overline{\normalfont\textbf{Sur}}_\mathbb{K}\otimes\Sigma^{-1} N_*(\Delta^\bullet))$ with the structure of a cosimplicial set with as coface maps (resp. codegeneracy maps) the coface maps (resp. codegeneracy maps) of the cosimplicial set $\overline{\mathcal{C}}\underset{\text{\normalfont H}}{\otimes}\overline{\normalfont\textbf{Sur}}_\mathbb{K}\otimes\Sigma^{-1} N_*(\Delta^\bullet)$ taken tensor-wise. Recall that the cosimplicial relations are given by the following:
\begin{itemize}
    \item If $i<j$, then $d^jd^i=d^id^{j-1};$
    \item If $i<j$, then $s^jd^i=d^is^{j-1};$
    \item $s^jd^j=s^jd^{j+1}=id;$
    \item If $i>j+1$, then $s^jd^i=d^{i-1}s^j;$
    \item If $i\leq j$, then $s^js^i=s^is^{j+1}.$
\end{itemize}

\noindent Note that we have an extra codegeneracy $s^{-1}:N_*(\Delta^n)\longrightarrow N_*(\Delta^{n-1})$ defined for every $0\leq a_0<\cdots<a_r\leq n$ by $s^{-1}(\underline{a_0\cdots a_r})=\underline{(a_0-1)\cdots (a_r-1)}$, with the convention $s^{-1}(\underline{a_0\cdots a_r})=0$ if $a_0=0$. One can easily check that the above relations are still satisfied with the addition of this degeneracy.

\begin{cons}\label{consbeta}
    We define a sequence of degree $-1$ morphisms $\beta^n:\overline{\mathcal{C}}\underset{\text{\normalfont H}}{\otimes}\overline{{\normalfont \textbf{Sur}}_\mathbb{K}}\otimes\Sigma^{-1} N_*(\Delta^n)\longrightarrow\mathcal{F}(\overline{\mathcal{C}}\underset{\text{\normalfont H}}{\otimes}\overline{{\normalfont \textbf{Sur}}_\mathbb{K}}\otimes\Sigma^{-1} N_*(\Delta^n))$ by induction on $n\geq 0$. Let $\partial^n:\mathcal{F}(\overline{\mathcal{C}}\underset{\text{\normalfont H}}{\otimes}\overline{{\normalfont \textbf{Sur}}_\mathbb{K}}\otimes\Sigma^{-1} N_*(\Delta^n))\longrightarrow\mathcal{F}(\overline{\mathcal{C}}\underset{\text{\normalfont H}}{\otimes}\overline{{\normalfont \textbf{Sur}}_\mathbb{K}}\otimes\otimes\Sigma^{-1} N_*(\Delta^n))$ be the morphism obtained from $\beta^n$ by the Leibniz rule.\\ We let $\beta^0$ to be such that $d_{\overline{\mathcal{C}}}+\partial^0$ is the differential of the cobar construction $B^c(\mathcal{C}\underset{\text{\normalfont H}}{\otimes}{{\normalfont \textbf{Sur}}_\mathbb{K}})\simeq\mathcal{F}(\overline{\mathcal{C}}\underset{\text{\normalfont H}}{\otimes}\overline{{\normalfont \textbf{Sur}}_\mathbb{K}}\otimes \Sigma^{-1}N_*(\Delta^0))$. We set $\beta_{(0)}^n=d_{{\normalfont \textbf{Sur}}_\mathbb{K}}+d_{\Sigma^{-1}N_*(\Delta^n)}$. For every $k\geq 1$, we define the component $\beta_{(k)}^n$ of $\beta^n$ given by trees with $k+1$ vertices by induction on $k$ and $n$. More precisely, we define $\beta_{(k)}^n$ on $F_l\overline{\mathcal{C}_{ns}}\otimes\Sigma^{-1} N_*(\Delta^n)$ by induction on $l\geq 0$. Let $c\in F_l\overline{\mathcal{C}_{ns}}$ and $\underline{x}\in N_*(\Delta^n)$ be a basis element.
    
    \begin{itemize}
        \item If $\underline{x}\neq\underline{0\cdots n}$, let $0\leq i\leq n$ be such that $\underline{x}=\underline{0\cdots (i-1)a_i\cdots a_r}$ with $i<a_i<\cdots<a_r\leq n$. We set
    $$\beta_{(k)}^n(c\otimes \Sigma^{-1}\underline{x})=d^i\beta_{(k)}^{n-1}(c\otimes \Sigma^{-1} s^{i-1}\underline{x});$$

    \item If $\underline{x}=\underline{0\cdots n}$, we set
    \begin{multline*}
    \beta_{(k)}^n(c\otimes\Sigma^{-1}\underline{0\cdots n})=(-1)^{|c|}H_n^{0}d^0\beta_{(k)}^{n-1}(c\otimes \Sigma^{-1}\underline{0\cdots (n-1)})\\ -H_n^0\beta_{(k)}^n(d_{{\normalfont\textbf{Sur}}_\mathbb{K}}(c)\otimes\Sigma^{-1}\underline{0\cdots n})-\sum_{\substack{p+q=k\\ p,q\neq 0}}H_n^{0}\partial_{(p)}^n\beta_{(q)}^n(c\otimes \Sigma^{-1}\underline{0\cdots n}).
    \end{multline*}
    \end{itemize}

    \noindent The morphism $\beta^n_{(k)}$ is then extended on $F_l\overline{\mathcal{C}_{ns}}\otimes\Sigma\otimes\Sigma^{-1} N_*(\Delta^n)$ by symmetry.
\end{cons}

\begin{lm}\label{comsim}
    For every $n,k\geq 0$, we have
    $$\forall 0\leq j\leq n-1, d^j\beta^{n-1}_{(k)}=\beta^{n}_{(k)}d^j;$$
    $$\forall 0\leq j\leq n, s^j\beta^n_{(k)}=\beta^{n-1}_{(k)}s^j,$$

    \noindent where we consider the morphisms $\beta_{(k)}^n$ defined in Construction \ref{consbeta}.
\end{lm}

\begin{proof}
    Since the coface maps and codegeneracy maps preserve the action of the symmetric groups on $\overline{\mathcal{C}}\underset{\text{\normalfont H}}{\otimes}\overline{\textbf{Sur}_\mathbb{K}}\otimes\Sigma^{-1} N_*(\Delta^n)$, it is sufficient to prove it on $F_l\overline{\mathcal{C}_{ns}}\otimes\Sigma^{-1} N_*(\Delta^n)$ for every $l\geq 0$. Let $c\in\overline{\mathcal{C}_{ns}}$ and let $\underline{x}\in N_*(\Delta^n)$ be a basis element. We prove the formulas by induction on $n,k,l\geq 0$. The assertion is obviously true for $n=0$, and for $n\geq 1$ and $k=0$. We now suppose that $n,k\geq 1$.\\
    
    We prove the first line of the lemma. Let $0\leq j\leq n-1$. If $\underline{x}=\underline{0\cdots (n-1)}$, then we indeed have $d^j\beta_{(k)}^{n-1}(c\otimes\Sigma^{-1}\underline{x})=\beta_{(k)}^nd^j(c\otimes\Sigma^{-1}\underline{x})$ by definition of $\beta_{(k)}^n$. Suppose now that $\underline{x}\neq\underline{0\cdots (n-1)}$. Then there exists $0\leq i\leq n-1$ such that $\underline{x}=\underline{0\cdots (i-1)a_i\cdots a_r}$ with $i<a_i<\cdots <a_r\leq n-1$.\\ If $j=i$, then
    \begin{center}
        $\begin{array}{lll}
            \beta_{(k)}^nd^j(c\otimes\Sigma^{-1}\underline{x}) & = &  \beta_{(k)}^n(c\otimes \Sigma^{-1}d^i\underline{x})\\
             & = & d^i\beta_{(k)}^n(c\otimes\Sigma^{-1} s^{i-1}d^i\underline{x})\\
             & = & d^j\beta_{(k)}^n(c\otimes\Sigma^{-1}\underline{x}),
        \end{array}$
    \end{center}

    \noindent since $s^{i-1}d^i=id$. If $j>i$, then
    \begin{center}
        $\begin{array}{lll}
        \beta_{(k)}^nd^j(c\otimes\Sigma^{-1} \underline{x}) & = & \beta_{(k)}^n(c\otimes\Sigma^{-1} d^j\underline{x})\\
        & = & d^i\beta_{(k)}^{n-1}(c\otimes\Sigma^{-1} s^{i-1}d^j\underline{x})\\
        & = & d^i\beta_{(k)}^{n-1}(c\otimes\Sigma^{-1} d^{j-1}s^{i-1}\underline{x}),
        \end{array}$
    \end{center}

    \noindent since $s^{i-1}d^j=d^{j-1}s^{i-1}$. By induction hypothesis on $n-1$, we deduce
    \begin{center}
        $\begin{array}{lll}
        \beta_{(k)}^nd^j(c\otimes\Sigma^{-1} \underline{x}) & = &  d^id^{j-1}\beta_{(k)}^{n-2}(c\otimes\Sigma^{-1} s^{i-1}\underline{x})\\
        & = & d^jd^i\beta_{(k)}^{n-2}(c\otimes \Sigma^{-1}s^{i-1}\underline{x})\\
        & = & d^j\beta_{(k)}^{n-1}(c\otimes\Sigma^{-1}\underline{x}).
        \end{array}$
    \end{center}

    \noindent If $j<i$, then
    \begin{center}
        $\begin{array}{lll}
        \beta_{(k)}^nd^j(c\otimes\Sigma^{-1} \underline{x}) & = &  d^j\beta_{(k)}^{n-1}(c\otimes\Sigma^{-1} s^{j-1}d^j\underline{x})\\
        & = & d^j\beta_{(k)}^{n-1}(c\otimes\Sigma^{-1}\underline{x}),
        \end{array}$
    \end{center}

    \noindent since $s^{j-1}d^j=id$. We thus have proved that $d^j\beta_{(k)}^{n-1}=\beta_{(k)}^nd^j$.\\
    
    We now prove the second line of the lemma. Let $\underline{x}\in N_*(\Delta^n)$. We first consider $\underline{x}=\underline{0\cdots n}$. Then, by definition of $\beta_{(k)}^n$,
    \begin{multline*}
    s^j\beta_{(k)}^n(c\otimes\Sigma^{-1}\underline{0\cdots n})=(-1)^{|c|}s^jH_n^{0}d^0\beta_{(k)}^{n-1}(c\otimes\Sigma^{-1}\underline{0\cdots (n-1)})\\-s^jH_n^0\beta_{(k)}^n(d_{\textbf{Sur}_\mathbb{K}}(c)\otimes\Sigma^{-1}\underline{0\cdots n})-\sum_{\substack{p+q=k\\ p,q\neq 0}}s^jH_n^{0}\partial_{(p)}^n\beta_{(q)}^n(c\otimes\Sigma^{-1}\underline{0\cdots n}).
    \end{multline*}

    \noindent Since $s^jH_n^0=H_{n-1}^0s^j$, we have
    \begin{multline*}
    s^j\beta_{(k)}^n(c\otimes\Sigma^{-1}\underline{0\cdots n})=(-1)^{|c|}H_{n-1}^{0}s^jd^0\beta_{(k)}^{n-1}(c\otimes\Sigma^{-1}\underline{0\cdots (n-1)})\\-H_{n-1}^0s^j\beta_{(k)}^n(d_{\textbf{Sur}_\mathbb{K}}(c)\otimes\Sigma^{-1}\underline{0\cdots n})-\sum_{\substack{p+q=k\\ p,q\neq 0}}H_{n-1}^{0}s^j\partial_{(p)}^n\beta_{(q)}^n(c\otimes\Sigma^{-1}\underline{0\cdots n}).
    \end{multline*}

    \noindent By induction hypothesis, we have that $s^j$ commutes with $\partial_{(p)}^n\beta_{(q)}^n$ for every $p,q\neq 0$ such that $p+q=n$. Since $s^j(\underline{0\cdots n})=0$, we have that the sum in the above identity is $0$. Analogously, by induction hypothesis on $l\geq 0$, we have that $s^j\beta_{(k)}^n(d_{\textbf{Sur}_\mathbb{K}}(c)\otimes\Sigma^{-1}\underline{0\cdots n})=0$. If $j>0$, we have $s^jd^0=d^0 s^{j-1}$ so that the first term is also $0$. If $j=0$, then $s^jd^0=id$. We thus have
    $$s^j\beta_{(k)}^n(c\otimes\Sigma^{-1}\underline{0\cdots n})=(-1)^{|c|}H_{n-1}^0\beta_{(k)}^{n-1}(c\otimes\Sigma^{-1}\underline{0\cdots (n-1)}).$$

    \noindent By definition of $\beta_{(k)}^{n-1}$, the term $\beta_{(k)}^{n-1}(c\otimes\Sigma^{-1}\underline{0\cdots (n-1)})$ is in the image of $H_{n-1}^0$. Since $H_{n-1}^0H_{n-1}^0=0$, we obtain $s^j\beta_{(k)}^n(c\otimes\Sigma^{-1}\underline{0\cdots n})=0$. We thus have proved that $s^j\beta_{(k)}^n(c\otimes\Sigma^{-1}\underline{0\cdots n})=\beta_{(k)}^ns^j(c\otimes\Sigma^{-1}\underline{0\cdots n})=0$. Suppose now that $\underline{x}\neq\underline{0\cdots n}$. Then there exists $0\leq i\leq n$ such that $\underline{x}=\underline{0\dots (i-1)a_i\cdots a_r}$ with $i<a_i<\cdots <a_r\leq n$. We thus have
    $$s^j\beta_{(k)}^n(c\otimes\Sigma^{-1}\underline{x})=s^jd^i\beta_{(k)}^{n-1}(c\otimes \Sigma^{-1}s^{i-1}\underline{x}).$$

    \noindent If $i<j$, then $s^jd^i=d^is^{j-1}$ so that
    $$s^j\beta_{(k)}^n(c\otimes\Sigma^{-1}\underline{x})=d^is^{j-1}\beta_{(k)}^{n-1}(c\otimes\Sigma^{-1} s^{i-1}\underline{x}).$$

    \noindent By induction hypothesis on $n-1$, we obtain
    \begin{center}
        $\begin{array}{lll}
        s^j\beta_{(k)}^n(c\otimes\Sigma^{-1}\underline{x}) & = & d^i\beta_{(k)}^{n-2}(c\otimes\Sigma^{-1} s^{j-1}s^{i-1}\underline{x})\\
        & = & d^i\beta_{(k)}^{n-2}(c\otimes \Sigma^{-1}s^{i-1}s^j\underline{x})\\
        & = & \beta_{(k)}^ns^j(c\otimes\Sigma^{-1}\underline{x}).
        \end{array}$
    \end{center}

    \noindent If $i=j,j+1$, then
    $$s^j\beta_{(k)}^n(c\otimes\Sigma^{-1}\underline{x})=\beta_{(k)}^{n-1}(c\otimes\Sigma^{-1} s^{i-1}\underline{x}).$$

    \noindent Since we have $s^{i-1}\underline{x}=s^i\underline{x}$, in any case, this gives
    $$s^j\beta_{(k)}^n(c\otimes\Sigma^{-1}\underline{x})=\beta_{(k)}^{n-1}s^j(c\otimes\Sigma^{-1}\underline{x}).$$

    \noindent If $i>j+1$, then, by induction hypothesis on $n-1$,
    \begin{center}
        $\begin{array}{lll}
        s^j\beta_{(k)}^n(c\otimes\Sigma^{-1}\underline{x}) & = & d^{i-1}\beta_{(k)}^{n-2}(c\otimes \Sigma^{-1}s^{j}s^{i-1}\underline{x})\\
        & = & d^{i-1}\beta_{(k)}^{n-2}(c\otimes \Sigma^{-1}s^{i-2}s^j\underline{x})\\
        & = & 0\\
        & =& \beta_{(k)}^ns^j(c\otimes\Sigma^{-1}\underline{x}),
        \end{array}$
    \end{center}

    \noindent since $s^j\underline{x}=0$. At the end, we have proved that $s^j\beta_{(k)}^n=\beta_{(k)}^{n-1}s^j$ and thus the lemma.
\end{proof}

\begin{remarque}
    In particular, this lemma implies that
    $$\beta_{(k)}^n(c\otimes\Sigma^{-1}\underline{0\cdots n})=-\sum_{\substack{p+q=k\\ p\neq 0}}H_n^{0}\partial_{(p)}^n\beta_{(q)}^n(c\otimes\Sigma^{-1}\underline{0\cdots n})$$

    \noindent for every $n,k\geq 0$ and $c\in\overline{\mathcal{C}_{ns}}$. Indeed, we have
    \begin{center}
        $\begin{array}{lll}
            \partial_{(k)}^n\beta_{(0)}^n(c\otimes\Sigma^{-1}\underline{0\cdots n}) & = & \displaystyle-(-1)^{|c|}\sum_{i=0}^n(-1)^i\beta_{(k)}^n(c\otimes\Sigma^{-1}\underline{0\cdots\widehat{i}\cdots n)}\\
            & & +\beta_{(k)}^n(d_{{\normalfont \textbf{Sur}}_\mathbb{K}}(c)\otimes\Sigma^{-1}\underline{0\cdots n}) \\
            & = & \displaystyle-(-1)^{|c|}\sum_{i=0}^n(-1)^id^i\beta_{(k)}^{n-1}(c\otimes\Sigma^{-1}\underline{0\cdots (n-1)})\\
            & & +\beta_{(k)}^n(d_{{\normalfont \textbf{Sur}}_\mathbb{K}}(c)\otimes\Sigma^{-1}\underline{0\cdots n}).
        \end{array}$
    \end{center}

    \noindent By an immediate computation, if $i\neq 0$, then $ d^iH_{n-1}^0=H_n^0d^i$. Since $\beta_{(k)}^{n-1}(c\otimes\Sigma^{-1}\underline{0\cdots (n-1)})$ is in the image of $H_{n-1}^0$ by construction, and that $H_{n}^0H_{n}^0=0$, we obtain
    $$H_n^0\partial_{(k)}^n\beta_{(0)}^n(c\otimes\Sigma^{-1}\underline{0\cdots n})=-(-1)^{|c|}H_n^0d^0\beta_{(k)}^{n-1}(c\otimes\Sigma^{-1}\underline{0\cdots (n-1)})+H_n^0\beta_{(k)}^n(d_{{\normalfont \textbf{Sur}}_\mathbb{K}}(c)\otimes\Sigma^{-1}\underline{0\cdots n})$$

    \noindent which proves the above formula.
\end{remarque}

\begin{thm}\label{repcossym}
    Let $n\geq 0$. The morphism $\partial^n:\mathcal{F}(\overline{\mathcal{C}}\underset{\text{\normalfont H}}{\otimes}\overline{{\normalfont\textbf{Sur}}_\mathbb{K}}\otimes\Sigma^{-1} N_*(\Delta^n))\longrightarrow \mathcal{F}(\overline{\mathcal{C}}\underset{\text{\normalfont H}}{\otimes}\overline{{\normalfont\textbf{Sur}}_\mathbb{K}}\otimes\Sigma^{-1} N_*(\Delta^n))$ is such that $d_{\mathcal{C}}+\partial^n$ is a derivation of operads. Moreover, the sequence $\partial^\bullet:\mathcal{F}(\overline{\mathcal{C}}\underset{\text{\normalfont H}}{\otimes}\overline{{\normalfont\textbf{Sur}}_\mathbb{K}}\otimes\Sigma^{-1} N_*(\Delta^\bullet))\longrightarrow \mathcal{F}(\overline{\mathcal{C}}\underset{\text{\normalfont H}}{\otimes}\overline{{\normalfont\textbf{Sur}}_\mathbb{K}}\otimes\Sigma^{-1} N_*(\Delta^\bullet))$ is a morphism of cosimplicial sets.
\end{thm}

\begin{proof}
By Lemma \ref{comsim}, the morphisms $\partial^\bullet$ preserve the cosimplicial structure of $\mathcal{F}(\overline{\mathcal{C}}\underset{\text{\normalfont H}}{\otimes}\overline{{\normalfont\textbf{Sur}}_\mathbb{K}}\otimes\Sigma^{-1} N_*(\Delta^\bullet))$. We need to prove that $d_\mathcal{C}+\partial^n$ is a derivation of operads. This is equivalent to prove that
    $$d_\mathcal{C}(\beta^n)+\partial^n\beta^n=0$$

    \noindent for every $n\geq 0$. By an immediate induction, we have $d_\mathcal{C}(\beta^n)=0$. It remains to prove that $\partial^n\beta^n=0$. This is equivalent to prove that
    $$\sum_{p+q=k}\partial^n_{(p)}\beta^n_{(q)}=0$$

    \noindent for every $k\geq 0$. We prove it on $\overline{\mathcal{C}_{ns}}\otimes\Sigma^{-1}N_*(\Delta^n)$ by induction on $n,k\geq 0$, since all the maps are symmetric by construction. It is true for $n=0$, and for $n\geq 1$ and $k=0$. We now suppose that $n\geq 1$ and $k\geq 1$. Let $c\in \overline{\mathcal{C}_{ns}}$ and $\underline{x}\in N_*(\Delta^n)$ be a basis element. If $\underline{x}\neq\underline{0\cdots n}$, then there exists $0\leq i\leq n-1$ and $\underline{y}\in N_*(\Delta^{n-1})$ such that $d^i\underline{y}=\underline{x}$. Since the morphisms $d_\mathcal{C}+\partial^\bullet$ are compatible with the cosimplicial structure of $\mathcal{F}(\overline{\mathcal{C}}\underset{\text{\normalfont H}}{\otimes}\overline{{\normalfont\textbf{Sur}}_\mathbb{K}}\otimes\Sigma^{-1} N_*(\Delta^n))$ by Lemma \ref{comsim}, we have
    $$\sum_{p+q=k}\partial_{(p)}^n\beta_{(q)}^n(c\otimes\Sigma^{-1}\underline{x})=\sum_{p+q=k}d^i\partial_{(p)}^{n-1}\beta_{(q)}^{n-1}(c\otimes\Sigma^{-1}\underline{y})$$
    
    \noindent which is $0$ by induction hypothesis on $n-1$. Suppose now that $\underline{x}=\underline{0\cdots n}$. By using that
$$\partial_{(0)}^n H_n^{0}+H_n^{0}\partial_{(0)}^n=id-\Phi_n^{0},$$

\noindent we have
\begin{multline*}
\partial_{(0)}^n\beta_{(k)}^n(c\otimes\Sigma^{-1}\underline{0\cdots n})=\sum_{\substack{p+q=k\\ p\neq 0}}H_n^{0}\partial_{(0)}^n\partial_{(p)}^n\beta_{(q)}^n(c\otimes\Sigma^{-1}\underline{0\cdots n})\\-\sum_{\substack{p+q=k\\ p\neq 0}}\partial_{(p)}^n\beta_{(q)}^n(c\otimes\Sigma^{-1}\underline{0\cdots n})+\sum_{\substack{p+q=k\\ p\neq 0}}\Phi_n^0\partial_{(p)}^n\beta_{(q)}^n(c\otimes\Sigma^{-1}\underline{0\cdots n}).
\end{multline*}

\noindent By Lemma \ref{comsim}, the morphism $\Phi_n^0$ commutes with the $\partial^n_{(p)}$'s. We thus have that the last sum is $0$, since $\phi_n^0(\underline{0\cdots n})=0$ because $n\geq 1$. Now, we claim that
$$\sum_{\substack{p+q=k\\ p\neq 0}}\partial_{(0)}^n\partial_{(p)}^n\beta_{(q)}^n(c\otimes\Sigma^{-1}\underline{0\cdots n})=0.$$

\noindent We write
$$\sum_{\substack{p+q=k\\ p\neq 0}}\partial_{(0)}^n\partial_{(p)}^n\beta_{(q)}^n(c\otimes\Sigma^{-1}\underline{0\cdots n})=\partial_{(0)}^n\partial_{(k)}^n\beta_{(0)}^n(c\otimes\Sigma^{-1}\underline{0\cdots n})+\sum_{\substack{p+q=k\\ p,q\neq 0}}\partial_{(0)}^n\partial_{(p)}^n\beta_{(q)}^n(c\otimes\Sigma^{-1}\underline{0\cdots n}).$$

\noindent We first deal with the sum at the right hand-side. Since $p<k$, we can use our induction hypothesis on $p$ to obtain
$$\sum_{\substack{p+q=k\\ p,q\neq 0}}\partial_{(0)}^n\partial_{(p)}^n\beta_{(q)}^n(c\otimes\Sigma^{-1}\underline{0\cdots n})=-\sum_{\substack{p+q=k\\ p,q\neq 0}}\sum_{\substack{s+t=p\\ s\neq 0}}\partial_{(s)}^n\partial_{(t)}^n\beta_{(q)}^n(c\otimes\Sigma^{-1}\underline{0\cdots n}).$$

\noindent By a variable substitution, this gives
$$\sum_{\substack{p+q=k\\ p,q\neq 0}}\partial_{(0)}^n\partial_{(p)}^n\beta_{(q)}^n(c\otimes\Sigma^{-1}\underline{0\cdots n})=-\sum_{\substack{s+t=k\\ s,t\neq 0}}\partial_{(s)}^n\left(\sum_{\substack{p+q=t\\ q\neq 0}}\partial_{(p)}^n\beta_{(q)}^n(c\otimes\Sigma^{-1}\underline{0\cdots n})\right).$$

\noindent Now, since applying $\beta_{(0)}^n$ on $c\otimes\Sigma^{-1}\underline{0\cdots n}$ allows us to apply our induction hypothesis, we have
$$\partial_{(0)}^n\partial_{(k)}^n\beta_{(0)}^n(c\otimes\Sigma^{-1}\underline{0\cdots n})=-\sum_{\substack{s+t=k\\ s,t\neq 0}}\partial_{(s)}^{n}\partial_{(t)}^{n}\beta_{(0)}^n(c\otimes\Sigma^{-1}\underline{0\cdots n}).$$

\noindent At the end, we obtain that
$$\sum_{\substack{p+q=k\\ p,q\neq 0}}\partial_{(0)}^n\partial_{(p)}^n\beta_{(q)}^n(c\otimes\Sigma^{-1}\underline{0\cdots n})=-\sum_{\substack{s+t=k\\ s,t\neq 0}}\partial_{(s)}^n\left(\sum_{p+q=t}\partial_{(p)}^n\beta_{(q)}^n(c\otimes\Sigma^{-1}\underline{0\cdots n})\right),$$

\noindent and this last sum is $0$ by induction hypothesis on $t<k$. We thus have proved that
$$\partial_{(0)}^n\beta_{(k)}^n(c\otimes\Sigma^{-1}\underline{0\cdots n})=-\sum_{\substack{p+q=k\\ p\neq 0}}\partial_{(p)}^n\beta_{(q)}^n(c\otimes\Sigma^{-1}\underline{0\cdots n}),$$

\noindent which is equivalent to
$$\sum_{p+q=k}\partial_{(p)}^n\beta_{(q)}^n(c\otimes\Sigma^{-1}\underline{0\cdots n})=0.$$

The theorem is proved.
\end{proof}

\begin{thm}
    Let $\mathcal{C}$ be a symmetric cooperad. Then
    $$B^c(\mathcal{C}\otimes{\normalfont\textbf{Sur}}_\mathbb{K})\otimes\Delta^\bullet:=(\mathcal{F}(\overline{\mathcal{C}}\underset{\text{\normalfont H}}{\otimes}\overline{\normalfont\textbf{Sur}}_\mathbb{K}\otimes \Sigma^{-1}N_*(\Delta^\bullet)),\partial^\bullet)$$

    \noindent where $\partial^\bullet$ is the twisting derivation constructed in Theorem \ref{repcossym} is a cosimplicial frame associated to $B^c(\mathcal{C})$.
\end{thm}

\begin{proof}
    The proof uses the same arguments as Theorem \ref{reperecos}, with the differentials constructed in Theorem \ref{repcossym}.
\end{proof}

\subsection{Computation of $\text{\normalfont Map}_{\Sigma\mathcal{O}p^0}(B^c(\mathcal{C}),\mathcal{P})$}\label{sec:263}

We now describe a mapping space $\text{\normalfont Map}_{\Sigma\mathcal{O}p^0}(B^c(\mathcal{C}\underset{\text{\normalfont H}}{\otimes}\textbf{Sur}_\mathbb{K}),\mathcal{P})$ for some coaugmented connected cooperad $\mathcal{C}$ and for some augmented connected operad $\mathcal{P}$. We recall the following definition.

\begin{defi}
    Let $M\in\Sigma\text{\normalfont Seq}_\mathbb{K}^0$ be an augmented sequence $M\simeq I\oplus\overline{M}$ with differential $d$. The sequence $M$ is an {\normalfont operad up to homotopy} if there exists a derivation of cooperads of the form $d+\partial$ on $\mathcal{F}^c(\Sigma\overline{M})$ with $\partial_{|\Sigma\overline{M}}=0$.
\end{defi}

In this situation, we say that $\partial$ is a \textit{twisting morphism}, and that $d+\partial$ is a \textit{twisted derivation}. Recall that giving such a differential is equivalent to giving a morphism $\beta:\mathcal{F}^c(\Sigma \overline{M})\longrightarrow\Sigma\overline{M}$ such that $\beta_{|\Sigma\overline{M}}=0$ and, if we denote by $\partial$ the morphism obtained from $\beta$ by the Leibniz rule on $\mathcal{F}^c(\Sigma\overline{M})$, then
$$d(\beta)+\beta\partial=0.$$
    
\begin{prop}
    Let $M\in\Sigma\text{\normalfont Seq}_\mathbb{K}^0$ be an operad up to homotopy. Then $\mathcal{L}(M):=\bigoplus_{n\geq 2}M(n)^{\Sigma_n}$ is endowed with the structure of a $\Gamma(\mathcal{P}re\mathcal{L}ie_\infty,-)$-algebra.
\end{prop}

\begin{proof}
    Let $\partial$ be the twisting part of the differential on $\mathcal{F}^c(\Sigma\overline{M})$. We denote by $\beta$ its composite with the projection on $\Sigma\overline{M}$. Let $x,y_1,\ldots,y_n\in\mathcal{L}(M)$ be elements with homogeneous degrees and arities, and $r_1,\ldots,r_n\geq 0$. We let $E$ to be the symmetric sequence spanned by abstract invariant variables $Y_1,\ldots,Y_n, dY_1,\ldots,dY_n$ of the same arities and degrees as $y_1,\ldots,y_n,d(y_1),\ldots,d(y_n)$. Let $\psi: {E}\longrightarrow \Sigma \overline{M}$ be the morphism which sends the $Y_i$'s to the $y_i$'s. This gives a morphism of coalgebras $\psi:\Gamma(\mathcal{L}( E))\longrightarrow\Gamma(\Sigma\mathcal{L}(M))$. Let $r_1,\ldots,r_n\geq 0$ be such that $r:=r_1+\cdots +r_n\neq 0$ (with, for every $1\leq i\leq n$, the assumption that $r_i=1$ if $Y_i$ has an odd degree). We set $x\llbrace\rrbrace=d(x)$ and
    $$x\llbrace y_1,\ldots,y_n\rrbrace_{r_1,\ldots,r_n}=\sum_{\underline{T}\in\mathcal{TM}({r+1})}\beta\left(\underline{T}^\vee\otimes x\otimes\psi\mathcal{O}(Y_1^{r_1} \cdots Y_n^{r_n})\right),$$

    \noindent where we consider the orbit map $\mathcal{O}$ defined in Proposition \ref{orbit} and where, in this sum, we identify every tensor $\underline{T}^\vee\otimes z$ such that $z\notin\bigotimes_{i=1}^{r+1}\Sigma \overline{M}(\text{\normalfont val}_{\underline{T}}(i))$ with $0$. In particular, the above sum is finite.\\
    
    We first note that these operations preserve $\mathcal{L}(M)$. Indeed, the symmetry relations in the cooperad $\mathcal{F}^c(\Sigma\overline{M})$ will only make involve either actions of symmetric groups elements on $x$ and the $Y_i$'s, which are invariant, or actions on the tensors given by $\mathcal{O}(Y_1^{\otimes r_1}\cdots Y_n^{\otimes r_n})$, which is invariant under the action of $\Sigma_{r}$. It remains to prove formulas of Theorem \ref{relationsprelieinfinite}. It is an immediate check that the operations $-\llbrace -,\ldots,-\rrbrace_{r_1,\ldots,r_n}$ satisfy the relations $(i)-(v)$. We now check relation $(vi)$. First, we have
$$\sum_{\underline{T}\in\mathcal{TM}({r+1})}d(\beta(\underline{T}^\vee\otimes x\otimes\psi\mathcal{O}(Y_1^{r_1}\cdots Y_n^{r_n}))) = x\llbrace y_1,\ldots,y_n\rrbrace_{r_1,\ldots,r_n}\llbrace\rrbrace;$$
             $$\sum_{\underline{T}\in\mathcal{TM}({r+1})} \beta(\underline{T}^\vee\otimes d(x)\otimes\psi\mathcal{O}(Y_1^{r_1}\cdots Y_n^{r_n})) = x\llbrace\rrbrace\llbrace y_1,\ldots,y_n\rrbrace_{r_1,\ldots,r_n} ;$$

\noindent and  
\medskip
             \begin{multline*}
                 \sum_{\underline{T}\in\mathcal{TM}({r+1})}\beta(\underline{T}^\vee\otimes x\otimes\psi\mathcal{O}(dY_k\cdot  Y_1^{r_1}\cdots Y_k^{r_k-1}\cdots Y_n^{r_n}))\\ = x\llbrace y_k\llbrace\rrbrace,y_1,\ldots,y_n\rrbrace_{1,r_1,\ldots, r_k-1,\ldots,r_n}.
             \end{multline*}
             \medskip
    
    \noindent for every $1\leq k\leq n$. This gives
    \medskip
    \begin{multline*}
        \sum_{\underline{T}\in\mathcal{TM}({r+1})}d(\beta)(\underline{T}^\vee\otimes x\otimes\psi\mathcal{O}(Y_1^{r_1}\cdots Y_n^{r_n}))\\=x\llbrace y_1,\ldots,y_n\rrbrace_{r_1,\ldots,r_n}\llbrace\rrbrace+x\llbrace\rrbrace\llbrace y_1,\ldots,y_n\rrbrace_{r_1,\ldots,r_n}\\+\sum_{k=1}^n\pm x\llbrace y_k\llbrace\rrbrace,y_1,\ldots,y_n\rrbrace_{1,r_1,\ldots,r_k-1,\ldots,r_n}.
    \end{multline*}
    \medskip
    
    \noindent Let $\Delta:\mathcal{F}^c(\Sigma \overline{M})\longrightarrow\mathcal{F}^c(\mathcal{F}^c(\Sigma \overline{M}))$ be the cooperad morphism induced by the cooperad structure of $\mathcal{F}^c(\Sigma \overline{M})$. Let $\pi_{\Sigma \overline{M}}:\mathcal{F}^c(\Sigma \overline{M})\longrightarrow\Sigma \overline{M}$ be the projection on $\Sigma \overline{M}$. We keep the notation $\partial$ for the morphism defined on $\mathcal{F}^c(\mathcal{F}^c(\Sigma \overline{M}))$ obtained from $\partial$ by the Leibniz rule. Since $\partial$ is compatible with the cooperad structure on $\mathcal{F}^c(\Sigma \overline{M})$, we have the following commutative diagram:
    \medskip
\[\begin{tikzcd}
	{\mathcal{F}^c(\Sigma\overline M)} & {\mathcal{F}^c(\mathcal{F}^c(\Sigma\overline M))} \\
	{\mathcal{F}^c(\Sigma\overline M)} & {\mathcal{F}^c(\mathcal{F}^c(\Sigma\overline M))} \\
	& {\mathcal{F}^c(\Sigma\overline M)}
	\arrow["\Delta", from=1-1, to=1-2]
	\arrow["\partial"', from=1-1, to=2-1]
	\arrow["\partial", from=1-2, to=2-2]
	\arrow["\Delta"', from=2-1, to=2-2]
	\arrow["{=}"', no head, from=2-1, to=3-2]
	\arrow["{\mathcal{F}^c(\pi_{\Sigma\overline M})}", from=2-2, to=3-2]
\end{tikzcd}.\]
\medskip

\noindent Let $\mathcal{F}^c_{(\geq 2)}(\Sigma \overline{M})$ be the sub symmetric sequence of $\mathcal{F}^c(\Sigma \overline{M})$ given by trees with at least $2$ vertices, so that $\mathcal{F}^c(\Sigma \overline{M})\simeq\Sigma \overline{M}\oplus\mathcal{F}^c_{(\geq 2)}(\Sigma \overline{M})$ as a symmetric sequence. We denote by $\mathcal{F}^c(\Sigma \overline{M};\mathcal{F}^c_{(\geq 2)}(\Sigma \overline{M}))$ the sub symmetric sequence of $\mathcal{F}^c(\mathcal{F}^c(\Sigma\overline{M}))$ given by trees with only one vertex in $\mathcal{F}^c_{(\geq 2)}(\Sigma \overline{M})$, and the other in $\Sigma \overline{M}$. Then we have the following commutative diagram:
\medskip
\[\begin{tikzcd}
	{\mathcal{F}^c(\mathcal{F}^c(\Sigma \overline{M}))} & {\mathcal{F}^c(\Sigma \overline{M};\mathcal{F}^c_{(\geq 2)}(\Sigma \overline{M}))} \\
	{\mathcal{F}^c(\mathcal{F}^c(\Sigma \overline{M}))} & {\mathcal{F}^c(\Sigma \overline{M};\mathcal{F}^c(\Sigma \overline{M}))} \\
	{\mathcal{F}^c(\Sigma \overline{M})} & {\mathcal{F}^c(\Sigma \overline{M};\Sigma \overline{M})}
	\arrow[two heads, from=1-1, to=1-2]
	\arrow["\partial"', from=1-1, to=2-1]
	\arrow["{\mathcal{F}^c(\Sigma \overline{M};\partial)}", from=1-2, to=2-2]
	\arrow["{\mathcal{F}^c(\pi_{\Sigma \overline{M}})}"', two heads, from=2-1, to=3-1]
	\arrow["{\mathcal{F}^c(\Sigma \overline{M};\pi_{\Sigma \overline{M}})}", from=2-2, to=3-2]
	\arrow["{\mathcal{F}^c(id_{\Sigma \overline{M}}\oplus id_{\Sigma \overline{M}})}", two heads, from=3-2, to=3-1]
\end{tikzcd}.\]
\medskip

\noindent The above commutative diagrams prove that $\partial:\mathcal{F}^c_{(\geq 2)}(\Sigma \overline{M})\longrightarrow\mathcal{F}^c(\Sigma \overline{M})$ is given by the composite:
\medskip
\[\begin{tikzcd}
	{\partial:\mathcal{F}^c_{(\geq 2)}(\Sigma\overline{M})} & {\mathcal{F}^c(\Sigma\overline{M};\mathcal{F}^c_{(\geq 2)}(\Sigma\overline{M}))} && {\mathcal{F}^c(\Sigma \overline{M};\Sigma\overline{M})} && {\mathcal{F}^c(\Sigma\overline{M})}
	\arrow["{\Delta_{(1)}}", from=1-1, to=1-2]
	\arrow["{\mathcal{F}^c(\Sigma\overline{M};\beta)}", from=1-2, to=1-4]
	\arrow["{\mathcal{F}^c(id_{\Sigma\overline{M}}\oplus id_{\Sigma\overline{M}})}", two heads, from=1-4, to=1-6]
\end{tikzcd}\]
\medskip
    
    \noindent where $\Delta_{(1)}$ is the composite of $\Delta:\mathcal{F}^c_{(\geq 2)}(\Sigma \overline{M})\longrightarrow\mathcal{F}^c(\mathcal{F}^c(\Sigma \overline{M}))$ with the projection on trees with only one vertex in $\mathcal{F}^c_{(\geq 2)}(\Sigma \overline{M})$. By definition, for every $\underline{T}\in\mathcal{TM}({r+1})$, we have
    \medskip
    \begin{multline*}
        \Delta_{(1)}(\underline{T}^\vee\otimes x\otimes\psi\mathcal{O}(Y_1^{r_1}\cdots Y_n^{r_n}))\\=\sum_{\substack{p+q=r\\ p\neq 0}}\sum_{\substack{\underline{U}\in\mathcal{TM}({q+1})\\\underline{V}\in\mathcal{TM}({p+1})\\\underline{U}\bullet_1\underline{V}=\underline{T}}}\sum_{\substack{p_i+q_i=r_i\\ p_1+\cdots+p_n=p\\ q_1+\cdots +q_n=q}}\pm\underline{U}^\vee\otimes (\underline{V}^\vee\otimes x\otimes\psi\mathcal{O}(Y_1^{p_1}\cdots Y_n^{p_n}))\otimes\psi\mathcal{O}(Y_1^{q_1}\cdots Y_n^{q_n})\\
        +\sum_{\substack{p+q=r-1\\ p\neq 0}}\sum_{k=0}^{q}\sum_{\substack{\underline{U}\in\mathcal{TM}({q+1})\\\underline{V}\in\mathcal{TM}({p+1})\\\underline{U}\bullet_k\underline{V}=\underline{T}}}\sum_{\substack{s_i+p_i'+t_i=r_i\\ p_1'+\cdots+p_n'=p+1\\ s_1+\cdots+s_n=k\\ t_1+\cdots+t_n=q-k}}\pm\underline{U}^\vee\otimes x\otimes\psi\mathcal{O}(Y_1^{s_1}\cdots Y_n^{s_n})\\\otimes\left(\underline{V}^\vee\otimes\psi\mathcal{O}(Y_1^{p_1'}\cdots Y_n^{p_n'})\right)\otimes\psi\mathcal{O}(Y_1^{t_1}\cdots Y_n^{t_n}).
    \end{multline*}
    \medskip

    \noindent By some variable substitutions, summing over $\underline{T}\in\mathcal{TM}({r+1})$ gives
    \medskip
    \begin{multline*}
    \sum_{\substack{p+q=r\\ p\neq 0}}\sum_{\substack{p_i+q_i=r_i\\ p_1+\cdots+p_n=p\\ q_1+\cdots +q_n=q}}\sum_{\underline{U}\in\mathcal{TM}({q+1})}\pm\underline{U}^\vee\otimes \left(\sum_{\underline{V}\in\mathcal{TM}({p+1})}\underline{V}^\vee\otimes x\otimes\psi\mathcal{O}(Y_1^{p_1}\cdots Y_n^{p_n})\right)\otimes\psi\mathcal{O}(Y_1^{q_1}\cdots Y_n^{q_n})\\
        +\sum_{\substack{p+q=r-1\\ p\neq 0}}\sum_{\substack{p_i'+q_i=r_i\\ p_1'+\cdots+p_n'=p+1\\ q_1+\cdots+q_n=q}}\sum_{\underline{U}\in\mathcal{TM}({q+1})}\pm\underline{U}^\vee\otimes x\\\otimes\text{\normalfont Sh}\left(\left(\sum_{\underline{V}\in\mathcal{TM}({p+1})}\underline{V}^\vee\otimes \psi\mathcal{O}(Y_1^{p_1'}\cdots Y_n^{p_n'})\right);\psi\mathcal{O}(Y_1^{q_1}\cdots Y_n^{q_n})\right),
    \end{multline*}
    \medskip

    \noindent where $\text{\normalfont Sh}$ is defined analogously as in Definition \ref{defsh}. We now use that
    $$\psi\mathcal{O}(Y_1^{p_1'}\cdots Y_n^{p_n'})=\sum_{k=1}^n\pm y_k\otimes\psi\mathcal{O}(Y_1^{p_1'}\cdots Y_k^{p_k'-1}\cdots Y_n^{p_n'})$$
    
    \noindent for every $p_1',\ldots,p_n'\geq 0$ (if $p_k'=0$ for some $k$, we just remove the corresponding term). For a chosed $1\leq k\leq n$, we set $p_i=p_i'$ for every $i\neq k$ and $p_k=p_k'-1$. Then, applying $\mathcal{F}^c(\Sigma\overline{M};\beta)$, gives
    \medskip
    \begin{multline*}
    \sum_{\substack{p+q=r\\ p\neq 0}}\sum_{\substack{p_i+q_i=r_i\\ p_1+\cdots+p_n=p\\ q_1+\cdots +q_n=q}}\sum_{\underline{U}\in\mathcal{TM}({q+1})}\pm\underline{U}^\vee\otimes \left(x\llbrace y_1,\ldots,y_n\rrbrace_{p_1,\ldots,p_n}\right)\otimes\psi\mathcal{O}(Y_1^{q_1}\cdots Y_n^{q_n})\\
        +\sum_{\substack{p+q=r-1\\ p\neq 0}}\sum_{k=1}^{n}\sum_{\substack{p_i+q_i=r_i, i\neq k\\p_k+q_k=r_k-1\\ p_1+\cdots+p_n=p\\ q_1+\cdots+q_n=q}}\sum_{\underline{U}\in\mathcal{TM}({q+1})}\pm\underline{U}^\vee\otimes x\otimes\text{\normalfont Sh}\left(y_k\llbrace y_1,\ldots,y_n\rrbrace_{p_1,\ldots,p_n};\psi\mathcal{O}(Y_1^{q_1}\cdots Y_n^{q_n})\right).
    \end{multline*}
    \medskip

    \noindent Applying $\beta$ again gives
    \medskip
    \begin{multline*}
    \sum_{\substack{p_i+q_i=r_i\\ p_1+\cdots+p_n\neq 0\\ q_1+\cdots +q_n\neq 0}}\pm x\llbrace y_1,\ldots,y_n\rrbrace_{p_1,\ldots,p_n}\llbrace y_1,\ldots,y_n\rrbrace_{q_1,\ldots,q_n}\\
        +\sum_{k=1}^{n}\sum_{\substack{p_i+q_i=r_i, i\neq k\\p_k+q_k=r_k-1\\ p_1+\cdots+p_n\neq 0}}\pm x\llbrace y_k\llbrace y_1,\ldots,y_n\rrbrace_{p_1,\ldots,p_n},y_1,\ldots,y_n\rrbrace_{1,q_1,\ldots,q_n}.
    \end{multline*}
    \medskip

    \noindent Finally, the equation $d(\beta)+\beta\partial=0$ applied on $\sum_{\underline{T}\in\mathcal{TM}({r+1})}\underline{T}\otimes x\otimes\psi\mathcal{O}(Y_1^{r_1}\cdots Y_n^{r_n})$ gives
    \begin{multline*}
        \sum_{p_i+q_i=r_i}\pm x\llbrace y_1,\ldots,y_n\rrbrace_{p_1,\ldots,p_n}\llbrace y_1,\ldots,y_n\rrbrace_{q_1,\ldots,q_n}\\
+\sum_{k=1}^n\sum_{\substack{p_i+q_i=r_i,i\neq k\\ p_k+q_k=r_k-1}}\pm x\llbrace y_k\llbrace y_1,\ldots,y_n\rrbrace_{p_1,\ldots,p_n},y_1,\ldots,y_n\rrbrace_{1,q_1,\ldots,q_n}=0
    \end{multline*}

    \noindent as desired.
\end{proof}

\begin{cor}\label{gammaprelieinfcomplet}
    Let $M\in\Sigma\text{\normalfont Seq}_\mathbb{K}^0$ be an operad up to homotopy. Then the completion of $\mathcal{L}(\Sigma M)$, which is $\prod_{n\geq 1}\Sigma M(n)^{\Sigma_n}$, is endowed with the structure of a complete $\Gamma\Lambda\mathcal{PL}_\infty$-algebra.
\end{cor}

\begin{proof}
    It is the same proof as for \cite[Corollary 2.18]{moi}.
\end{proof}

We now apply this proposition to $M=\text{\normalfont Hom}({\mathcal{C}}\underset{\text{\normalfont H}}{\otimes}{\normalfont \textbf{Sur}}_\mathbb{K}\otimes N_*(\Delta^n),{\mathcal{P}})$ for every $n\geq 0$, which will give Theorem \ref{theoremH}.

\begin{thm}
    Let $\mathcal{C}$ be a symmetric cooperad and $\mathcal{P}$ be a symmetric augmented operad such that $\mathcal{P}(0)=\mathcal{C}(0)=0$ and $\mathcal{P}(1)=\mathcal{C}(1)=\mathbb{K}$. Then, for every $n\geq 0$, the symmetric sequence $\text{\normalfont Hom}({\mathcal{C}}\underset{\text{\normalfont H}}{\otimes}{\normalfont \textbf{Sur}}_\mathbb{K}\otimes N_*(\Delta^n),{\mathcal{P}})$ is an operad up to homotopy such that the underlying $\widehat{\Gamma\Lambda\mathcal{PL}_\infty}$-algebra structure on $\Sigma\text{\normalfont Hom}_{\Sigma\text{\normalfont Seq}_\mathbb{K}}(\overline{\mathcal{C}}\underset{\text{\normalfont H}}{\otimes}\overline{\normalfont \textbf{Sur}}_\mathbb{K}\otimes N_*(\Delta^n),\overline{\mathcal{P}})$ satisfies
    $$\text{\normalfont Map}^h_{\Sigma\mathcal{O}p^0}(B^c(\mathcal{C}),\mathcal{P})\simeq\mathcal{MC}(\Sigma\text{\normalfont Hom}_{\Sigma\text{\normalfont Seq}_\mathbb{K}}(\overline{\mathcal{C}}\underset{\text{\normalfont H}}{\otimes}\overline{\normalfont \textbf{Sur}}_\mathbb{K}\otimes N_*(\Delta^\bullet),\overline{\mathcal{P}})),$$

    \noindent where we have set $\text{\normalfont Map}^h_{\Sigma\mathcal{O}p^0}(B^c(\mathcal{C}),\mathcal{P})=\text{\normalfont Map}_{\Sigma\mathcal{O}p^0}(B^c(\mathcal{C}\otimes{\normalfont\textbf{Sur}}_\mathbb{K}),\mathcal{P})$
\end{thm}

\begin{proof}
    Let $n\geq 0$. We first note that we have an isomorphism
    $$\Sigma\text{\normalfont Hom}(\overline{\mathcal{C}}\underset{\text{\normalfont H}}{\otimes}\overline{\normalfont\textbf{Sur}}_\mathbb{K}\otimes N_*(\Delta^n),\mathcal{P})\simeq\text{\normalfont Hom}(\overline{\mathcal{C}}\underset{\text{\normalfont H}}{\otimes}\overline{\normalfont\textbf{Sur}}_\mathbb{K}\otimes\Sigma^{-1} N_*(\Delta^n),\mathcal{P}).$$
    
    \noindent We thus need to construct a morphism
    $$\beta:\mathcal{F}^c_{(\geq 2)}(\text{\normalfont Hom}(\overline{\mathcal{C}}\underset{\text{\normalfont H}}{\otimes}\overline{\normalfont \textbf{Sur}}_\mathbb{K}\otimes\Sigma^{-1} N_*(\Delta^n),\overline{\mathcal{P}}))\longrightarrow\text{\normalfont Hom}(\overline{\mathcal{C}}\underset{\text{\normalfont H}}{\otimes}\overline{\normalfont \textbf{Sur}}_\mathbb{K}\otimes\Sigma^{-1} N_*(\Delta^n),\overline{\mathcal{P}})$$
    \noindent such that, if we denote by $d$ the differential induced by the internal differential of $\text{\normalfont Hom}(\overline{\mathcal{C}}\underset{\text{\normalfont H}}{\otimes}\overline{\normalfont \textbf{Sur}}_\mathbb{K}\otimes\Sigma^{-1} N_*(\Delta^n),\overline{\mathcal{P}})$, and if we denote by $\partial:\mathcal{F}^c_{(\geq 2)}(\text{\normalfont Hom}(\overline{\mathcal{C}}\underset{\text{\normalfont H}}{\otimes}\overline{\normalfont \textbf{Sur}}_\mathbb{K}\otimes\Sigma^{-1} N_*(\Delta^n),\overline{\mathcal{P}}))\longrightarrow\mathcal{F}^c(\text{\normalfont Hom}(\overline{\mathcal{C}}\underset{\text{\normalfont H}}{\otimes}\overline{\normalfont \textbf{Sur}}_\mathbb{K}\otimes\Sigma^{-1} N_*(\Delta^n),\overline{\mathcal{P}}))$ the morphism obtained from $\beta$ by the Leibniz rule, then $d(\beta)+\beta\partial=0$.\\
    
    \noindent We set, for every $f_1,\ldots,f_m\in \text{\normalfont Hom}(\overline{\mathcal{C}}\underset{\text{\normalfont H}}{\otimes}\overline{\normalfont \textbf{Sur}}_\mathbb{K}\otimes\Sigma^{-1} N_*(\Delta^n),\overline{\mathcal{P}})$ and $\underline{T}\in\mathcal{TM}(m)$,
    $$\beta(\underline{T}^\vee\otimes f_1\otimes\cdots\otimes f_m)=\gamma_{(\underline{T})}\circ (f_1\otimes\cdots\otimes f_m)\circ\beta^n_{(\underline{T})}$$

    \noindent where $\beta^n_{(\underline{T})}$ is the composite of the morphism $\beta^n$ defined in Construction \ref{consbeta} with the projection on the $\underline{T}$-component. We first note that $d(\beta)=0$, since $d(\beta^n)=0$. Now, note that the $\underline{T}$-component of $\beta\partial(\underline{T}^\vee\otimes f_1\otimes\cdots\otimes f_m)$ is
    $$\sum_{\underline{S}\subset\underline{T}}\gamma_{(\underline{T}/\underline{S})}\circ \left(f_1\otimes\cdots\otimes f_{r(\underline{S})-1}\otimes\left(\gamma_{(\underline{S})}\circ\left(\bigotimes_{i\in V_{\underline{S}}}f_i\right)\circ \beta_{(\underline{S})}^n\right)\otimes\bigotimes_{i\in V_{\underline{T}/\underline{S}}\setminus\{1,\ldots,r(\underline{S})-1\}}f_i\right)\circ\beta_{(\underline{T}/\underline{S})}^n.$$

    \noindent By Lemma \ref{compbulletsym}, this is equal to
    $$-\gamma_{(\underline{T})}\circ (f_1\otimes\cdots\otimes f_m)\circ\left(\sum_{\underline{S}\subset\underline{T}}\beta_{(\underline{T}/\underline{S})}^n\circ_{\underline{S}}\beta_{(\underline{S})}^n\right).$$

    \noindent This terms is $0$, since the sum $\sum_{\underline{S}\subset\underline{T}}\beta_{(\underline{T}/\underline{S})}^n\circ_{\underline{S}}\beta_{(\underline{S})}^n$ is precisely the $\underline{T}$-component of $\partial^n\beta^n$, which is $0$ by the proof of Theorem \ref{repcossym}.\\

    The computation of $\text{\normalfont Map}^h_{\Sigma\mathcal{O}p^0}(B^c(\mathcal{C}),\mathcal{P})$ comes from the construction of $B^c(\mathcal{C}\underset{\text{\normalfont H}}{\otimes}{\normalfont\textbf{Sur}}_\mathbb{K})\otimes\Delta^\bullet$ given by Construction \ref{consbeta} and from the complete $\Gamma(\mathcal{P}re\mathcal{L}ie_\infty,-)$-algebra structure given by Corollary \ref{gammaprelieinfcomplet}.
\end{proof}